\documentclass{amsart}
 \usepackage[top=1in, bottom=1in, left=1.375in, right=1.375in]{geometry}
 \usepackage{amsmath,amssymb}
 \usepackage{algpseudocode}
 \usepackage{algorithm}
 \usepackage{algorithmicx}
 \usepackage{caption}
 \usepackage{subcaption}
 \usepackage{amsthm}
 \numberwithin{equation}{section}
 \usepackage{amsfonts}
 \usepackage{graphicx}

\usepackage{amsmath}
\usepackage{amssymb}
\usepackage{amsthm}
\usepackage{amsmath}
\usepackage{setspace}
\usepackage{xcolor}
\usepackage{fancyhdr}
\usepackage{amssymb}
\usepackage{amsthm}
\usepackage{listings}
    \usepackage{cite}
    \usepackage{tabularx}
    \usepackage{graphicx}
    \usepackage{epstopdf}
    \usepackage{epsfig}
    \usepackage{float}
    \usepackage{appendix}
     \usepackage{hyperref}
 \theoremstyle{plain}
 \newtheorem{thm}{Theorem}[section]
 \newtheorem{prop}[thm]{Proposition}
 \newtheorem{lem}[thm]{Lemma}
 \newtheorem{cor}[thm]{Corollary}
 \theoremstyle{definition}
 \newtheorem{definition}[thm]{Definition}
 
 \theoremstyle{remark}
 \newtheorem{remark}[thm]{Remark}

 \let\pa=\partial
 \let\al=\alpha
 \let\b=\beta
 \let\d=\delta
 \let\g=\gamma
 \let\e=\varepsilon

 \let \kp = \kappa
 \let\lam=\lambda
 
 \let\s=\sigma
 \let\f=\frac
 \let \les = \lesssim
 \let\om=\omega
 
 \let \th = \theta
 \let \vth=\vartheta
 \let \vp = \varphi
 \let\G= \Gamma
\let\B = \Big
 \let\D=\Delta

 \let\Om=\Omega
 \let\td = \tilde
 \let\wt=\widetilde
 \let\wh=\widehat

 \let\teq \triangleq
 
 \let\pa=\partial

 \def\cA{{\mathcal A}}
 
 \def\cC{{\mathcal C}}

 \def\cH{{\mathcal H}}

 \def\cL{{\mathcal L}}

 \def\cP{{\mathcal P}}
 
 \def\cR{{\mathcal R}}
 
 \def\cT{{\mathcal T}}

 \def\cW{{\mathcal W}}

 \def\na{\nabla}
 \def\la{\langle}
 \def\ra{\rangle}
\def\lt{\left}
\def\rt{\right}
\def\one{\mathbf{1}}

 \newcommand{\beq}{\begin{equation}}
 \newcommand{\eeq}{\end{equation}}
  \newcommand{\bal}{\begin{aligned} }
  \newcommand{\eal}{\end{aligned}}
 \newcommand{\ben}{\begin{eqnarray}}
 \newcommand{\een}{\end{eqnarray}}
 \newcommand{\beno}{\begin{eqnarray*}}
 \newcommand{\eeno}{\end{eqnarray*}}

 \newcommand{\ee}{\mathbf{e}}

 \newcommand{\uu}{\mathbf{u}}

 \newcommand{\xx}{\mathbf{x}}

 \newcommand{\RR}{\mathbf{R}}
 \newcommand{\R}{\mathbb{R}}

  \newcommand{\BT}{\mathbb{T}}
  \newcommand{\BZ}{\mathbb{Z}}

 \newcommand{\supp}{\mathrm{supp}}

 \author{Jiajie Chen and Thomas Y. Hou}
 \address{Applied and Computational Mathematics, Caltech, Pasadena, CA 91125. Emails: jchen@caltech.edu, hou@cms.caltech.edu}
 \date{\today}

\title[Finite time blowup of 2D Boussinesq equations]{Finite time blowup of 2D Boussinesq 
and 3D Euler equations with $C^{1,\alpha}$ velocity and boundary}
 
\begin{document}
 \begin{abstract}
 Inspired by the numerical evidence of a potential 3D Euler singularity by Luo-Hou \cite{luo2013potentially-1,luo2013potentially-2} and the recent breakthrough by Elgindi \cite{elgindi2019finite} on the singularity formation of the 3D Euler equation without swirl with $C^{1,\alpha}$ initial data for the velocity, we prove the finite time singularity for the 2D Boussinesq and the 3D axisymmetric Euler equations in the presence of boundary with $C^{1,\alpha}$ initial data for the velocity (and density in the case of Boussinesq equations). Our finite time blowup solution for the 3D Euler equations and the singular solution considered in \cite{luo2013potentially-1,luo2013potentially-2} share many essential features, including the symmetry properties of the solution, the flow structure, and the sign of the solution in each quadrant, except that we use $C^{1,\alpha}$ initial data for the velocity field. We use a dynamic rescaling formulation and follow the general framework of analysis developed by Elgindi in \cite{elgindi2019finite}. We also use some strategy proposed in our recent joint work with Huang in \cite{chen2019finite} and adopt several methods of analysis in \cite{elgindi2019finite}
to establish the linear and nonlinear stability of an approximate self-similar profile. The nonlinear stability enables us to prove that the solution of the 3D Euler equations or the 2D Boussinesq equations with $C^{1,\alpha}$ initial data will develop a finite time singularity. Moreover, the velocity field has finite energy before the singularity time.

In the previous version of this paper, we proved the blowup results for the 3D axisymmetric Euler equations with initial data $(u_0^{\theta})^2,  u_0^r, u_0^z \in C^{1,\alpha}$ and $\omega_0^{\theta} \in C^{\alpha}$. Though the velocity $u^r, u^z$ in the axisymmetric setting is $C^{1,\alpha}$, our interpretation that the velocity is $C^{1,\alpha}$ is not correct since the velocity in 3D 
$u = u^r e_r + u^z e_z + u^{\theta} e_{\theta}$ also depends on $u^{\theta}$, which is not $C^{1,\alpha}$. This oversight can be fixed easily with a minor change in the construction of the approximate steady state and a minor modification to localize the approximate steady state. In particular, the proof of the blowup for the Boussinesq equations does not require this change and 
holds true both with or without this change, and this modification does not affect the nonlinear stability estimates of the 3D Euler equations. 
\end{abstract}

 \maketitle

\section{Introduction}
The three-dimensional (3D) incompressible Euler equations in fluid dynamics describe the motion of ideal incompressible flows. It has been used to model ocean currents, weather patterns, and other fluids related phenomena. Despite their wide range of applications, the question regarding the global regularity of the 3D Euler equations has remained open. The interested readers may consult the excellent surveys \cite{bt2007,constantin2007euler,gibbon2008three,hou2009blow,kiselev2018,
majda2002vorticity} and the references therein. 
The main difficulty associated with the regularity properties of the 3D Euler equations is due to the presence of vortex stretching, which is absent in the 2D Euler equations. To better illustrate this difficulty, we consider the so-called vorticity-stream function formulation:
\begin{equation}
  \omega_{t} + u \cdot \nabla \omega = \omega \cdot \nabla u,
  \label{eqn_eu_w}
\end{equation}
where $\omega = \nabla \times u$ is the \emph{vorticity vector} of the fluid, and $u$ is related to $\omega$ via the \emph{Biot-Savart law}. Under some decay conditions in the far field, one can show that $\omega$ satisfies the property 
\begin{equation}
  \|\omega\|_{L^{p}} \leq \|\nabla u\|_{L^{p}} \leq C_{p} \|\omega\|_{L^{p}},\qquad 1 < p < \infty.
  \label{eqn_du_Lp}
\end{equation}
Thus, the vortex stretching term $\omega \cdot \nabla u$ formally scales  like $\omega^{2}$. If such nonlinear alignment persists in time, the 3D Euler equations may develop a finite-time singularity. However, due to the nonlocal nature of the vortex stretching term, such nonlinear alignment may deplete itself dynamically (see e.g. \cite{hou2006dynamic}). Despite considerable efforts, whether the 3D Euler equations with smooth initial data of finite energy can develop a finite time singularity has been one of the most outstanding open questions in nonlinear partial differential equations. 

In \cite{luo2013potentially-1,luo2013potentially-2}, Luo and Hou presented some convincing numerical evidence that the 3D axisymmetric Euler equations with a solid boundary develop a potential finite time singularity for a class of smooth initial data with finite energy. 
The presence of the boundary and the odd-even symmetry of the solution along the axial direction play an important role in generating a stable and sustainable finite time singularity. 
The singularity scenario reported in\cite{luo2013potentially-1,luo2013potentially-2} has generated great interests and has inspired a number of subsequent developments, see e.g. \cite{kiselev2013small,chklsy2017,kryz2016} and the excellent survey article \cite{kiselev2018}.

Despite all the previous efforts, there is still lack of theoretical justification of the finite time singularity for the 3D axisymmetric Euler equations reported in \cite{luo2013potentially-1,luo2013potentially-2}. Very recently, Elgindi made a breakthrough on the 3D Euler equation singularity \cite{elgindi2019finite} by constructing the self-similar blowup solutions to the 3D axisymmetric Euler equations with $C^{1,\alpha}$ velocity and without swirl.

\subsection{Main results}\label{sec:thm}
 In this paper, inspired by the computation of Hou-Luo \cite{luo2013potentially-1,luo2013potentially-2} and Elgindi's work \cite{elgindi2019finite}, we study the singularity formation of 3D axisymmetric Euler equations and the 2D Boussinesq equations with boundary.
Since the singularity of the 3D axisymmetric Euler equations reported in \cite{luo2013potentially-1,luo2013potentially-2} occurs at the boundary, away from the symmetry axis, 
it is well known that the 3D axisymmetric Euler equations are similar to the 2D Boussinesq equations \cite{majda2002vorticity}.
Thus, it makes sense to investigate the finite time singularity of the 2D Boussinesq equations.

The main results of this paper are summarized by the following two theorems. In our first main result, we prove finite time blowup of the Boussinesq equations with $C^{1,\al}$ initial data for the velocity field and the density.
\begin{thm}\label{thm:bous}
Let $\omega$ be the vorticity and $\theta$ be the density in the 2D Boussinesq equations described by \eqref{eq:bous1}-\eqref{eq:biot}.
There exists $\al_0 > 0$ such that for $0 < \al < \al_0$, the unique local solution of the 2D Boussinesq equations in the upper half plane develops a focusing asymptotically self-similar singularity in finite time for some initial data $\om \in C_c^{\al}(\R_+^2) , \th \in C_c^{1, \al}(\R_+^2)$. In particular, the velocity field is $C^{1, \al}$ with finite energy. Moreover, the self-similar profile $(\om_{\infty}, \th_{\infty})$ satisfies $\om_{\infty}, \na \th_{\infty} \in C^{\f{\al}{40}}$.


\end{thm}

By asymptotically self-similar, we mean that the solution in the dynamic rescaling equations (see Definition in Section \ref{sec:dsform}) converges to the self-similar profile in a suitable norm. We will specify the norm in the convergence in Section \ref{sec:converge}.


In our second result, we prove the finite time singularity formation for the 3D axisymmetric Euler equations with large swirl in a cylinder $D = \{ (r, z) : r \leq 1,  z \in \BT\}$ that is periodic in $z$ (axial direction) with period $2$, where $r$ is the radial variable and $\BT = \R / (2 \BZ)$.

\begin{thm}\label{thm:euler}
Consider the 3D axisymmetric Euler equations in the cylinder $r,z \in [0, 1] \times \BT$. Let $\om^{\th}$ be the angular vorticity and $u^{\th}$ be the angular velocity.
There exists $\al_0 > 0$ such that for $0 < \al < \al_0$, the unique local solution of the 3D axisymmetric Euler equations given by \eqref{eq:euler1}-\eqref{eq:euler21} develops a 
singularity in finite time for some initial data 
$\om^{\th} \in C^{\al}(D) , u^{\th} \in C^{1, \al}(D)$ supported away from the axis $r=0$ with $u^{\th} \geq 0$. Moreover, the velocity field in each period has finite energy, $u^{\th}$ is even in $z$, and $\om^{\th}$ is odd in $z$.
\end{thm}

Our analysis shows that the singular solution in Theorem \ref{thm:euler} in the dynamic rescaling formulation remains very close to an approximate blowup profile in some norm (see Section \ref{sec:euler}) for all time (or up to the blowup time in the original formulation). It is conceivable that it converges to a self-similar blowup profile of 2D Boussinesq at the blowup time so that the blowup solution is asymptotically self-similar. However, we cannot prove this result using the current analysis since the domain $D$ is not invariant under dilation. We leave it to our future work.

\begin{remark}
In \cite{chen2019finite2} and the previous arXiv version of this paper, we prove Theorem \ref{thm:euler} with initial data $\om_0^{\th}, (u_0^{\th})^2 \in C^{1, \al}$. In the axisymmetric Euler equations \eqref{eq:euler1}, $u^r, u^z$ play the essential role as velocity and are in $C^{1,\al}$, and $u^{\th}$ does not appear in the advection term.
Yet, our previous interpretation that the velocity is $C^{1,\al}$ is not correct since the velocity vector $\uu(x) = u^r(r, z) e_r + u^z(r, z) e_z + u^{\th}(r, z) e_{\th} $ in 3D is not $C^{1,\al}$. Note that the quantity used in our stability estimates is $(u^{\th})^2$, not $u^{\th}$. The quantity $A(r, z) = (u_0^{\th}(r,z))^2 \in C^{1,\al}$ has a vanishing rate $|z|^{1+\al}$ near $z=0$, and thus $u_0^{\th}(r, z) = A^{1/2}$ has a vanishing rate $|z|^{(1+\al)/2}$ only near $z=0$ and is not in $C^{1,\al}$. This oversight can be fixed easily by modifying the construction of the approximate steady state $\bar \th$ as follows 
\[
\bar \th = 1 + \int_0^x \bar \th_x(z, y) dz ,
\]
where $\bar \th_x = \bar \eta$ has an explicit formula \eqref{eq:profile}. In \cite{chen2019finite2} and the previous arXiv version, we use $\bar \th(0, y) = 0$ and do not have the factor $1$.  Note that in the whole analysis, we estimate $\na \bar \th$ and do not use the variable $ \bar \th$ directly.
The constant $1$ does not change $\na \bar \th$. 

We remark that the analysis of the Boussinesq equations and the proof of Theorem \ref{thm:bous} do not require this change and hold true both with or without this change. Moreover, this modification does not affect the nonlinear stability estimates of the 3D Euler equations. It only leads to a slightly modified perturbation to localize the approximate steady state in \eqref{eq:profile2}, which remains small in the energy norm. See Lemma \ref{lem:small}.  With this modification, since $\bar \th \in C^{1,\al}$ is even and $\bar \th \geq 1$, we get $\bar \th^{1/2} \in C^{1,\al}$. The quantity $A(r, z)$ defined above relates directly to $\bar \th$ and becomes $ 1 + C(r)|z|^{1+\al}$ near $z= 0$, and thus $A^{1/2} \in C^{1,\al}$, which  implies $u_0^{\th} = A^{1/2} \in C^{1,\al}$. See Section \ref{subsec:non_blowup} for the discussion. Thus, we can prove the blowup results in Theorem \ref{thm:euler} with initial data $\om_0^{\th} \in C^{1,\al}, u_0^{\th} \in C^{1,\al}$. In particular, we have $\uu_0 \in C^{1,\al}$. 
\end{remark}

\subsection{Main ingredients in our analysis}\label{sec:ideas}

Our analysis follows the general framework developed by Elgindi in \cite{elgindi2019finite}. We use the Boussinesq equations to illustrate the main ideas in our analysis. 

As in our previous work \cite{chen2019finite}, we reformulate the equations using an equivalent dynamic rescaling formulation (see e.g. \cite{mclaughlin1986focusing,landman1988rate}). We follow \cite{elgindi2019finite} to derive the leading order system. In the derivation, we have used the argument in \cite{elgindi2019finite} to obtain the leading order approximation of the stream function for small $\alpha$. Moreover, as observed by Elgindi and Jeong in \cite{Elg17} (see also \cite{elgindi2019finite}), the advection terms are relatively small compared with the nonlinear vortex stretching term when we work with $C^\alpha$ solution with small $\al$ for vorticity or $\nabla \theta$, which vanishes weakly near the origin, e.g. $|x|^{\al}$. In the 2D Boussinesq equations \eqref{eq:bous1}-\eqref{eq:bous2}, the vortex stretching term for the $\omega$ equation is given by $\theta_x$. Within the above $C^{\al}$ class of solution, the transport term $ \uu \cdot \na \om $ may not be smaller than $\th_x$. For example, one can choose $\om, \th$ so that $\uu \cdot \na \om = O(1)$ and $\th_x = O(1)$. We further look for solutions of the 2D Boussinesq equation \eqref{eq:bous1}-\eqref{eq:bous2} by letting $ \omega = \alpha \tilde{\omega}, \; \theta = \alpha \tilde{\theta}$ with $\tilde{\omega} = O(1)$ and $\tilde{\theta} = O(1)$ as $\alpha \rightarrow 0$. 
Formally, the nonlinear transport term $\uu \cdot \na \om$ becomes relatively small compared with $\th_x$ due to the weakening effect of advection for $C^{\al}$ data and the weak nonlinear effect due to the fact that $\om = O(\al)$ and $\th_x = O(\al)$ for small $\al$ at a given time. Thus, we can ignore the contributions from the advection terms for small $\alpha$ when we work with this class of $\omega$ and $\theta$. See more discussion in Section \ref{sec:decoup}. In addition, inspired by our own computation of the Hou-Luo singularity scenario \cite{luo2013potentially-1,luo2013potentially-2}, we look for $\th$ that is anisotropic in the sense that $\th_y$ is small compared with $\theta_x$. We will justify that this property is preserved dynamically for our singular solution. As a result, we can decouple the $\theta_y$ equation from the leading order equations for $\omega$ and $\theta_x$. This gives rise to a leading order coupled system of Riccati type for $\omega$ and $\theta_x$, which is similar to the scalar leading order equation obtained in \cite{elgindi2019finite}.
Inspired by the solution structure of the leading order system in \cite{elgindi2019finite}, we are able to find a class of closed form solutions of this leading order system.

The most essential part of our analysis is to establish linear stability of the approximate steady state using the dynamic rescaling equations. As in \cite{elgindi2019finite} and our previous work with Huang \cite{chen2019finite}, we design some singular weights to extract the damping effect from the linearized operator around the approximate steady state. 
In order for the perturbation from the approximate steady state to be well defined in the weighted norm with a more singular weight, we impose some vanishing conditions on the perturbation at the origin by choosing some normalization conditions. This leads to some nonlocal terms related to the scaling parameter $c_\omega$ in the linearized equations, which are not present in \cite{elgindi2019finite}. 

Compared with the scalar linearized equation considered in \cite{elgindi2019finite},
the linearized equations for the 2D Boussinesq equations lead to a more complicated coupled system and we need to deal with a few more nonlocal terms that are of $O(1)$ as $\al \rightarrow 0$. Thus we cannot apply the coercivity estimate of the linearized operator in \cite{elgindi2019finite}, which is one of the key steps in constructing the self-similar solution in \cite{elgindi2019finite}. 
One of the main difficulties in our linear stability analysis is to control the nonlocal terms. If we use a standard energy estimate to handle these nonlocal terms, we will over-estimate their contributions to the linearized equations and would not be able to obtain the desired linear stability result. Since the damping term has a relatively small coefficient, we need to exploit the coupling structure in the system and take into account the cancellation among different nonlocal interaction terms in order to obtain linear stability. For this purpose, we design our singular weights that are adapted to the approximate self-similar profile and contain different powers of $R^{-k}$ to account the interaction in the near field, the intermediate field and the far field. 
To control the nonlocal scaling parameter $c_\omega$, we will derive a separate ODE for $c_\omega$, which captures the damping effect of $c_{\om}$.

We have used the elliptic estimate and several nonlinear estimates from \cite{elgindi2019finite} in our nonlinear stability analysis. The presence of swirl (the angular velocity $u^\theta$) or density ($\theta$) introduces additional technical difficulties. Since the approximate steady state for $\na \th$ does not decay in certain direction, we need to design different weighted Sobolev spaces carefully for different derivatives and further develop several nonlinear estimates. To obtain the $L^{\infty}$ estimate of a directional derivative of $\th$, which is necessary to close the nonlinear stability analysis, we make use of the hyperbolic flow structure. Once we obtain nonlinear stability, as in \cite{chen2019finite}, we establish finite time blowup from a class of compactly supported initial data $\omega_0$ and $\theta_0$ with finite energy by truncating the approximate steady state and using a rescaling argument.
We further establish convergence of the solution of the dynamic rescaling equations to the self-similar profile using a time-differentiation argument. This argument has also been used in our recent joint work with Huang in \cite{chen2019finite} and developed independently in \cite{elgindi2019finite}.

\vspace{-0.1in}
\subsection{ From the 2D Boussinesq to the 3D Euler equations}\label{sec:idea_euler}

For the 3D Euler equations, we consider the domain within one period, i.e.
$D_1 =\{ (r,z) : r \in [0,1], |z| \leq 1 \} $. 
We will construct a singular solution that is supported near $r=1, z=0$ up to blowup time and blows up at $r=1, z=0$. Since the support is away from the symmetry axis, we show that the 3D Euler equations are essentially the same as the 2D Boussinesq equations up to some lower order terms. This connection is well known; see e.g. \cite{majda2002vorticity}. Then  we generalize the proof of Theorem \ref{thm:bous} to prove Theorem \ref{thm:euler}.
 To justify this connection rigorously, we need two steps. The first step is to establish the elliptic estimates in the new domain. The second step is to control the support of the solution and show that it remains close to $r=1, z=0$ up to the blowup time.


\subsubsection{Control of the support }\label{sec:idea_supp}

The reason that the support of the singular solution remains close to $(r, z)=(1,0)$ is due to the
following properties of the singular solution. Firstly, the singular solution is focusing, which is characterized by the rescaling parameters $c_l(\tau) > \f{1}{2\al}$ for all $\tau > 0$. See the definition of $c_l$ in Section \ref{sec:dsform}. Secondly, the velocity in the dynamic rescaling formulation has sublinear growth in the support of the solution. These properties hold for the singular solution of the 2D Boussinesq equations. We prove that they remain true for the 3D Euler in Section \ref{sec:euler}. Using these properties, we derive an ODE to control the size of the support and show that it remains small up to the blowup time. See more discussion in Section \ref{sec:grow_supp}. Similar ideas and estimates to control the support have been used in \cite{chen2019finite} to generalize the singularity formation of the De Gregorio type model from the real line to a circle.

\subsubsection{The elliptic estimates}\label{sec:idea_biot}

The elliptic equation for the stream function $\td \psi(r, z)$ in $D_1$ reads 
\beq\label{eq:idea_euler2}
\cL \td \psi \teq -(\pa_{rr} + \f{1}{r} \pa_{r} +\pa_{zz}) \td{\psi} + \f{1}{r^2} \td{\psi} = \om^{\th},
\eeq
where $\om^{\th}$ is the angular vorticity. We impose the periodic boundary condition in $z$ and a no-flow boundary condition on $r=1: \td \psi(1, z) = 0$. See \cite{luo2013potentially-2,majda2002vorticity}. Since the solution is supported near $r=1, z=0$, we will only use $\td \psi(r,z)$ for $(r,z)$ near $(1,0)$ in our analysis. In this case, $r^{-1} \approx 1$ and the term $-\f{1}{r} \pa_r \td \psi + \f{1}{r^2} \td \psi$ in $\cL \td \psi$ is of lower order compared with $\partial_{rr} \tilde{\psi} + \partial_{zz} \tilde{\psi}$. In the dynamic rescaling equations, we obtain a small factor $C_l(\tau)$ for the term $-\f{1}{r} \pa_r \td \psi + \f{1}{r^2} \td \psi$ and treat it as a perturbation in $\cL \td \psi$. Moreover, if we relabel the variables $(r,z)$ as $(y, x)$ in $\R^2$, we formally have $\cL \td \psi \approx -\D_{2D} \td \psi$. 
In Section \ref{sec:3Delli}, we will justify this connection rigorously and then generalize the elliptic estimates that we obtain for the 2D Boussinesq to the 3D Euler equations.

\subsection{Connections to the Hou-Luo scenario}

Many settings of our problem are similar to those considered in \cite{luo2013potentially-1,luo2013potentially-2}. See more discussions after Lemma \ref{lem:selfsim}. The driving mechanism for the finite time singularity that we consider in this paper is essentially the same as that for the 3D axisymmetric Euler equations with solid boundary considered in \cite{luo2013potentially-1,luo2013potentially-2}. In both cases, the swirl (the angular velocity $u^\theta$ ) and the boundary play an essential role in generating a sustainable finite time singularity. It is the strong compression of the angular velocity $u^\theta$ toward the symmetry plane $z=0$ along the axial ($z$) direction on the boundary $r=1$ that creates a large gradient in $u^\theta$. Then the nonlinear forcing term $\partial_z(u^\theta)^2$ induces a rapid growth in the angular vorticity $\omega^\theta$, ultimately leading to a finite time blowup. 
Moreover, the singularities 
that we consider occur at the solid boundary, which are the same as the one reported in \cite{luo2013potentially-1,luo2013potentially-2}. 

We would like to emphasize that the presence of boundary plays a crucial role in the singularity formation even with $C^{1,\alpha}$ initial data for the velocity and $\theta$. If we remove the boundary, a promising potential blowup scenario for the 2D Boussinesq equation is to have a hyperbolic flow structure near the origin with 4-fold symmetry for $\theta$, i.e. $\th$ is odd in $y$ and even in $x$. Similar scenario has been used in \cite{zlatovs2015exponential}. Since $\theta (x,y)$ is odd with respect to $y$ and $\theta \in C^{1,\al}$, a typical $\theta$ is of the form: $\theta(x,y) \approx c_1 \alpha x^{1+\alpha} y + l.o.t., c_1 \neq 0$ near the origin. From our derivation and analysis of the leading order system, it is the nonlinear coupling between $\omega$ and $\theta_x$ that generates the blow-up mechanism. However, without the boundary, $\theta_x \approx c_1 \alpha (1+\alpha) x^\alpha y + l.o.t.$ and it does not vanish to the order $O(|y|^{\kp})$ near $y=0$ with a small exponent $\kp>0$. The advection of $\theta_x$ along the $y$ direction is not small compared with the vortex stretching term $-u_x \th_x$ in the $\th_x$ equation \eqref{eq:bous21}. Thus, we can no longer neglect the contribution from the $y$ advection term and we cannot derive our leading order system in this case. In fact, the transport of $\theta_x$ along the $y$ direction provides a strong destabilizing effect to the singularity formation and would likely destroy the self-similar focusing blowup mechanism \cite{hou2008dynamic,lei2009stabilizing}. 
 
If we approximate the velocity field $(u,v)$ by $(xu_x(0,0,t), yv_y(0,0,t))$ (note that $u_x(0,0,t)+v_y(0,0,t)=0$) as was done in a toy model introduced in \cite{elgindi2019finite}, we have the following result. For any $\om_0, \na \th_0 \in C_c^{\al}(\R^2)$, which is in the local well-posedness class for the 2D Boussinesq equations \cite{chae_kim_nam_1999}, under the 4-fold symmetry assumption, the solution of the toy model exists globally. The key point is that due to the odd symmetry of $\theta_0$ with respect to $y$ and the assumption that $\th_0 \in C^{1,\al}$, $\th_0$ must vanish linearly in $y$, i.e. $|\th(x, y)| \les |y|$. The proof follows an estimate similar to that presented in \cite{elgindi2019finite} and we defer it to Appendix \ref{app:toy}.

In the presence of the boundary ($y=0$), $\th$ can be nonzero on $y=0$, which removes the above constraint $|\th(x, y)| \les |y|$. Then we can further weaken the transport terms in the 2D Boussinesq as discussed in Section \ref{sec:ideas}. Although the leading order system for the 2D Boussinesq equations and the 3D Euler equations with $C^{1,\alpha}$ initial velocity and the boundary looks qualitatively similar to that for the 3D Euler equations without swirl and without boundary obtained in \cite{elgindi2019finite}, the physical driving mechanisms of the finite time singularity behind these two blowup scenarios are quite different. In our case, the swirl and the boundary play a crucial role. 
Our numerical study suggests that even for smooth initial data, $\theta_x$ is an order of magnitude larger than $\theta_y$ and the effect of advection is relatively weak compared with the vortex stretching term. More importantly, our computation reveals that the real parts of the eigenvalues of the discretized linear operator are all negative and bounded away from zero by a finite spectral gap. See also Section 3.4 in \cite{liu2017spatial} for an illustration of the eigenvalue distribution of the discretized linearized operator. This is a strong evidence that the linearized operator should be stable even for smooth initial data. The essential step in proving this rigorously is the linear stability analysis, which requires us to estimate the Biot-Savart law without the availability of the leading order structure for $C^{1,\al}$ velocity and control a few more nonlocal terms that we can neglect using the $C^{1,\alpha}$ initial data. In some sense, our blowup analysis for $C^{1,\alpha}$ initial data captures certain essential features of the Hou-Luo scenario \cite{luo2013potentially-1,luo2013potentially-2} and some essential difficulties in the analysis of such scenario.

\subsection{Review of other related works}

In the recent works \cite{elgindi2017finite,elgindi2018finite}, Elgindi and Jeong proved finite time singularity formation for the 2D Boussinesq and 3D axisymmetric equations in a physical domain with a corner and $\mathring{C}^{0, \al}$ data. 
The domain we study in this paper does not have a corner. In the case of the 3D Euler equations, our physical domain includes the symmetry axis. 
In comparison, the domain studied in \cite{elgindi2018finite} does not include the symmetry axis.

In \cite{HOANG20187328,KISELEV201834}, the authors studied a modified 2D Boussinesq equations with $\th_x$ in \eqref{eq:bous1} replaced by $\th / x$ and using a simplified Biot-Savart law. 
In these works, the simplified Biot-Savart law has a positive kernel and the authors have been able to prove finite time blowup for smooth initial data using a functional argument. 

After we completed our work, we learned from Dr. Elgindi that the stability of the self-similar blowup solutions in \cite{elgindi2019finite} and the construction of finite-energy $C^{1,\alpha}$ solutions that become singular in finite time have been established recently in \cite{elgindi2019stability}.

\vspace{0.1in}
\paragraph{\bf{Organization of the paper}}
In Sections \ref{sec:derive}-\ref{sec:2Ddyn}, we provide some basic set-up for our analysis, including the derivation of the leading order system, the dynamic rescaling formulation, the reformulation using the polar coordinates $(R,\beta)$, and the construction of the approximate self-similar solution. Section \ref{sec:lin} is devoted to the linear stability analysis of the leading order system. In Section \ref{sec:H2}, we perform higher order estimates of the leading order system as part of the nonlinear stability analysis.
Sections \ref{sec:elli} and \ref{sec:non} are devoted to the nonlinear stability analysis of the original system. In Section \ref{sec:euler}, we extend our analysis for the 2D Boussinesq equations to the 3D axisymmetric Euler equations. Some concluding remarks are provided in Section \ref{sec:remark} and some technical estimates are deferred to the Appendix.

\vspace{0.1in}
\paragraph{\bf{Notations}}
We use $\la \cdot, \cdot \ra, || \cdot||_{L^2}$ to denote the inner product in $(R,\b)$ 
and its $L^2$ norm 
\beq\label{eq:inner_L2}
\la f, g \ra  = \int_0^{\infty} \int_0^{\pi/2} f(R,\b) g(R, \b) d R d \b, \quad 
|| f||_{L^2} = \sqrt{\la f, f \ra}.
\eeq
We also simplify $|| \cdot ||_{L^2}$ as $|| \cdot ||_2$.
We remark that we use $d R d \b$ in the definition of the inner product rather than $R d R d \b$.


We use the notation $A \les B$ if there is some absolute constant $C>0$ with $A \leq CB$, and denote $A \asymp B$ if $A \les B$ and $B\les A$. The notation $\bar{\cdot}$ is reserved for the approximate steady states, e.g. $\bar{\Om}$ denotes the approximate steady state for $\Om$. We will use $C, C_1, C_2$ for some absolute constant, which may vary from line to line. We use $K_1, K_2,..$ and $\mu_1, \mu_2,...$ to denote some absolute constant which does not vary.

\section{Derivation of the leading order system}\label{sec:derive}

In this section, we will derive the leading order system that will be used for our analysis later in the paper. We first recall that the 2D Boussinesq equations on the upper half space are given by the following system:
\begin{align}
\om_t +  \uu \cdot \na \om  &= \th_{x},  \label{eq:bous1}\\
\th_t + \uu \cdot  \na \th & =  0 , \label{eq:bous2} 
\end{align}
where the velocity field $\uu = (u , v)^T : \R_+^2 \times [0, T) \to \R^2_+$ is determined via the Biot-Savart law
\beq\label{eq:biot}
 - \D \psi = \om , \quad  u =  - \psi_y , \quad v  = \psi_x,
\eeq
with no flow boundary condition 
\[
\psi(x, 0 ) = 0   \quad x \in \RR
\]
and $\psi$ is the stream function. The reader should not confuse the vector field $\uu$ with its first  component $u$.

The 2D Boussinesq equations have the following scaling-invariant property. If $(\om, \th)$ is a solution pair to \eqref{eq:bous1}-\eqref{eq:biot}, then 
\beq\label{eq:scal}
\om_{\lam, \tau}(x, t) = \f{1}{\tau} \om\lt( \f{x}{\lam},  \f{t}{\tau} \rt) , 
\quad \th_{\lam, \tau}(x , t)= \f{\lam}{\tau^2} \th \lt( \f{x}{\lam}, \f{t}{\tau} \rt)
\eeq
is also a solution pair to \eqref{eq:bous1}-\eqref{eq:biot} for any $\lam, \tau > 0$.

Next, we follow the ideas in Section \ref{sec:ideas} to derive the leading order system for the solutions $\om, \na \th \in C^\alpha$ with small $\alpha$. 



\subsection{The setup}
We look for a solution of \eqref{eq:bous1}-\eqref{eq:biot} with the following symmetry 
\[
\om(x, y) = - \om(x, -y),  \quad \th(x, y) = \th( - x,  y) 
\]
for all $x, y \geq 0 $. Accordingly, the stream function $\psi$ \eqref{eq:biot} is odd with respect to $x$
\[
 \quad \psi(x, y) = -\psi(-x, y)  .
\]
It is easy to see that the equations \eqref{eq:bous1}-\eqref{eq:biot} preserve these symmetries during time evolution. With these symmetries, it suffices to solve \eqref{eq:bous1}-\eqref{eq:biot} on $(x,y) \in [0, \infty) \times [0, \infty) $ with the following boundary conditions 
\[
\psi(x, 0) = \psi(0, y) = 0
\]
for the elliptic equation \eqref{eq:biot}. 

Taking $x, y$ derivative on \eqref{eq:bous2}, respectively,  we obtain 
\begin{align}
 \om_t +   \uu \cdot \na \om &=  \th_x   \label{eq:bous20}, \\ 
\th_{xt} +  \uu \cdot \na \th_x  &= -u_x \th_x - v_x  \th_y  \label{eq:bous21} ,\\
\th_{yt} +  \uu \cdot \na \th_y & = -u_y \th_x - v_y  \th_y  \label{eq:bous22} .
\end{align}
Under the odd symmetry assumption, we have $u(0, y) =0$. If the initial data $\th(0, y) =0$, this property is preserved. Therefore, we can recover $\th$ from $\th_x$ by integration. We will perform a-prior estimate of the above system, which is formally a closed system for $(\om, \th_x, \th_y)$.

\subsection{Reformulation using polar coordinates}

Next, we reformulate \eqref{eq:bous20}-\eqref{eq:bous22} using the polar coordinates introduced by Elgindi in \cite{elgindi2019finite}. We assume that $\al < 1/10$. 
We introduce 
\[r = \sqrt{x^2 + y^2}, \quad \b = \arctan(y / x), \quad R = r^{\al},
\]
Notice that $r \pa_r = \al R \pa_R$. 
We denote 
\beq\label{eq:nota_om}
\Om(R, \b, t) = \om(x, y, t), \quad \Psi = \f{1}{r^2} \psi, \quad \eta( R, \b, t) = (\th_x) ( x, y, t), \quad \xi( R, \b, t) = (\th_y)(x, y, t).
\eeq

We have 
\beq\label{eq:simp1}
\bal
\pa_x &  = \cos (\b) \pa_{r}  -\f{\sin(\b)}{r} \pa_{\b} =  \f{\cos(\b)}{r} \al R  \pa_R  -\f{\sin(\b)}{r} \pa_{\b} , \\ 
 \pa_y  & =  \sin(\b) \pa_{r}  + \f{\cos (\b)}{ r} \pa_{\b} =   \f{\sin(\b)}{r} \al R\pa_R + \f{\cos (\b)}{ r} \pa_{\b} .  \\
 \eal 
 \eeq
 Then using \eqref{eq:biot}, we derive 
\beq\label{eq:simp2}
\bal
u &=  - (r^2 \Psi)_y = - 2 r \sin \b \Psi - \al r R \sin \b \pa_R\Psi  -  r\cos \b  \pa_{\b}\Psi , \\
v  & = (r^2 \Psi)_x = 2 r \cos \b \Psi + \al r R \cos \b \pa_R \Psi - r \sin \b \pa_{\b} \Psi . \\
\eal
\eeq
Using the new variables $R, \b$, we can reformulate the Biot-Savart law \eqref{eq:biot} as 
\beq\label{eq:biot2}
- \al^2 R^2 \pa_{RR} \Psi - \al(4 +\al) R\pa_R \Psi - \pa_{\b \b} \Psi - 4 \Psi = \Om
\eeq
with boundary condition 
\[
\Psi(R,0) = \Psi(R, \f{\pi}{2}) = 0.
\]
For the transport term in \eqref{eq:bous20}-\eqref{eq:bous22}, we use \eqref{eq:simp1} to derive
\beq\label{eq:trans}
u \pa_x + v \pa_y \to - (\al R \pa_{\b} \Psi) \pa_R + ( 2\Psi +  \al R \pa_R \Psi) \pa_{\b}.
\eeq

Recall the notations $\Om, \eta, \xi$ in \eqref{eq:nota_om} for $\om, \th_x, \th_y$ in the $(R, \b)$ coordinates. Using \eqref{eq:trans}, we can rewrite \eqref{eq:bous20}-\eqref{eq:bous22} in $(R, \b)$ coordinates as follows 
\begin{align}
 \Om_t + \B(  - (\al R \pa_{\b} \Psi) \pa_R + ( 2\Psi +  \al R \pa_R \Psi) \pa_{\b} \B) \Om &= \eta,  \label{eq:bous_sim11}  \\
\eta_t + \B(  - (\al R \pa_{\b} \Psi) \pa_R + ( 2\Psi +  \al R \pa_R \Psi) \pa_{\b} \B)  \eta &= -u_x \eta - v_x  \xi , \label{eq:bous_sim12} \\
\xi_t + \B(  - (\al R \pa_{\b} \Psi) \pa_R + ( 2\Psi +  \al R \pa_R \Psi) \pa_{\b} \B)  \xi & = -u_y \eta - v_y  \xi  \label{eq:bous_sim13}.
\end{align}
The formulas of $\na u$ in $(R,\b)$ coordinates are rather lengthy and presented in \eqref{eq:simp5}.

\subsection{ Leading order approximations  of the Biot-Savart law and the velocity}
Next, we use an important result of Elgindi in \cite{elgindi2019finite} to obtain a leading order approximation of the modified stream function. Using this approximation, we can simplify the transport terms and $\na u$, and further derive the the leading order system of \eqref{eq:bous_sim11}-\eqref{eq:bous_sim13}.

Following \cite{elgindi2019finite}, we decompose the modified stream function $\Psi$ as follows
\beq\label{eq:biot3}
\bal
\Psi &= \f{1}{ \pi \al} \sin(2\b) L_{12} (\Omega) + \textrm{ lower order terms}, \\
L_{12}(\Om) & =  \int_R^{\infty} \int_0^{\pi/2} \f{ \sin (2\b) \Om(s, \b) }{s} ds d \b.
\eal
\eeq
For $\om \in C^{\al}$ with sufficiently small $\al > 0$, the leading order term in $\Psi$ is given by the first term on the right hand side. The lower order terms (l.o.t.) are relatively small compared to the first term and we will control them later using the elliptic estimates. We will perform the $L^2$ estimate for the solution of \eqref{eq:biot2} and 
one can see that the a-priori estimate blows up as $\al \to 0$. For $\al = 0$, \eqref{eq:biot2} becomes 
\[
L_0(\Psi) = - \pa_{\b \b} \Psi - 4 \Psi ,
\]
with boundary conditions $\Psi(R, 0) = \Psi(R, \pi/ 2) =0$, which is self-adjoint and has kernel $\sin(2\b)$. In this case, to solve $L_0(\Psi) = \Om$, a necessary and sufficient condition is that $\Om$ is orthogonal to $\sin 2\b$. Imposing this constraint 
when we perform the elliptic estimate leads to the leading order term in $\Psi$ \eqref{eq:biot3}.

Following the same procedure as in \cite{elgindi2019finite}, we drop the $O(\al)$ terms in \eqref{eq:simp1}, \eqref{eq:simp2} and the lower order terms in \eqref{eq:biot3} to extract the leading order term of the velocity $u, v$ 
\beq\label{eq:simp3}
\bal
u&  = -\f{2 r \cos \b }{\pi \al} L_{12}(\Om) + l.o.t., \quad  v  = \f{2r \sin \b}{\pi \al} L_{12}(\Om) + l.o.t., \\
  u_x  & = -v_y   = -\f{2}{\pi \al} L_{12} (\Om) + l.o.t., \quad u_{y}  =  l.o.t. , \quad v_x =l.o.t. .
\eal
\eeq
The complete calculation and the formulas of the lower order terms are given in \eqref{eq:simp5}-\eqref{eq:simp512}.

Similarly, the leading order term in the transport terms \eqref{eq:trans} is 
\beq\label{eq:trans2}
- (\al R \pa_{\b} \Psi) \pa_R + ( 2\Psi +  \al R \pa_R \Psi) \pa_{\b} 
= - \f{2}{\pi} \cos(2\b) L_{12}(\Om) R\pa_R + \f{2}{ \pi \al} \sin(2\b) L_{12}(\Om) \pa_{\b} + l.o.t. .
\eeq

Later on, we will prove that the self-similar blowup is non-linearly stable and we will control the above lower order terms using the elliptic estimates. These terms will be treated as small perturbations and are harmless to the self-similar blowup.

\subsection{Decoupling and simplifying the system}\label{sec:decoup}
We will look for solution $\th$ of \eqref{eq:bous20}-\eqref{eq:bous22} (or equivalently \eqref{eq:bous_sim11}-\eqref{eq:bous_sim13}) such that $\th_x \in C^{\al}$, $\th_x$ is odd, and $\th_y$ is relatively small compared to $\th_x$, i.e. $\th$ is {\it not} isotropic. The anisotropic property of $\th$ will enable us to further simplify \eqref{eq:bous_sim11}-\eqref{eq:bous_sim13}.
The reason that we have this property is due to the following key observation.
For the purpose of illustration, we construct a function $\theta$ that has the same qualitative feature as our solution $\th$. We first construct  $\theta_x$ of the form: $\th_x = \f{x^{\al}}{ 1 + (x^2 + y^2)^{\al/2} }$ for $x , y \geq 0$. Then for $x, y $ close to $0$, we have 
\beq\label{eq:key_ob2_example}
\th \approx \f{1}{1+\al} \cdot \f{x^{1+\al}}{ 1 + (x^2 + y^2)^{\al/2}}, \quad 
|\th_y | \approx  \B| \f{\al}{1+\al}  \cdot \f{xy}{x^2+ y^2}  \cdot \f{x^{\al}(x^2 + y^2)^{\al/2} }{ ( 1 + (x^2 + y^2)^{\al/2})^2}  \B| \les \al \th_x.
\eeq
Compared to $\th_x$, $\th_y$ is relatively small. Equivalently, $\xi$ is small relative to $\eta$.
Moreover, $\xi$ is weakly coupled with $\Om, \eta$ in \eqref{eq:bous_sim11}-\eqref{eq:bous_sim12} since $v_x = l.o.t.$ according to \eqref{eq:simp3}. Hence, we can decouple $\xi$ from the $\eta$ equation in \eqref{eq:bous_sim12} as follows
\[
\eta_t + \B(  - (\al R \pa_{\b} \Psi) \pa_R + ( 2\Psi +  \al R \pa_R \Psi) \pa_{\b} \B)  \eta = -u_x \eta + l.o.t. ,\\
\]
These key observations motivate us to focus on the system \eqref{eq:bous_sim11}-\eqref{eq:bous_sim12} about $\Om, \eta$.

Using the calculations of $\na u$ \eqref{eq:simp3}, the transport terms \eqref{eq:trans2} and treating $\xi \  (\th_y)$ as a lower order term,
we can simplify \eqref{eq:bous_sim11}-\eqref{eq:bous_sim13} as follows
\begin{align}
\Om_t   - \f{2}{\pi} \cos(2\b) L_{12}(\Om) R\pa_R  \Om+ \f{2}{ \pi \al} \sin(2\b) L_{12}(\Om) \pa_{\b} \Om
&= \eta  + l.o.t.,  \label{eq:sys0} \\
\eta_{t}  - \f{2}{\pi} \cos(2\b) L_{12}(\Om) R\pa_R  \eta+ \f{2}{ \pi \al} \sin(2\b) L_{12}(\Om) \pa_{\b} \eta &=  \f{2}{\pi \al} L_{12}(\Om) \eta + l.o.t., \label{eq:sys1}
\end{align}
where the equations are evaluated at $(R, \b)$ with $R = (x^2 + y^2)^{\al / 2}, \b = \arctan( y/x)$. Notice that in \eqref{eq:sys1}, the first transport term looks much smaller than the other transport term and the nonlinear term which contains a $1/\al$ factor.
Thus we can ignore it in our leading order approximation. For the angular transport term, we use an argument introduced in \cite{elgindi2019finite} and look for approximate solutions $(\Om, \eta)$ of the form
\beq\label{eq:form_al}
\Om(R, \b, t) = \al  \G(\b) \Om_*(R, t), \quad   \eta(R, \b, t) =\al \G(\b) \eta_*(R, t), \quad \G(\b) =
(\cos(\b))^{\al}.
\eeq
We have added the factor $\al$ in the above form, which is slightly different from \cite{elgindi2019finite}. For $\b \in [0, \pi /2]$, we gain a small factor $\al$ from the angular derivative:
\[
| \sin(2\b) \pa_{\b} \G(\b) |   =  |2 \al \sin^2(\b) (\cos(\b))^{\al} | \leq  2 \al \G(\b).
\]
Hence, the angular transport term in \eqref{eq:sys1} becomes smaller compared to the nonlinear term. 

Using \eqref{eq:form_al} and the above estimate, formally, we obtain that the transport terms in \eqref{eq:sys0} is of order $\al^2$ and $\eta$ in \eqref{eq:sys0} is of order $\al$. Therefore, we drop the transport terms in \eqref{eq:sys0}. This additional consideration is not required in \cite{elgindi2019finite} for 3D asymmetric Euler without swirl.

We remark that in our dynamic rescaling formulation, $\eta$ is comparable to the nonlinear term $\al^{-1} L_{12}(\Om) \eta$. Therefore, we drop the transport terms and the lower order terms in \eqref{eq:sys0},\eqref{eq:sys1} to derive a leading order system for $(\Om, \eta)$
\beq\label{eq:sys2}
\bal
&\Om_t  = \eta    , \quad \eta_{t}  =  \f{2}{\pi \al} L_{12}(\Om) \eta , \quad L_{12}(\Om) =  \int_R^{\infty} \int_0^{\pi/2} \f{\Om( s ,\b) \sin(2\b) }{s} ds d \b .
\eal
\eeq
It is not difficult to see that if the initial data $\Om, \eta$ are non-negative and are odd with respect to $x$, the solutions preserve these properties dynamically. In the first equation, $\Om$ tends to align with $\eta$ during the evolution. Then the nonlinear term in the second equation is of order $\eta^2$, which is the driving force of finite time singularity of the leading order  system.

\section{Self-similar solution of the leading order system}\label{sec:los}
The leading order system \eqref{eq:sys2} is crucial in our analysis and it captures the leading behavior of the 
blowup solution of the Boussinesq equations \eqref{eq:bous1}-\eqref{eq:biot}.
In this section, we construct the self-similar solution of the leading order system \eqref{eq:sys2} for $(\Om, \eta)$. Notice that $L_{12}(\Om)$ does not depend on the angular component $\b$. Inspired by the solution structure of the leading order system in \cite{elgindi2019finite}, we look for a self-similar solution in the form 
\[
\Om(R, \b, t) = (T-t)^{c_{\om}} \Om_*\lt( \f{R}{ (T-t)^{ \al \cdot c_l  } } \rt) \G(\b), \quad \eta(R, \b, t) = 
 (T-t)^{c_{\th} - c_l} \eta_*\lt( \f{R}{ (T-t)^{ \al \cdot c_l}} \rt) \G(\b) ,
\]
where $c_{\om}, c_l, c_{\th}$ are the scaling parameters. The reason that we use the scaling factor $(T-t)^{\al \cdot c_l}$ in the space variable $R$ is that $R = r^{\al}$ and 
$
\f{R}{ (T-t)^{\al \cdot c_l}} =  \lt( \f{ r}{ (T-t)^{c_l}} \rt)^{\al},$
where $r = \sqrt{x^2 + y^2}$. Factor $(T-t)^{c_l}$ corresponds to the scaling of the original variables $x, y$ and $(T-t)^{c_{\th}}$ is the scaling of $\th$ in \eqref{eq:bous20}-\eqref{eq:bous22}. See \eqref{eq:scal} for the scaling invariance of the Boussinesq equations.

Plugging the self-similar solutions ansatz into \eqref{eq:sys2}, we obtain
\beq\label{eq:selfsim1}
\bal
 & - (T-t)^{c_{\om} - 1}c_{\om} \Om_*(z) \G(\b)
 + (T-t)^{c_{\om} - 1} \al c_l z \pa_z \Om_*(z) \G(\b) =  (T-t)^{c_{\th} - c_l} \eta_*(z) \G(\b) , \\
 & - (T-t)^{c_{\th} - c_l - 1} (c_{\th} - c_l) \eta_*(z) \G(\b) +(T-t)^{c_{\th }- c_l - 1} \al c_l z \pa_z \eta_*(z) \G(\b) \\
 =& (T-t)^{c_{\th} - c_{l} + c_{\om}}\eta_*(z) \G(\b) \f{2}{\pi \al}  \int_z^{\infty} \f{\Om_*(s)}{s} ds\cdot \int_0^{\pi/2} \G(\b)\sin(2\b) d\b   ,
\eal
\eeq
where $z = R \cdot (T-t)^{-\al c_l} \geq 0$. From the above equations, we obtain that the scaling parameters $(c_{\om}, c_l, c_{\th})$ satisfies 
\[
c_{\om} - 1 = c_{\th} - c_l , \quad  c_{\th} - c_l -1 = c_{\om} + c_{\th} - c_l, 
\]
which implies 
\[
c_{\om} = -1,  \quad c_{\th} = c_l + 2.
\]
Denote 
\[
c = \f{2}{\pi} \int_0^{\pi} \G(\b) \sin(2 \b) d\b.
\]
Plugging the relations among the scaling parameters into \eqref{eq:selfsim1} and factorizing the temporal variable, we derive
\beq\label{eq:selfsim2}
\bal
 \al c_l z  \pa_z \Om_*  \G(\b)&=   - \Om_* \G(\b)+ \eta_* \G(\b), \\
 \al c_l z \pa_z \eta_* \G(\b) & = -2 \eta_*  \G(\b)+  \f{c}{ \al} \eta_* \G(\b)\int_z^{\infty} \f{\Om_*(s)}{s} ds .
 \eal
\eeq
We can factorize the angular part $\G(\b)$ to further simplify the above equations. Surprisingly, the above equations have explicit solutions of the form 
\[ 
\Om_*(z) = \f{az}{(b+z)^2} ,\quad  c_l = \f{1}{\al}  
\]
(recall that $z \geq 0$). We determine $\eta_*$ from the first equation in \eqref{eq:selfsim2}
\[
\eta_*(z) =  \al c_l z  \pa_z \Om_*    + \Om_*  = z \pa_z \Om_* + \Om_* =  \f{2ab z}{(b+z)^3}.
\]
Then $(\eta_*, \Om_*)$ solves \eqref{eq:selfsim2} exactly if and only if 
\[
 z \pa_z \eta_* + 2 \eta_* -  \f{c}{ \al} \eta_*\int_z^{\infty} \f{\Om_*(s)}{s} ds  =0
\]
which is equivalent to 
\[
  0 = z \lt( -\f{6ab z}{ (b+z)^4} + \f{2ab}{(b+z)^3} \rt) + \f{4ab z}{ (b+z)^3} - \f{c}{\al} \f{2abz}{ (b+z)^3} \f{a}{b+z} =  - \f{  2ab (  -3\al b + ac )z}{ \al (b+z)^4}.
\]
Hence, we obtain 
\[
a = \f{3 \al b }{c}.
\]
Using the above formula, we can derive the solutions $(\Om_*, \eta_*)$ of \eqref{eq:sys2}. We remark that there is a free parameter $b$ in the solutions $(\Om_*, \eta_*)$. After we impose a normalization condition, e.g. the derivative of $\Om_*$ at $z=0$, we can determine $b$. For simplicity, we choose $b=1$ and then $a$ becomes $a = 3 \al / c$. Consequently, we obtain the following result.

\begin{lem}\label{lem:selfsim}
The leading order system \eqref{eq:sys2} admits a family of self-similar solutions
\[
\Om(R, \b, t) = \f{ \al}{c} \f{1}{T-t}  \G(\b)  \Om_*\lt(\f{ R}{ T - t }  \rt), \quad 
\eta(R, \b , t) = \f{ \al}{c} \f{1}{ (T-t)^2}  \G(\b)  \eta_*\lt(\f{ R}{ T - t }  \rt) , 
\]
for some $T > 0$, where 
\[
\Om_*(z) = \f{3z}{ (1+z)^2}, \quad \eta_* = \f{ 6z }{ (1+z)^3 } , \quad c  = \f{2}{\pi} \int_0^{\pi / 2} 
\G(\b) \sin(2\b) d\b  \neq 0.
\]
\end{lem}

We will choose $\G(\b) = (\cos(\b))^{\al}$ in the later discussion.

\vspace{0.1in}
\paragraph{\bf{Properties of $\th_x, \om$}}
The self-similar profile $(\Om, \eta)$ of the leading order system \eqref{eq:sys2} in Lemma \ref{lem:selfsim} is indeed anisotropic in $x, y$ direction. Moreover, $\th_x$ and  $\om$ are positive in the first quadrant. For $\G(\b) = 
(\cos(\b))^{\al}$, the self-similar profile of $\th_x$ in the first quadrant is 
\[
\th_x = C\al \G(\b) \f{ R}{(1+R)^3} = C \al  \f{ |x|^{\al } }{  (1 + (x^2 + y^2)^{\al / 2})^3}  ,
\]
for some constant $C$. If $x^2 + y^2$ is small, the formal argument \eqref{eq:key_ob2_example} shows that $\th_y$ is relatively small compared to $\th_x$. We will estimate it precisely in Lemma \ref{lem:xi} in the Appendix. 

\vspace{0.1in}
\paragraph{\bf{Hyperbolic flow field}} 
The leading order of the flow structure corresponding to the self-similar solution of the leading order system can be obtained using 
\eqref{eq:simp3}
\[
\bal
L_{12}(\Om)(R, \b, t) &=\f{\pi \al}{2} \f{1}{T-t} \f{ 3   }{ 1 + R / (T-t)}   = \f{\pi \al}{2} \f{3}{ (T-t) +R},  \\
u(x , y, t) &=   - \f{3 r \cos \b}{ (T-t) + R} + l.o.t. , \quad v(x, y, t) = \f{3 r \sin(\b)}{ (T-t) + R} + l.o.t.. \\
\eal
\]
In the first quadrant, the flow is clockwise since $u < 0, v > 0$. Moreover, the odd symmetry of $\om$ implies that the flow is hyperbolic near the origin. These properties of the solutions are similar to those considered in \cite{luo2013potentially-1,luo2013potentially-2}.

\section{The dynamic rescaling formulation and the approximate steady state}\label{sec:2Ddyn}
In this section, we reformulate the problem using the dynamic rescaling equation and construct an approximate steady state based on the self-similar solution of the leading order system.

\subsection{Dynamic rescaling formulation}\label{sec:dsform}
Let $ \om(x, t), \th(x,t) , \uu(x, t)$ be the solutions of \eqref{eq:bous1}-\eqref{eq:biot}.
Then it is easy to show that 
\beq\label{eq:rescal1}
\bal
  \td{\om}(x, \tau) &= C_{\om}(\tau) \om(   C_l(\tau) x,  t(\tau) ), \quad   \td{\th}(x , \tau) = C_{\th}(\tau)
  \th( C_l(\tau) x, t(\tau)),  \\
    \td{\uu}(x, \tau) &= C_{\om}(\tau)  C_l(\tau)^{-1} \uu(C_l(\tau) x, t(\tau)) , 
\eal
\eeq
are the solutions to the dynamic rescaling equations
 \beq\label{eq:bousdy1}
\bal
\td{\om}_{\tau}(x, \tau) + ( c_l(\tau) \xx + \td{\uu} ) \cdot \na \td{\om}  &=   c_{\om}(\tau) \td{\om} + \td{\th}_x , \qquad 
\td{\th}_{\tau}(x , \tau )+ ( c_l(\tau) \xx + \td{\uu} ) \cdot \na \td{\th}  = 0,
\eal
\eeq
where $\uu = (u, v)^T = \na^{\perp} (-\D)^{-1} \td{\om}$, $\xx = (x, y)^T$, 
\beq\label{eq:rescal2}
\bal
  C_{\om}(\tau) = \exp\lt( \int_0^{\tau} c_{\om} (s)  d \tau\rt), \ C_l(\tau) = \exp\lt( \int_0^{\tau} -c_l(s) ds    \rt) , \  C_{\th}  =  \exp\lt( \int_0^{\tau} c_{\th} (s)  d \tau\rt),
\eal
\eeq
$  t(\tau) = \int_0^{\tau} C_{\om}(\tau) d\tau $ and  the rescaling parameter $c_l(\tau), c_{\th}(\tau), c_{\om}(\tau)$ satisfies 
\beq\label{eq:rescal3}
c_{\th}(\tau) = c_l(\tau ) + 2 c_{\om}(\tau).
\eeq

Let us explain the above relation. Using this relationship and \eqref{eq:rescal1}, we have $ \td {\uu} \cdot \na \td{\om} = C_{\om}(\tau)^2 \uu \cdot \na \om$ and $ \td \th_x = C_{\th}(\tau) C_l(\tau) \th_x$. To obtain \eqref{eq:bousdy1} from \eqref{eq:bous1}-\eqref{eq:biot}, we require that the scaling factors of $\td {\uu} \cdot \na \td{\om} $ and $ \td \th_x$ are the same, which implies $C_{\om}(\tau)^2 = C_{\th}(\tau) C_l(\tau)$. Using this relationship and \eqref{eq:rescal2}, we obtain \eqref{eq:rescal3}.

Recall that the Boussinesq equations have scaling-invariant property \eqref{eq:scal} with two parameters. We have the freedom to choose the time-dependent scaling parameters $c_l(\tau)$ and $c_{\om}(\tau)$ according to some normalization conditions. After we determine the normalization conditions for $c_l(\tau)$ and $c_{\om}(\tau)$, the dynamic rescaling equation is completely determined and the solution of the dynamic rescaling equation is equivalent to that of the original equation using the scaling relationship described in \eqref{eq:rescal1}-\eqref{eq:rescal2}, as long as $c_l(\tau)$ and $c_{\om}(\tau)$ remain finite.

We remark that the dynamic rescaling formulation was introduced in \cite{mclaughlin1986focusing,  landman1988rate} to study the self-similar blowup of the nonlinear Schr\"odinger equations. This formulation is also called the modulation technique in the literature and has been developed by Merle, Raphael, Martel, Zaag and others.
It has been a very effective tool to analyze the formation of singularities for many problems like the nonlinear Schr\"odinger equation \cite{kenig2006global,merle2005blow}, the nonlinear wave equation \cite{merle2015stability}, the nonlinear heat equation \cite{merle1997stability}, the generalized KdV equation \cite{martel2014blow}, and other dispersive problems. Recently, this method has been applied to study singularity formation in the De Gregorio model and the generalized Constantin-Lax-Majda model for the 3D Euler equations from smooth initial data \cite{chen2019singularity,chen2019finite,chen2020slightly}. It has also been applied to prove singularity formation in other equations in fluid dynamics, see e.g.  \cite{elgindi2019finite,collot2018singularity}.

If there exists $C>0$ such that for any $\tau > 0$, $c_{\om}(\tau) \leq -C <0$ and the solution $\td{\om}$ is nontrivial, e.g. $ || \td{\om}(\tau, \cdot) ||_{L^{\infty}} \geq c >0$ for all $\tau >0$, we then have 
\[
\bal
C_{\om}(\tau) &\leq e^{-C\tau}, \ t(\infty) \leq \int_0^{\infty}  e^{-C \tau } d \tau =C^{-1} <+ \infty \; ,
\eal
\]
and that $| \om(   C_l(\tau) x,  t(\tau) ) | = C_{\om}(\tau)^{-1}  |\td{\om}(x, \tau) | 
\geq e^{C\tau} |\td{\om}(x, \tau) | $  blows up at finite time $T = t(\infty)$.
If $(\td{\om}(\tau), \td{\th}(\tau),  c_{l}(\tau), c_{\om}(\tau), c_{\th}(\tau))$ converges to a steady state $(\om_{\infty}, \th_{\infty}, c_{l, \infty}, c_{\om,\infty}, c_{\th,\infty})$ of \eqref{eq:bousdy1} as $\tau \to \infty$, one can verify that 
\[
\om(x, t) = \f{1}{1 - t} \om_{\infty}\lt( \f{x}{ (1 - t)^{-c_{l, \infty} / c_{\om, \infty} }  }  \rt)
, \quad \th(x, t) = \f{1}{ (1-t)^{ c_{\th,\infty}/ c_{\om,\infty} }}\th_{\infty}\lt( \f{x}{ (1 - t)^{-c_{l, \infty} / c_{\om, \infty} }  }  \rt)
\]
is a self-similar solution of \eqref{eq:bous1}-\eqref{eq:biot}. 
To simplify our presentation, we still use $t$ to denote the rescaled time in the rest of the paper and drop $\td{\cdot}$ in \eqref{eq:bousdy1}.

\subsection{Reformulation using the $(R,\b)$ coordinates}
Taking $x, y$ derivative on the $\th$ equation in \eqref{eq:bousdy1}, we obtain a system similar to \eqref{eq:bous20}-\eqref{eq:bous22}.
\beq\label{eq:bousdy20}
\bal
 \om_t +  (c_l \xx +  \uu) \cdot \na \th_x &= c_{\om} \om +  \th_x , \\
\th_{xt} + (c_l \xx +  \uu) \cdot \na \th_x  &= ( c_{\th} - c_l -u_x) \th_x - v_x  \th_y ,\\
\th_{yt} + ( c_l \xx+  \uu )\cdot \na \th_y & =  (c_{\th} - c_l - v_y)  \th_y -u_y \th_x,
\eal
\eeq
where we have dropped $\td{\cdot}$ to simplify the notations. 
We make a change of variable $R = r^{\al},  \b  = \arctan(y/x)$ and introduce
\[
\Om(R, \b, t) = \om(x, y, t), \quad \eta(R, \b, t) = (\th_x)(x, y, t), \quad \xi(R, \b, t) = (\th_y)(x, y, t)
\]
in \eqref{eq:bousdy20} as we did in Section \ref{sec:derive}. Notice that the stretching term and the damping term satisfy
\[
c_l \xx \cdot \na \om(x, y, t) =  c_l r \pa_r \om(r, \b, t) = \al c_l R \pa_R \Om(R, \b, t), \quad c_{\om} \om(x, y, t) = c_{\om} \Om(R, \b, t),
\]
and similar relations hold for $\th_x, \th_y$. The reformulated system \eqref{eq:bousdy20}
under $(R, \b)$ coordinates reads
\beq\label{eq:bousdy2}
\bal
\Om_t  + \al c_l R \pa_R \Om + (\uu \cdot \na)  \Om  &=  c_{\om} \Om + \eta \\
\eta_t + \al c_l R \pa_R \eta + (\uu \cdot \na)  \eta & = (2 c_{\om} -  u_x) \eta - v_x \xi \\
\xi_t + \al c_l R \pa_R \xi + (\uu \cdot \na)  \xi & = (2 c_{\om}- v_y) \xi -  u_y \eta ,\\
\eal
\eeq
with the Biot-Savart law in the $(R,\b)$ coordinates \eqref{eq:simp2} and \eqref{eq:biot2}, where we have used $c_{\th} - c_l = 2 c_{\om}$ \eqref{eq:rescal3}. For now, we do not expand $u \cdot \na$ using \eqref{eq:trans} and $u_x, u_y, v_x, v_y$ due to their complicated expressions. 
Using the same argument as that in Section \ref{sec:decoup}, the leading terms in \eqref{eq:bousdy2} are given by 
\beq\label{eq:bous3}
\bal
\Om_t  + \al c_l R \pa_R \Om  &=  c_{\om} \Om + \eta  + l.o.t. ,\\
\eta_t + \al c_l R \pa_R \eta & = (2 c_{\om}  + \f{2}{\pi \al} L_{12}(\Om) ) \eta  + l.o.t. , \\
\xi_t + \al c_l R \pa_R \xi & = (2 c_{\om} - \f{2}{\pi \al} L_{12} (\Om)) \xi + l.o.t.,
\eal
\eeq
where we have dropped the transport terms and simplified $u_x, u_y, v_x, v_y, u/x, v / y$ using \eqref{eq:simp3}. 
We remark that the first two equations in \eqref{eq:bous3} are exactly the dynamic rescaling formulation of the leading order system \eqref{eq:sys2}. 

\subsection{Constructing an approximate steady state}
Notice that the system \eqref{eq:bous3} captures the leading order terms in the system \eqref{eq:bousdy2} and that the self-similar profile of \eqref{eq:sys2} corresponds to the steady state of the first two equations in \eqref{eq:bous3} after neglecting the lower order terms. It motivates us to use the self-similar solutions of \eqref{eq:sys2} in Lemma \ref{lem:selfsim} as the building block to construct the approximate steady state of \eqref{eq:bousdy2}. Firstly, we construct 
\beq\label{eq:profile}
\bal
\bar{\Om}(R, \b)& =  \f{\al}{c} \G(\b) \f{ 3R }{ (1 + R)^2}, \quad \bar{\eta}(R, \b) =  \f{\al}{c} \G(\b) \f{ 6 R }{ (1 + R)^3} , \quad \bar{c}_l  = \f{1}{\al} + 3, \quad \bar{c}_{\om} = -1 ,\\
\G(\b) &= (\cos(\b))^{\al}, \qquad  c =  \f{2}{\pi} \int_0^{\pi/2} \G(\b) \sin(2\b) d\b.
\eal
\eeq
Notice that $(\bar{\Om}, \bar{\eta})$ is a solution of \eqref{eq:selfsim2} with $c_l =\f{1}{\al}$. We modify $\bar{c}_l$ so that the approximate error vanishes quadratically near $R=0$, which will be discussed later. The corresponding $\bar{\th}$ can be obtained by integrating $\bar{\th}_x$ with condition $\bar{\th}(0, y)  = 0$, which is discussed in Appendix \ref{sec:xi}, and $\bar{u}, \bar{v}$ are obtained from the Biot-Savart law \eqref{eq:simp2}, \eqref{eq:biot2}. We can derive the leading order terms using \eqref{eq:biot3} and \eqref{eq:simp3}
\beq\label{eq:simp4}
\bal
&L_{12}(\bar{\Om}) = \int_R^{\infty} \int_0^{\pi/2} \f{ \bar{\Om}(s, \b) \sin(2\b)}{s} ds 
= \f{\pi}{2}  \f{3 \al }{1 + R} ,  \quad  \bar{\Psi}  =  \f{\sin(2\b)}{2} \f{3}{1 +R} + l.o.t. , \\
&  \bar{u}_x  = -\bar{v}_y = - \f{2}{\pi \al} L_{12}(\Om) + l.o.t.  = \f{3}{1+R} + l.o.t. , \quad \bar{u}_y , \  \bar{v}_x = l.o.t. .
\eal
\eeq

We will explain later why we choose the above $\G(\b)$. Lemma \ref{lem:one} in the Appendix shows that $\G(\b)$ is essentially equal to the constant $1$ in some weighted norm.
 


We define the error of the approximate steady state below
\beq\label{eq:error}
\bal
\bar{F}_{\om} &\teq \bar{c}_{\om} \bar{\Om} + \bar{\eta} - \al \bar{c}_l R \pa_R\bar{ \Om} - (\bar{\uu} \cdot \na)  \bar{\Om} ,\\
\bar{F}_{\eta} &\teq  (2 \bar{c}_{\om} -  \bar{u}_x) \bar{\eta} - \bar{v}_x \bar{\xi} 
-\al \bar{c}_l R \pa_R\bar{\eta} - (\bar{\uu} \cdot \na)  \bar{\eta} ,\\
\bar{F}_{\xi} & \teq (2 \bar{c}_{\om} -  \bar{v}_y) \bar{\xi} - \bar{u}_y \bar{\eta} 
-\al \bar{c}_l R \pa_R\bar{\xi} - (\bar{\uu} \cdot \na)  \bar{\xi} .\\
\eal
\eeq

The criteria to choose $\G$ in \eqref{eq:profile} is that $F_{\om}, F_{\eta}, F_{\xi}$ vanish quadratically near $R= 0$ since we will perform energy estimates with a singular weight in the later sections. Using the formula \eqref{eq:trans} for $\bar{u} \cdot \na $ and \eqref{eq:profile}, one can obtain the following expansion of $\bar{F}_{\om}$ near $R = 0$
\[
\bar{F}_{\om} = - 3\al R \pa_R \bar{\Om} - (\bar{u}\cdot \na) \bar{\Om} =\f{9\al R}{c}(  \al \G \cos(2\b) - \sin(2\b) \pa_{\b} \G - \al \G  )  + O(R^2) ,
\]
where we have used the explicit formula \eqref{eq:profile} in the first equality and the factor $3$ comes from $\bar{c}_l = \f{1}{\al} + 3$ in \eqref{eq:profile}. In order for $\bar{F}_{\om}$ to vanish quadratically near $R=0$, we have no choice but to set the coefficient in the $O(R)$ term
to be zero, which gives
\[
 \al \G \cos(2\b) - \sin(2\b) \pa_{\b} \G  - \al \G = 0.
\]
To solve the above first order ODE for $\Gamma$,  we choose the boundary condition $\Gamma (\pi/2)=0$ and requires $\Gamma(\beta) >0$ for $\beta \in (0,\pi/2]$. The solution of this ODE is exactly given by the formula of $\Gamma(\beta)$ in \eqref{eq:profile}. As we can see, such choice of $\Gamma$ is unique and is a consequence of the condition that $\bar{F}_\omega =O(R^2)$ near $R=0$. This condition plays an essential role in our stability analysis for the approximate self-similar profile.
With this $\G(\b)$, we also have $\bar{F}_{\eta} , \bar{F}_{\xi}= O(R^2)$ near $R =0$. 
 We justify these rigorously in Section \ref{sec:non}.


\subsection{Normalization conditions}
For initial data $\bar{\Om} + \Om, \bar{\eta} + \eta, \bar{\xi} + \xi$ of \eqref{eq:bousdy2}, we treat $\Om, \eta, \xi$ as perturbation and choose time-dependent scaling parameters $c_l + \bar{c}_l, c_{\om} + \bar{c}_{\om}$ as follows 
\beq\label{eq:normal}
c_{\om}(t) = -\f{2 }{\pi \al} L_{12}(\Om(t))(0), \quad c_l(t) =  - \f{1-\al}{\al} \f{2}{\pi \al}L_{12}(\Om(t))(0)= \f{1-\al}{\al} c_{\om}(t).
\eeq
Here, $c_l(t), c_{\om}(t)$ are treated as the perturbation of the scaling parameters $\bar{c}_l, \bar{c}_{\om}$. Suppose that $F_{\Om}(t), F_{\eta}(t), F_{\xi}(t)$ are the time-dependent update in \eqref{eq:bousdy2}
,i.e.
\[
F_{\Om}(t) = (c_{\om} + \bar{c}_{\om}) (\Om+ \bar{\Om}) + (\eta + \bar{\eta}) -
 \al ( c_l + \bar{c}_l) R \pa_R (\Om+\bar{\Om}) + ( ( \uu + \bar{\uu}) \cdot \na) ( \Om   + \bar{\Om}), 
\]
and so on. The reason we choose \eqref{eq:normal} is that we want $F_{\Om}(t), F_{\eta}(t), F_{\xi}(t)$ vanishes quadratically near $R=0$ for any perturbation $\Om(t), \eta(t), \xi(t)$ that vanishes quadratically near $R=0$, so that we can choose a singular weight to analyze the stability of the approximate steady state. Similar consideration has been used in our previous work with D. Huang on the asymptotically self-similar blowup of the Hou-Luo model from smooth initial data.
We will provide rigorous estimates for these terms in Section \ref{sec:non}.

\section{Linear stability}\label{sec:lin}
We present our linear stability analysis in this section.
In Section \ref{sec:L2_sys}, we linearize the dynamic rescaling formulation in the $(R, \b)$ coordinates \eqref{eq:bousdy2} around the approximate steady state $(\bar{\Om}, \bar{\eta}, \bar{\xi}, \bar{c}_l, \bar{c}_{\om})$. In Section \ref{sec:L2_outline}, we outline the steps in the linear stability analysis. In the rest of the Section, we establish the linear stability of the leading terms in the linearized system. Throughout this section, we use $\Om, \eta, \xi, c_l, c_{\om}$ to denote the perturbations around the approximate profile \eqref{eq:profile} and assume that $\Om  \in L^2(\vp), \eta \in L^2(\vp), \xi \in L^2(\psi)$ for some singular weights $\vp, \psi$ to be determined later.

\subsection{Linearized system}\label{sec:L2_sys}
We linearize \eqref{eq:bousdy2} around $(\bar{\Om}, \bar{\eta}, \bar{\xi}, \bar{c}_l, \bar{c}_{\om})$ 
\eqref{eq:profile} and derive the equations for the perturbation $\Om, \eta,\xi$ as follows 
\beq\label{eq:lin}
\bal
\Om_t + (1 +3\al) R \pa_R \Om + (\bar{\uu} \cdot \na )  \Om &= - \Om +  \eta  + 
c_{\om} ( \bar{\Om} - R \pa_R \bar{\Om}) + ( \al c_{\om} R\pa_R - ( \uu \cdot  \na ) ) \bar{\Om}+ \bar{F}_{\Om} + N_{\om} , \\
\eta_t + (1+3\al) R \pa_R \eta + (\bar{\uu} \cdot \na ) \eta &= (-2  - \bar{u}_x) \eta 
 - u_x  \bar{\eta}+ c_{\om} ( 2\bar{\eta} - R \pa_R \bar{\eta})   \\
&+( \al c_{\om} R\pa_R - ( \uu \cdot  \na ) )\bar{\eta} -  v_x \bar{\xi} -  \bar{v}_x \xi+ \bar{F}_{\eta} + N_{\eta} , \\
\xi_t +  (1+3\al) R \pa_R \xi + (\bar{\uu} \cdot \na)  \xi & = (- 2 -  \bar{v}_y ) \xi  - v_y  \bar{\xi} + c_{\om} ( 2 \bar{\xi} - R\pa_R \bar{\xi})\\
 & +( \al c_{\om} R\pa_R - ( \uu \cdot  \na ) ) \bar{\xi} - u_y \bar{\eta} - \bar{u}_y \eta + \bar{F}_{\xi} + N_{\xi} ,\\
\eal
\eeq
where we have used $\bar{c}_l = 1/\al + 3, \bar{c}_{\om} = -1$ \eqref{eq:profile}, $\al c_l(t) = c_{\om}(t) - \al c_{\om}(t)$ \eqref{eq:normal} and $- \al c_l R \pa_R \bar{g}  = - c_{\om} R\pa_R \bar{g}
+ \al c_{\om} R\pa_R \bar{g}$ for $g = \bar{\Om}, \bar{\eta}, \bar{\xi}$. The error $\bar{F}_{\Om}, \bar{F}_{\eta}, \bar{F}_{\xi}$ are defined in \eqref{eq:error} 
and the nonlinear terms are defined below 
\beq\label{eq:non}
\bal
N_{\Om} & = c_{\om} \Om + \eta   -  \al c_l R \pa_R \Om -  (\uu \cdot \na)  \Om , \\
N_{\eta} & =  (2 c_{\om} -  u_x) \eta - v_x \xi - \al c_l R \pa_R \eta - (\uu \cdot \na)  \eta , \\
N_{\xi} & = (2 c_{\om} -  v_y ) \xi -  u_y \eta -  \al c_l R \pa_R \xi  - (\uu \cdot \na)  \xi .\\
\eal
\eeq
We focus on the linearized equation of \eqref{eq:lin}. From \eqref{eq:trans2} and \eqref{eq:simp4}, we have
\beq\label{eq:trans3}
3 \al R\pa_R + \bar{u} \cdot  \na 
= 2 \bar{\Psi} \pa_{\b}  + \lt\{ -\al R \pa_{\b} \bar{\Psi} \pa_{R }
 + \al R \pa_R \bar{\Psi} \pa_{\b} \rt\}
 = \f{3 \sin (2 \b)}{1 + R} \pa_{\b} + l.o.t. . 
\eeq
We will justify the above decomposition using integration by parts to avoid loss of derivatives. We will also show that 
\beq\label{eq:trans32}
(\al c_{\om} R \pa_R - (\uu \cdot  \na)  ) \bar{\Om}, \quad (\al c_{\om} R \pa_R - (\uu \cdot  \na)  )\bar{\eta}, \quad (\al c_{\om} R \pa_R - (\uu \cdot  \na)  )\bar{\xi} 
\eeq
in \eqref{eq:lin} are lower order terms. Moreover, we will justify that $\bar{\xi}$ 
is small and is of order $\al^2$ in Lemma \ref{lem:xi} so that we can treat $v_x \bar{\xi}$ as a lower order term in the $\eta$ equation. 

Using \eqref{eq:simp3}, \eqref{eq:simp4}, \eqref{eq:trans3}, \eqref{eq:trans32} and then collecting the lower order terms with a small factor $\al$, the error terms $\bar{F}$ and the nonlinear terms $N$ in the remaining term $\cR$, we derive the leading order terms in the linearized equations 
\begin{align}
\Om_t +  R \pa_R \Om +\f{3 \sin (2 \b)}{1 + R} \pa_{\b} \Om &= - \Om +  \eta + c_{\om} ( \bar{\Om} - R \pa_R \bar{\Om}) + \cR_{\Om} , \label{eq:lin21}\\ 
\eta_t +  R \pa_R \eta + \f{3 \sin (2 \b)}{1 + R} \pa_{\b} \eta &= (-2  + \f{3}{1+R} ) \eta +  \f{2}{\pi \al} L_{12}(\Om)  \bar{\eta}+ c_{\om} ( 2\bar{\eta} - R \pa_R \bar{\eta}) + \cR_{\eta} ,\label{eq:lin22}\\
\xi_t + R \pa_R \xi + \f{3 \sin (2 \b)}{1 + R} \pa_{\b}  \xi &= (-2  - \f{3}{1+R}) \xi - \f{2}{\pi\al} L_{12}(\Om)  \bar{\xi} + c_{\om} ( 2\bar{\xi} - R \pa_R \bar{\xi}) + \cR_{\xi} \label{eq:lin23} ,
\end{align}
where the full expansion of $\cR$ is given in \eqref{eq:lin3} and their estimates are deferred to Section \ref{sec:non}. In the following subsections, we establish the linear stability for \eqref{eq:lin21}-\eqref{eq:lin23}. 
The contribution of $\cR$ is small.
Using this property, we can further establish the nonlinear stability of the approximate profile \eqref{eq:profile} using a bootstrap argument.

We introduce the following notation
\beq\label{eq:nota_ux}
\td{L}_{12}(\Om)(R) \teq L_{12}(\Om)(R) - L_{12}(\Om)(0)
= -\int_0^R \int_0^{\pi/2} \f{ \Om(s, \b) \sin(2\b) }{ s} d\b dx.
\eeq
According to the normalization condition of $c_{\om}$ \eqref{eq:normal}, we can simplify 
\beq\label{eq:nota_ux2}
c_{\om} + \f{2}{\pi \al} L_{12}(\Om)(R)  = \f{2}{\pi\al} \td{L}_{12} (\Om)(R) .
\eeq

\begin{definition}\label{def:op} We define the differential operators 
\[
D_R = R \pa_R , \quad D_{\b} = \sin(2\b) \pa_{\b}
\]
and the linear operators $\cL_i$
\beq\label{eq:op1}
\bal
\cL_1(\Om, \eta) &\teq - D_R \Om - \f{3}{1 + R} D_{\b} \Om -   \Om +  \eta + c_{\om} ( \bar{\Om} - D_R \bar{\Om}) ,\\
\cL_2(\Om, \eta ) & \teq  - D_R \eta - \f{3}{1 + R} D_{\b} \eta + (-2  + \f{3}{1+R} ) \eta + \f{2}{\pi \al } \td{L}_{12}(\Om)  \bar{\eta}+ c_{\om} ( \bar{\eta} - D_R \bar{\eta})  ,\\
\cL_3(\Om, \xi) & \teq  - D_R \xi - \f{3  }{1 + R} D_{\b} \xi+ (-2  - \f{3}{1+R}) \xi 
- \f{2}{\pi \al} \td{L}_{12}(\Om) \bar{\xi} +c_{\om} ( 3\bar{\xi} - D_R \bar{\xi}  ),
 \eal
\eeq
where $\td{L}_{12}(\Om) $ is defined in \eqref{eq:nota_ux} and $\bar{\Om}, \bar{\eta}$ are defined in \eqref{eq:profile}. 
Define the local part of $\cL_i$ by eliminating $c_{\om}, \td{L}_{12}(\Om)$
\beq\label{eq:op2}
\bal
\cL_{10}(\Om, \eta) &\teq  - D_R \Om - \f{3 }{1 + R} D_{\b} \Om -   \Om +  \eta, \quad \cL_{20}(\eta) \teq - D_R \eta - \f{3}{1 + R} D_{\b} \eta + (-2  + \f{3}{1+R} ) 
\eta , \\
\cL_{30}(\xi) & \teq  - D_R \xi - \f{3}{1 + R} D_{\b} \xi+ (-2  - \f{3}{1+R}) \xi .
\eal
\eeq
\end{definition}
With the above notations, \eqref{eq:lin21}-\eqref{eq:lin23} can be reformulated as 
\beq\label{eq:equiv}
\Om_t  = \cL_1(\Om, \eta) + \cR_{\Om} , \quad  \eta_t = \cL_2(\Om, \eta) + \cR_{\eta} , \quad 
\xi_t = \cL_3(\xi) + \cR_{\xi},
\eeq
where we have used the following identities to rewrite the $L_{12}(\Om),c_{\om}$ terms
in \eqref{eq:lin22}-\eqref{eq:lin23}
\[
\bal
 \f{2 L_{12}(\Om)  }{\pi \al}  \bar{\eta} +  c_{\om} ( 2 \bar{\eta} - D_R \bar{\eta})
 &=  \f{2 \td{L}_{12}(\Om) }{\pi \al} \bar{\eta} +   c_{\om} (  \bar{\eta} - D_R \bar{\eta}),  \\
  - \f{2 L_{12}(\Om)  }{\pi \al}  \bar{\xi} +  c_{\om} ( 2\bar{\xi} -D_R \bar{\xi} )
 &=  -\f{2 \td{L}_{12}(\Om) }{\pi \al} \bar{\eta}  + c_{\om} ( 3\bar{\xi} -D_R \bar{\xi} ) .
 \eal
\]

\subsubsection{Key observations }\label{sec:keyob}
There are several key observations that play a crucial role in our analysis. Firstly, the leading order terms in the $\Om, \eta$ equations
\eqref{eq:lin21}-\eqref{eq:lin22} do not couple the $\xi$ term, which is consistent with our derivation for the leading order system \eqref{eq:sys2}. 

Secondly, in the $\xi$ equation, the coupling between $\Om$ and $\xi$ through the nonlocal term $L_{12}(\Om)$ and $c_{\om}$ \eqref{eq:normal} is weak 
due to the fact that $\bar{\xi}$ is much smaller than $\bar{\Om}, \bar{\eta}$. After removing these nonlocal terms,
\eqref{eq:lin23} only involves local terms about $\xi$. By choosing a suitable singular weight, we will show that $\xi$ is linearly stable up to the weak nonlocal term. 

Thirdly, all the nonlocal terms in \eqref{eq:lin21}-\eqref{eq:lin22}, e.g. $c_{\om}, L_{12}(\Om)$, have coefficients with small angular derivative. For example, using \eqref{eq:profile}, we have
\beq\label{eq:nonlocal1}
c_{\om} (\bar{\Om} - R \pa_R \bar{\Om})  =c_{\om} \cdot \f{\al}{c} \G(\b) \f{6R^2}{ (1+R)^3}.
\eeq
We can apply the weighted angular derivative to gain a small factor $\al$
\[
| \sin(2\b) \pa_{\b} \G(\b)| = |2  \al \sin^2(\b) \G(\b)| \leq 2 \al \G(\b).
\]
A similar observation and estimate have been obtained in \cite{elgindi2019finite} for a different $\G$.

\subsubsection{The angular transport term}\label{sec:angle}
To understand the effect of the angular transport term in \eqref{eq:lin21}-\eqref{eq:lin23}, we choose a weight $\vp(R,\b) = A(R) (\sin(\b) )^{-\g_1} (\cos(\b))^{-\g_2}$ and then perform the $L^2$ estimate and use integration by parts to obtain 
\[
 \f{1}{2} \f{d}{dt} \la \Om^2 , \vp \ra = - \B\la \f{3 \sin (2 \b)}{1 + R} \pa_{\b} \Om , \Om \vp(R, \b)  \B\ra  + \textrm{ other terms (o.t.) }  
 = \B\la \f{ 3  (\sin(2\b) \vp)_{\b}}{2(1+R) \vp} , \Om^2 \vp \B\ra + o.t..
\]
It is not difficult to show that 
\[
\f{3  (\sin(2\b) \vp)_{\b}}{2(1+R) \vp} \B|_{R= 0} = 3(1 - \g_1) \cos^2(\b) - 3(1-\g_2) \sin^2(\b).
\]
Suppose that $\g_1, \g_2 \leq 1$. If $\b$ is small, the angular transport term contributes a growing factor $3(1 - \g_1) > 0$ to the energy norm.

To establish the linear stability, it is natural to first establish the (weighted) $L^2$ estimate of \eqref{eq:lin21}-\eqref{eq:lin23}. However, the above argument shows that for small $\b > 0$ the angular transport term destabilizes the profile of the singularity using the singular weights $A(R) (\sin(\b))^{-\g_1} (\cos(\b))^{-\g_2}$ with $\g_1 \leq 1$. A possible approach to address this issue in the estimate is to choose $\g_1$ close to or larger than $1$, i.e. a very singular weight in the $\b$ direction is desired. 
In \cite{elgindi2019finite}, $\g_1$ is chosen to be close to $1$ so that such growing factor is minimized. 
For \eqref{eq:lin21}-\eqref{eq:lin23}, due to the presence of the nonlocal term, e.g. $c_{\om} (\bar{\Om} - R \pa_R \bar{\Om})$, which only vanishes of order $\sin(2\b)^{\al /2}$ near $\b = 0, \pi/2$, if we use a very singular weight for the angular component $\b$, such nonlocal term will be very difficult to control. 

To handle the angular transport term in the $L^2$ estimate, we observe that $ \sin(2\b) \pa_{\b}\bar{\Om}$ is small since $\bar{\Om}$ varies slowly in $\b$. We expect that a similar smallness result holds for the perturbation term $ \sin(2\b) \pa_{\b}\Om$ and we will justify it in Section \ref{sec:stab_angle}. 
This observation motivates us \textit{not} to perform integration by parts for the angular transport term in the weighted $L^2$ estimate.

\subsection{Outline of the linear stability analysis}\label{sec:L2_outline}

We decompose the linear stability analysis of \eqref{eq:lin21}-\eqref{eq:lin23}, or equivalently \eqref{eq:equiv} into several steps. Based on the first observation in Section \ref{sec:keyob}, we separate the estimates of the system of $\Om, \eta$ \eqref{eq:lin21}-\eqref{eq:lin22} and the equation of $\xi$ \eqref{eq:lin23}.

In Section \ref{sec:L2_local}, we estimate the local part of the linearized operators $\cL_i$ \eqref{eq:op1}, i.e. $\cL_{i0}$ \eqref{eq:op2}.  The argument is mainly based on integration by parts. 

Instead of first performing the weighted $L^2$ estimate of the system, we perform the weighted 
$L^2$ estimate of the angular derivative in Section \ref{sec:stab_angle}. The motivation is that using the third observation in Section \ref{sec:keyob}, we gain a small factor $\al^{1/2}$ for the nonlocal terms in the equations of $D_{\b} \Om, D_{\b} \eta$. Therefore, we can treat the nonlocal terms as small perturbations and use the estimates of $\cL_{i0}$ in Section \ref{sec:L2_local} to establish the estimates of $D_{\b} \Om, D_{\b} \eta$. See also the motivation in Section \ref{sec:angle}. Once we obtain the estimates of $D_{\b} \Om, D_{\b} \eta$, we can treat the angular transport terms in the weighted $L^2$ estimates of the equations of $\Om, \eta$ \eqref{eq:lin21}-\eqref{eq:lin22} as perturbations. This overcomes the difficulty discussed in Section \ref{sec:angle}. 

In Section \ref{sec:L2_idea}, we use two models to illustrate the cancellations in \eqref{eq:lin21},\eqref{eq:lin22}, which are crucial for the estimates of $\td L_{12}(\Om), c_{\om}$. This motivates several technical estimates in Section \ref{sec:L2_less}.

In Section \ref{sec:L2_less}, we establish the weighted $L^2$ estimates of $\Om, \eta$ with less singular weights, and obtain the damping terms for $c_{\om}, \td{L}_{12}(\Om)$. We design the less singular weights carefully to fully exploit the cancellations discussed in Section \ref{sec:L2_idea}. 
This is the most difficult part in the whole analysis. 

After we obtain the damping terms for $c_{\om}, \td{L}_{12}(\Om)$, we can treat the nonlocal terms in \eqref{eq:lin21}-\eqref{eq:lin22} as perturbations. Using the estimates of the local operators $\cL_{i0}$ in Section \ref{sec:L2_local}, we further establish weighted $L^2$ estimates of $\Om, \eta$ with more singular weights that are introduced in \cite{elgindi2019finite} in Section \ref{sec:L2_more}. This enables us to apply several key estimates in \cite{elgindi2019finite} in our nonlinear estimates and simplify the whole estimates.

From the second observation in Section \ref{sec:keyob}, we treat the nonlocal terms in the $\xi$ equation \eqref{eq:lin23} as small perturbations. We estimate $D_{\b}\xi,\xi$ in Section \ref{sec:L2_xi} using the estimate of $\cL_{30}$ in Section \ref{sec:L2_local}.

\subsection{Estimates of $\cL_{10}, \cL_{20}, \cL_{30}$}\label{sec:L2_local}
We first introduce several singular weights that will be used throughout the paper.

\begin{definition}\label{def:wg}
Define $\vp_i, \psi_i$ by 
\beq\label{wg}
\bal
\vp_1 & \teq \f{(1+R)^4}{R^4}  \sin(2\b)^{ - \s}, \quad 
\vp_2 \teq \f{(1+R)^4}{R^4}  \sin(2\b)^{ - \g} , \\
\psi_1  & \teq \f{(1+R)^4}{R^4} (  \sin(\b)\cos(\b) )^{-\s},  \quad
\psi_2    \teq \f{(1+R)^4}{R^4}  \sin(\b)^{-\s}  \cos(\b)^{-\g} , \\
\eal
\eeq
where $\s = \f{99}{100},  \g = 1 + \f{\al}{10}$.
\end{definition}

The weights $\vp_1, \vp_2$ have been introduced in \cite{elgindi2019finite} for stability analysis. 

The weights $\vp_1$ and $\psi_1$ are essentially the same. We introduce $\psi_1$ for consistency and the following reasons. Firstly, we will apply the weights $\vp_i$ to $\Om, \eta$ and the weights $\psi_i$ to $\xi$. In particular, we will construct weighted $H^3$ norm $\cH^3(\vp)$ for $\Om,\eta$ and $\cH^3(\psi)$ for $\xi$ in \eqref{norm:H22}. Secondly, $\vp_1$ and $\vp_2$ have similar forms, and $\psi_1$ and $\psi_2$ also have similar forms.
It is easy to see that $\vp_1  \les \vp_2, \psi_1 \les \psi_2$. We choose $\psi_2$ less singular than $\vp_2$ for $\b$ close to $0$ since $\bar{\xi}$ does not decay in $R$ when $R\sin(\b)^{\al}$ is fixed and $\b$ is small. See Lemma \ref{lem:xi} regarding the estimate of $\bar{\xi}$.

Recall $\cL_{10}, \cL_{20}, \cL_{30}$ \eqref{eq:op2} in Definition \ref{def:op}. The following Lemmas will be used repeatedly. 
\begin{lem}\label{lem:dp}
For some $\d ,\d_1, \d_2> 0$, consider the weights 
\beq\label{wg:build}
\bal
\vp(R, \b) &=  \f{(1+R)^4}{R^4} (\sin(2\b))^{-\d} , \  \psi(R,\b) =  \f{(1+R)^4}{R^4} ( \sin(\b))^{-\d_1} (\cos(\b))^{-\d_2} .
\eal
\eeq
Assume $ \vp^{1/2} \Om, \vp^{1/2} \eta  \in L^2$. We have 
\beq\label{eq:dp1}
\bal
\la \cL_{10}(\Om, \eta) , \Om \vp \ra  +  \la \cL_{20}(\eta) , \eta \vp \ra 
&\leq ( -\f{1}{4} + 3|1- \d |  ) (  || \Om \vp^{1/2}||_2^2 + || \eta \vp^{1/2}||_2^2).
\eal
\eeq
Assume that $ \psi^{1/2} \xi \in L^2$. Denote $a \vee b \teq \max(a,b)$. Then it holds true that
\beq\label{eq:dp2}
\la \cL_{30}(\xi) , \xi \psi \ra   \leq  \B(-\f{1}{2} + 3 ( |1- \d_1 |\vee |1-\d_2| ) \B) 
|| \xi \psi^{1/2}||_2^2 .
\eeq

\end{lem}

We will apply Lemma \ref{lem:dp} to the singular weights in Definition \ref{def:wg}, i.e. $\vp = \vp_1$ or $\vp_2$ and $\psi = \psi_1$ or $\psi_2$. Hence, the exponents we will use are $\d = \s =\f{99}{100}$ or $\d = \g = 1 + \f{\al}{10}$, $\d_1 = \s$, $\d_2 = \s$ or $\d_2 = \g$. Since these exponents are very close to $1$, 
we have the order $|1-\d| \approx 0, |1-\d_1| \vee |1-\d_2| \approx 0$. The reader can regard the terms $|1-\d|, |1-\d_1| \vee |1-\d_2| \approx 0$.

\begin{remark}
The constant $-\f{1}{4}$ in \eqref{eq:dp1} can be improved to $- \f{1}{2}+ \e$ for any $\e> 0$ by  considering $ \lam_{\e} \la \cL_{10}(\Om, \eta) , \Om \vp \ra  +  \la \cL_{20}(\eta) , \eta \vp \ra $ for some $\lam_{\e} > 0$, and $ -\f{1}{2}$ in \eqref{eq:dp2} can be improved to $-\f{3}{2}$. Yet, we do not need these sharper estimates.
\end{remark}


\begin{proof}[Proof of Lemma \ref{lem:dp}]
By definition of $\vp, \psi$, we have
\beq\label{eq:dp_comp}
\bal
\f{  (3\sin(2\b) \vp)_{\b} }{2  (1+R)  \vp}  & = \f{3}{2(1+R)}  \f{( \sin(2\b)^{1-\d})_{\b}}{ \sin(2\b)^{-\d} } = \f{3 \cos (2\b) \cdot (1-\d)}{1+R}  \leq  3|1 - \d|, \\
\f{  (3\sin(2\b) \psi)_{\b} }{2  (1+R)  \psi}  & = \f{3}{(1+R)}  
 \f{ (  \sin(\b)^{1-\d_1} \cos(\b)^{1-\d_2} )_{\b} }{ \sin(\b)^{-\d_1} \cos(\b)^{-\d_2}}   \\
&=\f{3}{1+R} ( (1-\d_1) \cos^2(\b) - (1-\d_2) \sin^2(\b) ) \leq 3\max( |1-\d_1|, |1-\d_2| ),\\
\f{ (R \vp)_{R}}{ 2 \vp} &=\f{ (R \psi)_{R}} { 2 \psi} =\lt( \f{(1+R)^4}{R^3}  \rt)_R \f{ R^4 }{2 (1+R)^4}   
= \f{2R}{1 + R} - \f{3}{2} = \f{1}{2}  - \f{2 }{1+R} .  \\
\eal
\eeq

Using integration by parts for the transport terms in $\cL_{10}$ \eqref{eq:op2}, we yield 
\[
\la -D_R \Om, \Om \vp \ra 
= \B\la - R \vp, \f{1}{2} \pa_R \Om^2 \B\ra
=  \B\la \f{1}{2}(R\vp)_R, \Om^2 \B\ra 
= \B\la \f{ (R\vp)_R}{2 \vp }, \Om^2 \vp \B\ra.
\]
Similar calculation applies to $-\f{3}{1+R} D_{\b} \Om $ in $\cL_{10}$. Using  the above calculations, we get 
\[
\bal
&\la \cL_{10}(\Om, \eta), \Om \vp \ra  = \B\la \f{ (R \vp)_{R}}{ 2 \vp}  + \f{  (3\sin(2\b) \vp)_{\b} }{2  (1+R)  \vp} ,   \Om^2 \vp \B\ra   - \la \Om,  \Om \vp \ra + \la \Om, \eta \vp \ra \\
\leq &  \B\la   \f{1}{2}  - \f{2}{1+R} + 3 |1 -\d| -1, \  \Om^2 \vp \B\ra  + \la \Om, \eta \vp \ra
= \B\la  - \f{1}{2}  - \f{2}{1+R} + 3 |1 -\d| , \  \Om^2 \vp \B\ra  + \la \Om, \eta \vp \ra.
\eal
\]
Similarly, using integration by parts for the transport terms in $\cL_{20}$ \eqref{eq:op2} and 
\eqref{eq:dp_comp}, we get 
\beq\label{eq:dp_comp2}
\bal
&\la \cL_{20}( \eta), \eta \vp \ra 
= \B\la \f{ (R \vp)_{R}}{ 2 \vp}  + \f{  (3\sin(2\b) \vp)_{\b} }{2  (1+R)  \vp}, \eta^2 \vp \B\ra  + \B\la (-2 + \f{3}{1+R}) , \eta^2 \vp \B\ra \\
 \leq & \B\la   \f{ 2R}{1+R} - \f{3}{2}   + 3 |1 - \d| + (-2 + \f{3}{1+R}), \eta^2 \vp \B\ra 
 = \B\la - \f{1}{2} - \f{R}{1+R}  + 3 |1 - \d| , \eta^2 \vp \B\ra.
\eal
\eeq
We estimate the interaction term between $\Om, \eta$. Note that 
\[
4 ( \f{1}{4} + \f{2}{1+R} ) ( \f{1}{4} + \f{ R }{ 1+R})  > \f{2}{1+R} + \f{R}{1+R} \geq 1.
\]
Using the Cauchy-Schwarz inequality implies
\[
\la  \Om, \eta \vp  \ra  
\leq  \B\la  \f{1}{4} + \f{2}{1+R} ,   \Om^2    \vp  \B\ra
+ \B\la  \f{1}{4} + \f{R}{1+R}  ,  \eta^2  \vp \B\ra .
\]
Combining the above estimates, we prove 
\[
\bal
\la \cL_{10}(\Om, \eta), \Om \vp \ra &+\la \cL_{20}(\Om, \eta), \eta \vp \ra  
\leq  \B\la  - \f{1}{2}  - \f{2}{1+R} + 3 |1 -\d| , \  \Om^2 \vp \B\ra + 
\B\la - \f{1}{2} - \f{R}{1+R}+ 3 |1 - \d| , \eta^2 \vp \B\ra \\
&+  \B\la  \f{1}{4} + \f{2}{1+R}  , \Om^2    \vp   \B\ra +  \B\la   \f{1}{4} + \f{R}{1+R}   , \eta^2  \psi \B\ra
\leq \lt( -\f{1}{4} + 3|1- \d | \rt)  (  || \Om \vp^{1/2}||_2^2 + || \eta \vp^{1/2} ||_2^2).
\eal
\]
Recall $\cL_{30}$ in Definition \ref{def:op}. For \eqref{eq:dp2}, we use the computations \eqref{eq:dp_comp}-\eqref{eq:dp_comp2} to obtain 
\[
\bal
&\la \cL_{30}( \xi), \xi \psi \ra 
= \B\la \f{ (R \psi)_{R}}{ 2 \psi}  + \f{  (3\sin(2\b) \psi)_{\b} }{2  (1+R)  \psi}, \xi^2 \psi \B\ra  + \B\la (-2 - \f{3}{1+R}) , \xi^2 \psi \B\ra \\
 \leq & \B\la   \f{ 2R}{1+R} - \f{3}{2}   + 3 ( |1- \d_1 | \vee|1-\d_2|) + (-2 - \f{3}{1+R}), \xi^2 \psi \B\ra 
 \leq \B( - \f{1}{2} + 3 ( |1- \d_1 | \vee |1-\d_2|)   \B)  || \xi \psi^{1/2}||_2^2 .
 \eal
\]
\end{proof}

\subsection{Weighted $L^2$ estimate of the angular derivative $D_{\b}\Om, D_{\b}\eta$}\label{sec:stab_angle}


\begin{definition}\label{def:H1_b1}
Define an energy $E(\b,1) \geq 0 $ and a remaining term $\cR(\b,1)$ by
\beq\label{eg:b1}
E(\b,1) (\Om, \eta)  \teq  \lt(  || D_{\b} \Om \vp_2^{1/2} ||_2^2   + ||D_{\b} \eta \vp_2^{1/2}||_2^2 \rt)^{1/2} ,  \
\cR(\b, 1) \teq \la D_{\b} \cR_{\Om}, D_{\b}\Om\vp_2 \ra +  \la D_{\b} \cR_{\eta}, D_{\b}\eta \vp_2 \ra .
\eeq
\end{definition}

To simplify the notations, we drop $\Om, \eta$ in $E(\b, 1)$. 
The main result in this subsection is the following. This proposition enables us to treat the angular transport terms in \eqref{eq:lin21}-\eqref{eq:lin23} as perturbations. A similar estimate has been established in \cite{elgindi2019finite}.   
\begin{prop}\label{prop:b1}
Assume that $ \vp_2^{1/2}  D_{\b} \Om, \ \vp_2^{1/2} D_{\b}\eta \in L^2$. We have 
\beq\label{eq:H1_b5}
\bal
&\la D_{\b} \cL_1 (\Om, \eta), ( D_{\b} \Om ) \vp_2 \ra + \la  D_{\b} \cL_2( \Om , \eta)  , ( D_{\b}\eta) \vp_2 \ra \\
\leq & - ( \f{1}{5}  -\al) (E(\b,1) )^2 + 
C\al (L^{2}_{12}(\Om)(0) + || \td{L}_{12}(\Om) R^{-1}||^2_{L^2(R)} ),
\eal
\eeq
where $\cL_1, \cL_2$ are defined in Definition \ref{def:op}.
\end{prop}

We will use the following basic property of $D_{\b} = \sin(2\b) \pa_{\b}, \ \G(\b) = \cos(\b)^{\al}$ repeatedly 
\beq\label{eq:Dg}
D_{\b} \G(\b)= -2\al \sin^2(\b)\cos^{\al}(\b) =- 2\al \sin^2(\b) \G(\b), \quad
|D_{\b}\G(\b)| \leq  2\al  \sin(\b)\G(\b).
\eeq

\begin{proof}
Notice that the angular transport term in \eqref{eq:lin21}-\eqref{eq:lin22} can be written as $\f{3}{1+R} D_{\b}$ and that $D_{\b}$ commutes with the derivatives in \eqref{eq:lin21}-\eqref{eq:lin22} and $\cL_{10}, \cL_{20}$ \eqref{eq:op2}. We have 
 \beq\label{eq:H1_b1}
\bal
D_{\b} \cL_1(\Om, \eta) &=D_{\b}( \cL_{10}( \Om, \eta) + c_{\om}  D_{\b} (\bar{\Om} - R \pa_R \bar{\Om}) )= \cL_{10} (D_{\b} \Om, D_{\b}\eta) + c_{\om}  D_{\b} (\bar{\Om} - R \pa_R \bar{\Om}) ,\\
D_{\b} \cL_2(\Om, \eta)  &=D_{\b}( \cL_{20}( \Om, \eta) + \f{2}{\pi \al} \td{L}_{12}(\Om)  \bar{\eta}
+ c_{\om}   (\bar{\Om} - R \pa_R \bar{\Om})    ) \\
&= \cL_{10} (D_{\b} \Om, D_{\b}\eta) + \f{2}{\pi \al} \td{L}_{12}(\Om) D_{\b} \bar{\eta}  +c_{\om}  D_{\b} (\bar{\eta} - R \pa_R \bar{\eta}) ,
\\
\eal
\eeq
where we have used \eqref{eq:nota_ux2}. Applying Lemma \ref{lem:dp} with $\vp = \vp_2$ and $\d = \g = 1 + \f{\al}{10}$, we derive 
\beq\label{eq:H1_b2}
\bal
&\la \cL_{10} (D_{\b} \Om, D_{\b}\eta) , (  D_{\b} \Om ) \vp_2 \ra
+ \la \cL_{20} (D_{\b} \Om, D_{\b}\eta) , (  D_{\b} \eta ) \vp_2 \ra \\
\leq&  (- \f{1}{4} + 3 | 1 -\g | )  \B(   ||  D_{\b} \Om \vp_2^{1/2}||_2^2 
+   || D_{\b} \eta   \vp_2^{1/2}||_2^2   \B) 
\leq (- \f{1}{4} + \al )  \B(   ||  D_{\b} \Om \vp_2^{1/2}||_2^2 
+   || D_{\b} \eta   \vp_2^{1/2}||_2^2   \B).
\eal
\eeq

Recall $c_{\om} = -\f{2}{\pi\al} L_{12}(\Om)(0)$. Using \eqref{eq:bar_ing} in Lemma \ref{lem:bar} and the Cauchy-Schwarz inequality, we obtain 
\beq\label{eq:H1_b23}
\bal
&|\la c_{\om}  D_{\b} (\bar{\Om} - R \pa_R \bar{\Om}) ,  (D_{\b}\Om) \vp_2 \ra|
+| \la c_{\om}   D_{\b} (\bar{\eta} - R \pa_R \bar{\eta}) ,  ( D_{\b}\eta) \vp_2 \ra| \\
\les & \al^{1/2} |L_{12}(\Om)(0)|  (  || D_{\b} \Om \vp_2^{1/2} ||_2^2 + || D_{\b}\eta \vp_2^{1/2}||_2^2 )^{1/2}.
\eal
\eeq

Recall the notation $\td{L}_{12}(\Om)$ \eqref{eq:nota_ux}. Applying Lemma \ref{lem:ux} and \eqref{eq:bar_ux} in Lemma \ref{lem:bar}, we derive 
\[
\B|\B| \f{2}{\pi \al} \td{L}_{12}(\Om) D_{\b}\bar{\eta}  \vp_2^{1/2} \B|\B|_2
\les \al || \td{L}_{12}(\Om) R^{-1} ||^2_{L^2(R)} .
\]
Therefore, using the Cauchy-Schwarz inequality, we yield 
\beq\label{eq:H1_b3}
\bal
\la  \f{2}{\pi \al} \td{L}_{12}(\Om)  D_{\b}\bar{\eta},  D_{\b} (\eta) \vp_2 \ra
& \les \al^{1/2} || \td{L}_{12}(\Om) R^{-1} ||_{L^2(R)}  
|| D_{\b} \eta \vp_2^{1/2}||_2.
\eal
\eeq

Combining \eqref{eq:H1_b2}, \eqref{eq:H1_b23}, \eqref{eq:H1_b3} and adding the inner product about two terms in \eqref{eq:H1_b1}, we prove 
\[
\bal
&\la D_{\b} \cL_1(\Om, \eta) , (D_{\b}\Om) \vp_2 \ra + \la D_{\b} \cL_2(\Om, \eta) , (D_{\b} \eta) \vp_2 \ra
\leq -  ( \f{1}{4}  - \al) (  || D_{\b} \Om \vp_2^{1/2} ||_2^2   + ||D_{\b} \eta \vp_2^{1/2}||_2^2)\\
  &+ C \al^{1/2} | L_{12}(\Om)(0)| 
\B(  || D_{\b} \Om \vp_2^{1/2} ||_2^2   + ||D_{\b} \eta \vp_2^{1/2}||_2^2 \B)^{1/2} + C \al^{1/2} \B| \B| \td{L}_{12}(\Om) R^{-1} \B| \B|_{L^2(R)}  || D_{\b} \eta \vp_2^{1/2}||_2,
  \eal
\]
where $C$ is some absolute constant. Using the notation $E(\b,1)$ \eqref{eg:b1}, the Cauchy-Schwarz inequality concludes the proof of Proposition \ref{prop:b1} (notice that $-1/4 < -1/5$).
\end{proof}

\subsection{Ideas in the estimates of the nonlocal terms}\label{sec:L2_idea}

Recall $c_{\om}, \td L_{12}(\Om)$ from \eqref{eq:normal}, \eqref{eq:nota_ux}
\beq\label{eq:idea_sys2}
c_{\om} = -\f{2}{\pi \al} L_{12}(\Om)(0)
= -\f{2}{\pi \al} \int_0^{\infty}\int_0^{\pi/2} \f{ \Om \sin(2\b)}{R} dR d\b, \ \td L_{12}(\Om)(R) = - \int_0^{R} \int_0^{\pi /2}\f{ \Om \sin(2\b)}{R} dR d \b.
\eeq

The most difficult part in the linear stability analysis of \eqref{eq:equiv},\eqref{eq:op1} (or equivalently \eqref{eq:lin21}-\eqref{eq:lin23}) lies in the nonlocal terms $\td L_{12}(\Om), c_{\om}$. 
Note that the constant in the coercivity estimates of the local part of the linear operators $\cL_i$, i.e. $\cL_{i0}$, is small. For example, this constant is about $-\f{1}{4}$ in Lemma \ref{lem:dp}. We cannot estimate the nonlocal terms in some weighted Sobolev norm and treat them as 
small perturbations since these nonlocal terms are $O(1)$ for small $\alpha$. It is crucial for us to exploit the cancellation among various terms so that we can obtain sharp estimates of these nonlocal terms. 

We use two models to study  $\td L_{12}(\Om)$ and the $c_{\om}$ term. Similar models have been used in our previous work with D. Huang on the asymptotically self-similar blowup  of the Hou-Luo model. 

\subsubsection{Model 1 for nonlocal interaction}\label{sec:model1}

We consider the following coupled system 
\beq\label{eq:model1}
\pa_t \Om = \eta,  \quad \eta_t = \f{2}{\pi \al} \td L_{12}(\Om)  \bar{\eta}
\eeq
to study the cancellation between the nonlocal term $\f{2}{\pi \al} \td L_{12}(\Om)  \bar{\eta}$ in the $\eta$ equation and $\eta$ in the $\Om$ equation in \eqref{eq:equiv}. The above model is derived by dropping other terms in \eqref{eq:equiv}. The profile $\bar \eta$ satisfies $\bar \eta(0, \b) = 0$ and $\bar \eta > 0$ for $R >0$.

The motivation to exploit nonlocal cancellation is inspired by our previous joint works with Huang  on the De Gregorio model \cite{chen2019finite} and the Hou-Luo model for smooth initial data. In these works, the nonlocal cancellations between $Hf$ and $f$, where $H$ is the Hilbert transform, play an important role.

From Lemma \ref{lem:cancel}, we have a similar cancellation between $\td L_{12}(\Om)$ and $\Om$. Roughly speaking, $\td L_{12}(\Om)$ behaves like $- \Om$.  We perform $L^2(\rho_1)$ estimate on $\Om$ and $L^2(\rho_2)$ estimate on $\eta$ for some singular weights $\rho_1, \rho_2$ to be determined and combine both estimates 
\beq\label{eq:est_coupled}
\f{1}{2} \f{d}{dt} ( \la \Om, \Om \rho_1  \ra  + \la \eta, \eta \rho_2 \ra  )
=  \la \Om , \eta \rho_1 \ra + \la \f{2}{\pi \al} \td L_{12}(\Om)  \bar{\eta}, \eta \rho_2 \ra \teq I. 
\eeq
Formally, $I$ is the sum of the projections of $\eta$ onto two opposite directions.
To exploit this cancellation using Lemma \ref{lem:cancel}, we choose $\rho_1 = \sin (2\b) \rho_0, \rho_2 = \lam \f{\al \pi}{ 2\bar \eta} \rho_0$ with some $\lam > 0$ and singular weight $\rho_0$, such as $\rho_0 = R^{-3}, R^{-2}$, to obtain 
\[
I = \la \Om \sin(2\b), \eta \rho_0 \ra 
+ \la \lam \td L_{12}(\Om), \eta \rho_0 \ra =  \la \Om \sin(2\b ) +  \lam \td L_{12}(\Om), \eta \rho_0 \ra.
\]
For $k \in [ \f{3}{2}, 4]$, applying Young's inequality
$ab \leq s a^2 + \f{1}{4 s} b^2$ for some $s>0$, we yield 
\[
I \leq s || ( \Om \sin(2\b ) +  \lam \td L_{12}(\Om) ) R^{-k/2}||_2^2 + (4 s)^{-1}|| \eta \rho_0 R^{k/2}||_2^2 \teq A + B.
\]

If $k-1 > \f{\pi}{2} \lam$, using Lemma \ref{lem:cancel}, we obtain 
\[
A = s || \Om   \sin(2\b)^{1/2} R^{-k/2}||_2^2  - s ( (k-1) \lam -  \f{\pi}{2} \lam^2 ) \B| \B| \td{L}_{12}(\Om) R^{-k/2} \B| \B|^2_{L^2{(R)}} \leq  s || \Om   \sin(2\b)^{1/2} R^{-k/2}||_2^2.
\]

We remark that even estimating the first term in $I$, which is $\la \Om \sin(2\b), \eta \rho_0 \ra $ and does not involve the nonlocal term, we get an upper bound $s || \Om   \sin(2\b)^{1/2} R^{-k/2}||_2^2 + B$. The above calculation shows that by designing the weights $\rho_1, \rho_2$ carefully, we can exploit the nonlocal cancellation and obtain an even better estimate. Moreover, we gain a damping term for $\td L_{12}(\Om)$ from $A$.

We will use similar ideas to estimate the $\td L_{12}(\Om)$ term in the linearized equation \eqref{eq:lin21}-\eqref{eq:lin23}.

\subsubsection{Model 2 for the $c_{\om}$ term}\label{sec:model2}

We consider the following coupled system 
\beq\label{eq:model2}
\quad \pa_t \Om = \eta + c_{\om} \bar g, \quad \pa_t \eta = c_{\om} \bar f, 
\eeq
where $\bar f(0 , \b) = 0, \bar g(0, \b) = 0$,  $\bar f, \bar g >0$ for $R>0$ with $ \bar f R^{-1}, \bar g R^{-1} \in L^1$. Note that the profiles $\bar \eta - R \pa_R \bar \eta, \bar \Om - R \pa_R \bar \Om$ satisfy similar properties. This system models the $c_{\om}$ terms in the $\Om, \eta$ equations in \eqref{eq:equiv} by dropping other terms. 

Denote $W = \sin(2\b) R^{-1}$. Recall $c_{\om}$ in \eqref{eq:idea_sys2}. We have 
\[
c_{\om} = -\f{2}{\pi \al} \la \Om, \sin(2\b) R^{-1} \ra =  -\f{2}{\pi \al} \la \Om, W \ra.
\]

Denote $ B = \f{2}{\pi \al} \la \bar g , W\ra$. By definition, $B > 0$. We derive an ODE for $c_{\om}$ using the $\Om$ equation
\[
  \pa_t \la \Om, W\ra = c_{\om} \la \bar g, W \ra + \la \eta , W \ra
= - \f{2}{\pi \al} \la \Om, W \ra \la \bar g, W \ra  + \la \eta  , W \ra
=- B \la \Om, W \ra + \la \eta  , W \ra.
\]

Multiplying both sides by $(\f{2}{\pi\al})^2 \la \Om, W \ra = -\f{2}{\pi \al}c_{\om} $, we get 
\beq\label{eq:model2_2}
 \f{1}{2} \f{d}{dt} c_{\om}^2 =- B c_{\om}^2 - \f{2}{\pi \al} c_{\om}\la \eta  , W \ra \teq I_1 + I_2.
\eeq

We see that the $c_{\om} \bar g$ term in the $\Om $ equation in \eqref{eq:model2} provides a damping term for $c_{\om}$ in this ODE. In the $L^2(\rho_2)$ estimates of $\eta$ in \eqref{eq:model2}, we have
\[
 \pa_t \la \eta, \eta \rho_2 \ra =  c_{\om} \la \eta , \bar f \rho_2  \ra \teq I_3.
\]

Since $\bar f \rho_2, W >0$, we can exploit the cancellation between the integral $I_2$ in \eqref{eq:model2_2} and $I_3$. By combining the estimates of both terms, we can obtain better estimates of $I_2$ and $I_3$.

In the estimates of  \eqref{eq:lin21}-\eqref{eq:lin23}, we will derive a similar ODE for $c_{\om}$, which provides a damping term for $c_{\om}^2$. This damping term is crucial for us to control the nonlocal $c_{\om}$ terms in \eqref{eq:lin21}-\eqref{eq:lin23}. There is a coupling term $-c_{\om}\la \eta, W \ra$ in this ODE similar to $I_2$ in \eqref{eq:model2_2}. 
Using an idea similar to the one stated above, we will combine the estimates of such term and the $c_{\om}$ term in the $\eta$ equation in \eqref{eq:lin22}.

\subsection{Weighted $L^2$ estimate of $\Om, \eta$ with a less singular weight}\label{sec:L2_less}

In this subsection, we prove Proposition \ref{prop:L2less} to be introduced on the weighted $L^2$ estimate of $\Om, \eta$ with less singular weights.

The proof consists of several steps and we sketch it below. Firstly, we introduce the weights in our weighted estimates and motivate the choices of these weights. In Section \ref{subsec:L2_dp}, we estimate the local part of $\cL_1, \cL_2$ using mainly integration by parts argument, which is similar to that in Section \ref{sec:L2_local}. In Section \ref{subsec:L2_uw}, 
we use some ideas and estimates similar to those in Model 1 to estimate the interaction among $\Om, \eta$ and $\td L_{12}(\Om)$. In Section \ref{subsec:L2_cw1}, we use a direct calculation to estimate the $c_{\om}$ term in the $\Om$ equation in \eqref{eq:equiv}. Due to the special form of the weight $\vp_0$ in \eqref{wg:L2_R0}, the main term in this estimate is a damping term for $L_{12}^2(\Om)(0)$. In Section \ref{subsec:L2_cw2}, we use some ideas and estimates similar to those in Model 2 in Section \ref{sec:model2} to estimate the $c_{\om}$ term in the $\eta$ equation. In Section \ref{subsec:L2_angle}, we estimate the angular transport term in the $\Om, \eta$ equations in \eqref{eq:equiv} and treat it as perturbations. In Sections \ref{subsec:L2_sum}, \ref{subsec:L2_done}, we summarize these estimates, and establish some inequalities to conclude the proof of Proposition \ref{prop:L2less}.

Since the amount of damping in the energy estimate is small, we cannot overestimate several terms and need to track the coefficients in the estimates. Thus the estimates involve several explicit calculations, which will be presented in Appendix \ref{app:comp}. These calculations,\eqref{eq:cancel_coe} and \eqref{eq:cw_count} can also be verified with the aid of {\it Mathematica}. 
\footnote{ The Mathematica code for these calculations can be found via the link 
\url{https://www.dropbox.com/s/y6vfhxi3pa8okvr/Calpha_calculations.nb?dl=0}.} In view of Lemma \ref{lem:one}, in the following estimates, the reader can regard $\G(\b) \approx 1, c \approx \f{2}{\pi}$.

\begin{definition}\label{def:L2} To exploit the cancellation of the system, we define the following weights 
\beq\label{wg:L2_R0}
\bal
 \psi_0 & \teq \f{9}{8} \f{\al}{c \bar{\eta}}\lt(R^{-3} + \f{3}{2} \f{1+R}{R^2}    \rt) =\f{3}{16}\lt( \f{(1+R)^3}{R^4} +\f{3}{2} \f{(1+R)^4}{R^3}  \rt) \G(\b)^{-1} , \\ 
\vp_0  &\teq \f{(1+R)^3}{R^3} \sin(2\b), \quad \rho \teq  R^{-3 } + R^{-2 },
\eal
\eeq
where $\bar{\eta},  \ \G(\b) = \cos^{\al}(\b) $ are given in \eqref{eq:profile}.
\end{definition}
Compared to $\vp_2$ in \eqref{wg}, the above weights are less singular in the $R, \b$ components.

\subsubsection{The forms of the singular weights}\label{sec:wg_form}
There are several considerations to choose  the above weights $\psi_0, \vp_0$. Firstly, to obtain the damping terms in the energy estimate similar to that in Lemma \ref{lem:dp}, the weights in the $R$ direction can be a linear combinations of $R^{-k}$ with various $k$ \cite{chen2019finite,elgindi2019finite}. See also Lemma \ref{lem:dp}.  For $R$ near $0$, we need the weight to be singular, e.g. $R^{-k_1}$ for a large $k_1$. For very large $R$, we need the weight with slow decay, e.g. $R^{-k_2}$ with small $k_2$. However, using only these two powers $R^{-k_1}$ and $R^{-k_2}$ are not sufficient. Suppose that we use a weight $ \vp_0 = R^{-k_2} + c R^{-k_1}$ with well chosen $k_1, k_2, c$. Applying a calculation similar to that in \eqref{eq:dp_comp2} in Lemma \ref{lem:dp} to $\la L_{20} \eta, \eta \vp_0 \ra$, we can obtain $\la D, \eta^2 \vp_0\ra$ for some coefficient $D(R,\b)$. However, $D$ may not be negative in the whole domain as the one that we obtain in \eqref{eq:dp_comp2} or $| D(R,\b) |$ with $R= O(1)$ may become much smaller than $|D(0,\b)|$ and $|D(\infty, \b)|$. In either case, we cannot establish linear stability since the nonlocal terms are not small. Therefore, we need to add several powers $R^{-k}$ in $\vp_0, \psi_0$. The first formula of $\psi_0$ in \eqref{wg:L2_R0} is more important than the second, and it contains three different powers. 

Secondly, we add $\bar \eta$ in the denominator in $\psi_0$ to cancel the variable coefficient in our energy estimates, and design $\vp_0$ with the factor $\sin(2\b)$. These forms are similar to that in Model 1 in Section \ref{sec:model1}, where we choose $\rho_1 = \sin (2\b) \rho_0, \rho_2 =  \f{c}{\bar \eta} \rho_0$ for some weight $\rho_0$. These special forms are important and enable us to combine the estimates among $L_{12}(\Om), \Om$ and $\eta$.
This is the most important motivation in designing $\psi_0, \vp_0$ in \eqref{wg:L2_R0}. See Model 1 in Section \ref{sec:model1} and estimate \eqref{eq:L2_30}. 


Thirdly, we choose $\vp_0$ with the factor $\f{(1+R)^3}{R^3}$ to  derive a damping term for $L_{12}(\Om)(0)$ from the nonlocal term $c_{\om} ( \bar \Om - D_R \bar \Om)$. See \eqref{eq:L2_41}.



The main result in this section is the following
\begin{prop}\label{prop:L2less}
Define an energy $E(R,0)$ and a remaining term $\cR(R,0)$
\beq\label{eg:R0}
\bal
E(R,0)  &= ( || \Om \vp_0^{1/2}||_2^2 + || \eta \psi_0^{1/2}||_2^2  + \mu_0  L^2_{12}(\Om)(0 ) )^{1/2},
 \\ 
\cR(R, 0) &= \la \cR_{\Om}, \Om \vp_0 \ra + \la \cR_{\eta} ,\eta \psi_0 \ra +  \mu_0 L_{12}(\Om)(0) \la \cR_{\Om}, \sin(2\b) R^{-1} \ra,  \quad  \mu_0   = \f{81}{4\pi c}.
\eal
\eeq
Assume that $\Om, \eta$ satisfies that $E(R,0), E(\b) < + \infty$. For some absolute constant $\mu_1$, we have
\[
\bal
&\f{1}{2} \f{d}{dt} ( ( E(R,0)^2 + \mu_1 E(\b,1)^2 ) )
\leq - ( \f{1}{9} - C\al ) ( ( E(R,0)^2 + \mu_1 E(\b,1)^2 ) )  \\
&-  (4- C\al) L^2_{12}(\Om)(0) - ( \f{1}{4} -C\al) \B| \B| \td{L}^2_{12} \rho^{1/2} \B| \B|^2_{L^2(R)} +  \cR(R, 0) +\mu_1 \cR(\b, 1),
\eal
\]
where the energy $E(\b, 1)$ and the remaining term $\cR(\b, 1)$ are defined in \eqref{eg:b1}.
\end{prop}


Recall $\cL_{1}, \cL_2$ in Definition \ref{def:op}. A direct calculation with weights $\vp_0, \psi_0$ implies
\beq\label{eq:L2_1}
\bal
\la \cL_1(\Om, \eta), \Om \vp_0 \ra
&= -\la  R \pa_R \Om, \Om \vp_0 \ra  - \la \Om, \Om \vp_0 \ra+ \la \eta, \Om \vp_0 \ra
+ c_{\om} \la \bar{\Om} - R\pa_R \bar{\Om} , \Om \vp_0\ra  - \B\la  \f{3}{1+R} D_{\b} \Om, \Om \vp_0 \B\ra , \\
\la \cL_2(\Om, \eta), \eta \psi_0 \ra
& = -\la  R \pa_R \eta, \eta \psi_0 \ra  +  \B\la  (- 2 + \f{3}{1+R} ) \eta , \eta  \psi_0 \B\ra + \B\la  \f{2}{\pi \al} \td{L}_{12}(\Om) \bar{\eta}, \eta \psi_0 \B\ra \\
& +c_{\om} \la \bar{\eta} - R\pa_R \bar{\eta} , \eta \psi_0\ra   - \B\la  \f{3}{1+R} D_{\b} \eta, \eta \psi_0 \B\ra
 ,\\
\eal
\eeq
where we have used the notation $D_{\b} = \sin(2\b) \pa_{\b}$ to simplify the formula. We treat the sum of the first two terms on the right hand side as the damping terms.


\subsubsection{The damping terms}\label{subsec:L2_dp}
We first handle the first two terms on the right hand side of the $\cL_1$ equation in \eqref{eq:L2_1}. Using integration by parts for $\pa_R$, we derive
\[
\bal
&-\la  R \pa_R \Om, \Om \vp_0 \ra  - \la \Om, \Om \vp_0 \ra
= - \la R \vp_0 , \f{1}{2} \pa_R \Om^2 \ra- \la \Om, \Om \vp_0 \ra
=  \B\la \f{1}{2}(R\vp_0)_R  - \vp_0 , \Om^2 \B\ra,  \\
 &-\la  R \pa_R \eta, \eta \psi_0 \ra + \B\la  (- 2 + \f{3}{1+R} ) \eta, \eta  \psi_0 \B\ra 
 = \B\la \f{1}{2} (R\psi_0 )_R + (- 2 + \f{3}{1+R}) \psi_0, \eta^2 \B\ra.
\eal
\]
Using the formulas of $\psi_0, \vp_0$ \eqref{wg:L2_R0}, we compute the coefficients in the inner products in Appendix \ref{comp:L2_dp} and obtain 
\beq\label{eq:L2_2}
\bal
&-\la  R \pa_R \Om, \Om \vp_0 \ra  - \la \Om, \Om \vp_0 \ra= -  \B\la   (  2R^{-3} + \f{9}{2}  R^{-2}  + 3 R^{-1} + \f{1}{2} ) \sin(2\b) , \Om^2  \B\ra , \\
 &-\la  R \pa_R \eta, \eta \psi_0 \ra + \B\la ( - 2 + \f{3}{1+R}) \eta, \eta  \psi_0 \B\ra 
 =  - \B\la  \f{3(1+R)^2}{ 32R^4} ( 1 + 4 R + 3 R^2 + 3 R^3) \G(\b)^{-1} , \eta^2   \B\ra .
 \eal
 \eeq

\subsubsection{Estimate of interaction between $\Om$ and $\eta$}\label{subsec:L2_uw}
We use ideas in Model 1 in Section \ref{sec:model1} to combine the estimates of $\la \Om, \eta \psi \ra$ and $\la \f{2}{\pi \al} \td{L}_{12}(\Om) \bar{\eta}, \eta \psi_0 \ra$.
Using \eqref{eq:profile} and \eqref{wg:L2_R0}, we can compute
\[
\bal
I \teq \B\la \f{2}{\pi \al} \td{L}_{12}(\Om) \bar{\eta}, \eta \psi_0 \B\ra
 &=  \B\la \f{9}{ 4 \pi c}  \td{L}_{12}(\Om) , \eta \B(  \f{1}{R^3} + \f{3}{2} \f{1+R}{R^2} \B)\B\ra ,\\
 II \teq \la \Om, \eta \vp_0 \ra & = \B\la \Om \sin(2\b) , \eta \B( \f{1}{R^3} + 3 \f{1+R}{R^2} + 1 \B)  \B\ra ,
 \eal
\]
where $c$ is defined in \eqref{eq:profile} and satisfies $c  = \f{2}{\pi} + O(\al)$ (see Lemma \ref{lem:one}). We design $\psi_0$ \eqref{wg:L2_R0} so that the denominator in $\psi_0$ and the coefficient $\bar \eta$ in $I$ cancel.

Applying the Cauchy-Schwarz inequality, we yield 
\beq\label{eq:L2_30}
\bal
 I + II &= \B\la \Om \sin(2\b) + \f{9}{4\pi c}  \td{L}_{12}(\Om) , \eta R^{-3} \B\ra + \B\la \Om \sin(2\b) +  \f{9}{8\pi c} \td{L}_{12}(\Om) , 3 \eta  \f{1+R}{R^2} \B\ra  + \la \Om \sin(2\b) , \eta  \ra   \\
 &\leq  \f{4}{3} \B\la  (\Om \sin(2\b) + \f{9}{ 4 \pi c}  \td{L}_{12}(\Om) )^2 , R^{-3}  \B \ra
+ \f{1}{ 4 \cdot 4/3  } \la \eta^2 , R^{-3}  \ra  \\
&+ 6  \B\la ( \Om \sin(2\b) +  \f{9}{8\pi c} \td{L}_{12}(\Om)   )^2, R^{-2} \B\ra  + \f{3^2} {4 \cdot 6}  \B\la \eta^2 , \f{ (1 + R)^2}{R^{2 }}  \B\ra \\
&+  \f{1}{3} \B\la \Om^2, \f{1 + R}{R} \sin(2\b)^2 \B\ra
+ \f{3}{4}\B\la \eta^2, \f{R}{1+R}  \B\ra = \sum_{i=1}^6 J_i.
\eal
\eeq

We design the special forms $\psi_0, \vp_0$ in \eqref{wg:L2_R0} to obtain the good form $  \Om \sin(2\b) + C\td L_{12}(\Om) $ for some $C >0$ in $I + II$. Next, we exploit the cancellation between $\Om$ and $\td L_{12}(\Om)$ using Lemma \ref{lem:cancel}. We apply Lemma \ref{lem:cancel} with $k=2, 3$ to simplify $J_1, J_3$ defined above:
\[
\bal
J_1 + J_3 
=& \B\la   \lt( \f{4}{3} R^{-3} + 6 R^{-2 } \rt)\sin(2\b)^2  , \Om^2  \B\ra   - \f{4}{3} \lt( 2 \cdot \f{9}{4\pi c}   - \f{\pi}{2} \f{9^2}{ (4\pi c)^2}  \rt) ||  \td{L}_{12}(\Om) R^{-3/2 }||^2_{L^2(R)} \\
&- 6 (  \f{9}{8 \pi c}  - \f{\pi}{2} \f{9^2}{ (8 \pi c)^2}  ) \ || \td{L}_{12}(\Om), R^{-1 } ||^2_{L^2(R)} \teq M_1 + M_2 +M_3.
\eal
\]
We further simplify $M_2, M_3$ defined above. Using Lemma \ref{lem:one}, we have $| \pi c -2 | \les \al$ and
\beq\label{eq:cancel_coe}
\bal
- \f{4}{3} \cdot \f{9}{4\pi c} ( 2  - \f{\pi}{2} \f{9}{4\pi c} )
&\leq  - \f{4}{3}\cdot \f{9}{8} ( 2 - \f{\pi}{2} \cdot\f{9}{8}) + C \al   < - \f{1}{4}  + C\al ,\\
  - 6 \cdot \f{9}{8 \pi c} ( 1 - \f{\pi}{2} \f{9}{8 \pi c})
 &\leq  - 6 \cdot \f{9}{16} (1 - \f{\pi}{2} \cdot \f{9}{16}) + C \al < - \f{1}{4} + C \al ,
\eal
\eeq
for some absolute constant $C$.  It follows that
\[
M_2 + M_3 \leq  (-\f{1}{4} + C\al ) ( ||  \td{L}_{12}(\Om), R^{-3/2 }||^2_{L^2(R)} + 
|| \td{L}_{12}(\Om), R^{-1 } ||^2_{L^2(R)}) = (-\f{1}{4} + C\al )  || \td{L}_{12}(\Om)\rho^{1/2}||^2_{L^2},
\]
where we have used the notation $\rho$ defined in \eqref{wg:L2_R0}. Therefore, we yield the damping for $\td{L}_{12}(\Om)$.

\begin{remark}
The above computations of $J_1, J_2, J_3, J_4$ are exactly the same as those in Model 1 in Section \ref{sec:model1}. We choose the constants in the weights \eqref{wg:L2_R0} carefully so that when we apply Lemma \ref{lem:cancel}, the constant $ - ( (k-1) \lam - \f{\pi}{2}\lam^2) $ in \eqref{eq:cancel} is negative, i.e. \eqref{eq:cancel_coe}.

\end{remark}

Using \eqref{eq:L2_30}, the above estimate of $M_2 + M_3$ in $J_1 +J_3$ and $\sin(2\b)^2 \leq \sin(2\b)$, we prove
\beq\label{eq:L2_3}
\bal
&\B\la \f{2}{\pi \al} \td{L}_{12}(\Om) \bar{\eta}, \eta \psi_0 \B\ra + \la \Om, \eta \vp_0 \ra
= I + II \leq  \B\la \B( \f{4}{3} R^{-3}  + 6 R^{-2} + \f{1+R}{3 R} \B) \sin(2\b),   \Om^2 \B\ra   \\
+ & \B\la \f{3}{16} R^{-3}  + \f{3}{8}\f{ (1 + R)^2}{R^2} + \f{3}{4}\f{R}{1+R} , \eta^2   \B\ra 
  - (\f{1}{4} - C\al ) || \td{L}_{12}(\Om) \rho^{1/2} ||^2_{L^2(R)} .\\
\eal
\eeq

\subsubsection{Estimate of the projection $c_{\om}$ in the $\Om$ equation}\label{subsec:L2_cw1}
We estimate the terms involving $c_{\om}$ in \eqref{eq:L2_1} in this subsection. Notice that $c_{\om}$ defined in \eqref{eq:normal} is the projection of $\Om$ onto some function. Using \eqref{eq:profile} and \eqref{wg:L2_R0}, we can calculate 
\[
\bal
&\la \bar{\Om} - R\pa_{R} \bar{\Om} , \Om \vp_0 \ra 
=   \B\la \f{\al}{c} \G(\b) \f{6R^2}{ (1+R)^3}, \Om  \f{(1+R)^3}{R^3} \sin(2\b) \B\ra 
= \f{ 6\al}{c} \B\la \f{ \sin(2\b) \G(\b)}{R} , \Om  \B\ra.
\eal
 \]
We show that the above projection is almost equal to $L_{12}(\Om)(0)$. Notice that 
 \[
 \bal
& \f{1}{ c} \B\la \f{ \sin(2\b) \G(\b)  }{R} , \Om \B\ra
 - \f{\pi}{2} \B\la \f{  \sin(2\b) }{R} , \Om \B\ra  
 = \f{1}{c} \B\la \f{  \sin(2\b) (\G(\b) -1)}{R},  \ \Om \B\ra 
 +(\f{1}{c} - \f{\pi}{2}) \la \f{  \sin(2\b) }{R}, \Om \ra   \teq I + II.
 \eal
 \]
 Using Lemma \ref{lem:one}, \eqref{wg:L2_R0} and the Cauchy-Schwarz inequality, we have 
 \[
 \bal
|I| &\les \al\la \f{1}{R} \sin(2\b)^{1/2},  | \Om|  \ra  
 \les 
  \al  \B|\B| \Om \f{(1+R)^{3/2}}{R^{3/2}} \sin(2\b)^{1/2}\B|\B|_2 \B| \B| \f{1}{R} \cdot \f{R^{3/2}}{(1+R)^{3/2}}\B|\B|_2\les \al || \Om \vp_0^{1/2}||_2  , \\
   |II| & \les \al \la \f{ 1}{R} \sin(2\b), |\Om| \ra
 \les \al \B| \B| \Om  \f{ (1+R)^{3/2}}{R^{3/2}} \sin(2\b) \B|\B|_2
 \B|\B| \f{1}{R} \cdot \f{ R^{3/2}}{(1+R)^{3/2}} \sin(2\b) \B|\B|_2 \les \al || \Om \vp_0^{1/2}||_2 .\\
 \eal
 \]
 It follows that
 \[
\B| \f{1}{\al}\la \bar{\Om} - R\pa_{R} \bar{\Om} , \Om \vp_0 \ra 
- 6  \cdot \f{\pi}{2} \B\la \f{\sin(2\b)  }{R} ,\Om \B\ra \B| 
\leq 6 | I + II | \les \al || \Om \vp_0^{1/2}||_2 .
 \]
 Recall the definition of $c_{\om}$ in \eqref{eq:normal}. Using the above estimate and then the formula of $L_{12}(\Om)(0)$ \eqref{eq:biot3}, we have 
 \beq\label{eq:L2_41}
 \bal
 &c_{\om} \la \bar{\Om} - R\pa_{R} \bar{\Om} , \Om \vp_0 \ra 
= -\f{2}{\pi } L_{12}(\Om)(0)  \cdot \f{1}{\al}\la \bar{\Om} - R\pa_{R} \bar{\Om} , \Om \vp_0 \ra \\
\leq & -\f{2}{\pi } L_{12}(\Om)(0) \cdot 6 \cdot \f{\pi}{2} \B\la \f{\sin(2\b)  }{R} ,\Om \B\ra
+ C \al | L_{12}(\Om)(0)| \cdot || \Om \vp_0^{1/2}||_2   \\
=& - 6 ( L_{12}(\Om)(0))^2 + C \al | L_{12}(\Om)(0)| \cdot || \Om \vp_0^{1/2} ||_2  \\
\leq& - (6-C\al) ( L_{12}(\Om)(0))^2 + C \al || \Om \vp_0^{1/2} ||_2 .
\eal
 \eeq

By choosing $\vp_0$ in \eqref{wg:L2_R0} carefully, we obtain a damping term for $L_{12}(\Om)(0)^2$ from $c_{\om} ( \bar{\Om} - R\pa_{R} \bar{\Om}) $. This is one of the motivations to choose the special form of $\vp_0$. 

\subsubsection{Estimate of the projection $c_{\om}$ in the $\eta$ equation}\label{subsec:L2_cw2}
We use some ideas and estimates similar to those in Model 2 in Section \ref{sec:model2} to estimate the $c_{\om}$ term in the $\eta$ equation \eqref{eq:L2_1}.

Using $c_{\om} = -\f{2}{\pi \al} L_{12}(\Om)(0)$ \eqref{eq:normal} and expanding the coefficient $(\bar{\eta} - R\pa_R \bar{\eta}) \psi_0$ using the formulas \eqref{eq:profile} and \eqref{wg:L2_R0}, which is presented in Appendix \ref{comp:L2_cw1}, we derive
\beq\label{eq:L2_42}
c_{\om} \la \bar{\eta} - R\pa_R \bar{\eta} , \eta \psi_0\ra
=-\f{27}{ 4\pi c} L_{12}(\Om) (0) \B\la \eta, \f{1}{(1+R)R^2} \B\ra  
 - \f{81}{ 8 \pi c} L_{12}(\Om) (0) \B\la \eta, \f{1}{R} \B\ra  \teq A_1 + A_2 . 
\eeq

\vspace{0.1in}
\paragraph{\bf{An ODE for $L_{12}(\Om)(0)$}}
Using the $\Om$ equation in \eqref{eq:equiv} and derivation similar to that in Model 2 in Section \ref{sec:model2}, we derive the following ODE for $L_{12}(\Om)(0)$ in Appendix \ref{comp:L2_cw2}
\beq\label{eq:L2_44}
\bal
\f{1}{2}\f{d}{dt} 
\f{81 }{ 4 \pi c}  L^2_{12} (\Om)(0) &= \f{81 }{ 4 \pi c}  \B( 
- 4 L^2_{12} (\Om)(0) + L_{12}(\Om)(0) \B\la \eta,  \f{\sin(2\b)}{R} \B\ra   \\
&   -  L_{12}(\Om)(0)\B\la \f{3 \sin(2\b)}{(1+R)R} , D_{\b} \Om \B\ra 
+ L_{12}(\Om)(0) \B\la \cR_{\Om}, \f{\sin(2\b)}{R} \B\ra   \B) .
\eal
\eeq

Note that we have multiplied both sides of the ODE for $L^2_{12}(\Om)(0)$ by the constant $\f{81 }{ 4 \pi c}$ and will include $\f{81 }{ 4 \pi c}  L^2_{12} (\Om)(0)$ in the energy $E(R,0)$ \eqref{eg:R0}. The first term on the right hand side provides damping for $L^2_{12}(\Om)(0)$, which is similar to that in \eqref{eq:model2_2}. It enables us to control the term $A_1, A_2$ in \eqref{eq:L2_42}. Based on the idea in Model 2 in Section \ref{sec:model2} and the fact that the integrands in $A_2$ and $L_{12}(\Om)(0) \B\la \eta,  \f{\sin(2\b)}{R} \B\ra$ in \eqref{eq:L2_44} have different signs, we combine the estimate of $A_2$ in \eqref{eq:L2_42} and the $\eta$ term in \eqref{eq:L2_44} as follows to exploit cancellation 
\beq\label{eq:L2_51}
A_3 \teq A_2 + \f{81  }{ 4 \pi c}  L_{12}(\Om)(0) \B\la \eta,  \f{\sin(2\b)}{R} \B\ra
= \f{81}{8\pi c}   L_{12}(\Om)(0) \B\la \eta, \f{1}{R} (-1 + 2 \sin(2\b)) \B\ra .
\eeq

Next, we estimate $A_1$ in \eqref{eq:L2_42} and $A_3$ by treating them as perturbation. Applying the Cauchy-Schwarz inequality yields
\beq\label{eq:L2_520}
\bal
  A_3 & \leq \f{ 81 }{8 \pi c}  | L_{12}(\Om)(0)| \cdot  \B| \B| \eta \f{ (1+R)^2}{ R^{3/2}}  \B|\B|_2 \B| \B| \f{R^{3/2}}{(1+R)^2} \f{1}{R}(1 - 2  \sin(2\b))\B|\B|_2 ,\\
A_1 &\leq \f{27}{ 4\pi c} | L_{12}(\Om) (0)| \cdot \B|\B| \eta \f{(1+R)^{3/2}}{R^2}\B|\B|_2
\B| \B| \f{R^2}{(1+R)^{3/2}} \cdot \f{1}{(1+R)R^2}\B|\B|_2.  \\
\eal
\eeq
The integrals on $R, \b$ in \eqref{eq:L2_520} equal to $ \sqrt{ \f{1}{6}(\f{3 \pi}{2} -4 ) },\sqrt{ \f{\pi}{8}  }  $, respectively, which are computed in Appendix \ref{comp:L2_cw3}. Then we reduce \eqref{eq:L2_520} to 
\beq\label{eq:L2_521}
A_3 \leq b_1 | L_{12}(\Om)(0)| \B| \B| \eta \f{ (1+R)^2}{ R^{3/2}}  \B|\B|_2, 
 \quad A_1 \leq  b_2 | L_{12}(\Om) (0)|  \B|\B| \eta \f{(1+R)^{3/2}}{R^2}\B|\B|_2 .
\eeq
where $b_1, b_2$ are given by 
\[
b_1 \teq \f{81}{8 \pi c} \sqrt{ \f{1}{6}(\f{3 \pi}{2} -4 )} ,\quad  b_2 \teq \f{27}{ 4\pi c} \sqrt{ \f{\pi}{8}  }.
\]


Using the Young's inequality $ab \leq s a^2 + \f{1}{4 s}b^2$ for any $s > 0$, we get 
\beq\label{eq:L2_52}
\bal
A_1 + A_3  
&\leq  \f{1}{32}  \B| \B| \eta \f{ (1+R)^2}{ R^{3/2}}  \B|\B|_2^2
+ \f{9}{128}\B|\B| \eta \f{(1+R)^{3/2}}{R^2}\B|\B|_2^2
+ L^2_{12}(\Om)(0) \B(  \f{b^2_1}{4 \times 1/32} + \f{b_2^2}{4 \times 9/128} \B) . \\
\eal
\eeq

Using Lemma \ref{lem:one} for the estimate of $c$ and a direct calculation yield 
\beq\label{eq:cw_count}
\bal
&\f{b_1^2}{1/8} + \f{b_2^2}{9/32} - \f{81  }{4\pi c} \cdot 4 
= 8 \lt( \f{81}{8\pi c} \rt)^2 \f{1}{6} ( \f{ 3\pi}{2} -4 )
+ \f{32}{9}\lt( \f{ 27}{ 4\pi c} \rt)^2 \f{\pi}{8} - \f{81}{\pi c}   \\
\leq& \f{4}{3} \lt(\f{81}{16} \rt)^2 (\f{3 \pi }{2} - 4) 
+ \f{4\pi}{9} \lt( \f{27}{8} \rt)^2 - \f{81}{2}+ C \al <    C \al.
\eal
\eeq

Combining the identities \eqref{eq:L2_42}, \eqref{eq:L2_51}, the damping term of $L^2_{12}(\Om)(0)$ in \eqref{eq:L2_44} and the estimate \eqref{eq:L2_52}, we prove
\beq\label{eq:L2_53}
\bal
& c_{\om} \la \bar{\eta} - R\pa_R \bar{\eta} , \eta \psi_0\ra+\f{81 }{ 4 \pi c}  L_{12}(\Om)(0) \B\la \eta,  \f{\sin(2\b)}{R} \B\ra  - \f{81 }{4\pi c} \cdot 4 L^2_{12}(\Om)(0) \\
 = &
 A_1 + A_2 + \f{81 }{ 4 \pi c}  L_{12}(\Om)(0) \B\la \eta,  \f{\sin(2\b)}{R} \B\ra    - \f{81 }{4\pi c} \cdot 4 L^2_{12}(\Om)(0) 
 = A_1 + A_3 - \f{81 }{4\pi c} \cdot 4 L^2_{12}(\Om)(0) \\
= & \la  \eta^2 , \f{1}{32}\f{(1+R)^4}{R^{3 }} + \f{9}{128} \f{(1+R)^3}{R^{4 }}  \ra 
+ L^2_{12}(\Om)(0) \lt( \f{b_1^2}{1/8} + \f{b_2^2}{9/32} - \f{81 }{\pi c}  \rt)   \\
 \leq &  \f{3}{16} \B\la \eta^2  , \f{1}{6}\f{(1+R)^4}{R^3} + \f{3}{8} \f{(1+R)^3}{R^4} \B\ra
+C \al  L^2_{12}(\Om)(0) ,
\eal
\eeq
where we have used \eqref{eq:cw_count} to derive the last inequality.

\subsubsection{Estimate of the angular transport term}\label{subsec:L2_angle}
From the definition of the weights \eqref{wg}, \eqref{wg:L2_R0}, we have 
\[
\vp_0 \les \vp_2, \quad (1+R)^{-1}\psi_0 \les \psi_2, \quad  
\B| \B| \f{3 \sin(2\b)}{(1+R)R} \vp_2^{-1/2}\B|\B|_2 \les 1 .
\]
Therefore, we can estimate the angular transport terms in \eqref{eq:L2_1}, \eqref{eq:L2_44} as follows 
\[
\bal
&- \la  \f{3D_{\b} \Om}{1+R} , \Om \vp_0 \ra 
\les  || D_{\b} \Om  \vp_2^{1/2}||_2 || \Om \vp_0^{1/2} ||_2,  \quad - \la  \f{3 D_{\b} \eta}{1+R} , \eta \psi_0 \ra  \les  || D_{\b} \eta  \psi_2^{1/2}||_2 || \eta \psi_0^{1/2} ||_2, \\
 &- \f{81 }{ 4 \pi c}  L_{12}(\Om)(0)\B\la \f{3 \sin(2\b)}{(1+R)R} , D_{\b} \Om \B\ra 
 \les  | L_{12}(\Om)(0)| \B|\B| D_{\b} \Om \vp_2^{1/2}||_2,
 \eal
\]
where we have used $c^{-1} \les 1$ (see Lemma \ref{lem:one}). Using the energy notations $E(\b, 1)$ \eqref{eg:b1} and $E(R, 0)$ \eqref{eg:R0}, we further derive
\beq\label{eq:L2_6}
\bal
 &-\la  \f{3D_{\b} \Om}{1+R} , \Om \vp_0 \ra - \la  \f{3 D_{\b} \eta}{1+R} , \eta \psi_0 \ra 
 - \f{81 }{ 4 \pi c}  L_{12}(\Om)(0)\B\la \f{3 \sin(2\b)}{(1+R)R} , D_{\b} \Om \B\ra  
 \leq  K_1 E(R, 0) E(\b, 1) ,  
\eal
\eeq
for some absolute constant $K_1$. We remark that the absolute constants $K_1, K_2,..$ do not change from line to line.

\subsubsection{Completing the estimates with a less singular weight}\label{subsec:L2_sum}
Combining the estimates \eqref{eq:L2_1}-\eqref{eq:L2_3}, \eqref{eq:L2_41}, \eqref{eq:L2_44}, \eqref{eq:L2_53}, \eqref{eq:L2_6} and using the notations $E(R,0), \cR(R,0)$ \eqref{eg:R0}, we obtain
\beq\label{eq:L2_71}
\bal
\f{1}{2} \f{d}{dt} E(R,0)^2 &= \f{1}{2} \f{d}{dt} 
\lt(  
|| \Om \vp_0^{1/2}||_2^2 + || \eta \psi_0^{1/2}||_2^2
 +\f{81 }{ 4 \pi c}  L^2_{12} (\Om)(0) \rt) \\
 &\leq  \la \Om^2 , \sin(2\b) D(\Om) \ra 
 + \la \eta^2 , D(\eta) \ra  + (-\f{1}{4} + C\al ) || \td{L}_{12}(\Om) \rho^{1/2} ||^2_{L^2(R)} \\
 &\quad + L^2_{12}(\Om)(0) \lt( -6 + C \al \rt) +  C \al \la \Om^2, \vp_0 \ra + K_1 E(R, 0) E(\b, 1)  + \cR(R,0) ,
\eal
\eeq
where $D(\Om), D(\eta)$ are given by 
\beq\label{eq:L2_712}
\bal
D(\Om) & \teq  -2R^{-3} - \f{9}{2}  R^{-2}  - 3 R^{-1} - \f{1}{2} +  \f{4}{3} R^{-3}  + 6 R^{-2} + \f{1+R}{ 3 R} ,  \\
D(\eta)  &  \teq
- \f{3(1+R)^2}{ 32 R^4} ( 1 + 4 R + 3 R^2 + 3 R^3) \G(\b)^{-1} \\
&\quad +\lt( \f{3}{16} R^{-3}  + \f{3}{8}\f{ (1 + R)^2}{R^2} +  \f{3 R}{4(1+R)} \rt) 
  + \f{3}{16} \B( \f{1}{6}\f{(1+R)^4}{R^3} + \f{3}{8} \f{(1+R)^3}{R^4} \B).
\eal
\eeq

Recall the weights $\vp_0,\psi_0$ in \eqref{wg:L2_R0}. In Appendix \ref{comp:L2_sum}, we estimate $D(\Om), D(\eta)$ and prove 
\beq\label{eq:L2_72p}
\sin(2\b) D(\Om) \leq -\f{1}{6} \vp_0, \quad D(\eta ) \leq -\f{1}{8} \psi_0,
\eeq
which only involves elementary estimates.


For $L^2_{12}(\Om)(0)$ in \eqref{eq:L2_71}, we use Lemma \ref{lem:one} about $c$ ($c\pi = 2 + O(\al)$) to get
\[
-6 + C\al \leq -\f{1}{8}\times \f{81}{8} - 4 + C\al \leq -\f{1}{8}\times \f{81}{4\pi c} - 4 + C\al, 
\]
which implies
\beq\label{eq:L2_74}
(-6 + C\al ) L^2_{12}(\Om)(0) \leq - \f{1}{8} \cdot \f{81}{ 4\pi c} L^2_{12}(\Om)(0)
 - (4 - C\al) L^2_{12}(\Om)(0) ,
\eeq
where $C$ is some absolute constant and may vary from line to line. Observe that 
\beq\label{eq:L2_75}
K_1 E(R, 0) E(\b, 1) \leq \f{1}{100} E(R,0)^2 +  100 K^2_1 E^2(\b,1). 
\eeq

Recall $E(R,0)$ in \eqref{eg:R0}. Finally, substituting the estimates \eqref{eq:L2_72p}-\eqref{eq:L2_75} in \eqref{eq:L2_71}, we prove 
\beq\label{eq:L2_8}
\bal
&\f{1}{2} \f{d}{dt} E(R,0)^2 
\leq  -(\f{1}{6} - C\al) || \Om \vp_0^{1/2}||_2^2 - \f{1}{8} || \eta \psi_0^{1/2}||_2^2- \f{1}{8} \cdot \f{81}{ 4\pi c} L^2_{12}(\Om)(0)   - (4 - C\al) L^2_{12}(\Om)(0)  \\
& \qquad  \qquad \qquad - (\f{1}{4} - C\al )|| \td{L}_{12}(\Om) \rho^{1/2} ||_{L^2(R)}   + \f{1}{100} E(R,0)^2 +  100 K^2_1 E^2(\b,1) + \cR(R, 0)\\
 \leq & (-\f{1}{9} +C \al) E^2(R,0)
 - (4 - C\al) L^2_{12}(\Om)(0)    - (\f{1}{4} - C\al )|| \td{L}_{12}(\Om) \rho^{1/2} ||_{L^2(R)}   +  100 K^2_1 E^2(\b,1)+ \cR(R, 0),\\
\eal
\eeq
where we have used $- \f{1}{6} + C\al + \f{1}{100}  , - \f{1}{8} + \f{1}{100} < -\f{1}{9} + C\al$ to derive the last inequality. 


\subsubsection{Linear stability with a less singular weight}\label{subsec:L2_done}
Using the reformulation \eqref{eq:equiv}, and the notations $E(\b, 1)$ and $\cR(\b, 1)$ defined in \eqref{eg:b1}, we have
\beq\label{eq:equiv_l2less}
\f{1}{2} \f{d}{dt} (E(\b,1) )^2 
 = \la D_{\b} \cL_1 (\Om, \eta), ( D_{\b} \Om ) \vp_2 \ra + \la  D_{\b} \cL_2( \Om , \eta)  , ( D_{\b}\eta) \vp_2 \ra  + \cR(\b, 1) .
\eeq
Now we combine \eqref{eq:H1_b5} and \eqref{eq:L2_8} to establish the linear stability of \eqref{eq:lin21}-\eqref{eq:lin22} with the less singular weight \eqref{wg:L2_R0}. 
Firstly, we choose an absolute constant $\mu_1$ such that 
\[
100 K_1^2 <  \f{1}{20} \mu_1 ,
\]
where the absolute constant $K_1$ is determined in \eqref{eq:L2_6}. From \eqref{wg:L2_R0}, we have $R^{-2} \leq \rho$. Hence,
\[
|| \td{L}_{12}(\Om) R^{-1} ||^2_{L^2(R)} \leq || \td{L}_{12}(\Om) \rho^{1/2} ||^2_{L^2(R)}.
\]
Combining Proposition \ref{prop:b1}, \eqref{eq:L2_8}, the formulation \eqref{eq:equiv_l2less}, and the above estimates, we establish the estimate for $E(R,0)^2 + \mu_1 E(\b,1)^2$ 
\beq\label{eq:L2_9}
\bal
&\f{1}{2} \f{d}{dt} ( ( E(R,0)^2 + \mu_1 E(\b,1)^2 ) )
\leq  - ( \f{1}{9} - C\al ) ( ( E(R,0)^2 + \mu_1 E(\b,1)^2 ) )  \\
&-  (4-C\al) L^2_{12}(\Om)(0) - (\f{1}{4} - C\al ) || \td{L}^2_{12} \rho^{1/2} ||^2_{L^2(R)}
+ \cR(R, 0) + \mu_1 \cR(\b, 1)
.
\eal
\eeq
The proof of Proposition \ref{prop:L2less} is now complete.

\subsection{Weighted $L^2$ estimate of $\Om, \eta$ with a more singular weight}\label{sec:L2_more}
With the linear stability \eqref{eq:L2_9} with a less singular weight, we can proceed to perform the weighted $L^2$ estimate with a more singular weight.

\begin{definition}\label{def:L2_m} Define an energy $E(R, 1)$ and a remaining term $\cR(R, 1)$ by 
\beq\label{eg:R1}
E(R, 1) \teq  \lt(  || \Om \vp_1^{1/2}||_2^2 + || \eta \vp_1^{1/2}||_2^2 \rt)^{1/2}, \quad 
\cR(R, 1) \teq \la \cR_{\Om}, \Om \vp_1 \ra + \la \cR_{\eta}, \eta \vp_1 \ra , 
\eeq
where $\vp_1, \psi_1$ are given in Definition \ref{def:wg}.
\end{definition}

The main result in this Section is the following.
\begin{prop}\label{prop:L2}
Assume that $\Om \vp_1^{1/2} , \ \eta \vp_1^{1/2} \in L^2$. We have that 
\[
\bal
&\la  \cL_1 (\Om, \eta), \Om \vp_1 \ra + \la   \cL_2( \Om , \eta)  , \eta \vp_1 \ra 
\leq  - \f{1}{6} (E(R,1) )^2  +  K_3  \lt( L^2_{12}(\Om)(0) +  \B| \B| \td{L}_{12}(\Om) R^{-1} \B| \B|^2_{L^2(R)} \rt) ,
\eal
\]
where $\cL_1, \cL_2$ are defined in Definition \ref{def:op}, $K_3 >0$ is some fixed absolute constant.
\end{prop}

\begin{proof}[Proof of Proposition \ref{prop:L2}]
A direct calculation yields 
\beq\label{eq:L2_m1}
\bal
\la \cL_1(\Om, \eta), \Om \vp_1 \ra  & = \la \cL_{10}(\Om, \eta), \Om \vp_1 \ra + c_{\om} \la  \bar{\Om} - R \pa_R \bar{\Om}, \Om \vp_1 \ra , \\
\la \cL_2(\Om, \eta) , \eta \vp_1 \ra & = \la \cL_{20}(\eta), \eta \vp_1 \ra
+\f{2}{\pi \al} \la  \td{L}_{12}(\Om)  \bar{\eta}, \eta \vp_1 \ra + c_{\om}  \la  \bar{\eta} - R \pa_R \bar{\eta} ,\eta \vp_1 \ra .
\eal
\eeq
Applying Lemma \ref{lem:dp} with $\vp = \vp_1 $ and $\d = \s = \f{99}{100}$, we yield 
\[
\la \cL_{10}(\Om, \eta), \Om \vp_1 \ra  + 
\la \cL_{20}( \eta), \eta \vp_1 \ra \leq (-\f{1}{4} + 3 |1 -\s |) (|| \Om \vp_1^{1/2}||_2^2 + || \eta \vp_1^{1/2}||_2^2) < -\f{1}{5}( || \Om \vp_1^{1/2}||_2^2 + || \eta \vp_1^{1/2}||_2^2  ).
\]
Recall $c_{\om} = -\f{2}{\pi \al} L_{12}(\Om)(0)$ \eqref{eq:normal}. Using \eqref{eq:bar_ing} in Lemma \ref{lem:bar} and the Cauchy-Schwarz inequality, we obtain 
\[
| c_{\om} \la (\bar{\Om} - R \pa_R \bar{\Om} ) , \Om  \vp_1 \ra|
+ |c_{\om} \la (\bar{\eta} - R \pa_R \bar{\eta} ) , \eta \vp_1 \ra| \les  |L_{12}(\Om)(0)  | 
( || \Om \vp_1^{1/2}||_2^2 + || \eta \vp_1^{1/2}||_2^2)^{1/2}.
\]

For $\td{L}_{12}(\Om)$ in \eqref{eq:L2_m1}, using the Cauchy-Schwarz inequality, we derive 
\[
\B\la  \f{2}{\pi \al}\td{L}_{12}(\Om)  \bar{\eta},  \eta  \vp_1 \B\ra
\les \al^{-1} || \td{L}_{12}(\Om) \bar{\eta} \vp_1^{1/2}||_2 || \eta \vp_1^{1/2}||_2
  \les \B| \B| \td{L}_{12}(\Om) R^{-1} \B| \B|_{L^2(R)} || \eta \vp_1^{1/2}||_2,
\]
where we have applied Lemma \ref{lem:ux} and \eqref{eq:bar_ux} in Lemma \ref{lem:bar} in the second inequality.

Using the Cauchy-Schwarz inequality and the energy notation $E(R,1)$ \eqref{eg:R1}, we complete the proof of Proposition \ref{prop:L2}.
\end{proof}

\subsection{Weighted $L^2$ estimate of $D_{\b} \xi$ and $\xi$}\label{sec:L2_xi}
The estimates of $\xi$ are simpler since the main terms in the equation of $\xi$ \eqref{eq:lin23} do not couple with $\Om, \eta$ directly. We use the weights $\psi_1, \psi_2$ in Definition \ref{def:wg}.

\begin{prop}\label{prop:rho12}
Suppose that $\psi_1^{1/2} \xi, \ \psi_2^{1/2} D_{\b}\xi \in L^2$. We have 
\begin{align}
\la \cL_{3}(\Om, \xi), \xi \psi_1 \ra
 &\leq  (-\f{1}{3} + C\al )  || \xi \psi_1^{1/2}||_2^2
 +C \al \lt( L^2_{12}(\Om )(0)  + || \td{L}_{12} (\Om) R^{-1} ||^2_{L^2(R)}  \rt)
\label{eq:L2_xi2} ,\\
\la D_{\b} \cL_3(\Om, \xi), (D_{\b}\xi )\psi_2 \ra
&\leq  (-\f{1}{3} + C \al) || D_{\b}\xi \psi_2^{1/2}||_2^2
+C\al \lt( L^2_{12}(\Om )(0)   + || \td{L}_{12} (\Om) R^{-1} ||^2_{L^2(R)} \rt).
\label{eq:L2_xi3} 
\end{align}
\end{prop}

\begin{proof}[Proof of Proposition \ref{prop:rho12}]
Since $D_{\b}$ commutes with $\cL_3$ (see Definition \ref{def:op}) and $\td{L}_{12}(R)$ does not depend on $\b$, a direct calculation implies 
\beq\label{eq:L2_xi4}
\bal
\la \cL_3(\Om, \xi) ,\xi \psi_1 \ra &= \la \cL_{30}(\xi), \xi \psi_1 \ra  - \f{2}{\pi \al} \la \td{L}_{12}(\Om)\bar{\xi} , \xi \psi_1 \ra +c_{\om} \la  3\bar{\xi} -D_R \bar{\xi}  ,\xi \psi_1 \ra \\
\la D_{\b} \cL_3(\Om, \xi) , (D_{\b}\xi) \psi_2 \ra
&= \la \cL_{30}(D_{\b}\xi), (D_{\b}\xi) \psi_2 \ra  - \f{2}{\pi \al} \la \td{L}_{12}(\Om) D_{\b}\bar{\xi}  , (D_{\b}\xi) \psi_2 \ra  +c_{\om} \la D_{\b} ( 3\bar{\xi} - D_R \bar{\xi} ) ,\xi \psi_2 \ra.
\eal
\eeq
Applying \eqref{eq:dp2} in Lemma \ref{lem:dp} with $\psi = \psi_1$ (a constant multiple of $\psi$ does not change the estimate in \eqref{eq:dp2}) and with $\psi  = \psi_2$ (see Definition \ref{def:wg}), respectively, we derive 
\beq\label{eq:L2_xi40}
\bal
 \la \cL_{30}(\xi), \xi \psi_1 \ra &\leq (-\f{1}{2} + 3|1 - \s|) || \xi \psi_1^{1/2}||_2^2
 < -\f{3}{8} || \xi \psi_1^{1/2}||_2^2, \\
\la \cL_{30}(D_{\b}\xi), (D_{\b}\xi) \psi_2 \ra  & 
\leq (-\f{1}{2} + 3(|1 -\g| \vee |1-\s|) ) || D_{\b}\xi \psi_2^{1/2}||_2^2 \leq (-\f{3}{8} + \al) || D_{\b}\xi \psi_2^{1/2}||_2^2 ,
 \eal
\eeq
where $\g = 1 + \f{\al}{10}, \s = \f{99}{100}$. Using the Cauchy-Schwarz inequality, we yield 
\beq\label{eq:L2_xi41}
\bal
&\B|-\f{2}{\pi \al} \la  \td{L}_{12}(\Om)  \bar{\xi} , \xi \psi_1 \ra \B|
\les \al^{-1}  || \td{L}_{12}(\Om) \bar \xi \psi_1^{1/2}||_2 || \xi \psi_1^{1/2} ||_2
 \les  \al  || \td{L}_{12}(\Om) R^{-1} ||_2  || \xi \psi_1^{1/2} ||_2,
\eal
\eeq
where we have applied Lemma \ref{lem:ux} and \eqref{eq:xi_ux} in Lemma \ref{lem:xi} to derive the second inequality.

Using the Cauchy-Schwarz inequality, \eqref{eq:normal} and Lemma \ref{lem:xi}, we obtain 
\beq\label{eq:L2_xi42}
\bal
c_{\om} \la 3 \bar{\xi} - D_R \bar{\xi} , \xi \psi_1 \ra 
&\les \al^{-1} |L_{12}(\Om)(0)| \cdot ||  (3\bar{\xi} - D_R \bar{\xi} ) \psi_1^{1/2} ||_2
 || \xi \psi_1^{1/2} ||_2\\
 &\les \al  | L_{12}(\Om)(0)| \cdot || \xi \psi_1^{1/2} ||_2.
\eal
\eeq
Plugging \eqref{eq:L2_xi40}-\eqref{eq:L2_xi42} in \eqref{eq:L2_xi4} and using the Cauchy-Schwarz inequality
, we prove \eqref{eq:L2_xi2}.

The proof of \eqref{eq:L2_xi3} is completely similar. We apply estimates similar to those in \eqref{eq:L2_xi41}-\eqref{eq:L2_xi42} and Lemmas \ref{lem:ux}, \ref{lem:xi} to control the $c_{\om}$ and $\td{L}_{12}(\Om)$ terms. Combining these estimates, using the second inequality in \eqref{eq:L2_xi40} and then the Cauchy-Schwarz inequality prove \eqref{eq:L2_xi3}. 
\end{proof}

\subsubsection{The weighted $L^2$ energy}
Using the reformulation \eqref{eq:equiv}, we have 
\[
\bal
&\f{1}{2} \f{d}{dt} ( || \Om \vp_1^{1/2}||_2^2  + || \eta \vp_1^{1/2}||_2^2 )
= \la  \cL_1 (\Om, \eta), \Om \vp_1 \ra + \la   \cL_2( \Om , \eta)  , \eta \vp_1 \ra 
+ \la  \cR_{\Om}, \Om \vp_1 \ra + \la \cR_{\eta}, \eta \vp_1 \ra,  \\
&\f{1}{2} \f{d}{dt} || \xi \psi_1^{1/2} ||_2^2 = \la \cL_{3}(\xi), \xi \psi_1 \ra + 	\la \cR_{\xi}, \xi \psi_1 \ra
, \quad 
\f{1}{2} \f{d}{dt} || D_{\b}\xi  \psi_2^{1/2}||_2^2   = \la D_{\b} \cL_3(\xi), (D_{\b}\xi )\psi_2 \ra + \la D_{\b} \cR_{\xi}, D_{\b} \xi \psi_2 \ra.
\eal
\]
Recall the energy $E(R, 1)$ and the remaining term $\cR(R, 1)$ in Definition \ref{def:L2_m}. 
\[
E(R, 1)  =  ( || \Om \vp_1^{1/2}||_2^2  + || \eta \vp_1^{1/2}||_2^2  )^{1/2}, \quad \cR(R, 1) =  \la \cR_{\Om}, \Om \vp_1 \ra + \la \cR_{\eta}, \eta \vp_1 \ra .
\]
Combining the above reformulation, Propositions \ref{prop:L2less}, \ref{prop:L2}, \ref{prop:rho12}
and $R^{-2} \leq \rho$ \eqref{wg:L2_R0}, we know that there is some absolute constant $\mu_2$, which is small enough, e.g.
$\mu_2 K_3 < \f{1}{100}$, such that the following estimate holds 
\beq\label{eq:L2}
\bal
&\f{1}{2} \f{d}{dt} \lt(   E(R,0)^2 + \mu_1 E(\b,1)^2 + \mu_2 E ( R, 1)^2 
+  || \xi \psi_1^{1/2}||_2^2 +  || D_{\b}\xi  \psi_2^{1/2}||_2^2  \rt)  \\
\leq &  - ( \f{1}{9} - C\al) \lt(   E(R,0)^2 + \mu_1 E(\b,1)^2 + \mu_2 E ( R, 1)^2  
+ || \xi \psi_1^{1/2}||_2^2  +  || D_{\b}\xi \psi_2^{1/2}||_2^2 \rt) \\
&-  ( 3 - C\al)  L^2_{12}(\Om)(0) -(\f{1}{5} - C\al ) \B| \B| \td{L}^2_{12} \rho^{1/2} \B| \B|^2_{L^2(R)}  + \cR_0(\Om, \eta, \xi),
\eal
\eeq
where $\cR_0$ is defined below. We define the following weighted $L^2$ energy and the remaining term $\cR_0$ \footnote{In fact, $E_0$ contains a $L^2$ norm of the angular derivative $D_{\b}\Om, D_{\b} \eta, D_{\b} \xi$.}
\beq\label{eg:L2}
\bal
E_0(\Om, \eta, \xi) &\teq \lt(   E(R,0)^2 + \mu_1 E(\b,1)^2 + \mu_2 E ( R, 1)^2 
+  || \xi \psi_1^{1/2} ||_2^2 + ||  D_{\b}\xi  \psi_2^{1/2}||_2^2 \rt)^{1/2} ,\\
\cR_0(\Om, \eta, \xi) & \teq  \cR(R,0)  + \mu_1 \cR(\b,1) + \mu_2 \cR( R, 1) 
+  \la \cR_{\xi}, \xi   \psi_1 \ra +  \la D_{\b} \cR_{\xi}, ( D_{\b}  \xi) \psi_2 \ra ,
\eal
\eeq
where $(E(R,0), \cR(R,0) ), (E(\b,1) , \cR(\b, 1)), ( E(R,1), \cR(R, 1))$ are defined in \eqref{eg:R0}, \eqref{eg:b1} and \eqref{eg:R1}, respectively, and $\mu_i$ are some fixed absolute constants. 

We do not need the extra damping for $ \td{L}_{12}(\Om) \rho^{1/2}$ and $L_{12}(\Om)(0)$ in \eqref{eq:L2} due to Lemma \ref{lem:l12} and the fact that $E_0$ is stronger than $|| \Om \f{(1+R)^2 }{R^2}||_{L^2}$. Using \eqref{eq:l12}, we know that $C \al  || \td{L}_{12}(\Om) \rho^{1/2} ||^2_{L^2(R)}$, $C \al| L_{12}(\Om)(0) |^2$ can be bounded by $C \al E_0^2$. Hence, using the notation $E_0, \cR_0$, we derive the following result from \eqref{eq:L2}.
\begin{cor}\label{cor:lin}
Let $E_0(\Om, \eta, \xi),\cR_0(\Om, \eta,\xi)$ be the energy and the remaining term defined in \eqref{eg:L2}. Under the assumptions of Propositions \ref{prop:b1}, \ref{prop:L2} and \ref{prop:rho12}, we have 
\[
\f{1}{2}\f{d}{dt} E_0^2 \leq - ( \f{1}{9} - C\al) E_0^2   + \cR_0
.
\]

\end{cor}

\section{Higher order estimates and the energy functional}\label{sec:H2}
In this section, based on the weighted $L^2$ estimates established in Corollary \ref{cor:lin}, we proceed to perform the higher order estimates in the spirit of Propositions \ref{prop:L2}, \ref{prop:rho12} so that we can complete the nonlinear analysis. 
In subsection \ref{sec:nonH1}, we perform the weighted $H^1$ estimates of $\cL_i$ and illustrate how to apply several lemmas to control different terms in $D_R \cL_i$. In subsection \ref{sec:nonH2}, \ref{sec:H3}, we use a similar argument to establish weighted $H^2$ and $H^3$ estimates.  
In these estimates, we treat the nonlocal terms as perturbations and apply Lemma \ref{lem:dp} recursively.

Since $\bar{\xi}(x, y)$ does not decay in the $x$ direction when $y$ is fixed (see the estimates of $\bar{\xi}$ in Lemma \ref{lem:xi}), we cannot obtain the decay estimate for its perturbation $\xi$.
Hence, in order to obtain the $L^{\infty}$ control of $\xi$ and its derivatives, which will be used later to estimate the nonlinear terms, we cannot apply a $H^k \hookrightarrow L^{\infty}$ type Sobolev embedding. We perform the $L^{\infty}$ estimates of $\xi$ and its derivative directly in Section \ref{sec:inf}. This difficulty is not present in \cite{elgindi2019finite} by removing the swirl. The coefficient of the damping term in \eqref{eq:lin23} is given by $I_1 = -2 - \f{3}{1+R} \leq -2$. This simple inequality is actually related to the flow structure. In fact, $I_1$ is the leading order term of $-2 - \bar v_y $ (see \eqref{eq:lin} and \eqref{eq:simp4}), and the positive sign of $\bar v_y$ is related to the hyperbolic flow structure $\bar  u < 0, \bar v > 0$ and $\bar v(x, 0) = 0$. See more discussions after Lemma \ref{lem:selfsim}. 
The fact that $I_1$ is bounded uniformly away from $0$ enables us to establish the $L^{\infty}$ estimate of $\xi$.

\subsection{Weighted $H^1$ estimates} \label{sec:nonH1}
We remark that the weighted $H^1$ estimate with angular derivatives is already established in Section \ref{sec:stab_angle} about $D_{\b} \Om, D_{\b}\eta$ and Section \ref{sec:L2_xi} about $D_{\b}\xi$. 
Recall the weighted differential operator $D_R  = R \pa_R$ in Definition \ref{def:op}. We define an energy and a remaining term 
\beq\label{eg:R2}
\bal
E(R, 2)(\Om, \eta, \xi) & \teq \lt( || D_R \Om  \vp_1^{1/2}||_2^2 + || D_R \eta \vp_1^{1/2}||_2   
+ || D_R\xi  \psi_1^{1/2}||_2^2 \rt)^{1/2}, \\
\cR(R, 2)(\Om, \eta, \xi) & \teq \la D_R \cR_{\Om}, D_R \Om \vp_1 \ra 
+ \la D_R \cR_{\eta}, D_R \eta \vp_1 \ra + \la D_R \cR_{\xi}, D_R \xi \psi_1 \ra,
\eal
\eeq
where $\vp_1, \psi_1$ are defined in \eqref{wg}.

\begin{prop}\label{prop:H1}
Under the assumption of Corollary \ref{cor:lin} and that $\vp^{1/2}_1D_R \Om, 
\vp^{1/2}_1 D_R \eta , \psi^{1/2}_1 D_R \xi \in L^2$, we have 
\[
\la D_R \cL_1(\Om, \eta),( D_R \Om) \vp_1 \ra + \la D_R \cL_2(\Om, \eta), (D_R \eta) \vp_1 \ra 
+ \la D_R \cL_3(\xi),  (D_R \xi) \psi_1 \ra
\leq -\f{1}{6} E^2(R,2)+ K_4  E^2_0,
\]
where $K_4$ is some fixed absolute constant and $E_0, E(R,2)$ are defined in \eqref{eg:L2} and \eqref{eg:R2}.
\end{prop}	

\begin{proof}
Since $D_R$ commutes with $D_R, D_{\b}$ in $\cL_{i}, \cL_{i0}$ (see Definition \ref{def:op}), we have
\[
\bal
D_R \cL_1(\Om, \eta) 
&=\cL_{10}(D_R \Om, D_R \eta)  - D_R \f{3}{1+R} \cdot  D_{\b} \Om  + c_{\om} D_R( \bar{\Om} - R \pa_R \bar{\Om} ) = \cL_{10}(D_R \Om, D_R \eta) + \sum_{i=1}^2 {I_i} ,\\
D_R \cL_2(\Om, \eta) 
&= \cL_{20}(D_R \eta)- D_R \f{3}{1+R} \cdot  D_{\b} \eta + D_R(-2 + \f{3}{1+R}) \cdot \eta +\f{2}{\pi \al} \td{L}_{12}(\Om) \cdot D_R \bar{\eta} \\
&+\f{2}{\pi \al} D_R \td{L}_{12}(\Om) \cdot \bar{\eta}+c_{\om} D_R( \bar{\eta} - R \pa_R \bar{\eta} )  
= \cL_{20}( D_R \eta)+ \sum_{i=1}^5 II_i,  \\
D_R \cL_3(\Om,\xi) &= 
\cL_{30}(D_R \xi)- D_R \f{3}{1+R} \cdot D_{\b}\xi + D_R ( -2 - \f{3}{1+R}) \cdot \xi   
- \f{2}{\pi \al} \td{L}_{12}(\Om) \cdot D_R \bar{\xi} \\
 & - \f{2}{\pi \al} D_R \td{L}_{12}(\Om) \cdot  \bar{\xi} 
+c_{\om} D_R( 3\bar{\xi} - R \pa_R \bar{\xi} ) 
= \cL_{30}(D_R \xi) + \sum_{i=1}^5 III_i .
\eal
\]

Applying \eqref{eq:dp1} with $\vp = \vp_1$ (see \eqref{wg}), and 
\eqref{eq:dp2} with $ \psi = \psi_1$ (see \eqref{wg}) in Lemma \ref{lem:dp},
and $3|1-\s| < \f{1}{30}$, we yield 
\[
\bal
 \la \cL_{10}(D_R \Om, D_R \eta) , (D_R \Om) \vp_1 \ra
 + \la \cL_{20}(D_R \eta) , (D_R \eta) \vp_1 \ra
&\leq  -\f{1}{5} \B( || D_R \Om \vp_1^{1/2} ||_2^2 +  || D_R \eta   \vp_1^{1/2}||_2^2 \B) \\ 
\la \cL_{20} (D_R \xi) , (D_R \xi) \psi_1 \ra & \leq -\f{3}{8} || D_R\xi  \psi_1^{1/2}||_2^2  .\\
\eal
\]

Notice that $\vp_2$, $\psi_2$ \eqref{wg} satisfy $\vp_1 \leq \vp_2, \  \psi_1 \leq \psi_2.$ For the terms not involving $\td{L}_{12}(\Om), c_{\om}$, we use $E_0$ defined in \eqref{eg:L2} to control the weighted $L^2$ norm of $ D_{\b}\Om, D_{\b} \eta$. It is easy to see that 
\[
\bal
|| I_1 \vp_1^{1/2} ||_{L^2} &\les || D_{\b} \Om \vp_2^{1/2}||_{L^2} \les E_0 ,  \quad
|| II_1 \vp_1^{1/2} ||_{L^2} \les || D_{\b} \eta \vp_2^{1/2}||_{L^2} \les E_0 , \\
  || II_2 \vp_1^{1/2} ||_{L^2} & \les || \eta \vp_1^{1/2} ||_{L^2} \les E_0 , \qquad
|| III_1 \psi_1^{1/2} ||_{L^2}  \les  || D_{\b} \xi \psi_2^{1/2} ||_{L^2} \les E_0 ,\\ 
|| III_2 \psi_1^{1/2} ||_{L^2}  &\les  ||  \xi \psi_1^{1/2} ||_{L^2} \les E_0 .
\eal
\]
Recall $c_{\om} = -\f{2}{\pi \al} L_{12}(\Om)(0)$. Applying \eqref{eq:bar_ing} in Lemma \ref{eq:bar} 
to $I_2, II_5$ and \eqref{eq:xi_cw} in Lemma \ref{lem:xi} to $III_5$, we obtain 
\[
\bal
 || I_2 \vp_1^{1/2} ||_{L^2} &  \les |L_{12}(\Om)(0)| \les E_0 ,
 \quad  || II_5 \vp_1^{1/2} ||_{L^2}  \les | L_{12}(\Om)(0)|  \les E_0, \\
 || III_5 \psi_1^{1/2}||_{L^2} &\les \al |L_{12}(\Om)(0)| \les \al E_0.
 \eal
\]
Finally, for the $\td{L}_{12}(\Om)$ terms, we apply Lemma \ref{lem:ux}. 
To apply Lemma \ref{lem:ux}, we need the $L^{\infty}$ norm of some angular integrals, whose estimates are given in \eqref{eq:bar_ux} in Lemma \ref{lem:bar} about $\bar{\Om},\bar{\eta}$ and \eqref{eq:xi_ux} in Lemma \ref{lem:xi} about $\bar{\xi}$. Using these estimates, we obtain 
\[
\bal
|| II_3 \vp_1^{1/2} ||_{L^2} & \les || \td{L}_{12}(\Om) R^{-1}||_{L^2(R)} \les E_0, \qquad
||II_4 \vp_1^{1/2} ||_{L^2} \les || R^{-1} \Om||_{L^2} \les E_0,  \\
|| III_3 \psi_1^{1/2} ||_{L^2} & \les \al   || \td{L}_{12}(\Om) R^{-1} ||_{L^2(R)} \les \al  E_0, \quad  || III_4 \psi_1^{1/2} ||_{L^2} \les \al || R^{-1} \Om||_{L^2} \les \al E_0 . \\
\eal
\]

The result now follows using the Cauchy-Schwarz inequality (notice that $ - \f{1}{5}< -\f{1}{6}, \ \al < 1$)
and applying the energy notation \eqref{eg:R2}.
\end{proof}

Using the reformulation \eqref{eq:equiv}, we have
\[
\bal
&\f{1}{2} \f{d}{dt} E^2(R,2)  = \f{1}{2}\f{d}{dt} \lt( || D_R \Om \vp_1^{1/2}||_2^2  + || D_R \eta \vp_1^{1/2}||_2^2 + || D_R\xi  \psi_1^{1/2}||_2^2  \rt)^{1/2} \\
= &\la D_R \cL_1(\Om, \eta),( D_R \Om) \vp_1 \ra + \la D_R \cL_2(\Om, \eta), (D_R \eta) \vp_1 \ra 
+ \la D_R \cL_3(\xi),  (D_R \xi) \psi_1 \ra + \cR(R, 2).
\eal
\]
Therefore, it is not difficult to combine the above reformulation, Corollary \ref{cor:lin} and Proposition \ref{prop:H1} to prove the following results.
\begin{cor}\label{cor:H1}
Suppose that $\Om, \eta, \xi$ satisfy that $E_0(\Om, \eta, \xi), E(R,2)(\Om, \eta, \xi) < +\infty$, where $E_0, E(R,2)$ are defined in \eqref{eg:L2} and \eqref{eg:R2}, respectively. Then there exists an  absolute constant $\mu_3$, such that, the following statement holds true. The weighted $H^1$ energy $E_1$ and its associated remaining term $\cR_1$ defined by 
\beq\label{eg:H1}
E_1(\Om,\eta,\xi) \teq \lt( E_0^2(\Om,\eta,\xi) + \mu_3 E^2(R,2)(\Om, \eta,\xi) \rt)^{1/2}, \quad
\cR_1(\Om, \eta, \xi) \teq \cR_0 + \mu_3 \cR(R, 2),
\eeq 
where $\cR_0, \cR(R, 2)$ are defined in \eqref{eg:L2} and \eqref{eg:R2}, satisfy
\[
\f{1}{2} \f{d}{dt} 
 E_1^2
\leq ( -\f{1}{10} + C \al) E_1^2  + \cR_1.
\]
\end{cor}

\subsection{ Weighted $H^2$ estimates} \label{sec:nonH2}

We now proceed to perform the weighted $H^2$ estimates. Throughout this subsection, we assume that the following quantities are in $L^2$,
\[
\bal
 &\vp_2^{1/2} D^2_{\b}\Om ,\   \vp_2^{1/2} D_{\b} D_R \Om,  \  \vp_1^{1/2} D^2_R \Om , \quad 
 \vp^{1/2}_2 D^2_{\b}\eta,  \   \vp^{1/2}_2 D_R D_{\b}\eta,  
  \  \vp^{1/2}_1 D^2_{R}\eta,  \\
  &   \psi^{1/2}_2 D^2_{\b} \xi ,\  \psi^{1/2}_2 D_R D_{\b} \xi , 
  \  \psi_1^{1/2} D^2_{R} \xi .
 \eal
\]

We will use weights $\vp_1, \psi_1$ for $D_R^2$ derivative, $\vp_2, \psi_2$ for $D^2_{\b}$ and $D_R D_{\b}$ derivatives in the weighted $H^2$ norm to be constructed.
Recall $\cL_i, \cL_{i0}, D_{\b}$ in Definition \ref{def:op}. We perform the estimate of the second derivatives in the order of $D^2_{\b}, D_{\b} D_{R}, D_R^2$. The motivation to first estimate the angular derivative terms is the same as that in Sections \ref{sec:angle} and \ref{sec:stab_angle}. This order of energy estimates has been used in \cite{elgindi2019finite}. In these estimates, we treat the nonlocal terms as perturbations and apply Lemma \ref{lem:dp} recursively.

Notice that $D_{\b}$ commutes with $\cL_{i0},i=1,2,3$. For the $\cL_{i0}$ part in $\cL_i$, applying Lemma \ref{lem:dp} with $\vp = \vp_2, \d = \g = 1+ \al /10$, $\psi = \psi_2$ (see Definition \ref{def:wg} for $\vp_i,\psi_i$), we obtain 
\[
\bal
&\la D^2_{\b} \cL_{10}(\Om, \eta) , (D^2_{\b}\Om ) \vp_2 \ra 
+ \la D^2_{\b} \cL_{20}(\Om, \eta) , (D^2_{\b}\eta ) \vp_2 \ra  
+ \la D^2_{\b} \cL_{30}(\Om, \xi),  (D^2_{\b}\xi ) \psi_2 \ra   \\
\leq &    ( - \f{1}{4} + \al) \lt( || D^2_{\b}\Om  \vp_2^{1/2}||_2^2  
+ || D^2_{\b}\eta \vp_2^{1/2}||_2^2   \rt) + (-\f{3}{8} + \al) ||  D^2_{\b}\xi \psi_2^{1/2}||_2^2  .
\eal
\]
For the $D_R D_{\b}$ derivative, since $D_R$ does not commute with $\cL_{i0}$, we have
\beq\label{eq:H2_Rb}
(D_R D_{\b})\cL_{g0}(\Om, \eta, \xi)- \cL_{g0} (D_R D_{\b} \Om, 
D_R D_{\b} \eta, D_R D_{\b} \xi) = - D_R\f{3}{1+R} D^2_{\b} g + M,
\eeq
where $g = \Om, \eta, \xi$, we have used the notation $\cL_{\Om 0}= \cL_{10}, \cL_{\eta 0} = \cL_{20}, \cL_{\xi 0} = \cL_{30}$, and $M$ denotes some terms that are of lower order than the second derivatives of $\Om, \eta, \xi$ and have coefficients bounded by some absolute constant. For example, $M$ contains the term $D_R (-2 + \f{3}{1+R}) D_{\b} \eta$ in $D_R D_{\b} ( (-2 + \f{3}{1+R}) \eta )$. The term $M$ may vary from line to line but its weighted $L^2$ norm can be easily bounded by the weighted $H^1$ energy $E_1$ \eqref{eg:H1}. Applyiny Lemma \ref{lem:dp} with $\vp = \vp_2$ to $\cL_{g0} (D_R D_{\b} \Om, D_R D_{\b} \eta), g= \Om, \eta $, and with $\psi = \psi_2$ to $\cL_{\xi0}(D_R D_{\b} \xi) $, and then using the Cauchy-Schwarz inequality to control the right hand side of \eqref{eq:H2_Rb}, we yield 
\[
\bal
&\la D_R D_{\b}\cL_{10} ( \Om,  \eta) , (D_R D_{\b}\Om) \vp_2 \ra
+ \la  D_R D_{\b} \cL_{20} ( \eta)  , (D_R D_{\b} \eta) \vp_2 \ra  
+\la D_R D_{\b} \cL_{30}( \xi), ( D_R D_{\b} \xi) \psi_2 \ra \\
\leq &  (-\f{1}{5} + \al) \B( || D_R D_{\b}\Om  \vp_2^{1/2}||_2^2    + 
|| D_R D_{\b} \eta  \vp_2^{1/2}||_2^2  +  || D_R D_{\b} \xi  \psi_2^{1/2}||_2^2  \B)\\
&+ C \lt(  E^2_1 + || D^2_{\b}\Om  \vp_2^{1/2}||_2^2  
+ || D^2_{\b}\eta   \vp_2^{1/2}||_2^2  + || D^2_{\b}\xi \psi_2^{1/2}||_2^2   \rt) ,
\eal
\]
where the constant $-\f{1}{5}$ is a result from first applying Lemma \ref{lem:dp} to $(\Om, \eta)$ and $\xi$, which gives two damping factors $-\f{1}{4} + \al, -\f{3}{8} +\al$, and then using the Cauchy-Schwarz inequality ($-\f{3}{8}, -\f{1}{4} < -\f{1}{5}$). 

Similarly, for the $D_R^2$ derivative, we have for $g = \Om, \eta, \xi$
\[
(D^2_R )\cL_{g0}(\Om, \eta, \xi)- \cL_{g0} (D^2_R  \Om, 
D^2_R  \eta, D^2_R  \xi) = -2  D_R\f{3}{1+R} D_R D_{\b} g + M.
\]
Applyiny Lemma \ref{lem:dp} with $\vp = \vp_1$ to $\cL_{g0} (D^2_R  \Om, D^2_R  \eta) , g = \Om, \eta$, and $\psi = \psi_1$ to $\cL_{\xi0}(D_R^2 \xi)$ will give two damping factor $-\f{1}{4} + \f{3}{100}, -\f{3}{8}+ \al$. We then use the Cauchy-Schwarz inequality to yield
\[
\bal
&\la D^2_R \cL_{10} ( \Om,  \eta) , (D^2_R \Om) \vp_1 \ra
+ \la  D^2_R  \cL_{20} ( \eta)  , (D^2_R \eta) \vp_1 \ra  
+\la D^2_R  \cL_{30}( \xi), ( D^2_R  \xi) \psi_1 \ra \\
\leq &  (-\f{1}{5} +\al) \B( ||  D^2_R \Om \vp_1^{1/2}||_2^2    + 
||  D^2_R  \eta   \vp_1^{1/2}||_2^2  + ||  D^2_R  \xi   \psi_1^{1/2}||_2^2  \B)\\
&+ C \lt(  E^2_1 + || D_R D_{\b}\Om  \vp_2^{1/2}||_2^2 
+  || D_R D_{\b}\eta   \vp_2^{1/2}||_2^2  + || D_R D_{\b}\xi \psi_2^{1/2}||_2^2   \rt) ,
\eal
\]
where we have used $ \vp_1\les \vp_2, \psi_1  \les \psi_2$ to obtain $\la (D_RD_{\b} \Om)^2, \vp_1 \ra \les \la (D_RD_{\b} \Om)^2, \vp_2 \ra$ and other similar terms. We have also used $-\f{1}{4} + \f{3}{100}, -\f{3}{8} < -\f{1}{5}$ when we have applied the Cauchy-Schwarz inequality.

Notice that the remaining terms in $\cL_i$ except for $\cL_{i0}$ are the $\td{L}_{12}(\Om)$ terms and $c_{\om}$ terms. For the $\td{L}_{12}(\Om)$ terms, we use Lemma \ref{lem:ux} and then \eqref{eq:bar_ux} in Lemma \ref{lem:bar} about $\bar{\Om}, \bar{\eta}$ and \eqref{eq:xi_ux} in Lemma \ref{lem:xi} about $\bar{\xi}$ to estimate the $L^{\infty}$ norm of some angular integrals. For the $c_{\om}$ terms, we use the estimates in \eqref{eq:bar_ing} in Lemma \ref{lem:bar} about $\bar{\Om},\bar{\eta}$ and \eqref{eq:xi_cw} in Lemma \ref{lem:xi} about $\bar{\xi}$. We remark that from Proposition \ref{lem:ux}, the norm of $R^{-1} D_R^2 \td{L}_{12}(\Om)$ can be bounded by the norm of $R^{-1} D_R \Om$, which can be further bounded by $E_1$. 

Combining these estimates and using the Cauchy-Schwarz inequality, we obtain 
\beq\label{eq:H2_1}
\bal
&\la D^2_{\b} \cL_{1}(\Om, \eta) , (D^2_{\b}\Om ) \vp_2 \ra 
+ \la D^2_{\b} \cL_{2}(\Om, \eta) , (D^2_{\b}\eta ) \vp_2 \ra  
+ \la D^2_{\b} \cL_{3}(\Om, \xi),  (D^2_{\b}\xi ) \psi_2 \ra  
 \\
\leq & ( - \f{1}{6} + \al) \lt( ||  D^2_{\b}\Om  \vp_2^{1/2} ||_2^2  + || D^2_{\b}\eta \vp_2^{1/2}||_2^2  + || D^2_{\b}\xi \psi_2^{1/2}||_2^2  \rt) + C\al  E_1^2 ,\\
\eal
\eeq
\[
\bal
&\la D_R D_{\b} \cL_1(\Om, \eta) , (D_R D_{\b} \Om ) \vp_2 \ra
+\la D_R D_{\b} \cL_2 (\Om, \eta) ,  (D_R D_{\b} \eta ) \vp_2 \ra
+ \la D_R D_{\b} \cL_3(\Om, \xi) , (D_R D_{\b} \xi ) \psi_2 \ra
 \\
\leq & ( -\f{1}{6} + \al) \lt( || D_R D_{\b} \Om  \vp_2^{1/2}||_2^2  + || D_R D_{\b} \eta \vp_2^{1/2}||_2^2 + || D_R D_{\b} \xi  \psi_2^{1/2} ||_2^2
\rt)  \\
&  + K_5 \lt(  || D^2_{\b}\Om  \vp_2^{1/2}||_2^2 + || D^2_{\b}\eta \vp_2^{1/2}||_2^2  +  || D^2_{\b}\xi \psi_2^{1/2} ||_2^2   + E^2_1 \rt) , \\
&\la D^2_R  \cL_1(\Om, \eta) , (D_R^2 \Om ) \vp_1 \ra
+\la D^2_R  \cL_2 (\Om, \eta) ,  (D^2_R  \eta ) \vp_1 \ra  
+ \la  \cL_{30}( D^2_R  \xi), ( D^2_R  \xi) \psi_1 \ra \\
\leq &  (-\f{1}{ 6} + \al)  \lt( || D^2_R  \Om  \vp_1^{1/2}||_2^2 + || D^2_R  \eta \vp_1^{1/2} ||_2^2 + || D^2_R  \xi   \psi_1^{1/2} ||_2^2  \rt)   \\
  &+ K_5 \lt( 
 ||  D_R D_{\b}\Om  \vp_2^{1/2} ||_2^2  + || D_R D_{\b}\eta \vp_2^{1/2}||_2^2  
+ ||  D_R D_{\b} \xi  \psi_2^{1/2} ||_2^2
   + E^2_1 \rt) ,
   \eal
\]
where $K_5$ is some fixed absolute constant and we have used $-\f{1}{4}, -\f{1}{5} < -\f{1}{6}$ when we applied the Cauchy-Schwarz inequality to control the inner product between the $\td{L}_{12}(\Om), c_{\om}$ terms and the second derivative terms.

Combining Corollary \ref{cor:H1} with estimates \eqref{eq:H2_1}, we know that there exist some absolute constants $\mu_{2,k}$ that can be determined in the order of $k = 0, 1,2 $, such that the weighted $H^2$ energy functional $E_2$ 
and its associated remaining term $\cR_2$ defined below satisfy the estimates stated in Corollary \ref{cor:H2}.
\beq\label{eg:H2}
\bal
E^2_2(\Om, \eta, \xi)
&\teq E^2_1 +
 \sum_{0\leq k \leq 2}  \mu_{2, k}  \lt( 
 ||  D_R^{k} D^{2-k}_{\b} \Om  \vp_i^{1/2}||_2^2  
 +  || D_R^k D^{2-k}_{\b} \eta  \vp_i^{1/2}||_2^2  
 +  || D_R^k D^{2-k}_{\b} \xi  \psi_i^{1/2}||_2^2   \rt), \\
\cR_2(\Om, \eta,\xi)
&\teq \cR_1 +
 \sum_{ 0 \leq k \leq 2}  \mu_{2, k}  \lt( 
\la  D_R^{k} D^{2-k}_{\b} \cR_{\Om},  ( D_R^{k} D^{2-k}_{\b} \Om ) \vp_i \ra   
 +   \la  D_R^k D^{2-k}_{\b}  \cR_{\eta} , ( D_R^k D^{2-k}_{\b} \eta ) \vp_i \ra  \rt. \\
 & \lt. +  \la   D_R^k D^{2-k}_{\b} \cR_{\xi}, ( D_R^k D^{2-k}_{\b} \xi )\psi_i \ra  \rt), \\
\eal
\eeq
where  $E_1, \cR_1$ are defined in \eqref{eg:H1} and $(\vp_i, \psi_i) = (\vp_2, \psi_2)$ for $k=0,1$ and $(\vp_1, \psi_1)$ otherwise. 
\begin{cor}\label{cor:H2}
Suppose that $E_2(\Om, \eta, \xi) < +\infty$. Then the energy $E_2$ satisfies 
\[
\f{1}{2}\f{d}{dt} E_2^2(\Om, \eta,\xi) 
\leq  (- \f{1}{11} + C \al ) E_2^2  + \cR_2.
\] 
\end{cor}

\subsection{Weighted $H^3$ estimates}\label{sec:H3}

Recall the weights $\vp_i, \psi_i$ in Definition \ref{def:wg}. For $D_R^3 \Om, D_R^3 \eta$, we use weight $\vp_1$; for other third derivatives $D_R^i D^j_{\b}\Om, D_R^i D^j_{\b}\eta$, we use weight $\vp_2$. For $D_R^3 \xi$, we use weight $\psi_1$; for other third derivatives $D_R^i D^j_{\b} \xi$, we use weight $\psi_3$. The reason we perform weighted $H^3$ is to establish Proposition \ref{prop:tran3}.

In the same spirit of the weighted $H^2$ energy functional $E_2$ and Corollary \ref{cor:H2}, 
we can show that there exist some absolute constants $\mu_{3, k}$, which can be determined in the order $k=0, 1,2,3$, such that the weighted $H^3$ energy functional $E_3 \geq 0$ and its associated remaining term $\cR_3$ defined below satisfy the estimates stated in Corollary \ref{cor:H3}.
\beq\label{eg:H3}
\bal
E_3^2(\Om, \eta, \xi) &\teq E_2^2 + \sum_{ 0\leq k \leq 3}  \mu_{3, k}  \lt( 
 || D_R^{k} D^{3-k}_{\b} \Om  \vp_i^{1/2}||_2^2  
 +   || D_R^k D^{3-k}_{\b} \eta  \vp_i^{1/2}||_2^2  
 +  || D_R^k D^{3-k}_{\b} \xi  \psi_i^{1/2}||_2^2   \rt), \\
 \cR_3(\Om, \eta, \xi) & \teq \cR_2 +
 \sum_{ 0 \leq k \leq 3}  \mu_{3, k}  \lt( 
\la  D_R^{k} D^{3-k}_{\b} \cR_{\Om},  ( D_R^{k} D^{3-k}_{\b} \Om ) \vp_i \ra   
 +   \la  D_R^k D^{3-k}_{\b}  \cR_{\eta} , ( D_R^k D^{3-k}_{\b} \eta ) \vp_i \ra  \rt. \\
 & \lt. +  \la   D_R^k D^{3-k}_{\b} \cR_{\xi}, ( D_R^k D^{3-k}_{\b} \xi )\psi_i \ra  \rt), \\
 \eal
 \eeq
where $E_2, \cR_2$ are defined in \eqref{eg:H2}, $(\vp_i, \psi_i) = (\vp_3, \psi_3)$ for $k=0,1,2$ and $(\vp_1, \psi_1)$ otherwise. 
\begin{cor}\label{cor:H3}
Suppose that $E_3(\Om, \eta, \xi) < +\infty$. Then the energy $E_3$ satisfies 
\[
\f{1}{2}\f{d}{dt} E_3^2(\Om, \eta,\xi) 
\leq  (- \f{1}{12} + C \al ) E_3^2  + \cR_3.
\] 
\end{cor}

\subsection{$\cC^1$ estimates}\label{sec:inf}

We introduce the following weights for the weighted $C^1$ estimates
\beq\label{wg:c1}
\phi_1 = \f{1+R}{R}, \quad \phi_2 = 1 + (R \sin(2\b)^{\al})^{- \f{1}{40}},
\eeq
and the following $\cC^1$ norm
\beq\label{norm:c1}
\bal
||  f ||_{\cC^1} &\teq || f ||_{\infty} + || \phi_1 D_R f ||_{\infty}  + || \phi_2 D_{\b}f ||_{\infty}  \\
&= || f ||_{\infty} + || \f{1+R}{R} D_R f ||_{\infty}  + || ( 1 + (R \sin(2\b)^{\al})^{-\f{1}{40}} ) D_{\b}f ||_{\infty} .
\eal
\eeq

To close the nonlinear estimates, we need to control the $L^{\infty}$ norm of $\Om, \eta, \xi$ and their derivatives. For $\Om, \eta$, the weighted $H^3$ estimates that we have obtained guarantee that $\Om, \eta \in \cC^1$, which will be established precisely in later sections. For $\xi$, however, since the weight $\psi_2$ (see Definition \ref{def:wg}) is less singular in $\b$ for $\b$ close to $0$, the weighted $H^3$ space associated to $\xi$ is not embedded continuously into $\cC^1$. Alternatively, we perform $\cC^{1}$ estimates on $\xi$ directly.
This difficulty is absent in \cite{elgindi2019finite} by removing the swirl.

Firstly, the transport term in the $\xi$ equation in \eqref{eq:lin}, including the nonlinear part in $N_{\xi}$, is given by 
\beq\label{eq:cA}
\cA (\xi )\teq (1+ 3\al) D_R \xi + \al c_l D_R \xi + (\bar{\uu} \cdot \na) \xi + (\uu \cdot \na) \xi.
\eeq

The main damping term in the $\xi$ equation is $(-2 - \bar{v}_y) \xi$. \eqref{eq:simp4} shows that $-\bar{v}_y = -\f{3}{1+R} + l.o.t.$. Therefore, we consider 
\beq\label{eq:Xi1}
\bal
(-2 - \bar{v}_y  )\xi
&= ( -2 - \f{3}{1+R}  )\xi + \Xi_1 , \quad \Xi_1  \teq (\f{3}{1+R} - \bar{v}_y) \xi.
\eal
\eeq
We further introduce $\Xi_2$ to denote the lower order terms in the $\xi$ equation \eqref{eq:lin}
\beq\label{eq:Xi2}
\Xi_2 =  - v_y \bar{\xi} + c_{\om} ( 2 \bar{\xi} - R\pa_R \bar{\xi})
  +( \al c_{\om} R\pa_R - ( \uu \cdot  \na ) ) \bar{\xi} -  (u_y \bar{\eta} +\bar{u}_y \eta ).
\eeq
Then the $\xi$ equation in \eqref{eq:lin} can be simplified as 
\beq\label{eq:xi0_2}
\pa_{t} \xi + A(\xi) = (-2 - \f{3}{1+R}) \xi + \Xi_1 + \Xi_2 +  \bar{F}_{\xi} + N_o,
\eeq
where we have moved part of the nonlinear term $N_{\xi}$ defined in \eqref{eq:non} to the transport term $A(\xi)$ and $N_o$ is given by 
\beq\label{eq:non_xio}
 N_o = (2c_{\om} - v_y)\xi   -  u_y \eta .
\eeq

Notice that $ - \f{3}{1+R} \leq 0$. Multiplying $\xi$ on both sides and then performing $L^{\infty}$ estimate yield 
\beq\label{eq:xi_inf1}
\f{1}{2}\f{d}{dt} || \xi||^2_{\infty}  \leq -2 || \xi||^2_{\infty} 
+ || \xi||( || \Xi_1||_{L^{\infty}} + ||\Xi_2 ||_{\infty} + ||\bar{F}_{\xi}||_{\infty} + || N_{\xi}||_{\infty} ),
\eeq
where the transport term $A(\xi)$ vanishes.

Before we perform weighted $C^1$ estimates, 
we rewrite $\cA(\xi)$ defined in \eqref{eq:cA} as follows 
\beq\label{eq:cA2}
\cA(\xi) =( (1+ 3\al + \al c_l) D_R \xi +  \f{3 }{ 1+R} D_{\b} \xi )
+ (  ( (\uu + \bar{\uu}) \cdot \na - \f{3}{1+R} D_{\b} )  \xi  ) \teq \cA_1(\xi) + \cA_2(\xi).
\eeq

Recall the weights $\phi_1, \phi_2$ in \eqref{wg:c1}. Observe that $D_{\b}$ commutes with $\cA_1$ and $D_R$ commutes with $D_R, D_{\b}$. Denote by $[P, Q]$ the commutator $P Q - Q P$. A direct calculation shows that
\beq\label{eq:xi_com}
\bal
\phi_1 D_R \cA \xi - \cA( \phi_1 D_R \xi ) & =  
\phi_1 D_R \f{3}{1+R} \cdot D_{\b} \xi -  (1+ 3\al + \al c_l) D_R \phi_1 \cdot D_R \xi 
+ [ \phi_1 D_R ,\cA_2] \xi, \\
& = - \f{3}{1+R} D_{\b} \xi + (1+ 3\al + \al c_l) \f{1}{1 +R} \phi_1 D_R \xi + [ \phi_1 D_R ,\cA_2] \xi , \\
 \phi_ 2D_{\b} \cA \xi  - \cA ( \phi_2 D_{\b} \xi ) & = - \cA_1 (\phi_2 -1 ) \cdot D_{\b} \xi + [  \phi_ 2 D_{\b}, \cA_2] \xi ,
\eal
\eeq
where we have used $\cA_1 (1) = 0$ in the last equality. 
Hence, using \eqref{eq:xi0_2} and the above calculation, we obtain the equation of $\phi_1 D_R \xi$
\[
\bal
\pa_t ( \phi_1 D_R \xi) + \cA(\phi_1 D_R \xi )
&=  \f{3}{1+R} D_{\b} \xi - (1 + 3 \al + \al c_l)  \f{1}{1 + R} \phi_1 D_R \xi  - [\phi_1 D_R, \cA_1 \xi] \\
&+ \phi_1 D_R ( (-2 - \f{3}{1+R}) \xi )  +  \phi_1 D_R ( \Xi_1 + \Xi_2 + \bar{F}_{\xi} + N_o) .
\eal
\]
We remark that $- (1 + 3\al) \f{1}{1+R} \phi_1 D_R \xi$ is a damping term, though we will not use it. Performing $L^{\infty}$ estimate for $\phi_1 D_{\b} \xi$, we obtain the following estimate, which is similar to \eqref{eq:xi_inf1} 
\beq\label{eq:xi_inf2R}
\bal
& \f{1}{2} \f{d}{dt}  || \phi_1 D_R \xi ||^2_{\infty} 
\leq -(2 - |\al c_l| )  || \phi_1 D_R \xi ||^2_{\infty}  + 3 || \phi_1 D_R \xi ||_{\infty} || \xi ||_{\infty }  \\
& + || \phi_1 D_R \xi ||_{ L^{\infty} }
( 3 || D_{\b} \xi ||_{\infty} +|| [\phi_1 D_R, \cA_2] \xi ||_{\infty} +
 || \phi_1 D_R ( \Xi_1 + \Xi_2 + \bar{F}_{\xi} + N_o  ) ||_{L^{\infty}} ),
\eal
\eeq
where we have used $|\f{3}{1+R}| \leq 3$ and 
\[
\phi_1 D_R \xi \cdot \phi_1 D_R (- 2 - \f{3}{1+R}) \xi 
= \phi_1 D_R \xi  \cdot ((- 2 - \f{3}{1+R}) \phi_1 D_R \xi
+ \phi_1 \f{3R\xi }{ (1+R)^2} )
\leq -2 ( \phi_1 D_R \xi)^2 + 3  || \phi_1 D_R \xi ||_{\infty} ||\xi||_{\infty}.
\]

Similarly, using \eqref{eq:xi0_2}, \eqref{eq:xi_com} and then performing $L^{\infty}$ estimate on $\phi_2 D_{\b} \xi$, we obtain 
\beq\label{eq:xi_inf2b}
\bal
 \f{1}{2} \f{d}{dt}  || \phi_2 D_{\b} \xi||^2_{\infty} 
&\leq -2  || \phi_2 D_{\b}\xi||^2_{\infty}  + || \phi_2 D_{\b} \xi||_{\infty} 
|| \cA_1(\phi_2 - 1) \cdot D_{\b}\xi ||_{L^{\infty}}  \\
&+ 
 || \phi_2 D_{\b} \xi||_{\infty}  ( || [ \phi_2 D_{\b}, \cA_2] \xi ||_{\infty} +
||\phi_2 D_{\b} (\Xi_1 + \Xi_2 + \bar{F}_{\xi}  + N_o ) ||_{L^{\infty}} )  ,\\
\eal
\eeq
where we have used 
\[
\phi_2 D_{\b} \xi  \cdot \phi_2 D_{\b} (- 2 - \f{3}{1+R}) \xi 
\leq -2 ( \phi_2 D_{\b} \xi )^2 .
\]


We defer the estimates of the remaining terms in \eqref{eq:xi_inf1},\eqref{eq:xi_inf2R},\eqref{eq:xi_inf2b} which are small, to Section \ref{sec:non}.

\subsection{The energy functional and the $\cH^m$ norm}\label{sec:energy}

Using all the energy notations \eqref{eg:b1}, \eqref{eg:R0}, \eqref{eg:R1}, \eqref{eg:L2}, \eqref{eg:R2},\eqref{eg:H1}, \eqref{eg:H2} and \eqref{eg:H3}, we obtain the full expression of $E_3$ \eqref{eg:H3}
\beq\label{eg:H3_full}
\bal
E_3^2
&=    || \Om  \vp_0^{1/2}||_2^2  + || \eta  \psi_0^{1/2}||_2^2  + \f{81}{  4\pi c}  L^2_{12}(\Om)(0) + \mu_1 \lt( ||  D_{\b} \Om \vp_2^{1/2}||_2^2 + || D_{\b} \eta \psi_2^{1/2}||_2^2 \rt)  +  || D_{\b}\xi \psi_2^{1/2}||_2^2 
\\
&+\mu_2\lt( ||  \Om  \vp_1^{1/2}||_2^2  + || \eta  \psi_1^{1/2}||_2^2  \rt) 
+   || \xi \psi_1^{1/2}||_2^2   +
\mu_3 \lt( ||  D_R \Om  \vp_1^{1/2} ||_2^2  + ||  D_R \eta  \psi_1^{1/2} ||_2^2  
+ ||  D_R\xi  \psi_1^{1/2}||_2^2  \rt) \\
&+ \sum_{l=2,3} \ \sum_{ 0\leq k \leq l}  \mu_{l, k}  \lt( 
\la  ||  D_R^{k} D^{l-k}_{\b} \Om  \vp_i^{1/2}||_2^2    
 +   ||  D_R^k D^{l-k}_{\b} \eta  \vp_i^{1/2} ||_2^2 
 +  || D_R^k D^{l-k}_{\b} \xi  \psi_i^{1/2} ||_2^2 \rt) , 
\eal
\eeq
where $(\vp_i, \psi_i) = (\vp_1, \psi_1)$ for $k=l$, $(\vp_i, \psi_i) = (\vp_2, \psi_2)$ for $k\neq l$ and $l=2, 3$.

Recall $\vp_i, \psi_i$ in Definition \ref{def:wg}. We define the $\cH^m(\rho)$ norm with $m \geq 0$ as follows 
\beq\label{norm:H22}
|| f ||_{\cH^m(\rho)}  \teq \sum_{ 0\leq k \leq m}   || D^k_R f \rho_1^{1/2} ||_{L^2} 
+ \sum_{   \ i+ j \leq m - 1
} || D^{i}_R D^{j+1}_{\b}  f \rho_2^{1/2}||_{L^2} .
\eeq
The $\cH^0( \vp)$ norm is the same as $L^2(\vp_1)$ norm. For the $\cH^3(\vp)$ norm, we use \eqref{norm:H22} with $\rho_i = \vp_i$; for the $\cH^3(\psi)$ norm, we use \eqref{norm:H22} with $\rho_i = \psi_i, i=1,2$.
We simplify $\cH^3(\vp)$ as $\cH^3$. 
We apply the $\cH^3$ norm for $\Om, \eta$ and the $\cH^3(\psi)$ norm for $\xi$. We use the $\cH^m$ norm to establish the elliptic estimate in the next Section. We will only use the $\cH^2, \cH^2(\psi)$ and $\cH^3, \cH^3(\psi)$ norms. Remark that the $\cH^m$ norm is different from the canonical Sobolev $H^m$ norm.

From the Definition \ref{def:wg} of $\vp_i, \psi_i$, we have a simple relationship between $\cH^m$ and $\cH^{m}(\psi)$.
\begin{lem}\label{lem:H2H2}
For $  \f{\g - \s}{2} \leq \lam \leq \f{1}{2}$ and $m \leq 3$, we have 
\beq\label{eq:H2H2}
|| f||_{\cH^m(\psi)}  \les  || f||_{\cH^m} , \quad  || \sin(\b)^{\lam} f ||_{\cH^m}  \les || f||_{\cH^m(\psi)} .
\eeq
\end{lem}
The proof follows from several simple inequalities $ \psi_i \les \vp_i,  \sin(\b)^{\lam} \vp_i \les \psi_i$,  
$ D^i_{\b} \sin(\b)^{\lam} \cdot \vp_2 = 2 \lam \cos^2(\b)  \sin(\b)^{\lam} \vp_2 \les  \psi_1  $ for $i  \leq 3$, and expanding the norm.

We also define the corresponding inner products on $\cH^3$ and $\cH^3(\psi)$, which are equivalent to $\cH^3, \cH^3(\psi)$
\beq\label{eq:inner}
\bal
\la f, g \ra_{\cH^3} & \teq \mu_1 \la D_{\b} f , D_{\b} g \vp_2 \ra
+ \mu_2 \la  f ,   g \vp_1 \ra  + \mu_3 \la D_R f , D_R g \vp_1 \ra \\
&+ \sum_{ k=2,3} \mu_{k, k}   || D^k_R f \vp_1^{1/2} ||_{L^2}  
+ \sum_{  j \geq 1, \  2 \leq  i+ j \leq 3}  \mu_{ i + j, i} \la  D^i_R D^j_{\b}  f , D^i_R D^j_{\b}  g \vp_2\ra , \\
\la f, g \ra_{\cH^3(\psi)} & \teq   \la D_{\b} f , D_{\b} g \psi_2 \ra
+  \la  f ,   g \psi_1 \ra  + \mu_3 \la D_R f , D_R g \psi_1 \ra  \\
&+ \sum_{ k=2,3} \mu_{k, k}   || D^k_R f \psi_1^{1/2} ||_{L^2}  + \sum_{  j \geq 1, \  2 \leq i+ j \leq 3}  \mu_{ i + j, i} \la  D^i_R D^j_{\b}  f , D^i_R D^j_{\b}  g \psi_2\ra.\\
\eal
\eeq

Clearly, using these notations and \eqref{eg:L2}, \eqref{eg:R2}, \eqref{eg:H1}, \eqref{eg:H2}, \eqref{eg:H3}, we have  
\beq\label{eq:E3}
\bal
E_3^2 &= \f{81}{4 \pi c} L^2_{12}(\Om)(0) + \la \Om^2, \vp_0 \ra + \la \eta^2, \psi_0 \ra
+ \la \Om, \Om \ra_{\cH^3} + \la \eta,  \eta \ra_{\cH^3}+ \la \xi, \xi \ra_{\cH^3(\psi)}  ,\\
\cR_3 & = \la \cR_{\Om}, \Om \vp_0 \ra + \la \cR_{\eta} ,\eta \psi_0 \ra +  \f{81}{4\pi c} L_{12}(\Om)(0) \la \cR_{\Om}, \sin(2\b) R^{-1} \ra + 
 \la \cR_{\Om}, \Om \ra_{\cH^3} + 
 \la \cR_{\eta}, \eta \ra_{\cH^3} 
 + \la \cR_{\xi}, \xi \ra_{\cH^3(\psi)}.
 \eal
\eeq

We also have the following simple inequality 
\beq\label{eq:E2H2}
 || \Om||^2_{\cH^3}  + || \eta||^2_{\cH^3}   + || \xi||^2_{\cH^3(\psi)} \les  E_3^2(\Om, \eta, \xi).
\eeq

\section{Elliptic Regularity Estimates and Estimate of nonlinear terms}\label{sec:elli}

In this section, we first follow the argument in \cite{elgindi2019finite} to establish the $\cH^3$ estimates for the elliptic operator and justify that the leading order term of the (modified) stream function can be written as \eqref{eq:biot3} in subsection \ref{sec:elliH3}. Then to simplify our nonlinear estimates, we will generalize several estimates derived in \cite{elgindi2019finite} in subsection \ref{sec:nonterm}. 

The fact that $\bar{\xi}$ (see Lemma \ref{lem:xi}) and $\xi$ do not decay in certain direction makes the estimates of nonlinear terms complicated since we cannot apply the same weighted Sobolev norm to $\Om, \eta, \xi$. More precisely, the $\cH^k(\psi)$ norm for $\xi$ is weaker than the $\cH^k$ norm for $\Om, \eta$ (see \eqref{eq:H2H2}). To compensate this, we use a combination of $\cC^1$ norm and $\cH^k(\psi)$ norm for $\xi$. We will establish several estimates for $\xi$ in subsection \ref{sec:nonterm}. Moreover, estimating the $\cH^k$ norm of $v_x \xi$ in the $\eta$ equation \eqref{eq:lin} will be more difficult since $\xi$ is in a weaker Sobolev space. In subsection \ref{sec:new}, we will estimate the nonlinear term $v_x \xi$ in the $\eta$ equation \eqref{eq:lin}. We will also perform a new estimate of the transport term with weighted $H^3$ data.


Recall that the Biot-Savart law in $\R^2_+$ is given by \eqref{eq:biot},
which can be reformulated using the polar coordinates as 
\[
 - \pa_{rr} \psi - \f{1}{r} \pa_r  \psi - \f{1}{r^2} \pa_{\b \b } \psi   = \om ,
\]
where $r  = \sqrt{x^2 + y^2}, \b = \arctan(y / x)$. We introduce $R = r^{\al}$ and  $\Psi(R,\b) = \f{1}{r^2} \psi(r, \b), \Om(R, \b) = \om(r, \b)$. It is easy to verify that the above elliptic equation is equivalent to 
\beq\label{eq:elli}
\cL_{\al}(\Psi) \teq - \al^2 R^2 \pa_{RR} \Psi- \al(4+\al) R \pa_R \Psi - \pa_{\b \b} \Psi - 4 \Psi = \Om.
\eeq
The boundary condition of $\Psi$ is given by 
\beq\label{eq:ellibc}
\Psi(R,0) = \Psi(R, \pi/2) =0, \quad \lim_{R\to \infty} \Psi(R, \b) =  0.
\eeq

\subsection{$\cH^3$ estimates}\label{sec:elliH3}
Recall that the $\cH^m, m \geq 0$ norm defined in Section \ref{sec:energy} is given by
\beq\label{norm:H2}
\bal
|| f||_{\cH^m} &\teq \sum_{ 0\leq k \leq m}  || D^k_R f  \f{(1+R)^2 }{R^2 \sin(2\b)^{\s/2}} ||_{L^2} 
+ \sum_{   i+ j \leq m-1
} || D^i_R D^{j+1}_{\b}   f \f{(1+R)^2} { R^2 \sin(2\b)^{\g/2}  } ||_{L^2} , \\
\eal
\eeq
where $\s = 99/100, \g = 1 + \al / 10$ and we have used the definition of $\vp_i$ in Definition \ref{def:wg}. The $\cH^0$ norm is the same as $L^2(\vp_1)$ norm.

\begin{prop}\label{prop:elli}
Assume that $0 < \al \leq \f{1}{4}, 1 < \g \leq \f{5}{4}$, and $\Om$ satisfies $||\Om||_{\cH^3} < +\infty$ with 
\beq\label{eq:orth}
\int_0^{\pi/2} \Om(R, \b) \sin(2\b) d \b =0
\eeq
for every $R$. The solution of \eqref{eq:elli} satisfies
\[
 \al^2 || R^2 \pa_{RR} \Psi||_{\cH^3} +\al || R \pa_{R \b} \Psi||_{\cH^3} +|| \pa_{\b \b} \Psi ||_{\cH^3} \leq C || \Om ||_{\cH^3}
\]
for some absolute constant $C$ independent of $\al$ and $\g$.
\end{prop}

\begin{remark}
We need the orthogonality assumption \eqref{eq:orth} since $\sin(2\b)$ is in the null space of the self-adjoint operator $L_0(\Psi) = -\pa_{\b \b} \Psi - 4 \Psi$ with boundary condition $\Psi(0) = \Psi(\pi/2) =0$, which is the limiting operator in \eqref{eq:elli} as $\al \to 0$. 
\end{remark}


We only outline some key steps in the proof. Since the $\cH^2$ norm is the same as that in \cite{elgindi2019finite} and the $\cH^2$ estimates can be easily extended to the $\cH^3$ estimates, 
the complete proof follows from the same argument in \cite{elgindi2019finite}. Here, the proof is even simpler since there is no first order angular derivative term in \eqref{eq:elli}, i.e. $\pa_{\b} (\tan(\b) \Psi)$, which is one of the major difficulties in obtaining the elliptic estimate in \cite{elgindi2019finite}.

\vspace{0.1in}
\paragraph{\bf{The Orthogonality condition}} 
Define
\[
\Psi_*(R) = \int_0^{\pi/2} \Psi(R, \b) \sin(2\b) d \b.
\]
Using \eqref{eq:elli} and \eqref{eq:orth}, we derive
\[
\al^2 R^2 \pa_{RR} \Psi_* +  \al(\al + 4) R\pa_R \Psi_* = 0,
\]
which is an Euler equation and has an explicit solution 
\[
\Psi_*(R) = c_1 + c_2 R^{1- \f{4+\al}{\al}}.
\]
Recall the boundary condition in \eqref{eq:ellibc} for $\Psi$. We have $\Psi_* \to 0$ as $R \to \infty$. Since $R^2 \Psi$ vanishes at $R = 0$, we derive $c_1 = c_2 =0$ and $\Psi_* \equiv 0$.

Recall the boundary condition \eqref{eq:ellibc}. We can expand $\Psi(R, \b)$ in a series \\
$\Psi(R, \b) = \sum_{n \geq 1 } \Psi_n(R) \sin(2 n \b).$
Due to the orthogonality condition $\Psi_*(R) \equiv 0$, we have $\Psi_1(R) \equiv 0$. Therefore, we have 
\beq\label{eq:elli1}
|| \pa_{\b}\Psi(R, \b) g(R) ||_2^2 - 16 || \Psi(R, \b) g(R) ||_2^2
= \sum_{n \geq 2} (4n^2  - 16) || \Psi_n(R) g(R) ||^2_{L^2(R)}  \geq 0
\eeq
for some weight $g(R)$ such that $|| \pa_{\b}\Psi(R, \b) g(R) ||_2$ is finite. In particular, we have 
\[
|| \pa_{\b}\Psi(R, \b) g(R) ||^2_2 - 4 || \Psi(R, \b) g(R) ||_2^2 \geq \f{3}{4}|| \pa_{\b}\Psi(R, \b) g(R) ||^2_2.
\]
In \cite{elgindi2019finite}, the corresponding orthogonality condition is $\int_0^{\pi/2} \Om(R,\b) \cos^2(\b) \sin(\b) d\b =0$ for every $R$, which implies $|| \Psi_1(R) g(R)||^2_{L^2(R)} \leq \sum_{n \geq 2} || \Psi_n(R) g(R) ||^2_{L^2(R)}.$
Based on this, the positivity of the operator $\cL_{\al}$ \eqref{eq:elli} in the $L^2$ sense is established. Here, we simply have $\Psi_1(R) \equiv 0$.

Based on the above estimates, the proof of Proposition \ref{prop:elli} follows from the same argument as that in \cite{elgindi2019finite}.

\vspace{0.1in}
\paragraph{\bf{The singular term}}
In general the vorticity $\Om$ does not satisfy the assumption \eqref{eq:orth} in Proposition \ref{prop:elli}. 
Suppose that $\Psi$ is the solution of \eqref{eq:elli}. Consider $\td{\Psi} = \Psi + G \sin(2\b)$. The goal is to construct $G$ so that $\cL_{\al}(\td{\Psi})$ satisfies \eqref{eq:orth}, i.e.
$\int_0^{\pi/2} \cL_{\al}(\td{\Psi}) \sin(2\b) d \b =0$. Recall the notation $L_{12}(\Om)$ in \eqref{eq:biot3}. Following the argument in \cite{elgindi2019finite}, in Appendix \ref{app:singular}, we derive 
\beq\label{eq:G}
G = - \f{1}{\pi \al} L_{12}(\Om)(R) + \bar{G},
\quad 
\bar  G \teq   - \f{1}{\al \pi} R^{ -\f{4}{\al}} \int_0^{R} \int_0^{\pi/2} \Om(s, \b) \sin(2\b) s^{\f{4}{\al} - 1} ds.
\eeq

Although there is a large factor $1/\al$ in $\bar{G}$, it can be proved that $|| \bar{ G}||_{\cH^3}$ can be bounded by $C || \Om||_{\cH^3}$ using a Hardy-type inequality. We refer the reader to \cite{elgindi2019finite} and \cite{Elg17} for more details.

Using Proposition \ref{prop:elli} and an argument similar to that in \cite{elgindi2019finite}, we have the following result, which is similar to Theorem 2 in \cite{elgindi2019finite}.
\begin{prop}\label{prop:key}
Assume that $\al \leq \f{1}{4}$ and $\Om \in \cH^3$. Let $\Psi$ be the 
solution to \eqref{eq:elli} with boundary condition \eqref{eq:ellibc}. Then we have
\[
\bal
&\al^2 || R^2 \pa_{RR} \Psi ||_{\cH^3} + 
\al || R \pa_{R \b} \Psi ||_{\cH^3}  +|| \pa_{\b\b} (\Psi - \f{1}{\al \pi} \sin(2\b) L_{12}(\Om)) ||_{\cH^3} \leq C || \Om ||_{\cH^3}
\eal
\]
for some absolute constant $C$ independent of $\al, \g$ in the definition of $\cH^3$ \eqref{norm:H2}.
\end{prop}
\begin{remark}
The $\cH^3$ norm of $\al D_R \pa_{\b}\Psi$ is not included in Theorem 2 in \cite{elgindi2019finite}. Yet, 
the estimate of such term can be derived easily from Proposition \ref{prop:elli} and the estimate of $G$ defined in \eqref{eq:G}.
\end{remark}

\subsection{Estimates of nonlinear terms}\label{sec:nonterm}
In this subsection, we generalize several estimates of nonlinear terms derived in \cite{elgindi2019finite}, which will be used in our nonlinear stability estimate in the next section. 

We define the $\cW^{l, \infty}$ norm:
\beq\label{norm:W}
|| f ||_{\cW^{l,\infty}} \teq \sum_{0 \leq k+ j \leq l ,  j \neq 0} \B| \B| 
\sin(2\b)^{- \f{\al}{5}}D_R^k  \f{\lt(  \sin(2\b)  \pa_{\b}  \rt)^j}{\f{\al}{10} + \sin(2\b)}   f   \B|\B|_{L^{\infty}}  + \sum_{0 \leq k  \leq l } \B| \B| D_R^k f \B|\B|_{L^{\infty}}.
\eeq

A similar $\cW^{l,\infty}$ has been used in \cite{elgindi2019finite} and
our $\cW^{l,\infty}$ norm is slightly different from that in \cite{elgindi2019finite}. We replace the operator $(R+1)^k \pa^k_R $ by $D_R^k = (R\pa_R)^k$. The reason for doing this is that the stronger weight $(R+1)^k$ is not necessary in the derivation of the product rule in \cite{elgindi2019finite} related to $\cW^{l, \infty}$, and that the differential operator $D_R$ commutes with $\cL_{\al}$ in the elliptic equation \eqref{eq:elli}, while $\pa_R$ does not. Therefore, the higher order elliptic estimates related to $\pa_R$ can depend on the value of $\al$. We will only use these estimates when $\al$ is very small. 


\vspace{0.1in}
\paragraph{\bf{Functions in $\cW^{7,\infty}$}}
From Proposition \ref{prop:gam} in the Appendix, we know that $\G(\b), \bar{\Om}, \bar{\eta} \in \cW^{7,\infty}$. 
\begin{remark}
We do not apply the $\cW^{l,\infty}$ norm to $\bar{\xi}, \ \xi$. 
\end{remark}

Recall the $\cC^1$ norm in \eqref{norm:c1}. For the $\cC^1$ and $\cW^{1,\infty}$ norms, we have a simple result.
\begin{prop}\label{prop:c1}
For any $f, g \in \cC^1$ and $ \f{1+R}{R}p \in \cW^{1,\infty}$,
we have 
\[
|| f g ||_{\cC^1} \leq || f||_{\cC^1} || g ||_{\cC^1}, \quad || p||_{\cC^1} \les 
|| \f{1+R}{R}  p ||_{\cW^{1,\infty}}.
\]
\end{prop}
The $\cW^{4,\infty}$ version of the following result has been established in \cite{elgindi2019finite}, whose generalization to $\cW^{l, \infty}$ is straightforward.
\begin{prop}\label{prop:alg}
Assume that $f, g \in \cW^{l,\infty}$. Then we have
\[
|| f g ||_{\cW^{l,\infty}} \les_l || f  ||_{\cW^{l,\infty}}  ||  g ||_{\cW^{l, \infty}} .
\]
\end{prop}

Recall from \eqref{eq:simp4} that $L_{12}(\bar{\Om})= \f{3\pi \al}{2} \f{1}{1+R} $. We define $\bar{\Psi}$ by 
\[
\cL_{\al} ( \bar{\Psi} ) = - \al^2 R^2 \pa_{RR} \bar{\Psi} - \al( 4 + \al ) R \pa_R \bar{\Psi} - \pa_{\b \b} \bar{\Psi} - 4 \Psi = \bar{\Om},
\]
where $\cL_{\al}$ is the operator in \eqref{eq:elli}. We have the following estimates.

\begin{prop}\label{prop:psi}
For $\al \leq \f{1}{4}$, we have 
\[
\bal
|| \f{1+R}{ R} \pa_{\b \b} ( \bar{\Psi}  -  \f{\sin(2\b)}{\pi \al}  L_{12}(\bar{\Om})) ||_{\cW^{7,\infty}} \les \al , \qquad || L_{12}(\bar{\Om})||_{\cW^{7,\infty}} \les \al , \\
\al || \f{1+R}{R}  D_R^2 \bar{\Psi} ||_{\cW^{5,\infty}} +  \al||  \f{1+R}{R} \pa_{\b} D_R \bar{\Psi} ||_{\cW^{5,\infty}}  + || \f{1+R}{ R} \pa_{\b \b} ( \bar{\Psi}  -  \f{\sin(2\b)}{\pi \al}  L_{12}(\bar{\Om})) ||_{\cW^{5,\infty}} \les  \al  \notag. 
\eal
\]
\end{prop}

 \begin{proof}
The proof of the first inequality follows from the same argument in \cite{elgindi2019finite}.
Here, the proof is even simpler since there is no first order angular derivative term in \eqref{eq:elli}, i.e. $\pa_{\b} (\tan(\b) \Psi)$ in \cite{elgindi2019finite}. 

Using the formula \eqref{eq:simp4}, we know $L_{12}(\bar{\Om}) =\f{3\pi \al}{2(1+R)},  \f{\sin(2\b)}{\pi\al} L_{12}(\bar{\Om}) = \f{3  \sin(2\b) }{2 (1+R)}$. Since $L_{12}(\bar{\Om})$ does not depend on $\b$, the second inequality follows from a direct calculation. 

For $ 0 \leq i , j \leq 7$, we have
\[
 \al \B| \f{1+R}{R}  D^{i+1}_R \pa_{\b}^j \f{3 \sin(2\b)}{2 (1+R)} \B|
 +  \al  \B| \f{1+R}{R} D^{i+2}_R \pa_{\b}^j \f{3 \sin(2\b)}{2 (1+R)} \B| \les \al.
\]
Using $\f{\al}{5}+ \sin(2\b) \geq \sin(2\b)^{1- \f{\al}{5}} $, the definition of $\cW^{5,\infty}$ in \eqref{norm:W} and the first inequality, 
we complete the proof.
\end{proof}

\subsubsection{Some embedding Lemmas}

A similar version of the following estimate has been established in \cite{elgindi2019finite}. We remark that we have modified the weight for the $R$ variable in the $\cW^{l,\infty}$ norm. 
\begin{prop}\label{prop:W2}
Assume that $ \f{(1+R)^3}{R^2} f \in \cW^{3,\infty}$, then we have $f \in \cH^3$ and 
\[
|| f ||_{\cH^3} \les || \f{(1+R)^3}{R^2} f ||_{\cW^{3,\infty}}.
\]
\end{prop}

\begin{proof}
Recall the definition of $\vp_i$ in \eqref{wg} and $\cH^3$ in \eqref{norm:H2}, respectively. The main term to consider in $|| f||_{\cH^3}$ is $||D_R^3 f \vp_1^{1/2}||_{L^2}$. Observe that
\[
|| D_R^3 f  \f{(1+R)^2}{R^2} \sin(2\b)^{-\s/2} ||^2_{L^2}
\les || \f{(1+R)^3}{R^2}  D_R^3 f ||_{\infty} .
\]
It suffices to bound the last term by $|| \f{(1+R)^3}{R^2 } f||_{\cW^{3,\infty}}$. We have 
\[
\bal
\f{(1+R)^3}{R^2} D_R^3 f &=   D_R^3 ( \f{(1+R)^3}{R^2}   f) 
- 3 D^2_R \f{(1+R)^3}{R^2} D_R f -3 D_R \f{(1+R)^3}{R^2} D^2_R f 
- D_R^3( \f{(1+R)^3}{R^2} )  f \\
&  = I_1 + I_2 + I_3 + I_4.
\eal
\]
Notice that $ | D_R^k \f{(1+R)^3}{R^2}| \les  \f{(1+R)^3}{R^2}$ for $k=1,2,3$. 
Then by the definition of $\cW^{3,\infty}$, $I_1$ and $I_4$ can be bounded by $|| \f{(1+R)^3}{R^2 } f||_{\cW^{3,\infty}}$. For $I_2, I_3$, we have
\[
|I_2 |\les |\f{(1+R)^3}{R^2} D_R f| , \quad  |I_3| \les  |\f{(1+R)^3}{R^2} D_R^2 f |,
\]
which contains lower order derivatives of $f$ (compared to $D_R^3f$). The same argument implies that $|I_2|, |I_3| $ can be further bounded by $|| \f{(1+R)^3}{R^2 } f||_{\cW^{3,\infty}}$.
Other terms in $\cH^3$ norm can be estimated similarly. 
\end{proof}

We have the following decay estimate.
\begin{lem}\label{lem:xi_decay}
Suppose that $\xi \in \cH^2(\psi)$, we have 
\[
|| R^{1/2} \sin(2\b)^{1/4} \xi||_{L^{\infty}} \les || \xi||_{\cH^2(\psi)}.
\]
\end{lem}
The above estimate also holds for $\xi \in \cH^2$ since $\cH^2$ is stronger than $\cH^2(\psi)$ (see Lemma \ref{lem:H2H2}). 

\begin{proof}
Using a direct calculation yields
\[
\bal
&||  \sin(2\b)^{1/2} R \xi^2||_{L^{\infty}}
\les ||  \pa_R \pa_{\b} ( \sin(2\b)^{1/2} R\xi^2 )  ||_{L^1}
= || \pa_{\b} (\sin(2\b)^{1/2} (\xi^2 + 2\xi D_R \xi ) ) ||_{L^1} \\
 \les & || \sin(2\b)^{-1/2} (\xi^2 + 2\xi D_R \xi ) ||_{L^1}
+ || \sin(2\b)^{1/2} ( 2\xi \pa_{\b} \xi +2 \pa_{\b}\xi D_R \xi + 2\xi \pa_{\b}D_R\xi) ||_{L^1}.
\eal
\]
Recall the definition of $\cH^2(\psi)$ \eqref{norm:H22} and the weights in Definition \ref{def:wg}.
Using the Cauchy-Schwarz inequality concludes the proof.
\end{proof}


\begin{lem}\label{lem:inf}
We have 
\[
\bal
|| f ||_{L^{\infty}} &\les  \al^{-1/2} || f||_{\cH^2},  \\
  || f||_{\cC^1}  &= || f ||_{L^{\infty} }  + ||\f{1+R}{R} D_R f ||_{L^{\infty}} + || ( 1 + ( R \sin(2\b)^{\al})^{- \f{1}{40}} )D_{\b} f ||_{L^{\infty}} \les \al^{-1/2} || f||_{\cH^3},
 \eal
\]
provided that the right hand side is bounded.
\end{lem}

The first inequality has been established in \cite{elgindi2019finite}.
Recall the definition of $\cH^3$ and its associated weights in \eqref{norm:H2}. The proof of the $\cC^1$ estimates follows from the argument in the proof of Lemma \ref{lem:xi_decay}, the Cauchy-Schwarz inequality and 
 \[
  || \f{1}{1+R} \sin(2\b)^{ \g / 2 - 1}  ||_{L^2} \les \al^{-1/2},  \quad  
    || \f{ R^{1- \f{\al}{40}}}{ (1+R)^2} \sin(2\b)^{ \g / 2 - 1 - \f{\al}{40}}  ||_{L^2} \les \al^{-1/2}.
 \]

\subsubsection{The product rules}\label{sec:prod}
 In this subsection, we generalize the estimates of nonlinear terms and the transport terms derived in \cite{elgindi2019finite} to the $\cH^3(\psi)$ norm.


Denote the sum space $ X \teq \cH^3 \oplus \cW^{5,\infty}$ with sum norm 
\beq\label{norm:X}
|| f||_{X} \teq \inf \{ || g||_{\cH^3} + || h||_{\cW^{5,\infty}} :  f = g + h  \}.
\eeq


We use the following product rules to estimate the nonlinear terms. 
\begin{prop}\label{prop:prod1}
For all $ f \in X,  g \in \cH^3,  \xi \in \cH^3(\psi) \cap \cC^1$, we have 
\beq\label{eq:prod1}
\bal
|| f g  ||_{\cH^3} & \les \al^{-1/2}  || f||_{X} || g  ||_{\cH^3} ,  \\
|| f \xi||_{\cH^3(\psi)} & \les  \al^{-1/2}  || f ||_{X} ( \al^{1/2} ||\xi||_{\cC^1} + || \xi||_{\cH^3(\psi)}) . 
\eal
\eeq
\end{prop}

The first inequality has been proved in \cite{elgindi2019finite}. We will focus on the product rule with $\cH^3(\psi)$ norm.

\begin{proof}
If $f \in \cW^{5,\infty}$, applying the same argument in \cite{elgindi2019finite} yields 
\[
|| f \xi||_{\cH^3(\psi)} \les \al^{-1/2} || f ||_{\cW^{5,\infty}} ||  \xi||_{\cH^3(\psi)}.
\]

Now, we assume $f \in \cH^3$. We consider the third derivative $D^3 = D_R^i D^j_{\b}$ terms since other are terms are easier. If $(D^3, \psi_i ) = (D_R^3, \psi_1 ), ( D^3_{\b}, \psi_2)$, we use a $L^2 \times L^{\infty}$ interpolation
\[
\bal
||  D^3(f \xi)   \psi_i^{1/2}||_2^2  
&\les \sum_{k =0,  1} ||  D^k f  D^{3-k} \xi  \psi_i^{1/2}||_2^2 + \sum_{ k = 2,3} ||  D^k f    D^{3-k} \xi   \psi_i^{1/2}||_2^2  \\
& \les || f ||_{\cC^1} || \xi||_{\cH^3(\psi)}
+ ||f ||_{\cH^3(\psi)} || \xi||_{\cC^1}
\les \al^{-1/2} || f||_{\cH^3} || \xi||_{\cH^3(\psi)}+ ||f ||_{\cH^3} || \xi||_{\cC^1},
\eal
\]
where we have applied Lemma \ref{lem:inf} to $||f||_{\cC^1}$ and Lemma \ref{lem:H2H2} to obtain the last inequality.

If $D^3 = D_R^2 D_{\b}$ or $D_{\b}^2 D_R$, the corresponding singular weight in the $\cH^3(\psi)$ norm is $\psi_2$. 
We consider the term $D_R^2 \xi  D_{\b} f \psi_2^{1/2}$ in the $L^2$ estimate of $D^3 (f \xi) \psi_2^{1/2}$, which is a typical and the most difficult term.
The previous $L^2 \times L^{\infty}$ estimate fails since $D_R^2 \xi \psi_2^{1/2}$ is not in $L^2(R, \b)$. 
Recall the Definition \ref{def:wg} of $\psi_2, \vp_2$. Denote 
\beq\label{eq:PQS}
  W = \f{(1+R)^4}{R^4},  \quad  P = \sin(\b)^{-\s} \cos(\b)^{-\g} , \quad  
Q = \sin(2\b)^{-\g} ,\  S = \sin(2\b)^{-\s}, \   \lam = \g - \s .
\eeq
 Clearly, we have $\vp_2 = WQ, \psi_2 = W  P, \ \psi_1 \asymp W S$  and $P \les \sin(\b)^{\lam} Q$. We use a $L^2(R, L^{\infty}(\b)) \times L^{\infty}( R, L^2(\b) )$ estimate \footnote{The
$L^2(R, L^{\infty}(\b)) \times L^{\infty}( R, L^2(\b) )$ estimate of the mixed derivatives term in the $\cH^2$ norm is due to Dongyi Wei. We are grateful to him for telling us this estimate. We apply this idea to derive the estimates in the $\cH^3(\psi)$ norm.
 }
\beq\label{eq:drdb}
\bal
|| D_R^2 \xi  D_{\b} f  (W P)^{1/2}||_2^2 
& \leq \B|\B|  || \sin(\b)^{\lam / 2} D^2_R \xi (R, \cdot)||^2_{L^{\infty}(\b)}|| D_{\b} f Q^{1/2}(R, \cdot) ||^2_{L^2(\b)} W \B|\B|_{L^1(R)} \\ 
& \teq || A(R)^2   B(R)^2    W ||_{L^1(R)} .
\eal
\eeq
We further estimate the integrands $A(R), B(R)$. Using the Poincare inequality, we have
\[
A(R) \les 
|| \pa_{\b} (  \sin(\b)^{\lam / 2} D^2_R \xi (R, \cdot ) ) ||_{L^1(\b)} + || \sin(\b)^{\lam / 2} D^2_R \xi (R, \cdot)||_{L^2(\b)} 
\teq A_1(R) + A_2(R) .
\]
Using the Cauchy-Schwarz inequality, we can bound the first term as follows 
\[
\bal
A_1(R) &  \les || \sin(\b)^{\lam / 2-1}  D^2_R \xi (R, \cdot )  ||_{L^1(\b)}
+ ||  \sin(\b)^{\lam / 2}  \sin(2\b)^{-1} D_{\b} D^2_R \xi (R, \cdot )  ||_{L^1(\b)} \\
& 
\les || S^{1/2} D_R^2 \xi(R, \cdot) ||_{L^2(\b)}
|| S^{-1/2} \sin(\b)^{\lam/2 -1} ||_{L^2} \\
& \quad + || P^{1/2} D_{\b} D_R^2 \xi(R, \cdot) ||_{L^2(\b)}
|| P^{-1/2} \sin(\b)^{\lam/2}  \sin(2\b)^{-1}||_{L^2} .
\eal
\]
Recall $P,S, \lam$ defined in \eqref{eq:PQS} and $\g = 1 + \f{\al}{10}$. A simple calculation yields  
\[
\bal
|| S^{-1/2} \sin(\b)^{\lam/2-1} ||_{L^2} 
&\les || \sin(\b)^{\g / 2- 1} ||_{L^2(\b)} \les \al^{-1/2},  \\ 
|| P^{-1/2} \sin(\b)^{\lam/2}  \sin(2\b)^{-1}||_{L^2} 
&\les || \sin(\b)^{\g / 2-1} \cos(\b)^{\g/2-1} ||_{L^2(\b)}
\les \al^{-1/2}. 
\eal
\]
Combining the above estimates, we derive
\[
A \les A_1(R) + A_2(R)
\les \al^{-1/2} (
|| S^{1/2} D_R^2 \xi(R, \cdot) ||_{L^2(\b)} +
|| P^{1/2} D_{\b} D_R^2 \xi(R, \cdot) ||_{L^2(\b)} )
+ ||  D^2_R \xi (R, \cdot)||_{L^2(\b)} .
\]
Recall $W S \les \psi_1, WP \les \psi_2$. Consequently, we have 
\[
|| A^2(R)  W ||_{L^1(R)} \les  \al^{-1} || \xi||^2_{\cH^3(\psi)}.
\]

Recall $B(R)$ in \eqref{eq:drdb}. Since $ D_{\b} f Q^{1/2} W^{1/2} , D_R D_{\b} f Q^{1/2} W^{1/2} \in L^2$, we have 
$\lim \inf_{R\to 0} B(R) = 0$ and yield
\[
|| B^2||_{L^{\infty}(R)}
\leq   || \pa_R B^2 ||_{L^1(R)}
\les || \pa_R D_{\b} f Q^{1/2} ||_{L^2} || D_{\b} f Q^{1/2} ||_{L^2}
\les || f ||^2_{\cH^3},
\]
where we have used $\pa_R = R^{-1} D_R, R^{-1} \les W^{1/2}$ and $W Q = \vp_2$ to obtain the last inequality. Plugging the estimates of $A$ and $B$ in \eqref{eq:drdb}, we yield the desired estimate on $|| D_R^2 \xi D_{\b} f
\psi_2^{1/2}||_{L^2}$.
\end{proof}


We generalize the $\cH^2$ estimate of transport term derived in the earlier arXiv version of \cite{elgindi2019finite} as follows.

\begin{prop}\label{prop:tran1}
Assume that $u, \pa_{\b}u, D_R u \in \cH^3$ and $\Om \in \cH^3, \xi \in \cH^3(\psi) \cap \cC^1$ we have 
\[
\bal
| \la \Om , u D_R \Om \ra_{\cH^3} |& \les \al^{-\f{1}{2}} \lt( ||u||_{\cH^3} + || \pa_{\b} u ||_{\cH^3} 
+ || D_R u ||_{\cH^3}   \rt) || \Om ||^2_{\cH^3} , \\
| \la \xi , u D_R \xi\ra_{\cH^3(\psi)} | & \les \al^{-\f{1}{2}} \lt( ||u||_{\cH^3} + || \pa_{\b} u ||_{\cH^3} + || D_R u ||_{\cH^3}   \rt) ( || \xi ||_{\cH^3(\psi)} + \al^{1/2} ||\xi||_{\cC^1}  )^2.
\eal
\]
Moreover, for all $u, D_R u \in X = \cH^3 \oplus \cW^{5,\infty}$ and $\Om \in \cH^3, \xi \in \cH^3(\psi)
\cap \cC^1$, we have 
\[
\bal
| \la \Om , u D_{\b} \Om \ra_{\cH^3} | &\les \al^{-1/2} \lt( || u||_X + || D_R u ||_X ) \rt)  
|| \Om ||^2_{\cH^3} ,\\
| \la \xi , u D_{\b} \xi \ra_{\cH^3(\psi)} | & \les \al^{-1/2} \lt( || u||_X +
  || D_R u ||_X) \rt)  ( || \xi ||_{\cH^3(\psi)} + \al^{1/2} ||\xi||_{\cC^1} )^2.
\eal
\]
\end{prop}

The proof follows from the argument in the proof of Proposition \ref{prop:prod1} and that in the earlier arXiv version of \cite{elgindi2019finite}. Here, the proof is easier since the data is more regular (than $\cH^2$), i.e.$\cH^3$ or $\cH^3(\psi)$, and then the estimate of several nonlinear terms can be done by applying $L^{\infty}$ estimate on one term. To estimate the mixed derivative terms, e.g. $ \la D^2_R D_{\b} \xi, D^2_R D_{\b} (uD_{\b} \xi) \psi_2 \ra $, we apply the $L^2(R , L^{\infty}(\b)) \times  L^{\infty}(R, L^{2}(\b))$ argument similar to that in the proof of Proposition \ref{prop:prod1}.

The following result is a simple $\cH^3,\cH^3(\psi)$ generalization of another transport estimate in the earlier arXiv version of \cite{elgindi2019finite}.
\begin{prop}\label{prop:tran2}
Let $\cH^3(\rho)$ be either $\cH^3$ or $\cH^3(\psi)$.  For all $g \in \cH^3(\rho)$, $u$ with $ || D_R^i u ||_{L^{\infty}} < \infty $ for $i \leq 3$ and $|| D_R^i D^j_{\b} \pa_{\b} u ||_{L^{\infty}} < \infty$ for $ i+ j \leq 2$, we have 
\[
\bal
| \la g , u D_R g \ra_{\cH^3(\rho)} | & \les \al^{-1/2} ( 
\sum_{ 0 \leq i \leq 3} || D_R^i u ||_{L^{\infty}} + \sum_{i + j \leq 2}|| D_R^i D^j_{\b} \pa_{\b} u ||_{L^{\infty}}  ) || g ||^2_{\cH^3(\rho)} ,\\
\eal
\]
\end{prop}
The proof follows simply from applying $L^{\infty}$ estimate on the $u$ term and integration by parts.

\subsection{A new estimate of the transport term and the estimate of $v_x \xi$}\label{sec:new}
In this subsection, we establish a new estimate of the transport term which is necessary to close the nonlinear estimate and estimate $||v_x \xi||_{\cH^3}$ which is not covered by Proposition \ref{prop:prod1}.

\begin{prop}\label{prop:tran3}
 Let $\Psi$ be a solution of \eqref{eq:elli}. Suppose that $g, \Om \in \cH^3, \xi \in \cH^3(\psi) \cap \cC^1$. We have 
\[
\bal
| \la g ,  \f{1}{\sin(2\b)} D_R \Psi D_{\b} g \ra_{\cH^3} | &\les \al^{-3/2} || \Om ||_{\cH^3}   || g ||^2_{\cH^3} ,\\
| \la \xi , \f{1}{\sin(2\b)} D_R \Psi D_{\b}  \xi \ra_{\cH^3(\psi)} | & \les \al^{-3/2} 
|| \Om ||_{\cH^3}  ( || \xi ||_{\cH^3(\psi)} + \al^{1/2} || \xi||_{\cC^1} )^2.
\eal
\]
\end{prop}

If one apply Proposition \ref{prop:tran1} with $ u = \f{D_R \Psi}{\sin(2\b)} $, $|| D_R u ||_{\cH^3}$ in the upper bound cannot be bounded by $|| \Om||_{\cH^3}$.

\begin{proof}
Denote $u = \f{D_R \Psi}{\sin(2\b)} $. The estimate of the transport term is similar to that in Proposition \ref{prop:prod1} except that we need to perform integration by parts for the terms $\la D^3 g, u D^3 D_{\b} \vp \ra$ in the estimate. We focus on a typical and difficult term $\la D_R^2 D_{\b} \xi, D^2_R u D_{\b}^2 \xi \psi_2 \ra$ to see why we can improve the estimate in Proposition \ref{prop:tran1}. Other terms can be estimated similarly. 

For this term, it suffices to estimate the $L^2$ norm of $D_R^2 u D^2_{\b} \xi \psi_2^{1/2}$. Recall $\psi_2 = W P$ with $W, P$ defined in \eqref{eq:PQS}. We have
\[
|| D_R^2 u D^2_{\b} \xi \psi_2^{1/2}||_2 \leq   || D_R^2 u W^{1/2} ||_{L^2(R, L^{\infty}(\b))}
|| D^2_{\b} \xi P^{1/2}||_{L^{\infty}(R, L^2(\b))} \teq A \cdot B.
\]
The term $A$ can be bounded by $C \al^{-1/2} || u ||_{\cH^3}$, which is further bounded by $C \al^{-3/2} || \Om||_{\cH^3}$ using Proposition \ref{prop:key}. The term $B$ is bounded by $C || \xi ||_{\cH^3(\psi)}$. It is similar to the argument in the proof of Proposition \ref{prop:prod1} and we omit the detail.
\end{proof}


Finally, we estimate the nonlinear term $v_x \xi$ in the $\eta$ equation \eqref{eq:lin}. 

\begin{prop}\label{prop:prod3}
Let $\Psi, \bar{\Psi}$ be a solution of \eqref{eq:elli} with source term $\Om, \bar{\Om}$, respectively, and $V_1(\Psi)$ be the operator which is related to $v_x$ and is to be defined in \eqref{eq:simp5}. Assume that $\xi \in \cH^3(\psi) \cap \cC^1, \Om \in \cH^3$. We have 
\beq\label{eq:prod3}
\bal
|| V_1(\Psi) \xi||_{\cH^3}  &\les  \al^{-1/2}  || \Om||_{\cH^3} ( 
\al^{1/2} || \xi||_{\cC^1} + || \xi||_{\cH^3(\psi)}),  \\
\quad || V_1(\bar{\Psi}) \xi ||_{\cH^3}  &\les \al^{1/2}  || \xi||_{\cH^3(\psi)}.
\eal
\eeq
\end{prop}

The difficulty lies in that $\cH^3(\psi)$ is weaker than $\cH^3$ (see Lemma \ref{lem:H2H2}). We can not apply Proposition \ref{prop:prod1} directly to estimate $v_x \xi$. We need to use a key fact that $v_x$ vanishes on $\b = 0$.

\begin{proof}
We use the formula of $V_1(\Psi)$ \eqref{eq:simp512} to be derived 
\[
\bal
V_1(\Psi)  &= \al (1 + 2\cos^2 \b) D_R \Psi - \al D_R D_{\b}\Psi -  D_{\b} \Psi_* + 2 \Psi_*  + \sin^2 (\b) \pa^2_{\b} \Psi_*  + \al^2 \cos^2(\b) D_R^2 \Psi \\
&\teq A(\Psi)  + \al^2 \cos^2(\b) D_R^2 \Psi.
\eal
\]
where $\Psi_* = \Psi - \f{\sin(2\b)}{\pi \al} L_{12}(\Om)$. We first consider the second inequality in \eqref{eq:prod3}. Notice that $V_1(\bar{\Psi})$ vanishes on $\b=0$. More precisely, Proposition \ref{prop:psi} implies $\sin(\b)^{-1/2} V_1(\bar{\Psi}) \in \cW^{5,\infty}$. Applying the product rule in $\cH^3$ norm in Proposition \ref{prop:prod1}, Lemma \ref{lem:H2H2} and then Proposition \ref{prop:psi}, we yield 
\[
|| V_1(\bar{\Psi}) \xi||_{\cH^3} 
\les \al^{-1/2}  || \sin(\b)^{-1/2} V_1(\bar{\Psi})||_{\cW^{5,\infty}}  || \sin(\b)^{1/2} \xi ||_{\cH^3}  \les \al^{1/2} || \xi||_{\cH^3(\psi)} .
\]

Next, we consider the first inequality in \eqref{eq:prod3}. From Proposition \ref{prop:key}, we know that $\sin(\b)^{-1/2} A(\Psi) \in \cH^3$. Applying Propositions \ref{prop:prod1}, \ref{prop:key} and Lemma \ref{lem:H2H2}, we derive
\[
|| A(\Psi) \xi||_{\cH^3}
\les \al^{-1/2} || A(\Psi) \sin(\b)^{-1/2} ||_{\cH^3} || \xi \sin(\b)^{1/2} ||_{\cH^3(\psi)}
\les \al^{-1/2} || \Om||_{\cH^3} || \xi||_{\cH^3(\psi)}.
\]

Finally, we focus on the term $g \teq \al^2 D_R^2 \Psi$ in $V_1(\Psi)$. We consider the third derivative terms $D^3( D_R^2 \Psi \cdot \xi)$ with $D^3 = D_R^i D_{\b}^j, i+ j = 3$ in the $\cH^3$ estimate since other terms are easier. If $D^3 = D_R^3$, we need to estimate the $L^2$ norm of $D_R^3( g \xi ) \vp_1^{1/2}$. Since $\vp_1 \asymp \psi_1$, the estimate follows from the argument in the proof of Proposition \ref{prop:prod1} and we obtain 
\[
||  D_R^3( \al^2 D_R^2 \Psi \xi)  \vp_1^{1/2}||_2^2 
\les \al^{3/2} || \Om||_{\cH^3} ( || \xi||_{\cH^3(\psi)} + \al^{1/2} || \xi||_{\cC^1}) .
\] 

Otherwise, we need to estimate the $L^2$ norm of $D^2 D_{\b}( g \cdot \xi) \vp_2^{1/2}$ with $D^2 = D_R^i D_{\b}^j, i+ j =2$ (note that $D_{\b}$ commutes with $D_R$). We rewrite $D_{\b}(g\xi)$ as follows 
\[
D_{\b}( g\xi ) = \pa_{\b} g(\sin(2\b) \xi) + g D_{\b}\xi
= \sin(2\b)^{3/4} \pa_{\b} g (\sin(2\b)^{1/4} \xi) + \sin(2\b)^{1/4}(\sin(2\b)^{-1/2}g )   \sin(2\b)^{1/4} D_{\b}\xi.
\]
Notice that $\sin(2\b)^{1/4} \vp_2 \les \vp_1 ,\psi_1$. Using the idea in the discussion of Lemma \ref{lem:H2H2} and expanding the $\cH^2$ norm, one can verify easily that 
\[
|| D^2 ( D_{\b} (g \xi ) )  \vp_2^{1/2} ||_{L^2} 
\les || \sin(2\b)^{1/2} \pa_{\b}g \cdot \sin(2\b)^{1/4} \xi ||_{\cH^2} + || \sin(2\b)^{-1/2} g \sin(2\b)^{1/4} D_{\b}\xi ||_{\cH^2}.
\]
Applying the $\cH^2$ version of the product rule in Proposition \ref{prop:prod1} (it is given in \cite{elgindi2019finite}), Proposition \ref{prop:key} to $g = \al^2 D_R^2\Psi$, and Lemma \ref{lem:H2H2}, we obtain 
\[
\bal
|| D^2 ( D_{\b} (g \xi ) )  \vp_2^{1/2} ||_{L^2} 
&\les \al^{-1/2}
 || \sin(2\b)^{1/2} \pa_{\b} g ||_{\cH^2} || \sin(2\b)^{1/4}\xi||_{\cH^2} \\
 &+ \al^{-1/2}
 || \sin(2\b)^{-1/2}  g ||_{\cH^2} || \sin(2\b)^{1/4} D_{\b}\xi||_{\cH^2} 
 \les \al^{3/2} || \Om||_{\cH^3} || \xi||_{\cH^3(\psi)}.
 \eal
\]
Combining the estimates of $A(\Psi)$ and $\al^2 D_R^2\Psi$ completes the proof.
\end{proof}

\section{Nonlinear stability}\label{sec:non}
In this section, we complete the estimates of the remaining terms $\cR_3$ in Corollary \ref{cor:H3} and in \eqref{eq:xi_inf1},\eqref{eq:xi_inf2R},\eqref{eq:xi_inf2b}. We will prove the following for the energy $E_3$ in \eqref{eg:H3} and $E(\xi,\infty)$
\begin{align}
&\f{1}{2}\f{d}{dt}  E_3^2  \leq - \f{1}{12}    E_3^2  + C \al^{1/2} (E_3^2 + \al || \xi||^2_{\cC^1}) + C \al^{-3/2} ( E_3 + \al^{1/2} || \xi ||_{\cC^1} )^3 + C \al^2 E_3 , 
\label{eq:boot1} \\
&\f{1}{2} \f{d}{dt} E(\xi, \infty)^2
\leq   - E(\xi, \infty)^2 + 
C||\xi||_{\cC^1} ( \al^{1/2} E_3 + \al ||\xi||_{\cC^1} )  \label{eq:boot2}  \\ 
& \qquad \qquad + C || \xi||_{\cC^1}( \al^{-1} E^2_3 + \al^{-1} E_3 || \xi||_{\cC^1} )  + C\al^2 E(\xi, \infty) , \notag
\end{align}
for any initial perturbation $\Om, \eta, \xi$ with $E_3(\Om, \eta, \xi) < +\infty$ and $E(\xi, \infty) < +\infty$, where
\beq\label{eg:xi_inf}
E(\xi, \infty) \teq  ( || \xi ||^2_{\infty}  + || \phi_2 D_{\b}\xi||^2_{\infty} + 
\mu_4 ||  \phi_1 D_R \xi ||_{\infty}^2  )^{1/2}
\eeq
for some absolute constants $\mu_4$. $E(\xi,\infty)$ is equivalent to $||\xi||_{\cC^1}$ \eqref{norm:c1} once we determine the absolute constants $\mu_4$. 


The major step is the linear stability that gives the damping term $(- \f{1}{12} + C\al ) E_2^3$ and $ (-1 + C\al) E(\xi, \infty)^2$. We have already established the linear stability in Corollary \ref{cor:H3} and estimates \eqref{eq:xi_inf1}, \eqref{eq:xi_inf2R}, \eqref{eq:xi_inf2b}. The remaining terms $\cR_3$ in Corollary \ref{cor:H3} and in \eqref{eq:xi_inf1}, \eqref{eq:xi_inf2R}, \eqref{eq:xi_inf2b} contribute other terms in \eqref{eq:boot1}-\eqref{eq:boot2}. We will further construct an energy $E^2(\Om, \eta, \xi ) \teq \al  E(\xi, \infty)^2+ E_3^2(\Om, \eta, \xi)$ and these remaining terms are relatively small at the threshold $E =O(\al^2)$. Then we can close the nonlinear estimate.

We will first derive several formulas for later use in subsection \ref{sec:formu}. Then we estimate the remaining terms mentioned above. In subsection \ref{sec:trans} and \ref{sec:nonf}, we will apply the product rules obtained in subsection \ref{sec:nonterm} to estimate the transport terms and nonlinear terms and then complete the estimate \eqref{eq:boot1}. We will derive the $\cC^1$ estimate \eqref{eq:boot2} in subsection \ref{sec:C1} and prove finite time blowup in subsection \ref{sec:blowup}. We remark that estimates similar to the $\cC^1$ estimates \eqref{eq:boot2} are not required in \cite{elgindi2019finite} since there is no swirl.

 \vspace{0.1in}
 \paragraph{\bf{Notations}}
 Throughout this section, $\chi$ is the radial cutoff function in Lemma \ref{lem:l12}. We use $\Psi_*, \bar{\Psi}_*$ to denote the lower order terms in $\Psi, \bar{\Psi}$, i.e. 
\beq\label{eq:nota_psi}
\Psi_* \teq \Psi - \f{\sin(2\b)}{\pi \al}L_{12}(\Om), \quad \bar{\Psi}_* \teq \bar{\Psi} - \f{\sin(2\b)}{\pi \al}L_{12}(\bar{\Om}).
\eeq
$\Psi_*$ and $\Psi$ enjoys the elliptic estimate in Proposition \ref{prop:elli} and $\bar{\Psi},\bar{\Psi}_*$ satisfy Proposition \ref{prop:psi}. 

\subsection{Formulas of the velocity and related terms}\label{sec:formu}
In this subsection, we derive the formulas of the velocity in terms of the stream function in the $(R,\b)$ coordinates to be used later and then collect the remaining terms to be estimated in the nonlinear stability analysis.

Denote
\beq\label{eq:simp50}
u \teq U(\Psi),  \ v  \teq V(\Psi), \ u_x  \teq U_1(\Psi), \  u_y \teq U_2(\Psi), \  v_x \teq V_1(\Psi), \  v_y \teq V_2(\Psi) . 
\eeq
The formula of $U, V$ in terms of $\Psi$ are given in \eqref{eq:simp2}. We also collect them below.
Using \eqref{eq:simp1}-\eqref{eq:simp2}, $D_R = R \pa_R, r \pa_r = \al D_R $ and the incompressible condition $u_x + v_y =0$
, we compute 
\beq\label{eq:simp5}
\bal
U(\Psi) & = - 2 r \sin \b \Psi - \al r  \sin \b D_R \Psi  -  r\cos \b  \pa_{\b}\Psi ,\quad V(\Psi)= 2 r \cos \b \Psi + \al r  \cos \b D_R \Psi - r \sin \b \pa_{\b} \Psi , \\
U_1(\Psi) & = -\f{1}{2} \al^2  \sin (2\b ) D_R^2 \Psi - \f{\al }{2} \sin (2 \b)  D_R \Psi - \cos (2 \b) \pa_{\b} \Psi - \al \cos (2\b)  \pa_{\b} D_R \Psi + \f{\sin(2\b)}{2} \pa^2_{\b} \Psi ,\\
U_2(\Psi)  &=  \al (-1 - 2\sin^2 \b) D_R \Psi - \al D_R D_{\b}\Psi - D_{\b} \Psi - 2 \Psi - \al^2 \sin^2(\b) D_R^2 \Psi - \cos^2 (\b) \pa^2_{\b} \Psi , \\
V_1(\Psi) & =  \al (1 + 2\cos^2 \b) D_R \Psi - \al D_R D_{\b}\Psi -  D_{\b} \Psi + 2 \Psi + \al^2 \cos^2(\b) D_R^2 \Psi + \sin^2 (\b) \pa^2_{\b} \Psi , \\
V_2(\Psi) &= - U_1(\Psi). 
\eal
\eeq
Recall  $\Psi =\f{\sin(2\b)}{\pi \al} L_{12}(\Om) + \Psi_*$. 
For the terms not involving the $R$-derivative, e.g. $\Psi,\ \pa_{\b}\Psi$, we compute the contributions from the leading order part of $\Psi$, i.e. $\f{ \sin(2\b)}{\pi \al} L_{12}(\Om)$, and $\Psi_*$ separately,
\beq\label{eq:simp51}
\bal
U(\Psi) &= -\f{2r\cos(\b)}{\pi \al} L_{12}(\Om) - 2 r \sin(\b) \Psi_*  - \al r  \sin \b D_R \Psi  -  r\cos \b  \pa_{\b}\Psi_* \teq  -\f{2r\cos(\b)}{\pi \al} L_{12}(\Om) + U(\Psi, \Psi_*),\\
V(\Psi) &= \f{2r \sin(\b)}{\pi \al} L_{12}(\Om) + 2 r \cos \b \Psi_* + \al r  \cos \b D_R \Psi - r \sin \b \pa_{\b} \Psi_* \teq \f{2r \sin(\b)}{\pi \al} L_{12}(\Om) + V(\Psi, \Psi_*) ,\\ 
U_1(\Psi) &= -\f{2}{\pi \al} L_{12}(\Om) -\f{\al^2}{2}   \sin (2\b ) D_R^2 \Psi - \f{\al }{2} \sin (2 \b)  D_R \Psi - \cos (2 \b) \pa_{\b} \Psi_* - \al \cos (2\b)  \pa_{\b} D_R \Psi \\
&\quad + \f{\sin(2\b)}{2} \pa^2_{\b} \Psi_*   \teq  -\f{2}{\pi \al} L_{12}(\Om) + U_1(\Psi, \Psi_*), \quad V_2(\Psi) = -U_1(\Psi) =  \f{2}{\pi \al} L_{12}(\Om)  - U_1(\Psi, \Psi_*) .
\eal
\eeq
The first term in the formulas of $U, V, U_1, V_2$ is the leading order term. Observe that 
\[
-D_{\b} \sin(2\b) - 2 \sin(2\b) -\cos^2(\b) \pa^2_{\b} \sin(2\b) = 0,  \quad 
-D_{\b} \sin(2\b) + 2 \sin(2\b) + \sin^2(\b)  \pa^2_{\b} \sin(2\b)  = 0.
\]
For the terms not involving the $R$-derivative in $U_2(\Psi), V_1(\Psi)$ \eqref{eq:simp5}, the contributions from $\sin(2\b) L_{12}(\Om)$ cancel each other. Hence, we have 
\beq\label{eq:simp512}
\bal
U_2(\Psi)  &=  \al (-1 - 2\sin^2 \b) D_R \Psi - \al D_R D_{\b}\Psi - D_{\b} \Psi_* - 2 \Psi_* - \al^2 \sin^2(\b) D_R^2 \Psi - \cos^2 (\b) \pa^2_{\b} \Psi_* ,\\ 
V_1(\Psi) & = \al (1 + 2\cos^2 \b) D_R \Psi - \al D_R D_{\b}\Psi -  D_{\b} \Psi_* + 2 \Psi_* + \al^2 \cos^2(\b) D_R^2 \Psi + \sin^2 (\b) \pa^2_{\b} \Psi_*  .
\eal
\eeq
We decompose $U,V$ in \eqref{eq:simp51}-\eqref{eq:simp512} so that we can apply the elliptic estimate in Propositions \ref{prop:key}, \ref{prop:psi} to $U(\Psi, \Psi_*), V(\Psi, \Psi_*)$, $U_1(\Psi, \Psi_*), V_2(\Psi, \Psi_*), U_2(\Psi), V_1(\Psi)$.

Recall the formula of $\uu \cdot \na $ in \eqref{eq:trans}
\[
u \cdot \na  = - (\al R \pa_{\b} \Psi) \pa_R + ( 2\Psi +  \al R \pa_R \Psi) \pa_{\b}.
\]
Since $\Psi =\f{\sin(2\b)}{\pi \al} L_{12}(\Om) + \Psi_*, D_{\b} = \sin(2\b) \pa_{\b}$, we have
\beq\label{eq:trans41}
\bal
u \cdot \na &= 	
 (-\f{ 2\cos(2\b)}{\pi } L_{12}(\Om)  
- \al \pa_{\b} \Psi_* ) D_R + ( \f{2 }{\pi \al} L_{12}(\Om) + \f{ 2 \Psi_* + \al D_R \Psi }{\sin(2\b)}  ) D_{\b} \teq \f{2}{\pi \al} L_{12}(\Om)  D_{\b} + \cT(\Om) , \\
\cT(\Om) &\teq -\f{ 2\cos(2\b)}{\pi } L_{12}(\Om)  D_R 
- \al \pa_{\b} \Psi_* D_R + \f{  2\Psi_* + \al D_R \Psi}{\sin(2\b)}  D_{\b} .
\eal
\eeq
Using \eqref{eq:simp4}, we have $\f{2}{\pi \al} L_{12}(\bar{\Om})= \f{3}{1+R}$ and
\[
\bar{\uu} \cdot \na = \f{3}{1+R} D_{\b} + \cT(\bar{\Om}) .\]

Recall the formulations \eqref{eq:lin21}-\eqref{eq:lin23} and their equivalence \eqref{eq:equiv}.
We use the notations \eqref{eq:simp50} to rewrite $u_x,u_y$ and so on, and the above computations to expand the remaining terms $\cR$ in \eqref{eq:lin21}-\eqref{eq:lin23}. $\cR$ consists of three parts: the lower order terms in the linearized equation (denote as $P$), the error term $\bar{F}$ \eqref{eq:error} and the nonlinear term $N$ \eqref{eq:non}. The formula of $P$ is given below
\beq\label{eq:lin3}
\bal
\cR_{\Om} = & P_{\Om} + \bar{F}_{\Om} + N_{\Om} , \quad
\cR_{\eta} =  P_{\eta} + \bar{F}_{\eta} + N_{\eta} ,  \quad 
\cR_{\xi}  =  P_{\xi} + \bar{F}_{\xi} + N_{\xi} , \\
 \quad P_{\Om}  = & ( -3 \al D_R- \cT(\bar{\Om}) )\Om  + ( \al c_{\om} D_R- (\uu \cdot \na) ) \bar{\Om},  \\
P_{\eta}  =& ( -3 \al D_R - \cT(\bar{\Om}) )\eta   + ( \al c_{\om} D_R- (\uu \cdot \na) ) \bar{\eta}
- ( U_1(\bar{\Psi}) + \f{3}{1+R}) \eta   - ( U_1(\Psi) + \f{2}{\pi \al} L_{12}(\Om)) \bar{\eta}\\
&- ( V_1(\bar{\Psi}) \xi  + V_1(\Psi) \bar{\xi} ) ,\\
P_{\xi} = & ( -3 \al D_R- \cT(\bar{\Om}) ) \xi    + ( \al c_{\om} D_R- (\uu \cdot \na) ) \bar{\xi}
+ (- V_2(\bar{\Psi}) +\f{3}{1+R} ) \xi  + ( - V_2(\Psi)+  \f{2}{\pi \al} L_{12}(\Om) ) \bar{\xi} \\&-  ( U_2(\Psi) \bar{\eta}  + U_2(\bar{\Psi}) \eta).
\eal
\eeq
We remark that $P$ is the difference between the linear part of \eqref{eq:lin} and \eqref{eq:lin21}-\eqref{eq:lin23}.

Recall $\bar{c}_{\om} = -1, \bar{c}_l = \f{1}{\al} + 3$ and $\bar{\Om}, \bar{\eta}$ in \eqref{eq:profile}. Notice that  $ c_l = \f{1}{\al}, \Om_* =\f{3\al}{c} \f{ R}{(1+R)^2}, \eta_* = \f{6\al}{c} \f{R}{(1+R)^3}, \G = \cos(\b)^{\al}$ is a solution of \eqref{eq:selfsim2} and $\bar{\Om}, \bar{\eta}$ satisfy
$
\bar{\Om} = \Om_* \G(\b),  \ \bar{\eta} = \eta_* \G(\b) , \\  \f{c}{\al} \int_R^{\infty} \f{\Om_*}{s} ds  =  \f{3}{1+R}.$
Hence, we have 
\[
D_R \bar{\Om} = \bar{c}_{\om} \bar{\Om} + \bar{\eta}, \quad 
D_R \bar{\eta} = 2\bar{c}_{\om}  \bar{\Om} + \f{3}{1+R} \bar{\eta}.
\]

Hence, we can simplify $\bar{F}_{\Om}, \bar{F}_{\eta}$ in \eqref{eq:error} as 
\beq\label{eq:error2}
\bar{F}_{\Om} = (-3\al D_R - \bar{\uu} \cdot \na) \bar{\Om}, \quad 
\bar{F}_{\eta} = (-\f{3}{1+R} - U_1(\bar{\Psi})) \bar{\eta} - V_1(\bar{\Psi}) \bar{\xi} 
+  (-3\al D_R - \bar{\uu} \cdot \na) \bar{\eta} ,
\eeq
where we have used the notations in \eqref{eq:simp50} for $\bar{u}_x, \bar{u}_y, \bar{v}_x, \bar{v}_y$.

Recall the definition of the $\cH^3, \cH^3(\psi)$ inner product in \eqref{eq:inner} and the remaining terms 
$\cR_3$ in \eqref{eg:H3},\eqref{eq:E3}. See also the full expression of the weighted $H^3$ energy $E_3$ \eqref{eg:H3_full} related to $\cR_3$  Clearly, we have 
\beq\label{eq:remR3}
\cR_3  = \la \cR_{\Om}, \Om \vp_0 \ra + \la \cR_{\eta} ,\eta \psi_0 \ra +  \f{81}{4\pi c} L_{12}(\Om)(0) \la \cR_{\Om}, \sin(2\b) R^{-1} \ra + 
 \la \cR_{\Om}, \Om \ra_{\cH^3} + 
 \la \cR_{\eta}, \eta \ra_{\cH^3} 
 + \la \cR_{\xi}, \xi \ra_{\cH^3(\psi)}.
\eeq
We remark that $\la \cdot, \cdot \ra$ in the first three terms is the $L^2$ inner product defined in \eqref{eq:inner_L2}. We assume that $\Om, \eta \in \cH^3, \Om \in L^2(\vp) , \eta \in L^2(\psi)$, $\xi \in \cH^3(\psi), \xi \in \cC^1$. We will choose initial perturbations $\Om, \eta, \xi$ in these classes. In subsection \ref{sec:trans}, we estimate the transport terms in the last three terms in $\cR_3$. In subsection \ref{sec:nonf}, we estimate the nonlinear terms in the last three terms in $\cR_3$. In subsection \ref{sec:remR3}, we estimate the first three terms in $\cR_3$.

\subsection{Analysis of the transport terms in $P, N, F$} \label{sec:trans} \quad

In this subsection, we estimate the transport terms in $P$, $N$ and $F$ in $\cH^3$ or $\cH^3(\psi)$ norm. Our main tools in this and the next few subsections are the product rules, the elliptic estimates obtained in Section \ref{sec:elli} and Lemma \ref{lem:l12} on $L_{12}(\Om)$. The reader should pay attention to the subtle cancellation near $R=0$ in the estimates in subsections \ref{sec:tran2}, \ref{sec:tran3}.

\subsubsection{Transport terms I : $( -3\al D_R-\cT(\bar{\Om})) g$ in $P$} \label{sec:tran1}
We estimate 
\[
I_1 = | \la  (-3\al D_R -\cT(\bar{\Om}))  \Om, \Om \ra_{\cH^3}  |, \ 
I_2 = | \la ( -3\al D_R -\cT(\bar{\Om}) ) \eta, \eta \ra_{\cH^3}  |, \ 
I_3 = | \la  ( -3\al D_R -\cT(\bar{\Om}) ) \xi, \xi \ra_{\cH^3(\psi)}  | .
\]
Recall $\cT(\bar{\Om})$ in \eqref{eq:trans41}
\[
3\al D_R+ \cT(\bar{\Om}) = 3\al D_R -\f{ 2\cos(2\b)}{\pi } L_{12}(\bar{\Om})  D_R 
- \al \pa_{\b} \bar{\Psi}_* D_R + \f{1}{\sin(2\b)} ( 2 \bar{\Psi}_* + \al D_R \bar{\Psi} ) D_{\b} .
\]
Applying Proposition \ref{prop:psi} to estimate the above coefficients, then Proposition \ref{prop:tran1} to the $D_{\b}$ transport terms and Proposition \ref{prop:tran2} to the $D_R$ transport terms yield
\[
I_1 \les \al^{1/2} || \Om||^2_{\cH^3}, \quad I_2 \les \al^{1/2} || \eta||^2_{\cH^3},  \quad I_3  \les \al^{1/2} || \xi||^2_{\cH^3(\psi)}.
\]

\subsubsection{Transport term $II$ : $-  \al c_l R \pa_R g -  (\uu \cdot \na) g$ in $N$ \eqref{eq:non}}
We are going to estimate
\[
| \la  (-\al c_l D_R - (\uu \cdot \na))  \Om, \Om \ra_{\cH^2}  |, \quad
| \la ( -\al c_l D_R -(\uu \cdot \na) ) \eta, \eta \ra_{\cH^2}  |,  \quad
| \la  ( -\al c_l D_R -(\uu \cdot \na) ) \xi, \xi \ra_{\cH^2(\psi)}  | .
\]
Recall $\al c_l =  - \f{2 (1 - \al) }{\pi \al} L_{12}(\Om)(0)
$ in \eqref{eq:normal} and the computation about $\uu \cdot \na$ in \eqref{eq:trans41}
\[
\bal
(-\al c_l D_R - (\uu \cdot \na))
&=  (  \f{2 (1-\al) }{\pi \al}  L_{12}(\Om)(0) 
+ \f{ 2\cos(2\b)}{\pi } L_{12}(\Om)   +  \al \pa_{\b} \Psi_* ) D_R \\
 &- ( \f{2}{\pi \al} L_{12}(\Om) + \f{2 \Psi_*  }{\sin(2\b)} 
+ \f{ \al D_R \Psi}{\sin(2\b)} ) D_{\b}  .
 \eal
\]
For the first two $D_R$ transport terms, we apply Proposition \ref{prop:tran2} and Lemma \ref{lem:l12} to estimate $||D_R^k L_{12}(\Om)||_{L^{\infty}}$ for $k \leq 3$. For the third, fourth 
($(\f{2}{\pi \al}L_{12}(\Om))D_{\b}$) and fifth ($\f{2\Psi_*}{\sin(2\b)})D_{\b}$) transport terms, 
 we apply Proposition \ref{prop:tran1}, Proposition \ref{prop:key} to $\pa_{\b} \Psi_*, \f{\Psi_*}{\sin(2\b)}$ and \eqref{eq:l12X} in Lemma \ref{lem:l12} to $L_{12}(\Om)$. For the last transport term, we use Proposition \ref{prop:tran3}. Hence, we derive 
\[
| I_1 | \les \al^{-3/2} || \Om||^3_{\cH^3} ,
\quad | I_2| \les \al^{-3/2} || \Om||_{\cH^3} || \eta||^2_{\cH^3},
\quad |I_3| \les \al^{-3/2} || \Om||_{\cH^3} ( || \xi||_{\cH^3(\psi)} + \al^{1/2} || \xi||_{\cC^1} )^2.
\]

The largest term is $\f{2}{\pi \al}L_{12}(\Om) D_{\b}$, which leads to $\al^{-3/2}$ in the upper bound.

\subsubsection{Transport term III : $( \al c_{\om} D_R-(\uu \cdot \na) ) \bar{g}$ in $P$}\label{sec:tran2}  Next, we estimate 
\[
||  \al c_{\om} D_R - (u \cdot \na)) \bar{\Om}||_{\cH^3} , \quad || \al c_{\om} D_R- (u \cdot \na) ) \bar{\eta} ||_{\cH^3}  , \quad
|| \al c_{\om} D_R- (u \cdot \na) )\bar{\xi} ||_{\cH^3(\psi)} .
\]
Recall that $\cH^3$ contains a singular weight $\f{(1+R)^4}{R^4}$. 
We use the explicit form $\G(\b) = \cos(\b)^{\al}$ and a careful calculation to cancel the singular weight $R^{-4}$ near $R=0$.Using the formula for $c_{\om}$ in \eqref{eq:normal} and the computation in \eqref{eq:trans41}, we have 
\beq\label{eq:tran_R21}
\bal
 (\al c_{\om} D_R - (u \cdot \na) )g 
 &= \lt( -\f{2}{\pi} L_{12}(\Om)(0)  D_R 
 + \f{2\cos(2\b)} {\pi} L_{12}(\Om) D_R  - \f{2}{\pi \al}L_{12}(\Om) D_{\b} \rt) g \\
&+ ( \al \pa_{\b} \Psi_* D_R - (\sin(2\b))^{-1} (2 \Psi_* + \al D_R\Psi) D_{\b} ) g \teq I(g) + II(g).
\eal
\eeq
Denote $Q = L_{12}(\Om)- \chi L_{12}(\Om)(0)$. We use $ L_{12}(\Om) = Q +\chi L_{12}(\Om)(0)$ to rewrite $I(g)$
 \beq\label{eq:tran_R22}
I = \f{2 }{\pi } L_{12}(\Om)(0)   (- D_R g + \cos(2\b) \chi D_R g - \f{1}{\al} \chi D_{\b} g )   + 
 \f{2}{\pi} Q (\cos(2\b) D_R g - \f{1}{\al} D_{\b} g )  \teq I_1 + I_2. 
\eeq

Using \eqref{eq:Dg} and the formula of $g = \bar{\Om}, \bar{\eta}$ in \eqref{eq:profile}, we have
\[
 D_{\b}\G  = - 2\al \sin^2(\b) \G , \quad  D_{\b} g  =  - 2 \al \sin^2(\b) g.
\]

It follows that
\beq\label{eq:tran_R23}
I_1 =  \f{2}{\pi}   L_{12}(\Om)(0) ( - D_R g + \cos(2\b)\chi D_R g + 2\sin^2(\b)\chi g   )
=  \f{2}{\pi}   L_{12}(\Om)(0) ( -(1-\chi) D_R g + 2\sin^2(\b) \chi(- D_R g + g)  ).
\eeq
Since the smooth cutoff function $\chi$ satisfies $1-\chi(R) =0$ for $R \leq 1$. $I_1$ vanishes quadratically near $R=0$. For $(g, \cH^3(\rho)) = (\bar{\Om}, \cH^3), ( \bar{\eta}, \cH^3)$ or $(\bar{\xi}, \cH^3(\psi))$, applying Lemma \ref{lem:bar} to $g = \bar{\Om} ,\bar{\eta}$, \eqref{eq:xi_cw} in Lemma \ref{lem:xi} to $g = \bar{\xi}$ and using a direct calculation yield
\[
||I_1(g)||_{\cH^3(\rho)} \les |L_{12}(\Om)(0) | ( || (1 -\chi) g||_{\cH^3(\rho)} + || D_R g - g||_{\cH^3(\rho)} ) \les  \al | L_{12}(\Om)(0)| \les \al || \Om||_{\cH^3},
\]
where we have used \eqref{eq:l12} in Lemma \ref{lem:l12} in the last inequality. 

Recall $Q = L_{12}(\Om) - \chi L_{12}(\Om)(0)$ and $I_2, II(g)$ in \eqref{eq:tran_R21}, \eqref{eq:tran_R22}. For $g = \bar{\Om}, \bar{\eta}$, applying the product estimate in Proposition \ref{prop:prod1}, we get
\[
\bal
|| I_2(g ||_{\cH^3}
&\les \al^{-1/2} || Q||_{\cH^3} ( || D_Rg||_{\cW^{5,\infty}}
+ \al^{-1} || D_{\b} g||_{\cW^{5,\infty}} ) \les \al^{1/2} || \Om||_{\cH^3} ,\\ 
 ||II(g)||_{\cH^3} & \les  \al^{-1/2} || \Om||_{\cH^3} ( \al ||D_R g||_{\cW^{5,\infty}} + || D_{\b } g||_{\cW^{5,\infty}} ) \les \al^{3/2} || \Om||_{\cH^3},
\eal
\]
where we have applied Proposition \ref{prop:key} to $\Psi$, Lemma \ref{lem:l12} to $Q$ and Proposition \ref{prop:gam} to $g = \bar{\Om}, \bar{\eta}$. For $g = \bar{\xi}$, applying Proposition \ref{prop:prod1} yields
\footnote{The estimate of $I_2(\bar{\xi}),II(\bar{\xi})$ can be improved to $\al^{3/2} || \Om||_{\cH^3}$ but we do not need this extra smallness here.}
\[
\bal
|| I_2(\bar{\xi}) ||_{\cH^3(\psi)}
&\les \al^{-1/2} || Q||_{\cH^3} ( \al^{1/2} || D_R \bar{\xi}||_{\cC^1} + ||D_R\bar{\xi}||_{\cH^3(\psi)}  + \al^{1/2} || D_{\b} \bar{\xi} ||_{\cC^1}
+ || D_{\b} \bar{\xi}||_{\cH^3(\psi)} ) \les \al^{1/2} || \Om||_{\cH^3}
 ,\\ 
 ||II(\bar{\xi})||_{\cH^3(\psi)} & \les  \al^{-1/2} || \Om||_{\cH^3} ( \al^{3/2} ||D_R \bar{\xi}||_{\cW^{\cC^1}} + \al || D_R \bar{\xi}||_{\cH^3(\psi)} +
\al^{1/2} || D_{\b } \bar{\xi}||_{\cC^1} + || D_{\b} \bar{\xi} ||_{\cH^3(\psi)}  )  \les \al^{1/2} || \Om||_{\cH^3},
 \eal
\]
where we have used Lemma \ref{lem:xi} to estimate the norm of $\bar{\xi}$. Hence, we prove 
\[
||  \al c_{\om} D_R - (u \cdot \na)) \bar{\Om}||_{\cH^3} + || \al c_{\om} D_R- (u \cdot \na) ) \bar{\eta} ||_{\cH^3}  + 
|| \al c_{\om} D_R- (u \cdot \na) )\bar{\xi} ||_{\cH^3(\psi)} 
\les \al^{1/2} || \Om ||_{\cH^3}.
\]

\subsubsection{Transport term $IV$: $ (-3\al D_R - \bar{\uu} \cdot \na) g $ in $\bar{F}_{\Om}, \bar{F}_{\eta}, \bar{F}_{\xi}$}\label{sec:tran3} We will prove for $(g, \cH^3(\rho)) = 
(\bar{\Om},\cH^3) ,  (\bar{\eta}, \cH^3)  , (\bar{\xi},\cH^3(\psi))$
\beq\label{eq:error_tran}
|| (-3\al D_R - \bar{\uu} \cdot \na) g ||_{\cH^3(\rho)} \les \al^2 .
\eeq

From \eqref{eq:simp4}, we have $\f{2}{\pi} L_{12}(\overline{\Om})(0) = 3\al$. Hence, we can apply the decomposition in \eqref{eq:tran_R21}-\eqref{eq:tran_R22} to $(-3\al D_R - \bar{\uu} \cdot \na) g$ to get
\beq\label{eq:tran_R24}
\bal
 &(-3\al D_R - \bar{\uu} \cdot \na)  g = I_1(g) + I_2(g) + II(g),    \quad
II(g) = ( \al \pa_{\b} \bar{\Psi}_* D_R - (\sin(2\b))^{-1} (2 \bar{\Psi}_* + \al D_R \bar{\Psi})D_{\b} ) g  \\
 & 
 I_1(g) = \f{2 }{\pi } L_{12}(\bar{\Om})(0)   (- D_R g + \cos(2\b) \chi D_R g - \f{1}{\al} \chi D_{\b} g ), 
  \quad I_2(g) = \f{2}{\pi} \bar{Q} (\cos(2\b) D_R g - \f{1}{\al} D_{\b} g ) ,
 \eal
\eeq
where $\bar{Q} = L_{12}(\bar{\Om}) -\chi L_{12}(\bar{\Om})(0)$. Notice that the computation \eqref{eq:tran_R23} still holds for $g = \bar{\Om}, \bar{\eta}$
\[
I_1(g)  = \f{2}{\pi}   L_{12}(\bar{\Om})(0) ( -(1-\chi) D_R g + 2\sin^2(\b) \chi(- D_R g + g) .
\]
Recall $L_{12}(\bar{\Om}) = \f{3\al\pi}{2(1+R)}$. Notice that $(1-\chi) D_Rg , D_R g - g, Q D_R g, Q D_{\b} g$
 vanish quadratically near $R=0$.  Applying Lemma \ref{lem:bar} to $g = \bar{\Om}, \bar{\eta} $ and using a direct calculation yield
\[
|| I_1(g)||_{\cH^3} \les \al | L_{12}(\bar{\Om})(0)|  \les \al^2,
\quad || I_2(g)||_{\cH^3} \les \al^2.
\]
Since $\bar{\xi}$ already vanishes quadratically near $R=0$, using Lemma \ref{lem:xi} for $\bar{\xi}$ and a direct calculation give
\[
|| I_1(\bar{\xi})||_{\cH^3(\psi)} \les \al | L_{12}(\bar{\Om})(0)|  \les \al^2,
\quad || I_2(\bar{\xi})||_{\cH^3(\psi)} \les \al^2.
\]

For $II(g)$ with $g= \bar{\Om}, \bar{\eta}$, we apply Propositions \ref{prop:W2}, \ref{prop:alg} and the triangle inequality to yield
\[
\bal
||II(g)||_{\cH^3} &\les || \f{(1+R)^3}{R^2} II(g)||_{\cW^{3,\infty}}
\les  || \f{1+R}{R}\al \pa_{\b} \bar{\Psi}_*||_{\cW^{5,\infty} } || \f{(1+R)^2}{R} D_R g ||_{\cW^{3,\infty}}\\
 &+ || \f{1+R}{R} (\sin(2\b))^{-1} (2 \bar{\Psi}_* + \al D_R \bar{\Psi})  ||_{\cW^{5,\infty} }|| \f{(1+R)^2}{R} D_{\b} g||_{\cW^{3,\infty}}  \les \al^2,
\eal
\]
where we have applied 
Proposition \ref{prop:psi} to $\bar{\Psi}, \bar{\Psi}_*$ and Proposition \ref{prop:gam} to $g = \bar{\Om}, \bar{\eta}$. 

For $II(\bar{\xi})$, we use Propositions \ref{prop:prod1}, \ref{prop:psi} and Lemma \ref{lem:xi} to get
\[
\bal
 ||II(\bar{\xi})||_{\cH^3(\psi)} & \les
 \al^{-1/2}  || \pa_{\b} \bar{\Psi}_* ||_{\cW^{5,\infty}} 
( \al^{3/2} ||D_R \bar{\xi}||_{\cW^{\cC^1}} + \al || D_R \bar{\xi}||_{\cH^3(\psi)} ) \\
&+ \al^{-1/2} ||  (\sin(2\b))^{-1} (2 \bar{\Psi}_* + \al D_R \bar{\Psi}) ||_{\cW^{5,\infty}} 
(\al^{1/2} || D_{\b } \bar{\xi}||_{\cC^1} + || D_{\b} \bar{\xi} ||_{\cH^3(\psi)}  )
\les \al^{5/2}.
\eal
\]

\subsection{Nonlinear forcing terms in $P, N , F$}\label{sec:nonf}
The estimates in this subsection are obtained by applying the product estimates in subsection \ref{sec:nonterm} directly. The reader should pay attention to the cancellation near $R=0$ in the estimates in subsection \ref{sec:nonf2}.
\subsubsection{ Nonlinear forcing term in $P_{\eta}, P_{\xi}$} We are going to estimate
\[
\bal
I_1&= || - ( U_1(\bar{\Psi}) + \f{3}{1+R}) \eta  
 - ( U_1(\Psi) + \f{2}{\pi \al} L_{12}(\Om)) \bar{\eta} ||_{\cH^3}, \quad I_2 = ||  V_1(\bar{\Psi}) \xi  + V_1(\Psi) \bar{\xi}  ||_{\cH^3} , \\
 II_1 &= || (- V_2(\bar{\Psi}) +\f{3}{1+R} ) \xi  + ( - V_2(\Psi)+  \f{2}{\pi \al} L_{12}(\Om) ) \bar{\xi} ||_{\cH^3(\psi)} , \quad
 II_2 =  ||  U_2(\Psi) \bar{\eta}  + U_2(\bar{\Psi}) \eta ||_{\cH^3(\psi)} .
 \eal
\]

From \eqref{eq:simp4}, $\f{2}{\pi \al} L_{12}(\bar{\Om}) = \f{3}{1+R}$. Recall the formula of $U_i, V_j$ in \eqref{eq:simp51}-\eqref{eq:simp512}. Applying Propositions \ref{prop:psi}, \ref{prop:key}, we obtain
\beq\label{eq:small_uv}
\bal
&|| U_1(\bar{\Psi}) + \f{2}{\pi \al} L_{12}(\bar{\Om}) ||_{\cW^{5,\infty}}
= || -V_2(\bar{\Psi}) + \f{2}{\pi \al} L_{12}(\bar{\Om}) ||_{\cW^{5,\infty}} \les \al,
\quad || U_2(\bar{\Psi}) ||_{\cW^{5,\infty}} \les \al , \\
&|| U_1(\Psi) + \f{2}{\pi \al} L_{12}( \Om) ||_{\cH^3}
= || -V_2( \Psi) + \f{2}{\pi \al} L_{12}( \Om) ||_{\cW^{5,\infty}} \les || \Om||_{\cH^3}
\quad || U_2(\Psi) ||_{ \cH^3} \les  || \Om||_{\cH^3} . \\
\eal
\eeq

Applying Proposition \ref{prop:prod1}, Lemma \ref{prop:gam} to $\bar{\eta}$ and Lemma \ref{lem:xi} to $\bar{\xi}$, we yield
\[
\bal
I_1 & \les \al^{1/2} || \eta||_{\cH^3} + \al^{-1/2} || \Om||_{\cH^3} ||  \bar{\eta}||_{\cW^{5,\infty}} \les  \al^{1/2} (  || \eta||_{\cH^3} + || \Om||_{\cH^3} ), \\ 
 II_1 
&\les \al^{1/2} ( \al^{1/2} || \xi||_{\cC^1} + || \xi||_{\cH^3(\psi)})+ \al^{-1/2} || \Om||_{\cH^3} ( \al^{1/2} || \bar{\xi}||_{\cC^1} + || \bar{\xi}||_{\cH^3(\psi)}) \\
& \les  \al^{1/2} ( \al^{1/2} || \xi||_{\cC^1} + || \xi||_{\cH^3(\psi)})+ \al^{3/2} || \Om||_{\cH^3} ,
\eal
\]
where we have used Lemma \ref{lem:xi} in the last inequality. Using Lemma \ref{lem:H2H2} and Proposition \ref{prop:prod1}, we derive
\[
II_2 \les ||  U_2(\Psi) \bar{\eta}  + U_2(\bar{\Psi}) \eta ||_{\cH^3}
\les \al^{-1/2} (  || \Om||_{\cH^3} || \bar{\eta}||_{\cW^{5,\infty}} + || U_2(\bar{\Psi})||_{\cW^{5,\infty}}  || \eta||_{\cH^3})
\les \al^{1/2} ( || \Om||_{\cH^3}  + || \eta||_{\cH^3} ).
\]
For $I_2$, we use Proposition \ref{prop:prod3} and Lemma \ref{lem:xi} to obtain 
\[
I_2 \les \al^{1/2} || \xi||_{\cH^3(\psi)} + 
\al^{-1/2} ||\Om||_{\cH^3} ( \al^{1/2} || \bar{\xi}||_{\cC^1} + || \bar{\xi}||_{\cH^3(\psi)})
\les  \al^{1/2} || \xi||_{\cH^3(\psi)}  + \al^{3/2} ||\Om||_{\cH^3}.
\]

\subsubsection{Nonlinear forcing term in $N$ \eqref{eq:non}: $c_{\om} \Om, \ (2 c_{\om} -  U_1(\Psi)) \eta - V_1(\Psi) \xi, (2 c_{\om} -  V_2(\Psi) ) \xi -  U_2(\Psi) \eta $} \label{sec:nonf2} \quad

Recall the formula of $U_1, V_2$ in \eqref{eq:simp51}. We use the following decomposition 
\[
-V_2(\Psi) = U_1(\Psi)= ( U_1(\Psi) + \f{2}{\pi \al} L_{12}(\Om) ) - \f{2}{\pi \al} L_{12}(\Om)
 = I + II.
\]

Applying Proposition \ref{prop:key} to $I$ and Lemma \ref{lem:l12} to $II$, we obtain 
\beq\label{eq:U1X}
|| V_2(\Psi)  ||_X = || U_1(\Psi) ||_X 
\les || I||_{\cH^3} + \al^{-1} || L_{12}(\Om)||_X \les \al^{-1} || \Om||_{\cH^3} .
\eeq

Applying Propositions \ref{prop:prod1}, \ref{prop:key}, we get
\[
|| U_1(\Psi) \eta ||_{\cH^3} \les \al^{-3/2} || \Om||_{\cH^3} || \eta||_{\cH^3}
,\quad
|| (V_2(\Psi)  \xi  ||_{\cH^3(\psi)} \les 
\al^{-3/2} || \Om||_{\cH^3} ( ||  \xi||_{\cH^3(\psi)} +\al^{1/2} || \xi||_{\cC^1} ).
\]

Applying Proposition \ref{prop:prod3} to $V_1 \xi$, Proposition \ref{prop:prod1} and Lemma \ref{lem:H2H2} to $U_2 \eta$ yields 
\[
\bal
|| - V_1(\Psi) \xi ||_{\cH^3} &\les
\al^{-1/2} || \Om ||_{\cH^3} ( || \xi||_{\cH^3(\psi)} + \al^{1/2} || \xi||_{\cC^1} ), \\
  || - U_2(\Psi) \eta ||_{\cH^3(\psi)} &\les 
 ||  U_2(\Psi) \eta ||_{\cH^3} \les \al^{-1/2} || \Om||_{\cH^3} || \eta||_{\cH^3}.
 \eal
\] 
Finally, from \eqref{eq:normal}, \eqref{eq:l12}, the scalar $c_{\om}$ satisfies 
$| c_{\om} | \les \al^{-1} || \Om||_{\cH^3}$. Hence, we obtain 
\[
|| c_{\om} \Om||_{\cH^3} \les \al^{-1} || \Om||_{\cH^3}^2,  \ 
|| c_{\om}  \eta||_{\cH^3} \les \al^{-1} || \Om||_{\cH^3} || \eta||_{\cH^3},  \ 
 || c_{\om} \xi||_{\cH^3(\psi)} \les \al^{-1} || \Om||_{\cH^3} || \xi||_{\cH^3(\psi)} .
\]

\subsubsection{Nonlinear forcing terms in $F$} Recall that we have estimated the transport term 
$(-3 \al D_R - \bar{\uu} \na ) g$ in $F_{\Om}, F_{\eta}, F_{\xi}$ in \eqref{eq:error_tran}. The remaining terms in $\bar{F}_{\eta}$ and $\bar{F}_{\xi}$ (see \eqref{eq:error}, \eqref{eq:error2}) are 
\beq\label{eq:error_rem}
I = (-\f{3}{1+R} - U_1(\bar{\Psi})) \bar{\eta} - V_1(\bar{\Psi}) \bar{\xi} 
, \quad 
II = (2 \bar{c}_{\om} - V_2(\bar{\Psi})) \bar{\xi} - U_2(\bar{\Psi}) \bar{\eta} 
- D_R\bar{\xi}  ,\\
\eeq
where we have used $ -\al \bar{c}_l D_R   = - D_R -3\al D_R$ since $\bar{c}_l = \f{1}{\al} +3$ \eqref{eq:profile}. From \eqref{eq:simp4}, we have $\f{2}{\pi \al}L_{12}(\bar{\Om}) = \f{3}{1+R}$. Using $U_i, V_j$ in \eqref{eq:simp51}-\eqref{eq:simp512}, $\bar{\eta}$ \eqref{eq:profile} and Proposition \ref{prop:psi}, we have 
\[
\bal
 || \f{1+R}{R} (U_1(\bar{\Psi}) + \f{3}{1+R}) ||_{\cW^{5,\infty}}
\les \al , \quad
|| \f{1+R}{R} U_2(\bar{\Psi} ) ||_{\cW^{5,\infty}} \les \al , \quad 
 || \f{(1+R)^2}{R} \bar{\eta}||_{\cW^{5,\infty}} \les \al .\\
\eal
\]
Applying the embedding in Proposition \ref{prop:W2} and then the algebra property of $\cW^{3,\infty}$ in Proposition \ref{prop:alg} to $\bar{\eta}$ and the above estimates, we get  
\[
|| (-\f{3}{1+R} - U_1(\bar{\Psi})) \bar{\eta}||_{\cH^3} \les \al^2,  \quad 
|| U_2(\bar{\Psi}) \bar{\eta}||_{\cH^3(\psi)} \les  || U_2(\bar{\Psi}) \bar{\eta}||_{\cH^3} \les \al^2 ,
\] 
where we have used \eqref{eq:H2H2} in the second inequality. 
Applying the product estimates in Propositions \ref{prop:prod1}, \ref{prop:prod3}, Proposition \ref{prop:psi} to $V_2(\bar{\Psi})$ and Lemma \ref{lem:xi} to $\bar{\xi}$, we yield 
\[
\bal
|| ( V_2(\bar{\Psi} )-\f{3}{1+R}) \bar{\xi}||_{\cH^3(\psi)} & \les 
\al^{-1/2} \cdot \al ( \al^{1/2} || \bar{\xi}||_{\cC^1} + || \bar{\xi}||_{\cH^3(\psi)} )  
\les \al^{5/2} , \\
|| V_1(\bar{\Psi}) \bar{\xi}||_{\cH^3}  &\les \al^{1/2}  || \bar{\xi} ||_{\cH^3(\psi)}
\les \al^{5/2}.
\eal
\]


For the remaining part in $II$, we simply use $\bar{c}_{\om} = -1$ and Lemma \ref{lem:xi} to get
\[
|| 2 \bar{c}_{\om} \bar{\xi}-D_R\bar{\xi}||_{\cH^3(\psi)}   + || \f{3}{1+R} \bar{\xi}||_{\cH^3(\psi)}  \les \al^2.
\]
Therefore, combining the formula of $\bar{F}$ in \eqref{eq:error}, \eqref{eq:error2}, the estimate
\eqref{eq:error_tran} and the above estimates of $I,II$, we prove
\beq\label{eq:error3}
|| \bar{F}_{\Om}||_{\cH^3} \les \al^2, \quad || \bar{F}_{\eta}||_{\cH^3} \les \al^2, \quad || \bar{F}_{\xi}||_{\cH^3(\psi)} \les \al^2.
\eeq

\subsection{Analysis of the remaining terms in $\cR_3$}\label{sec:remR3}
It remains to estimate 
\beq\label{eq:remR31}
\la \cR_{\Om}, \Om \vp_0 \ra , \quad \la \cR_{\eta} ,\eta \psi_0 \ra , \quad \f{81}{4\pi c} L_{12}(\Om)(0) \la \cR_{\Om}, \sin(2\b) R^{-1} \ra ,
\eeq
in $\cR_3$ \eqref{eq:remR3}. Recall the definition of $\vp_0, \psi_0$ in Definition \ref{wg:L2_R0} and $\vp_1$ in Definition \ref{def:wg}. Note that $\psi_0(R,\b)$ grows linearly for large $R$. Clearly, we have 
 \[
\vp_0 \les \vp_1, \quad \psi_0 = \f{9}{32} R \G(\b)^{-1}  +\f{3}{16}\lt( \f{(1+R)^3}{R^4} +\f{3}{2} \f{(1+R)^4}{R^3} - \f{3}{2} R \rt) \G(\b)^{-1} \teq \psi_{0,1} + \psi_{0,2}.
 \]
Since the weights $\vp_0, \psi_{0,2}, R^{-1} \sin(2\b)$ are much weaker than the weights $\vp_1$, the estimates of 
\[
\la \cR_{\Om}, \Om \vp_0 \ra , \quad \la \cR_{\eta} ,\eta \psi_{0,2} \ra  ,\quad \f{81}{4\pi c} L_{12}(\Om)(0) \la \cR_{\Om}, \sin(2\b) R^{-1} \ra 
\]
follows from the same argument as that in the last two sections and a similar bound can be derived. It remains to estimate $\la R_{\eta} ,\eta R \G(\b)^{-1} \ra $. Compared to $\vp_1$, $R \G(\b)^{-1}$ is much less singular in $R$ and $\b$. We focus on how to control the growing factor $R$. We use the decay estimate of $\bar{\eta}$ in Lemma \ref{lem:bar} and $\bar{\xi}$ in Lemma \ref{lem:xi}. In particular, for $i+j \leq 7$ we have
\beq\label{eq:decayxi}
| D_R^i D_{\b}^j \bar{\eta} | \les \al (1+R)^{-2}, \quad |D_R^i D_{\b}^j \bar{\xi}| \les |\bar{\xi}| \les \al^2 (1+R)^{-2} \sin(\b)^{-2\al}  .
\eeq

Recall the decomposition of $\cR_{\eta}$ in \eqref{eq:lin3} and the error $\bar{F}_{\eta}$ defined in \eqref{eq:error2}. We use argument similar to that in the last subsection to estimate 
$|| \bar{F}_{\eta} ( R \G(\b)^{-1} )^{1/2}||_2 $. A typical term in $\bar{F}_{\eta}$ can be estimated as follows 
\[
\int_0^{\infty}\int_0^{\pi/2} V_1(\bar{\Psi})^2 \bar{\xi}^2  R \G(\b)^{-1} d R d \b
\les \al^2  \int_0^{\infty}\int_0^{\pi/2}  \al^4 (1+R)^{-4} \sin(\b)^{-4\al}  R \G(\b)^{-1} dR d\b \les \al^6 \les \al^4,
\]
where we have applied Proposition \ref{prop:psi} to estimate $V_1(\bar{\Psi})$ and used $\al < \f{1}{8}$ (we will choose $\al$ sufficiently small). Similarly, we have
$ ||  \bar{F}_{\eta} ( R \G(\b)^{-1})^{1/2} ||_2^2  \les \al^4$. Hence, using the Cauchy-Schwarz inequality, we get
	\[
| \la \bar{F}_{\eta}, \eta R \G(\b)^{-1} \ra|
\les || \bar{F}_{\eta}  ( R\G(\b)^{-1})^{1/2}||_2  ||  \eta  ( R \G(\b)^{-1})^{1/2} ||_2
\les \al^2  || \eta  \psi_0^{1/2} ||_2  \les \al^2 E_3,
	\]
	where we have used \eqref{eq:E3} to derive the last inequality.

Recall $P_{\eta}$ in \eqref{eq:lin3}, $N_{\eta}$ in \eqref{eq:non} and the formula of $\uu \cdot \na$ in \eqref{eq:trans41}. We use integration by parts and then a $L^{\infty}$ estimate to estimate the transport terms in $P_{\eta}, N_{\eta}$. A typical term in these transport terms can be estimated as follows
\[
\bal
| \la \f{2}{\pi \al} L_{12}(\Om) D_{\b} \eta, \eta R \G^{-1} \ra|
&= | \la \f{2}{\pi \al} L_{12}(\Om) \pa_{\b} (\sin(2\b) \G^{-1} ) , \eta^2 R \ra |
\les \al^{-1} || L_{12}(\Om) ||_{\infty} ||  \eta ( R \G^{-1} )^{1/2}||_2^2  \\
& \les \al^{-1} || \Om \vp_1^{1/2} ||_{L^2} ||  \eta \psi_0^{1/2}||_2^2 \les \al^{-1} E_3^3 ,
\eal
\]
where we have used $\G(\b) = \cos(\b)^{\al}, |\sin(2\b)\pa_{\b} \G(\b)^{-1}| \les \G(\b)^{-1} $
in the first inequality, Lemma \ref{lem:l12} in the second inequality and \eqref{eq:E3} in the last inequality.

For the nonlinear terms related to $\eta$, i.e. $( 2 c_{\om} -U_1(\Psi)) \eta$ in $N_{\eta}$ \eqref{eq:non} and $-(U_1(\Psi) + \f{3}{1+R}) \eta$ in $P_{\eta}$ \eqref{eq:lin3}, we also apply a $L^{\infty}$ estimate. For example, we have
\[
| \la (2c_{\om} - U_1(\Psi)) \eta, \eta R \G(\b)^{-1} \ra |
\les || 2c_{\om} -U_1(\Psi) ||_{L^{\infty}} ||  \eta \psi_0^{1/2}||_2^2
\les \al^{-1} || \Om||_{\cH^3} ||\eta \psi_0^{1/2}||_2^2  \les \al^{-1} E_3^3,
\]
where we have used \eqref{eq:U1X} and $|c_{\om}|  = \f{2}{\pi \al} |L_{12}(\Om)(0)| \les \al^{-1} || \Om||_{\cH^3}$ (see Lemma \ref{lem:l12}) in the last inequality.

For the terms related to $\bar{\eta}, \bar{\xi}$ in $P_{\eta}$ \eqref{eq:lin3}, i.e. $
( U_1(\Psi) + \f{2}{\pi \al} L_{12}(\Om)) \bar{\eta}, \ V_1(\Psi) \bar{\xi} )$, they can be estimated easily by using the fast decay of $\bar{\xi}, \bar{\eta}$ \eqref{eq:decayxi}.

Finally, for the terms related to $\xi$, i.e. $V_1(\Psi) \xi$ in $N_{\eta}$ \eqref{eq:non} and $V_1(\bar{\Psi}) \xi$ in \eqref{eq:lin3}, we get
\[
\bal
&|  \la V_1(\bar{\Psi}) \xi, \eta R \G^{-1} \ra|
+ |  \la V_1(\Psi) \xi, \eta R \G^{-1} \ra| \\
\les & || \eta, R^{1/2} \G^{-1/2} ||_{L^2}
|| \xi R^{1/2} \sin(2\b)^{1/4} ||_{L^{\infty}}
( || V_1(\bar{\Psi}) \sin(2\b)^{-1/4} \G^{-1/2} ||_{L^2}
+ || V_1(\Psi) \sin(2\b)^{-1/4} \G^{-1/2} ||_{L^2}  ) \\
 \les & || \eta \psi_0^{1/2} ||_{L^2}  || \xi||_{\cH^3(\psi)}
( || V_1(\bar{\Psi}) \sin(2\b)^{-\s/ 2} ||_{L^2}
+ || V_1(\Psi) \sin(2\b)^{-\s/ 2} ||_{L^2}  )  \\
\les & E_3^2
( || \bar{\Om} \sin(2\b)^{-\s/2} ||_{L^2}+ ||\Om  \sin(2\b)^{-\s/2}||_{L^2}  ) 
\les E_3^2 (\al + E_3),
\eal
\]
where we have applied Lemma \ref{lem:xi_decay} in the second inequality, the weighted $L^2$ (with weight $\sin(2\b)^{-\s}, \s = \f{99}{100}$) version of Proposition \ref{prop:key} in the third inequality and a direct computation using \eqref{eq:profile} in the last inequality. 

Combining the estimates of $\bar{F}_{\eta}, P_{\eta}, N_{\eta}$, we have 
\[
| \la \cR_{\eta}, \eta R \G^{-1} \ra | \les \al^{-3/2} E_3^3 + \al^{1/2} E_3^2 + \al^2 E_3.
\]




\subsubsection{Completing the $\cH^3$ and $\cH^3(\psi)$ estimates}
From \eqref{eq:E2H2}, we can use $E_3$ to bound $|| \Om||_{\cH^3}$ , $||\eta||_{\cH^3},||\xi||_{\cH^3(\psi)}$. Combining the estimates 
in the last few subsections, we prove 
\[
| \la \cR_{\Om} , \Om \ra_{\cH^3} | \ , 
| \la \cR_{\eta} , \eta \ra_{\cH^3} | \ , | \la \cR_{\xi} , \xi \ra_{\cH^3(\psi)} | 
\les \al^{1/2} ( E_3^2 + \al || \xi||^2_{\cC^1}) + \al^{-3/2} ( E_3 + \al^{1/2}||\xi||_{\cC^1})^3 + \al^2 E_3 ,
\]
where $E_3$ is defined in \eqref{eg:H3}. Combining Corollary \ref{cor:H3} and the above estimates, we prove \eqref{eq:boot1}.

\subsection{Remaining terms in the $\cC^1$ estimate of $\xi$} \label{sec:C1} \quad

Recall that we perform $L^{\infty}$ estimates of $\xi$ and its derivatives in subsection \ref{sec:inf}. In this subsection, we complete the estimate of the remaining terms in these estimates and derive \eqref{eq:boot2}. We group together the remaining terms in \eqref{eq:xi_inf1}, \eqref{eq:xi_inf2R}, \eqref{eq:xi_inf2b}, which remain to be estimated.
They can be bounded by 
\[
\bal
 &|| \xi||_{\cC^1} ( || \Xi_1 ||_{\cC^1}+ || \Xi_2 ||_{\cC^1} + ||  \bar{F}_{\xi}||_{ \cC^1}  + || N_o||_{\cC^1} ),   \quad  ||  \xi||_{\cC^1}  || [ \phi_1 D_R, \cA_2] \xi ||_{\infty} , \\
  &  || \xi||_{\cC^1}    || [ \phi_2 D_{\b}, \cA_2] \xi ||_{\infty} ,  \quad
|\al c_l|  || \phi_1 D_R \xi ||^2_{L^{\infty}}, 
\quad || \phi_2 D_{\b} \xi||_{\infty} || \cA_1(\phi_2 - 1) \cdot D_{\b}\xi ||_{L^{\infty}}.
\eal
\]

\subsubsection{Analysis of  $\Xi_1, \Xi_2, N_o$}\label{sec:Xi1}
 Recall $\Xi_1, \Xi_2, N_o$ in \eqref{eq:Xi1}, \eqref{eq:Xi2},\eqref{eq:non_xio}
\[
\bal
&\Xi_1 = (\f{3}{1+R} - V_2(\bar{\Psi})) \xi  , \quad
\Xi_2 =  - V_2(\Psi) \bar{\xi} + c_{\om} ( 2 \bar{\xi} - R\pa_R \bar{\xi})  +( \al c_{\om} R\pa_R - ( \uu \cdot  \na ) ) \bar{\xi} -  (U_2(\Psi) \bar{\eta} + U_2(\bar{\Psi}) \eta ), \\
 & N_o = (2c_{\om} - V_2(\Psi))\xi  - U_2(\Psi) \eta,
 \eal 
\]
where we have used $V_2(\Psi) = v_y, U_2(\Psi) =u_y $ \eqref{eq:simp50}. Recall  
  \eqref{eq:simp4}, \eqref{eq:normal}, \eqref{eq:nota_psi}. We have
\[\f{2}{\pi \al} L_{12}(\bar{\Om}) = \f{3}{1+R}, \
c_{\om} = -\f{2}{\pi\al} L_{12}(\Om)(0), \  \Psi_* = \Psi - \f{\sin(2\b)}{\pi \al} L_{12}(\Om).
\] 
 Then we obtain $V_2(\bar{\Psi}) - \f{3}{1+R} = -U_1 ( \bar{\Psi}, \bar{\Psi}_*)$ (see \eqref{eq:simp51}) .

For the transport term $( \al c_{\om} D_R - (\uu \cdot \na) ) \bar{\xi}$, we use the decomposition \eqref{eq:tran_R21}-\eqref{eq:tran_R22} with $g= \bar{\xi}$. Then each term in $\Xi_1, \Xi_2, N_o$
depends only on $ L_{12}(\Om), \Psi, \eta, \xi$ and their approximate steady state, e.g. $V_2(\bar{\Psi})$. To estimate the $\cC^1$ norm of the product in $\Xi_1, \Xi_2, N_o$, using Proposition \ref{prop:c1}, we only need to estimate the $\cC^1$ norm of each single term. 


For the terms depending on $\Psi,\Psi_*$, e.g. $V_2(\Psi) - \f{2}{\pi \al} L_{12}(\Om)$ (see \eqref{eq:simp51}-\eqref{eq:simp512}), we apply Proposition \ref{prop:key}  and Lemma \ref{lem:inf} to obtain the $\cC^1$ estimate. For the terms depending on $\bar{\Psi}, \bar{\Psi}_*$, we apply Propositions \ref{prop:psi} and \ref{prop:c1} to estimate the $\cC^1$ norm. 

For the terms depending on $L_{12}(\Om)$, we use \eqref{eq:l12X} in Lemma \ref{lem:l12} to estimate the $\cC^1$ norm. 

The slightly difficult term is $ V_2(\Psi)$. Using the formula of $V_2(\Psi)$ in \eqref{eq:simp51}, \eqref{eq:simp512}, Propositions \ref{prop:key}, and Lemmas \ref{lem:inf}, \ref{lem:l12}, we get
\beq\label{eq:xi_c11}
\bal
&||V_2(\Psi)||_{\cC^1}  
\les ||V_2(\Psi) - \f{2}{\pi \al} L_{12}(\Om)||_{\cC^1} + \f{2}{\pi \al} || L_{12}(\Om)||_{\cC^1} \les  (\al^{-1/2} + \al^{-1}) ||\Om||_{\cH^3} \les \al^{-1} ||\Om||_{\cH^3}.
\eal
\eeq

Using \eqref{eq:xi0}-\eqref{eq:xi} in Lemma \ref{lem:xi} and Lemma \ref{lem:bar}, we have $|| \bar{\xi}||_{\cC^1}  + || D_R \bar{\xi}||_{\cC^1} \les \al^2, || \bar{\eta}||_{\cC^1} \les \al $. From \eqref{eq:simp4}, we know $|| L_{12}(\bar{\Om})||_{\cC^1} \les \al$. Therefore, we get
\[
|| \Xi_1||_{\cC^1} \les  \al || \xi||_{\cC^1} , \quad 
|| \Xi_2||_{\cC^1} \les \al^{1/2} || \Om||_{\cH^3} + \al^{1/2} || \eta||_{\cH^3}, \quad
|| N_o ||_{\cC^1} \les  \al^{-1} || \xi||_{\cC^1} || \Om||_{\cH^3} 
+ \al^{-1} || \Om||_{\cH^3} || \eta||_{\cH^3}.
\]
The largest term in $\Xi_2$ is given by $ (U_2(\Psi) \bar{\eta} + U_2(\bar{\Psi})\eta)$, which leads to the above upper bound.


\subsubsection{Analysis of $\bar{F}_{\xi}$} Recall $\bar{F}_{\xi}$ and $\bar{\uu} \cdot \na$ defined in \eqref{eq:error} and \eqref{eq:trans41}	
\[
\bal
& \bar{F}_{\xi}  = (2 \bar{c}_{\om} -  V_2(\bar{\Psi})) \bar{\xi} -  U_2(\bar{\Psi}) \bar{\eta} 
-  \al \bar{c}_l R \pa_R\bar{\xi} - (\bar{\uu} \cdot \na)  \bar{\xi} ,\\
 & \bar{\uu} \cdot \na \bar{\xi} =   (-\f{ 2\cos(2\b)}{\pi } L_{12}(\Om)   - \al \pa_{\b} \Psi_* ) D_R\bar{\xi} + ( \f{2 }{\pi \al} L_{12}(\Om) + \f{ 2 \Psi_* + \al D_R \Psi }{\sin(2\b)}  ) D_{\b}  \bar{\xi}.
\eal
\]
For  $\bar{\xi}$ terms, we use $ || D_R^i D^j_{\b}\bar{\xi}||_{\cC^1} \les \al^2, i + j \leq 2$  
from \eqref{eq:xi0}-\eqref{eq:xi} in Lemma \ref{lem:xi}. For other terms, we use $|| \bar{\eta}||_{\cC^1} \les \al $ from Lemma \ref{lem:bar} and apply the strategy in the last subsection to estimate the $\cC^1$ norm
. We get
\[
||  \bar{F}_{\xi} ||_{\cC^1} \les \al^2.
\]

\subsubsection{$|| [ \phi_2 D_{\b}, \cA_2] \xi ||_{\infty} $, $  || [ \phi_1 D_R, \cA_2] \xi ||_{\infty}$} 
Recall $\cA_2$ defined in \eqref{eq:cA2}. Using \eqref{eq:trans41}, we have
\[
\bal
&\cA_2 (\xi) =\f{2}{\pi \al} L_{12}(\Om) D_{\b} \xi + (  \cT(\bar{\Om}) +\cT( \Om) )\xi
= \f{2}{\pi \al} L_{12}(\Om) D_{\b} \xi
 - \f{2}{\pi} \cos(2\b) (  L_{12}(\Om) +L_{12}(\bar{\Om}) ) D_R \xi\\
 &- \al (\pa_{\b} \Psi_* + \pa_{\b}\bar{\Psi}_*)D_R \xi
 +  \f{  2\Psi_* + \al D_R \Psi + 2 \bar{\Psi}_* + \al D_R \bar{\Psi}  }{\sin(2\b)}   D_{\b} \xi
  \teq ( H_1 D_{\b}  + H_2 D_R + H_3 D_{R}+ H_4 D_{\b})\xi.
\eal
\]

Recall $\phi_1, \phi_2$ defined in \eqref{wg:c1}. For $D = D_R, D_{\b}$ and $\phi = \phi_1, \phi_2$, a direct computation yields
\beq\label{eq:cong}
| \phi^{-1}  D \phi |\les 1.
\eeq
Let $H \td{D}$ be a term in the above formula of $\cA_2$ 
and $(D, \phi) =(D_R,\phi_1)$ or $(D_{\b}, \phi_2)$. Using \eqref{eq:cong} and 
the $\cC^1$ norm defined in \eqref{norm:c1} to control the $L^{\infty}$ norm of  $\phi D H, \phi D \xi,  \td{D} \xi, H$, we obtain 
\[
\bal
| [ \phi D, H \td{D}] \xi |
&= | \phi D H \cdot \td{D} \xi - H \td{D} \phi \cdot D \xi |
\leq  || H||_{\cC^1} || \xi ||_{\cC^1}+ || H||_{L^{\infty}} || \phi^{-1} \td{D} \phi ||_{L^{\infty} } || \phi D \xi ||_{L^{\infty}}  \les
|| H ||_{\cC^1} || \xi ||_{\cC^1}.
\eal
\]

Applying the strategy in Section \ref{sec:Xi1} to estimate the $\cC^1$ norm of $\Psi, \bar{\Psi}, L_{12}(\Om)$ terms, we get
\[
|| H_1||_{\cC^1} \les  \al^{-1} || \Om||_{\cH^3}, \quad || H_2||_{\cC^1} \les  || \Om||_{\cH^3} + \al,  \quad || H_3||_{\cC^1} \les \al^{1/2} || \Om||_{\cH^3} + \al^2, \quad 
|| H_4 ||_{\cC^1} \les \al^{-1/2} || \Om||_{\cH^3} +  \al.
\]
The largest term is $\al^{-1} L_{12}(\Om)$ in $H_1$, which is estimated by \eqref{eq:l12X} in Lemma \ref{lem:l12} and using $D_{\b} L_{12}(\Om) = 0$.

Combining the above estimates, we conclude that
\[
|| [D_R, \cA_2] \xi  ||_{\infty} , \ || [D_{\b}, \cA_2] \xi  ||_{\infty}
\les || \xi||_{\cC^1}  ( \al^{-1} || \Om||_{\cH^3} + \al  ).
\]

\subsubsection{Analysis of $|\al c_l| ,  || \cA_1(\phi_2 - 1) \cdot D_{\b}\xi ||_{L^{\infty}}$}
Using \eqref{eq:normal} and \eqref{eq:l12} in Lemma \ref{lem:l12}, we obtain 
\[
|\al c_l| \leq C \al^{-1} | L_{12}(\Om)(0)| \leq C\al^{-1} || \Om||_{\cH^3}.
\]
Using the formulas of $\phi_2, \cA_1$ in \eqref{wg:c1}, \eqref{eq:cA2}, we get 
\[
\bal
&| \phi_2^{-1} \cA_1 (\phi_2 - 1) |
= | \phi_2^{-1} ( (1+ 3\al + \al c_l) D_R +  \f{3 }{ 1+R} D_{\b}  ) ( R \sin(2\b)^{\al})^{-1/40} |\\
\leq & \phi_2^{-1} ( \f{1}{40} (1+ 3\al + \al c_l)  
+ C \al  )  (R \sin(2\b)^{\al})^{-1/40} 
\leq \f{1}{40} (1 + 3 \al + C \al^{-1} || \Om||_{\cH^3}) + C\al,
\eal
\]
where we have used $D_R (R \sin(2\b)^{\al})^{-1/40} = -\f{1}{40}(R \sin(2\b)^{\al})^{-1/40}, 
| D_{\b} (R \sin(2\b)^{\al})^{-1/40}| \les \al | (R \sin(2\b)^{\al})^{-1/40} |$ in the first inequality. Therefore, we get 
\[
 || \cA_1(\phi_2 - 1) \cdot D_{\b}\xi ||_{L^{\infty}}
 \leq (\f{1}{40} + C \al + C \al^{-1} || \Om||_{\cH^3} ) || \phi_2 D_{\b} \xi ||_{L^{\infty}}.
\]

\subsubsection{Completing the $\cC^1$ estimates} 
From \eqref{eq:E2H2}, we can use $E_3$ to further bound $|| \Om||_{\cH^3},||\eta||_{\cH^3},||\xi||_{\cH^3(\psi)}$. Plugging all the above estimates of the remaining terms in \eqref{eq:xi_inf1}, \eqref{eq:xi_inf2R}, \eqref{eq:xi_inf2b}, we prove
\[
\bal
\f{1}{2} \f{d}{dt} || \xi||^2_{\infty}
&\leq -2 || \xi||^2_{\infty} + C || \xi||_{\cC^1}( \al^{1/2} E_3 + \al || \xi||_{\cC^1}
+ \al^{-1} E^2_3 + \al^{-1} E_3 || \xi||_{\cC^1} ) + C\al^2 || \xi||_{\infty} ,\\
\f{1}{2} \f{d}{dt} || \phi_2 D_{\b} \xi||^2_{\infty}
&\leq -(2 - \f{1}{40}) || \phi_2 D_{\b}\xi||^2_{\infty}   \\
&+ C || \xi||_{\cC^1}( \al^{1/2} E_3 + \al || \xi||_{\cC^1}
+ \al^{-1} E^2_3 + \al^{-1} E_3 || \xi||_{\cC^1} ) + C\al^2 || \phi_2 D_{\b}\xi||_{\infty} ,\\ 
\f{1}{2} \f{d}{dt} || \phi_1 D_R \xi||^2_{\infty}
&\leq -2 || \phi_1 D_R \xi||^2_{\infty}  +  3 ||\phi_1 D_R \xi||_{\infty}( || \phi_2 D_{\b} \xi||_{\infty} + ||\xi||_{\infty} ) \\
&+C || \xi||_{\cC^1}( \al^{1/2} E_3 + \al || \xi||_{\cC^1}
+ \al^{-1} E^2_3 + \al^{-1} E_3 || \xi||_{\cC^1} ) + C\al^2 || \phi_1 D_{R}\xi||_{\infty} .\\
\eal
\]
Hence, for some absolute constant $\mu_4$, e.g. $\mu_4 = \f{1}{10}$, the energy defined in \eqref{eg:xi_inf} satisfies \eqref{eq:boot2}.

\subsection{Finite time blowup with finite energy velocity field}\label{sec:blowup}

\subsubsection{The bootstrap argument}
Now, we construct the energy 
\beq\label{eg}
E( \Om, \eta, \xi) =( E_3(\Om, \eta, \xi)^2 + \al E(\xi, \infty)^2)^{1/2}.
\eeq

Adding the estimates  \eqref{eq:boot1} and $ \al \times$\eqref{eq:boot2}, we have
\beq\label{eq:boot4}
\f{1}{2} \f{d}{dt} E^2(\Om, \eta,\xi)
\leq  -\f{1}{12} E^2 + K \al^{1/2} E^2 + K \al^{-3/2} E^3 + K \al^2 E ,
\eeq
for some universal constant $K$, where we have used the fact that $E(\xi, \infty)$ is equivalent to $|| \xi||_{\cC^1}$ since $\mu_4$ is an absolute constant.
We know that there exists a small absolute constant $\al_1 < \f{1}{1000}$ and $K_*$, such that, for any $\al < \al_1$ and $E = K_* \al^2$, we have 
\beq\label{eq:reuse}
 -\f{1}{12} E^2 + K \al^{1/2} E^2 + K \al^{-3/2} E^3 + K \al^2 E  < 0.
\eeq
If $E(\Om(\cdot, 0), \eta(\cdot ,0), \xi(\cdot, 0)) < K_* \al^2$, we have 
\beq\label{eq:boot42}
E(\Om(t), \eta(t), \xi(t) )  < K_* \al^2 ,
\eeq
for all time $t > 0$, where we have used the time-dependent normalization condition \eqref{eq:normal} for $c_{\om}(t), c_l(t)$. Applying Lemma \ref{lem:l12} to $L_{12}(\Om)(0)$ and Lemma \ref{lem:inf} to $\Om, \eta$, we derive
\[
\bal
& | c_{\om}(t)| = \f{2}{\pi \al} |L_{12}(\Om)(0)| < C \al^{-1} || \Om||_{\cH^3} \leq C \al^{-1} E \leq K_9 \al, \quad
 | c_l(t)  | = |\f{1-\al}{\al} \f{2}{\pi\al} L_{12}(\Om)(0)| < C \al^{-2} E \leq  K_9, \\
 &|| \Om ||_{L^{\infty}} + || \eta ||_{L^{\infty}}  
 < C E \leq C \al^2 \leq K_9 \al \min( ||\bar{\Om}||_{L^{\infty}}, || \bar{\eta}||_{L^{\infty}}), 
\quad || \xi||_{L^{\infty}} < C \al^{-1/2} E  \leq K_9 \al^{3/2} ,
 \eal
\]
where we have used $|| \bar{\Om} ||_{L^{\infty}}, || \bar{\eta} ||_{L^{\infty}} \geq C \al$ according to \eqref{eq:profile} and Lemma \ref{lem:one} in the last inequality, and $K_9 > 0$ is some absolute constant.
We further take 
\beq\label{eq:al0}
\al_0 = \min(\al_1, \   \f{3\pi}{4K_*}, \  \f{K_*^2}{4 K^2_{10}} , \ \f{1}{ 16 (K_9 + 1)^4}),
\eeq
where $K_{10}$ is the constant defined in Lemma \ref{lem:small}. For $\al < \al_0$, using $\bar{c}_{\om} = -1, \bar{c}_l = \f{1}{\al} + 3$ and the formula of $\bar{\Om}, \bar{\eta}$ in \eqref{eq:profile}, we further yield
 \beq\label{eq:boot43}
 \bal
& - \f{3}{2} < c_{\om}  + \bar{c}_{\om} < -\f{1}{2}, \quad c_l + \bar{c}_l > \f{1}{2\al} +3, \\
&|| \Om + \bar{\Om} ||_{L^{\infty}} \asymp  || \bar{\Om}||_{L^{\infty}} \asymp \al, \quad 
|| \eta + \bar{\eta} ||_{L^{\infty}} \asymp  || \bar{\eta}||_{L^{\infty}} \asymp \al,  \quad  || \xi ||_{L^{\infty}} \leq \f{1}{2} \al^{3/2}. 
\eal
\eeq

\subsubsection{Finite time blowup}\label{subsec:blowup}

For H\"older initial data, the local well-posedness 
of the solutions follows from the argument in \cite{chae_kim_nam_1999} for the 2D Boussinesq equations. 
The Beale-Kato-Majda type blowup criterion still applies to the Boussinesq equations in the specified domain. The time integral of $|| \nabla \th ||_{L^{\infty}}$ controls the breakdown of the solutions in the 2D Boussinesq equations \cite{chae_kim_nam_1999}. We will control this quantity and show that there exists $T_0$ such that $\int_0^{T_0}  || \nabla \th(\cdot, s)||_{\infty} ds  = \infty$ 
in the 2D Boussinesq equations. The solutions remain 
in the same regularity class as that of the initial data before the blowup time. In particular, the velocity field is in $C^{1,\alpha}$ before the blowup time.

Let $\chi(\cdot) : [0, \infty) \to [0, 1]$ be a smooth cutoff function, such that $\chi(R) = 1$ for $R \leq 1$ and $\chi(R) = 0$ for $R \geq 2$. We choose perturbation $\Om   =  ( \chi( R  / \lam) - 1) \bar{\Om}, \th(R, \b) = ( \chi( R  / \lam)  - 1)  \bar{\th}$ and $\eta = \th_x$, $\xi = \th_y$ can be obtained accordingly, where $\bar{\th}(x, y)$ is recovered from $\bar{\th}_x$ by integration \eqref{eq:th}. Obviously, $\Om, \eta, \xi \equiv 0$ for $R \leq \lam$. Using Lemma \ref{lem:small} for $\Om, \eta, \xi$ and $\al < \al_0$ (see \eqref{eq:al0}),  we obtain that these initial perturbations satisfy $E(\Om(0), \eta(0),\xi(0)) < 2 K_{10}\al^{5/2} \leq K_* \al^2$
for sufficiently large $\lam$. We remark that the initial perturbation is of size $C \al^{5/2}$ even for extremely large $\lam$ because $\bar{\xi}$ does not decay in the $\cC^1$ norm for large $R$. It is important to add a small weight $\al$ in $E(\xi, \infty)$ when we define the final energy in \eqref{eg}.

In particular, the initial data $ \bar{\Om} + \Om = \chi(R/\lam) \bar{\Om}$ (recall $\Om(R, \b) = \om(x, y)$), $\bar{\th} + \th
= \chi(R/\lam) \bar{\th}$ have compact support and thus we have finite energy $|| u + \bar{u}||_{L^2} < +\infty, || \th + \bar{\th}||_{L^2} < +\infty$. $c_{\om}(t), c_{l}(t)$ are determined by \eqref{eq:normal}.

Denote by $\om_{phy}, \th_{phy}$ the corresponding solutions in the original Boussinesq equation \eqref{eq:bous1}-\eqref{eq:bous2}, which are related to the rescaled variables $\om, \th$ via the rescaling formula \eqref{eq:rescal1}, \eqref{eq:rescal2}
\beq\label{eq:rescal_BKM}
\bal
& \om_{phy}(   x, t(\tau) ) = C_{\om}(\tau)^{-1} (\om + \bar{\om})(  C_l(\tau)^{-1} x, \tau), \quad \th_{phy}(   x, t(\tau) ) = C_{\th}(\tau)^{-1}  (\th+ \bar{\th})( C_l(\tau)^{-1} x, \tau), \\
& C_{\om}(\tau) = \exp\lt( \int_0^{\tau} c_{\om}(s) + \bar{c}_{\om}  ds \rt), \quad 
C_l(\tau) = \exp\lt( - \int_0^{\tau} c_{l}(s) + \bar{c}_{l}  ds \rt), \quad 
t(\tau) = \int_0^{\tau} C_{\om}(\tau) d \tau .
\eal
\eeq

We remark that the scaling parameters in \eqref{eq:rescal2} become $(c_{\om} + \bar{c}_{\om}, c_{l} + \bar{c}_l)$. Denote 
\[
M(\tau) \teq \int_0^{ t(\tau)} || \na  \th_{phy} (s) ||_{L^{\infty}} ds.
\]


Using a change of variable $s = t(p)$ and $\pa_x (\th + \bar{\th})  = (\eta + \bar{\eta}), \pa_y(\th + \bar{\th}) = (\xi + \bar{\xi})$, we obtain 
\[
M(\tau)  = \int_0^{ \tau} || \na  \th_{phy} ( t(p)) ||_{L^{\infty}}  C_{\om}(p) d p
= \int_0^{\tau} C_{\om}(p)^{-1}( || (\eta + \bar{\eta})(p) ||_{L^{\infty}} + || (\xi + \bar{\xi})(p) ||_{L^{\infty}} ) dp, 
\]
where we have used the formula \eqref{eq:rescal_BKM} and $C_{\th}^{-1}(p) C_l^{-1}(p) = C_{\om}(p)^{-2}$ according to \eqref{eq:rescal2},\eqref{eq:rescal3} in the second equality. Using the bootstrap estimates \eqref{eq:boot43} and Lemma \ref{lem:xi} about $\bar{\xi}$, we obtain 
\[
M(\tau) \asymp  \al \int_0^{\tau} C_{\om}(p)^{-1} dp.
\]

Using \eqref{eq:boot43} and \eqref{eq:rescal_BKM}, we have $e^{-3p /2  } < C_{\om}(p) < e^{-p/2}$.
Therefore, we obtain $M(\tau)  < +\infty  \quad \forall \tau < + \infty$ and
\[
\int_0^\infty M(\tau) d \tau  \geq C \al \int_0^\infty \int_0^{\tau} e^{p/2} dp d \tau = \infty , \quad t(\infty) \leq \int_0^{\infty} e^{- p/2} dp < + \infty.
\]

Denote $T^* = t(\infty)$.
Applying the BKM type blowup criterion in \cite{chae_kim_nam_1999}, we obtain that the solutions remain in the same regularity class as that of
the initial data before $T^* $ and develop a finite time singularity at $T^*$. Similarly, by rescaling the time variable, we prove that $|| \om_{phy}||_{L^{\infty}}$ and $ || \na \th_{phy}||_{L^{\infty}}$ blowup at $T^*$.


\begin{remark}
The crucial nonlinear estimate \eqref{eq:boot4} and {\it a priori} estimate \eqref{eq:boot42}, i.e. the bootstrap estimate for small perturbation, offer strong control on the perturbation and the exact solution before the blowup time. In particular, it allows us to truncate the far field of the approximate steady state, which leads to a small perturbation only, to obtain initial data with finite energy.   
\end{remark}

\subsubsection{Convergence to the self-similar solution}\label{sec:converge}
Taking the time derivative of \eqref{eq:lin}, using the {\it a priori} estimate \eqref{eq:boot42} for the small perturbation and analysis similar to that in the previous Section, we can further perform $\cH^2$ estimates on $\Om_t, \eta_t$, $\cH^2(\psi)$ and $L^{\infty}$ estimates on $\xi_t$.
In particular, following the argument in our previous joint work with Huang \cite{chen2019finite}, 
we can further obtain that there exists an exact self-similar solution $\Om_{\infty}, \eta_{\infty} \in \cH^3, \xi_{\infty} \in \cH^3(\psi) \cap L^{\infty}$, such that the solution of the dynamic rescaling equation with initial data constructed in Subsection \ref{subsec:blowup} converges to $(\Om_{\infty}, \eta_{\infty}, \xi_{\infty})$ exponentially fast. The convergence is in the $\cH^2$ norm for the variables $\Om, \eta$ and both $\cH^2(\psi)$ and $L^{\infty}$ norm for the variable $\xi$.

Using the {\it a-priori} estimate \eqref{eq:boot42} and Lemma \ref{lem:xi}, we have $ || \bar{\xi} + \xi(t) ||_{\cC^1} \leq C\al^{3/2} $ for all time in the dynamic rescaling equation. Using Lemma \ref{lem:h_embed}, we know that the space $\cC^1$ (the weighted $C^1$ space) can be embedded continuously into the standard H\"older space $C^{ \al / 40}$.
 Therefore, the $\cC^1$ estimate of $\bar{\xi} + \xi$ implies that $\bar{\xi} + \xi(t) \in C^{ \alpha /40}$ with uniform H\"older norm. Since $\bar{\xi} + \xi(t)$ converges to $\xi_{\infty}$ in $L^{\infty}$, we have $\xi_{\infty} \in C^{ \alpha /40}$. Finally, using the same argument, the fact that $\Om_{\infty}, \eta_{\infty} \in \cH^3$ and the embedding $\cH^3 \hookrightarrow \cC^1$ in Lemma \ref{lem:inf}, we conclude $\Om_{\infty}, \eta_{\infty}, \xi_{\infty} \in C^{ \al / 40}$. 

Notice that $c_l + \bar{c}_l > \f{1}{2\al}$ from \eqref{eq:boot43}. Thus, the self-similar blowup is focusing. This completes the proof of Theorem \ref{thm:bous}.




\section{Finite time blowup of 3D axisymmetric Euler equations with solid boundary}\label{sec:euler}

In this section, we prove Theorem \ref{thm:euler}. We first review the setup of the problem. In Section \ref{sec:dyn}, we reformulate the 3D Euler equations and discuss the connection between the 3D Euler and 2D Boussinesq; see e.g. \cite{majda2002vorticity}.
In Section \ref{sec:3Delli}, we establish the elliptic estimates. In Section \ref{sec:3Dnon}, we will construct initial data and control the support of the solution under some bootstrap assumptions. With these estimates, the rest of the proof follows essentially the nonlinear stability analysis of the 2D Boussinesq equations and is sketched in the same subsection.

\vspace{0.1in}
\paragraph{\bf{Notations}} 
In this Section, we use $x_1, x_2,  x_3$ to denote the Cartesian coordinates in $\R^3$, and 
\beq\label{eq:cylinder}
r= \sqrt{x_1^2 + x_2^2}, \quad z = x_3, \quad \vartheta = \arctan(x_2 /x_1) 
\eeq
 to denote the cylindrical coordinates. The reader should not confuse $r$ with the radial variable in the 2D Boussinesq.

Let $\uu$ be the axi-symmetric velocity and $\mathbf{\om} = \na \times \uu$ be the vorticity vector. In the cylindrical coordinates, we have the following representation
\[
\uu(r, z) = u^r(r, z) \ee_r + u^{\th}(r, z ) \ee_{\th}  + u^z(r, z ) \ee_z, 
\quad 
\mathbf{\om} = \om^r(r, z) \ee_r + \om^{\th}(r, z ) \ee_{\th}  + \om^z(r, z ) \ee_z,
\]
where $\ee_r, \ee_{\th}$ and  $\ee_z$ are the standard orthonormal vectors defining the cylindrical coordinates, 
\[
\ee_r = ( \f{x_1}{r}, \f{x_2}{r}, 0 )^T , \quad \ee_{\th} = ( \f{x_2}{r},  - \f{x_1}{r}, 0  )^T, \quad \ee_z = (0, 0, 1)^T, 
\]
and $r = \sqrt{x_1^2 + x_2^2}$ and $z = x_3$. 

We study the 3D axisymmetric Euler equations in a cylinder $D = \{ (r,z) : r \in [0,1], z \in \BT \}, \BT = \R / ( 2 \BZ)$ that is periodic in $z$. The equations are given below:
\beq\label{eq:euler1}
\pa_t (ru^{\th}) + u^r (r u^{\th})_r + u^z (r u^{\th})_z = 0, \quad 
\pa_t \f{\om^{\th}}{r} + u^r ( \f{\om^{\th}}{r} )_r + u^z ( \f{\om^{\th}}{r})_z = \f{1}{r^4} \pa_z( (r u^{\th})^2 ).
\eeq
The radial and axial components of the velocity can be recovered from the Biot-Savart law
\beq\label{eq:euler2}
-(\pa_{rr} + \f{1}{r} \pa_{r} +\pa_{zz}) \td{\psi} + \f{1}{r^2} \td{\psi} = \om^{\th}, 
 \quad  u^r = -\td{\psi}_z, \quad u^z = \td{\psi}_r + \f{1}{r} \td{\psi}  
\eeq
with a no-flow boundary condition on the solid boundary $r = 1$
\beq\label{eq:euler21}
\td{\psi}(1, z ) = 0
\eeq
and a periodic boundary condition in $z$. 

We consider solution $\om^{\th}$ with odd symmetry in $z$, which is preserved by the equations dynamically. Then $\td \psi$ is also odd in $z$. Moreover, since $\td \psi$ is 2-periodic in $z$, we obtain 
\beq\label{eq:euler22}
\td \psi(r, 2k-1)  = 0 . \quad \textrm{for all \ } k \in \BZ
\eeq
This setup of the problem is essentially the same as that in \cite{luo2013potentially-1,luo2013potentially-2}. 

Equation \eqref{eq:euler2} is equivalent to $-\D( \td \psi \sin(\vth)) = \om \sin(\vth)$, where $\vth = \arctan(x_2 / x_1)$ 
and $\D$ is the Laplace operator in $\R^3$. We further assume that $\om^{\th} \in C^{\al}(D)$ with support away from $r=0$. It follows $\om^{\th} \sin(\vth) \in C^{\al}(D)$. Note that the cylinder $D_{k,l} \teq\{ (r,z) : r \in [0,1], 2k-1 \leq z \leq 2l-1 \} $ satisfies the exterior sphere condition. 
Under the boundary condition \eqref{eq:euler21}-\eqref{eq:euler22}, using Theorems 4.3, 4.6 in \cite{gilbarg2015elliptic} 
we obtain a unique solution $\td \psi \sin \vth \in C^{2,\al}(D_{k,l}) \cap C( \bar D_{k,l})$ for any $k <l, k, l \in \BZ$. This further implies the existence and the uniqueness of solution of \eqref{eq:euler2}-\eqref{eq:euler22}. 

Due to the periodicity in $z$ direction, it suffices to consider the equations in the first period $D_1 =\{ (r,z) : r \in [0,1], |z| \leq 1 \} $. 
We have the following pointwise estimate on $\td \psi$, which will be used to estimate $\td \psi$ away from the $\supp(\om^{\th})$ in Section \ref{sec:3Delli}.

\begin{lem}\label{lem:biot1}
Let $\td{\psi}$ be a solution of \eqref{eq:euler2}-\eqref{eq:euler21}, and $\om^{\th} \in C^{\al}(D_1)$ for some $\al>0$ be odd in $z$ with $\supp(\om^{\th}) \cap D_1\subset \{ (r ,z) : (r-1)^2 + z^2 < 1/4  \}$. 
For $  \f{1}{4} < r \leq 1, |z|\leq 1$, we have 
\[
| \td{\psi}(r, z) |  \les \int_{D_1} | \om^{\th}(r_1, z_1) | \B( 1 + |\log( (r-r_1)^2 + (z- z_1)^2) |  \B) r_1 d r_1 d z_1.
\] 
\end{lem}

If the domain of the equation \eqref{eq:euler2} is $\R^3$, the estimate is straightforward by using the Green function. For the domain we consider, the Green function would be complicated. The proof is based on comparing $\td{\psi}\sin(\vth)$ with the solutions of $- \D (\psi_{\pm} \sin(\vth)) = f_{\pm}(r,z) \sin(\vth)$ in $\R^3$, where $f_{\pm}$ are some functions related to $\om^{\th}$. 
We defer the proof to Appendix \ref{app:stream}.

If the initial data $u^{\th}$ of \eqref{eq:euler1}-\eqref{eq:euler21} is non-negative, 
$u^{\th}$ remains non-negative before the blowup, if it exists.
Then, $u^{\th}$ can be uniquely determined by $(u^{\th} )^2$.
We introduce the following variables 
\beq\label{eq:omth}
\td{\th} \teq (r u^{\th})^2,  \quad \td{\om} = \om^{\th} / r.
\eeq

We reformulate \eqref{eq:euler1}-\eqref{eq:euler21} as 
 \beq\label{eq:euler31}
\bal
\pa_t \td{\th} + u^r \td{\th}_{ r} + u^z \td{\th}_{z} &= 0,  \quad  \pa_t \td{\om} + u^r \td{\om}_r + u^z \td{\om}_{ z} = \f{1}{r^4} \td{\th}_z  ,  \\
-( \pa^2_{r} +\f{1}{r} \pa_r + \pa_z^2 -\f{1}{r^2} )  \td{\psi} & = r \td{\om} , \quad \td{\psi}(1, z) = 0 ,  \quad u^r = -  \td{\psi}_{ z} , \quad 
u^z = \f{1}{r} \td{\psi }+  \td{\psi}_{ r} .
\eal
\eeq

\subsection{Dynamic rescaling formulation}\label{sec:dyn}
We introduce new coordinates $(x, y)$ centered at $r = 1, z= 0$ and its related polar coordinates
\beq\label{eq:euler_polar}
x =  C_l(\tau)^{-1} z, \quad  y = (1-r) C_l(\tau)^{-1}, \quad  \rho = \sqrt{x^2 + y^2}, \quad \b = \arctan(y/x) , \quad R = \rho^{\al},
\eeq
where $C_l(\tau)$ is defined below \eqref{eq:rescal42}. The reader should not confuse $\rho$
with the notations for the weights, and the relation $R = \rho^{\al}$ with $R = r^{\al}$ in the 2D Boussinesq. By definition, we have 
\beq\label{eq:label_rs}
z = C_l(\tau)x, \quad r = 1 - C_l(\tau) y = 1 - C_l(\tau) \rho \sin(\b).
\eeq

We consider the following dynamic rescaling formulation centered at $r = 1, z= 0$
\beq\label{eq:rescal41}
\bal
\th(x, y, \tau) &= C_{\th}(\tau) \td{\th}( 1 - C_l(\tau) y,   C_l(\tau) x , t(\tau) ), \\
  \om(x, y, \tau) &=  C_{\om}(\tau) \td{\om}( 1 - C_l(\tau) y,  C_l(\tau) x  , t(\tau)) , \\
 \psi(x, y, \tau)  & =  C_{\om}(\tau) C_l(\tau)^{-2} \td{\psi} (1 - C_l(\tau) y,  C_l(\tau) x, t(\tau)),
\eal
\eeq
where $C_l(\tau), C_{\th}(\tau), C_{\om}(\tau), t(\tau)$ are given by $C_{\th}  = C^{-1}_l(0) C^2_{\om}(0) \exp\lt( \int_0^{\tau} c_{\th} (s)  d \tau\rt)$, 
\beq\label{eq:rescal42}
\bal
  C_{\om}(\tau) = C_{\om}(0) \exp \lt( \int_0^{\tau} c_{\om} (s)  d \tau \rt), \ C_l(\tau) =C_l(0) \exp\lt( \int_0^{\tau} -c_l(s) ds \rt) , \   t(\tau) = \int_0^{\tau} C_{\om}(\tau) d\tau ,
\eal
\eeq
and the rescaling parameter $c_l(\tau), c_{\th}(\tau), c_{\om}(\tau)$ satisfies $c_{\th}(\tau) = c_l(\tau ) + 2 c_{\om}(\tau)$. We remark that $C_{\th}(\tau)$ is determined by $C_l, C_{\om}$ via $C_{\th} = C^2_{\om} C_l^{-1}$. We have this relation due to the same reason as that of \eqref{eq:rescal3}. 
We choose $(r,z)=(1,0)$ as the center of the above transform since the singular solution is concentrated near this point.
We have $0\leq y \leq C_l^{-1}, |x| \leq C_l^{-1}$ since $r \in [0, 1], |z| \leq 1$. 
We have a minus sign for $\pa_y$
\[
\pa_y \th = -C_{\th} C_l(\tau) \td{\th}_r , \quad  \pa_y \om = - C_{\om} C_l(\tau) \td{\om}_r,  \quad \pa_y \psi = - C_{\om} C_l(\tau)^{-1} \td{\psi}_r.
\]

Let $(\td{\th}, \td{\om})$ be a solutions of \eqref{eq:euler31}. It is easy to show that $\om, \th$ satisfy
\[
\th_t +  c_l \xx \cdot \na \th + (-u^r) \th_y +  u^z  \th_x  = c_{\th} \th , \quad \om_t + c_l \xx \cdot \na \om  + (-u^r)  \om_y +  u^z \om_x = c_{\om } \om + \f{1}{r^4} \th_x .
\]
The Biot-Savart law in \eqref{eq:euler31} depends on the rescaling parameter $C_l, \tau$
\[
-(\pa_{xx}  + \pa_{yy}) \psi + \f{1}{r} C_l \pa_y \psi  + \f{1}{r^2 } C^2_l \psi = r \om ,
\quad   u^r(r, x) = -  \psi_x,  \quad u^z(r, x) =  \f{1}{r}  C_l(\tau)  \psi -  \psi_y,  
\]
where $r = 1 - C_l(\tau) y$ \eqref{eq:label_rs}. We introduce $u = u^z, v = - u^r$. Then, we can further simplify
\beq\label{eq:euler4}
\bal
& \th_t + (c_l \xx + \uu \cdot \na ) \th = c_{\th} \th , \quad  \om_t  + ( c_l \xx + \uu \cdot \na ) \om = \th_x + \f{1 - r^4}{r^4} \th_x, \\
& -(\pa_{xx}  + \pa_{yy}) \psi + \f{1}{r} C_l \pa_y \psi + \f{1}{r^2} C_l^2 \psi= r \om , \quad 
u(x, y) =    - \psi_y + \f{1}{r} C_l \psi, \quad v  = \psi_x ,
\eal
\eeq
with boundary condition $\psi(x, 0 ) \equiv 0$. 
If $C_l$ is extremely small, we expect that the above equations are essentially the same as the dynamic rescaling formulation \eqref{eq:bousdy1} of the Boussinesq equations. 
We look for solutions of \eqref{eq:euler4} with the following symmetry 
\[
\om(x, y)= -\om(-x, y), \quad \th(x, y) = \th(-x, y). 
\]
Obviously, the equations preserve these symmetries and thus it suffices to solve \eqref{eq:euler4} on $x, y \geq 0$ with boundary condition  $\psi(x, 0) = \psi(y, 0) = 0$ for the elliptic equation.

\subsection{The elliptic estimates} \label{sec:3Delli} 

In this Section, we use the ideas in Section \ref{sec:idea_biot} to estimate the time-dependent elliptic equation in \eqref{eq:euler4}. We first estimate $\psi$ away from $\supp(\om)$. 
In Section \ref{sec:local_outline}, we outline the estimates. In the remaining subsections, we  localize the elliptic equation and establish the $\cH^3$ elliptic estimates.

Under the polar coordinates \eqref{eq:euler_polar} $\rho  = \sqrt{x^2 + y^2}, \b = \arctan( y/x)$, we reformulate \eqref{eq:euler4} as
\beq\label{eq:nota_om2}
 - \pa_{ \rho \rho } \psi - \f{1}{\rho} \pa_{\rho}  \psi - \f{1}{\rho^2} \pa_{\b \b } \psi 
+ \f{C_l}{r} \sin(\b) \pa_{\rho}  \psi + \f{ C_l}{r} \f{\cos(\b)}{\rho} \pa_{\b} \psi  + \f{C_l^2}{r^2} \psi =r \om.
\eeq

Recall $R = \rho^{\al}$ from \eqref{eq:euler_polar}. Denote
\beq\label{eq:nota_om3}
\Psi(R,\b) = \f{1}{\rho^2} \psi(\rho, \b), \quad \Om(R, \b) = \om(\rho, \b), \quad \eta(R, \b) = (\th_x)(\rho, \b),
\quad \xi(R, \b) = (\th_y)(\rho, \b).
\eeq

Since we rescale the cylinder $D_1 = \{ (r, z) : r \leq 1, |z|\leq 1  \}$, the domain for $(x, y)$ is 
\beq\label{eq:rescale_D}
\td D_1 \teq \{ (x, y) :  |x| \leq C_l^{-1}, y \in [0, C_l^{-1}] \}.
\eeq

We focus on the sector $\rho \leq C_l^{-1}$, or equivalently $R \leq C_l^{-\al}$, and $ \b \in [0, \pi/2]$ due to the symmetry of the solutions.
Notice that $\rho \pa_{\rho} = \al R \pa_R= \al D_R$. It is easy to verify that \eqref{eq:nota_om2} is equivalent to 
\beq\label{eq:elli_euler}
\bal
&- \al^2 R^2 \pa_{RR} \Psi- \al(4+\al) R \pa_R \Psi - \pa_{\b \b} \Psi - 4 \Psi \\
&+ \f{ C_l \rho}{r}( \sin(\b) (2  + \al D_R) \Psi + \cos(\b)  \pa_{\b} \Psi ) 
+ \f{C_l^2 \rho^2}{r^2} \Psi = r \Om. 
\eal
\eeq
We keep the notation $\rho = R^{1/\al}, r = 1 - C_l \rho \sin(\b)$ to simplify the formulation. The boundary condition of $\Psi$ is given by (in the sector $R \leq C_l^{-\al}$)
\beq\label{eq:elli_euler2}
\Psi(R,0) = \Psi(R, \pi/2) =0. 
\eeq

\begin{definition}\label{def:supp}
We define the size of support of $(\th, \om)$ of \eqref{eq:euler4} 
\[
S(\tau) = \mathrm{ess}\inf  \{\rho :  \th(x, y,\tau) =0, \om(x, y, \tau ) = 0 \textrm{ for } x^2 + y^2 \geq \rho^2 \} .
\]
\end{definition}
Obviously, the support of $\Om, \eta$ defined in \eqref{eq:nota_om3} is $S(\tau)^{\al}$. After rescaling the spatial variable, the support of $(\td{\th}, \td{\om})$ of \eqref{eq:euler31} satisfies 
\[
\mathrm{supp}  \ \td{\th}(t(\tau)),  \ \mathrm{supp}  \ \td{\om}(t(\tau)) \subset \{ (r, z) :  ( (r-1)^2 + z^2)^{1/2} \leq  C_l(\tau) S(\tau)  \}.
\]
We will construct initial data of \eqref{eq:euler4} with compact support $S(0) < + \infty$ and use the idea described in Section \ref{sec:idea_supp} to prove that $C_l(\tau) S(\tau)$ remains sufficiently small for all $\tau >0$.

\begin{remark}\label{rem:order1}
There are several small parameters $\al, C_l(\tau), C_l(\tau) S(\tau)$ in the following estimates. We will choose $\al$ to be small. For most estimates, the constants are independent of $C_l(\tau)$. We will choose $C_l(0)$ to be much smaller than $\al$ at the final step. This allows us to prove that $C_l(\tau), C_l(\tau) S(\tau) ,  ( C_l(\tau) S(\tau))^{\al} $ are very small. One can regard $C_l(\tau) \approx 0$. Recall the relation \eqref{eq:label_rs} about $r$.
In the support of the solution, we have $r = 1 - C_l \rho \sin(\b) \approx 1$. We treat the error terms in these approximations as small perturbations.
\end{remark}

Recall the $L^2$ inner product defined in \eqref{eq:inner_L2}. Using the estimate in Lemma \ref{lem:biot1}, we obtain in the following Lemma that the $L^2$ norm of $\Psi$ away from the support of the solution is small. It will be used later to localize\eqref{eq:elli_euler}.

\begin{lem}\label{lem:far}
Suppose that the assumptions in Lemma \ref{lem:biot1} hold true. Let $S(\tau)$ be the support size of $\om(\tau), \th(\tau)$. 
Assume $C_l(\tau) S(\tau) < \f{1}{4}$. For any $M > (2 S(\tau) )^{\al}$, we have 
\[
|| \Psi  \one_{M \leq R \leq (2C_l)^{-\al}} ||_{L^2}  \les C(M) \cdot || \Om||_{L^2}, \quad
C(M) \teq  (1 + | \log( C_l M^{ 1 / \al}) | )  S M^{-1/\al}  || \Om||_{L^2}.
\]
\end{lem}

The proof follows from the estimate in Lemma \ref{lem:biot1}. We defer it to Appendix \ref{app:stream}. 
We will choose $M$ so that  $C(M)$ is small, e.g. $C(M) \les 1$ or $C(M) \les 3^{-1/\al}$.
If we use an estimate similar to Proposition \ref{prop:key} and then restrict it to $M \leq R \leq (2C_l)^{-\al}$, the constant in the upper bound is $\al^{-1}$, which is not sufficient for our purpose. 

\begin{remark}
We restrict the domain of the integral $D_I$ to $R \leq (2C_l)^{-\al}$, which is equivalent to $\rho\leq (2C_l)^{-1} $ due to \eqref{eq:euler_polar}, so that $D_I$ is in $\td D_1$ \eqref{eq:rescale_D}. We impose $R \geq M > (2S(\tau))^{\al}$ so that $D_I$ is away from the support of the solution. Since $S(\tau), C_l(\tau)$ are the variables defined in $(x,y)$ coordinates, when we pass to $(R,\b)$ coordinates, we have a $\al$ power for these variables, e.g. $( S(\tau) )^{\al}, ( C_l(\tau) )^{\al}$.
\end{remark}

\subsubsection{Outline of the estimates }\label{sec:local_outline}
In Section \ref{sec:local_derive}, we use \eqref{eq:elli_euler} to derive the elliptic equation 
\eqref{eq:elli2} for $\chi \Psi$ with some cutoff function $\chi$. The equation is similar to \eqref{eq:elli} in the 2D Boussinesq and has an extra error term $Z_{\chi}$.
We first establish the $L^2$ estimate of $\chi_1 \Psi$ in the same Section \ref{sec:local_derive}. To estimate the terms involving derivatives of $\chi$, e.g. $D_R^2 \chi \Psi$, we use Lemma \ref{lem:far}. The $L^2$ estimate enables us to estimate the error term $Z_{\chi}$. The advantage of localizing \eqref{eq:elli_euler} is that $\chi \Psi$ can be treated as a solution of the elliptic equation \eqref{eq:elli} in $\R_2^+$. Then, in Section \ref{sec:local_H3}, we apply the $\cH^k$ version of the key elliptic estimate in Proposition \ref{prop:key} recursively to $\chi_i \Psi$ with $\chi_i$ that has smaller support, and establish the higher order elliptic estimates.

\subsubsection{Localizing the elliptic equation}\label{sec:local_derive}
We will take advantage of the fact that $C_l(\tau) S(\tau)$ can be extremely small and localize the elliptic equation. Firstly, we assume that $C_l(\tau) S(\tau) < \f{1}{4}$. Recall the relation \eqref{eq:label_rs} about $r$. Then we have $r = 1 - C_l \rho \sin(\b) \geq \f{3}{4}, r^{-1} \les 1$.

Let $\chi_1(\cdot) : [0, \infty) \to [0, 1]$ be a smooth cutoff function, such that $\chi_1(R) = 1$ for $R \leq 1$, $\chi_1(R) = 0$ for $R \geq 2$ and $(D_R \chi_1)^2  \les \chi_1$. 
This assumption can be satisfied if $\chi_1 = \chi_0^2$ where $\chi_0$ is another smooth cutoff function. Denote $\chi_{\lam}(R) = \chi_1(R/ \lam)$. It is easy to verify that 
 \beq\label{eq:chi2}
(D_R \chi_{\lam})^2 = (R / \lam  \pa_R\chi_1(R /\lam) )^2 \les \chi_1(R /\lam) = \chi_{\lam}(R),\quad
|D_R^k \chi_{\lam}| \les \one_{\lam \leq R \leq 2 \lam} ,
 \eeq
for $k\leq 5$, where we have used the property $|D^2_R\chi_1| \les \chi_1$ in the first inequality. Denote 
\[
 \Psi_{\chi} = \Psi \chi_{ \lam } , \quad \Om_{\chi} = \Om  \chi_{\lam}.
\]
At this moment, we just simplify $\chi_{\lam}$ as $\chi$. Observe that $R^2 \pa_{RR} + R \pa_R = D_R^2$ and
\beq\label{eq:com_chi}
\bal
& r \Om_{\chi} = (1 - C_l \rho \sin(\b)) \Om_{\chi} = \Om_{\chi} -  C_l \rho \sin(\b) \Om_{\chi}, \quad  \al D_R (\chi \Psi) = \al D_R \chi \Psi +  \al \chi D_R \Psi,   \\
& \al^2 D_R^2( \chi \Psi ) =  \al^2 \chi D_R^2 \Psi + 2 \al D_R \chi \cdot \al D_R \Psi + \al^2 D_R^2 \chi \Psi . \\ 
\eal
\eeq
Multiplying $\chi$ on both sides of \eqref{eq:elli_euler}, and using \eqref{eq:com_chi} and 
a direct calculation yield 
\beq\label{eq:elli2}
\bal
&-\al^2 D_R^2 \Psi_{\chi} 
- 4 \al D_R \Psi_{\chi} - \pa_{\b \b} \Psi_{\chi} - 4 \Psi_{\chi}
= \Om_{\chi} + Z_{\chi}  ,  \quad Z_{\chi} = Z_1 + Z_2 + Z_3, 
\eal
\eeq
with boundary condition \eqref{eq:elli_euler2}, where 
$Z_1, Z_2$ and $Z_3$ are given below
\beq\label{eq:Zchi}
\bal
& Z_1 =  -\f{ C_l  \rho }{r}  (\sin(\b) (2 \Psi_{\chi} + \al D_R \Psi_{\chi}  ) + \cos(\b) \pa_{\b} \Psi_{\chi}) -\f{C_l^2 \rho^2}{r^2} \Psi_{\chi} ,\\
& Z_2 =\f{ C_l \sin(\b) \rho }{r} \al D_R \chi \Psi -( \al^2 D_R^2 \chi + 4\al D_R \chi )\Psi -2  \al^2  D_R \chi D_R \Psi, \quad  Z_3 = -C_l \rho \sin(\b) \Om_{\chi}.
\eal
\eeq

Recall that $R = \rho^{\al}, r = 1 - C_l y = 1 - C_l \rho \sin(\b)$ from \eqref{eq:euler_polar}, \eqref{eq:label_rs} and $L_{12}(f)(0)$ from \eqref{eq:biot3}. Next, we derive $L_{12}(Z_{\chi_{\lam}})(0)$. It will be used in Section \ref{sec:local_H3} when we apply Proposition \ref{prop:key}.


Firstly, for sufficiently smooth $\Om, \Psi$ with $\Om$ vanishing at least linear near $R=0$, we show that $L_{12}(Z_{\chi_{\lam}})(0)$ is independent of the cutoff radial $\lam$ for $\lam \geq ( S(\tau) )^{\al}$. From $\lam \geq ( S(\tau))^{\al}$, we have $\Om =\Om \cdot \chi_{\lam}= \Om_{\chi_{\lam}}$. For any $\e>0$, using integration by parts, we get
\[
\bal
\la   \pa_{\b \b} \Psi_{\chi} + 4 \Psi_{\chi}, \  \sin(2\b) R^{-1}  \one_{R \geq \e} \ra
&= \la -4 \Psi_{\chi} + 4 \Psi_{\chi},  \sin(2\b) R^{-1}  \one_{R \geq \e} \ra = 0,  \\
\la \al^2 D_R^2 \Psi_{\chi} + 4 \al D_R \Psi_{\chi} , \ \sin(2\b) R^{-1} \ra
&= \la \al^2 \pa_R (D_R \Psi_{\chi}) + 4 \al \pa_R \Psi_{\chi}, \sin(2\b)  \ra  \\
&= -4\al \int_0^{\pi/2} \Psi(0, \b) \sin(2\b) d \b.
\eal
\]
Note that $\Psi$ may not vanish at $R=0$. Since $\rho = R^{1/\al}$ vanishes at $R=0$, it is easy to see that $Z_{\chi}$ vanishes at $R=0$. 
Therefore, integrating both sides of \eqref{eq:elli2} with $\sin(2\b)R^{-1}\one_{R\geq \e}$, and then using the above computations and taking $\e \to 0$, for $\lam \geq ( S(\tau)^{\al} )$, we have 
\beq\label{eq:l12Z0}
L_{12}(Z_{\chi_{\lam}})(0) =  - L_{12}(\Om)(0)  + 4\al \int_0^{\pi/2} \Psi(0, \b) \sin(2\b) d \b.
\eeq

Next, we perform $L^2$ estimate for $\Psi_{\chi}$. It will be used later to estimate 
$Z_{\chi}$ in \eqref{eq:elli2}.

\begin{lem}\label{lem:in}
There exists $\al_2>0$ such that if $\al < \al_2, C_l S < 4^{-1/\al -1}$, for $\lam = \f{1}{4} C_l^{-\al} $, the solution of 
\eqref{eq:elli2} satisfies
\[
  \al^2|| D_R \Psi_{\chi_{\lam}}||^2_{L^2}  
+  \al || \Psi_{\chi_{\lam}} ||^2_{L^2} 
+ \al || \pa_{\b} \Psi_{\chi_{\lam}} ||^2_{L^2} \les \al^{-1}  || \Om||_{L^2}^2.
\]
\end{lem}

Firstly, we have $\lam = \f{1}{4} C_l^{-\al} >  S^{\al}$ and $\Om_{\chi} =\Om \chi_{\lam} = \Om$.
We impose $C_l S < 4^{-1/\al -1}$ so that $\lam > (2S)^{\al}$ and $C(M) \les C \al^{-1}$ in Lemma \ref{lem:far} with $M = \lam$. 
At this step, this bound is good enough for us to treat $Z_{2}$ in \eqref{eq:elli2}-\eqref{eq:Zchi} as perturbation. In the following estimates, we treat the small factor $C_l$ in $Z_{\chi}$ approximately equal to zero and $r\approx 1$. See also Remark \ref{rem:order1}. 

\begin{proof} We simplify $\chi_{\lam}$ as $\chi$. Multiplying \eqref{eq:elli2} by $\Psi_{\chi}$ and using integration by parts, we get 
\beq\label{eq:elli22}
\bal
I & \teq \al^2 || R\pa_R \Psi_{\chi}||^2_{L^2} 
+ \f{4\al - \al^2}{2} || \Psi_{\chi}||^2_{L^2} + || \pa_{\b} \Psi_{\chi}||^2_{L^2} - 4 || \Psi_{\chi}||^2_{L^2} \\
&=  \la \Om, \Psi_{\chi} \ra  +  \la Z_1  +Z_3, \Psi_{\chi} \ra + \la Z_2 , \Psi_{\chi} \ra . \\
\eal
\eeq
Using the Fourier series expansion with basis $\{ \sin(2 n \b) \}_{n \geq 1}$, one can verify that
\[
|| \pa_{\b} \Psi_{\chi}||^2_{L^2} \geq 4 ||\Psi_{\chi}||_{L^2},
\]
which is sharp with equality when $\Psi_{\chi} = \sin(2\b)$. Therefore, multiplying the above inequality by $1 -\f{\al}{4}$ and then applying it to the left hand side of \eqref{eq:elli22} yields
\[
I \geq \al^2 || D_R \Psi_{\chi}||^2_{L^2} + \f{2\al -\al^2}{2} || \Psi_{\chi}||^2_{L^2}  
+ \f{\al}{4} || \pa_{\b} \Psi_{\chi}||^2_{L^2} \geq 
\al^2 || D_R \Psi_{\chi}||^2_{L^2} + \f{\al }{2} || \Psi_{\chi}||^2_{L^2}  
+ \f{\al}{4} || \pa_{\b} \Psi_{\chi}||^2_{L^2},
\]
where we have used $\al \leq 1$.

Within the support of $\chi= \chi_{\lam}$, we have $ R \leq 2 \lam$.  By assumption, we have $\lam = \f{1}{4} C_l^{-\al} > 4^{\al} S^{\al}$. It follows that
\beq\label{eq:Cllam0}
 C_l \rho  \one_{R \leq 2\lam} = C_l R^{ \f{1}{\al}}  \one_{R \leq 2\lam}\leq  C_l (2\lam)^{\f{1}{\al}} = 2^{-\f{1}{\al}} \les \al^2, \quad |\log (C_l \lam^{\f{1}{\al}} )|\les \al^{-1}, \quad  2 S \leq \lam^{1/\al}.
 \eeq

The $Z_1, Z_3$ terms \eqref{eq:Zchi} contain the small factor $C_l \rho$.
Since $r^{-1} \les 1$, we get
\[
\bal
&|| Z_1 ||_{L^2} \les \al^2 (  || \Psi_{\chi}||_{L^2} 
+ ||\al D_R \Psi_{\chi}||_{L^2} +|| \pa_{\b} \Psi_{\chi}||_{L^2} )
\les  \al^2 \al^{-1/2}  I^{1/2} \les \al^{3/2} I^{1/2} , \\
& || Z_3||_{L^2} \les \al^2 || \Om||_{L^2} \les || \Om||_{L^2}.
\eal
\]

We perform integration by parts for the last term $-2\al^2 D_R \chi D_R \Psi$ in $Z_2$ \eqref{eq:Zchi}
\[
-2\al^2 \la D_R \chi D_R \Psi, \Psi \chi \ra 
= - \al^2 \la R \chi D_R \chi, \pa_R \Psi^2 \ra
=  \al^2 \la (R \chi D_R \chi)_R, \Psi^2 \ra 
= \al^2 \la  (D_R \chi)^2  + \chi D^2_R \chi + \chi D_R \chi,  \Psi^2 \ra .
\]

Using the above identity, \eqref{eq:chi2} for $|D_R^k \chi|$ and \eqref{eq:Cllam0}, we obtain 
\[
|\la Z_2 , \Psi_{\chi} \ra|
\les ( \al^2 + \al) || \Psi \one_{\lam \leq R \leq 2 \lam} ||^2_{L^2} \les \al || \Psi
\one_{\lam \leq R \leq 2 \lam} ||^2_{L^2}  \les  \al || \Psi
\one_{\lam \leq R \leq (2C_l)^{-\al}} ||^2_{L^2}  ,
\]
where we have used $2 \lam < (2C_l)^{-\al}$ in the last inequality.
Since $(2S)^{\al}\leq \lam$ and $ S \lam^{-1/\al} \les 1$ (see \eqref{eq:Cllam0}), we apply Lemma \ref{lem:far} with $M = \lam$ and \eqref{eq:Cllam0} to get 
\[
\bal
|\la Z_2 , \Psi_{\chi} \ra|
&\les \al (1 + |\log( C_l \lam^{1/ \al})|)^2 (S  \lam^{-1/\al})^2 || \Om||^2_{L^2}
\les \al^{-1}|| \Om||^2_{L^2}.
\eal
\]

Plugging the estimates of $Z_1, Z_2, Z_3$ and $|| \Psi_{\chi}||_{L^2}\les \al^{-1/2} I^{1/2}$ into \eqref{eq:elli22} and then using the Cauchy-Schwarz inequality, 
we prove 
\[
I \leq C \al^{-1/2} I^{1/2}   || \Om||_{L^2} + C\al \cdot I +C \al^{-1} || \Om||_{L^2}^2.
\]
Now we choose 
\beq\label{eq:al2}
\al_2 =  \min( (2C)^{-1}, 4^{-1}). 
\eeq
Then for $\al < \al_2$, we have $C\al < \f{1}{2}$.
Solving the above inequality yields $I \les \al^{-1} || \Om||^2_{L^2}$. 
\end{proof}

\subsubsection{Localized $\cH^3$ estimates}\label{sec:local_H3}

Notice that the elliptic equation \eqref{eq:elli2}  is localized to $ R \leq 2 \lam \leq \f{1}{2} C_l^{-\al}$, which is away from the boundary of the rescaled domain $\td D_1 = [0, C_l^{-1}]\times [-C_l^{-1}, C_l^{-1} ]$. Therefore, $\Psi_{\chi}$ can be treated as a solution of \eqref{eq:elli2} in the whole space $R \geq 0, \b \in [0, \pi/2]$ with source term $\Om_{\chi} + Z_{\chi}$. We can apply Proposition \ref{prop:key} to improve the elliptic estimate in Lemma \ref{lem:in}. In this estimate, we need to further estimate $\Om_{\chi} + Z_{\chi}$ and $L_{12}(\Om_{\chi} + Z_{\chi} )$.

The term in $Z_{\chi}$ \eqref{eq:elli2},\eqref{eq:Zchi} either has a small factor $C_l \rho \approx 0 $ (see Remark \ref{rem:order1}), or is localized to $\lam \leq R \leq 2\lam$ due to the factor $(D_R \chi)^k$, where $\lam$ is the parameter in the cutoff function $\chi(R /\lam)$. To show that the second type of term is small, we use Lemma \ref{lem:far} and interpolation. 
Using the smallness of these variables and Lemmas \ref{lem:far} and \ref{lem:in}, we can treat $Z_{\chi}$ as a small perturbation. Since $\rho = R^{1/\al}$ and $D_R \chi =0$ for $|R|\leq 1$,  the singular weight $W = \f{(1+R)^k}{R^k},k=1,2$ is treated approximately as $1$ in the following estimates of terms involving $\rho, D_R \chi$.

\begin{prop}\label{prop:elli_euler1}
Let $\Psi$ be the solution of \eqref{eq:elli_euler} 
and $W = \f{(1+R)^k}{R^k}$ for $k=1$ or $2$. 
If $\al < \al_2$ \eqref{eq:al2}, $ C_l S < \al \cdot 8^{-1/\al -1}$, for $\lam = \f{1}{8} C_l^{-\al} $, we have 
\[
\bal
&\al^2 || R^2 \pa_{RR} \Psi_{\chi_{\lam}} W||_{L^2} + 
\al || R \pa_{R \b} \Psi_{\chi_{\lam}} W ||_{L^2} \\
& +|| \pa_{\b\b} (\Psi_{\chi_{\lam}} - \f{\sin(2\b)  }{\al \pi} 
(L_{12}(\Om) + \chi_1 L_{12}(Z_{\chi_{\lam}})(0) )
)  W ||_{L^2} \les  || \Om W||_{L^2} , \\
\eal
\]
where $Z_{\chi}$ is defined in \eqref{eq:elli2},\eqref{eq:Zchi} and $\chi_1$ is the cutoff function. Moreover, for $\nu \geq ( S(\tau))^{\al}$, $L_{12}(Z_{\chi_{\nu}})(0)$ does not depend on $\nu$ and satisfies 
\beq\label{eq:elli_euler1_l12}
 | L_{12}(Z_{\chi_{\lam}})(0) |  =  | L_{12}(Z_{\chi_{\nu}})(0) |  \les 
( 4^{- \f{1}{\al}} \al^{-1} +\min(\al,  (8^{1/\al} C_l S)^{1/2} ) ) || \Om \f{1+R}{R}||_{L^2}.
\eeq
\end{prop}

\begin{remark}
Let $\chi_{\lam_1}$ be the cutoff function in Lemma \ref{lem:in}. We choose $\lam = \f{1}{8} C_l^{-\al}$ so that $\chi_{\lam_1} \equiv 1$ in $\supp(\chi_{\lam})$. This allows us to apply Lemma \ref{lem:in} to estimate various terms in $\supp(\chi_{\lam})$. We use $L_{12}(Z_{\chi_{\lam}})(0)$ to correct $\Psi$ so that $\Psi_{\chi_{\lam}} - \f{\sin(2\b)}{\pi \al} ( L_{12}(\Om) - \chi_1 L_{12}(Z_{\chi_{\lam}})(0) )$ vanishes near $R =0$. Choosing small $C_l S, \al$ later, we use \eqref{eq:elli_euler1_l12} to show that $L_{12}(Z_{\chi_{\lam}})(0)$ is very small.

\end{remark}

\begin{proof}
{\bf{Step 1}.}
We apply the elliptic estimate in Proposition \ref{prop:key} in the weighted $L^2$ case, which can be proved using the same argument in \cite{elgindi2019finite}, to obtain 
\beq\label{eq:elli3}
\bal
I &\teq \al^2 || R^2 \pa_{RR} \Psi_{\chi_{\lam}} W||_{L^2} + 
\al || R \pa_{R \b} \Psi_{\chi_{\lam}} W ||_{L^2} \\
&\quad  +|| \pa_{\b\b} (\Psi_{\chi_{\lam}} - \f{\sin(2\b)  }{\al \pi} 
(L_{12}(\Om_{\chi} + Z_{\chi}) ) W ||_{L^2}  \les  || (\Om_{\chi}+Z_{\chi}) W||_{L^2}.
\eal
\eeq
Under the assumption $C_l S < \al 8^{-1/\al - 1}$, 
we have $ (2S)^{\al} < \f{1}{8} C_l^{-\al} = \lam$. Thus, $ \Om_{\chi} =  \chi_{\lam} \Om = \Om$. Recall $Z_{\chi} = Z_1 + Z_2 + Z_3$ in \eqref{eq:Zchi} and $\rho = R^{1/\al}$. Within the support of $\chi$, we have 
\beq\label{eq:F10}
C_l \rho W = C_l R^{1/\al -2} (1+R)^2 \les C_l (2\lam)^{1/\al} \leq 4^{- 1 /\al}.
\eeq
We can apply Lemma \ref{lem:in} to estimate the $L^2(W^2)$ norm of $Z_1$
\beq\label{eq:F1}
|| Z_1 W ||_{L^2} \les || C_l \rho W\chi||_{L^{\infty}} (|| \Psi_{\chi} ||_{L^2} 
+ \al || D_R \Psi_{\chi}||_{L^2} + ||\pa_{\b} \Psi_{\chi}||_{L^2} ) \les 4^{-1/\al} \al^{-1}|| \Om||_{L^2}.
\eeq
Estimate of $Z_3$ defined in \eqref{eq:Zchi} is trivial
\beq\label{eq:F3}
|| Z_3 W||_{L^2} \les 4^{-1/\al } || \Om||_{L^2}.
\eeq

Recall $Z_2$ defined in \eqref{eq:Zchi}.
Notice that the support of $Z_2$ lies in $\lam \leq R \leq 2\lam$ due to the $D_R \chi$ term. Within this annulus, we get $W \les 1$. Due to the smallness of $C_l \rho $ from \eqref{eq:F10}, we have 
\beq\label{eq:F20}
|| Z_2 W||_{L^2} \les \al || \Psi \one_{\lam \leq  R \leq 2 \lam}||_{L^2} +  \al^2 || D_R \chi D_R\Psi ||_{L^2} .
\eeq

Using $\lam = \f{1}{8}C_l^{-\al}$ and $C_l S < \alpha 8^{-1/\alpha-1}$, we obtain
\beq\label{eq:Cllam}
|\log(C_l \lam^{1/ \al})|  = |\log(8^{-1/\al})| \les \al^{-1} , \quad
S \lam^{-1/\al} =  8^{1/\al} C_l S < \al.
\eeq
Since $\lam \geq (2S(\tau))^{\al}$, applying Lemma \ref{lem:far} with $M = \lam$
and \eqref{eq:Cllam} to $C(M)$, we get 
\beq\label{eq:F21}
|| \Psi \one_{\lam \leq  R \leq 2 \lam}||_{L^2} \les \al^{-1}  8^{1/\al} C_l S  || \Om ||_{L^2} \les  || \Om ||_{L^2}. 
\eeq

Applying  Lemma \ref{lem:in} to $D_R \Psi_{\chi}$, and using \eqref{eq:F20}, \eqref{eq:F21}, we yield
\beq\label{eq:F2}
 || Z_2 W||_{L^2} \les \al || \Om||_{L^2} + \al^{1/2} || \Om||_{L^2} \les \al^{1/2} || \Om||_{L^2} .
\eeq


Plugging \eqref{eq:F1}-\eqref{eq:F2} into \eqref{eq:elli3} and using $4^{-1/ \al} \al^{-1} \les 1$, we prove 
\beq\label{eq:elli32}
 I   \les || \Om_{\chi} W ||_{L^2} \les || \Om W||_{L^2},
\eeq

\textbf{Step 2: Smallness of $Z_2$.} We use interpolation and the smallness of $|| \Psi \one_{\lam \leq  R \leq 2 \lam}||_{L^2}$ \eqref{eq:F21} to refine the estimate of $Z_2$ in \eqref{eq:F2}. 
The refinement is used to estimate the term $\al^{-1} L_{12}(Z_{\chi})$ in $I$ \eqref{eq:elli3}, and is important to prove \eqref{eq:elli_euler1_l12}. Using integration by parts, we obtain 
\beq\label{eq:intep_J0}
\bal
 J & \teq || \al^2 D_R \chi D_R \Psi ||_2^2 = \al^4  \la  R( D_R \chi)^2 D_R \Psi, \ \pa_R \Psi \ra  \\
&= -\al^4 \la  \pa_R (R  (D_R\chi)^2)  D_R \Psi ,\Psi \ra
 -\al^4 \la   R  (D_R\chi)^2 \pa_R D_R \Psi ,\Psi \ra . 
\eal
\eeq

Using \eqref{eq:chi2}, we get $|\pa_R (R  (D_R\chi)^2)| \les |D_R \chi| \one_{\lam \leq R \leq 2\lam}, (D_R\chi)^2 \leq \chi \one_{\lam \leq R \leq 2 \lam} $.
Using the Cauchy-Schwarz inequality, we yield 
\beq\label{eq:ref0}
J   \les  \al^2 ( \al^2  || D_R \chi D_R \Psi||_2
  +  \al^2  ||\chi D_R^2 \Psi ||_{L^2})
|| \Psi \one_{\lam \leq R \leq 2\lam} ||_{L^2} .
\eeq

We further estimate $\al^2  ||\chi D_R^2 \Psi ||_{L^2}$. Using $|D_R^2 \chi| \les \one_{\lam \leq R \leq 2\lam}$ and \eqref{eq:com_chi},  we obtain 
\[
\al^2 || \chi_{\lam} D_R^2\Psi ||_{L^2}
\les \al^2  || D_R^2  \Psi_{\chi}  ||_{L^2}
+ \al^2 || D_R \chi D_R \Psi ||_{L^2} + \al^2 || \Psi  \one_{\lam \leq R \leq 2\lam}||_{L^2}.
\]

By definition, we have $  \al^2 || D_R \chi D_R \Psi ||_{L^2}  = J^{1/2}$. 
Using  \eqref{eq:elli3},\eqref{eq:elli32},\eqref{eq:F21}, we obtain 
\beq\label{eq:ref1}
\al^2 || \chi_{\lam} D_R^2\Psi ||_{2} \les || \Om W||_2 + J^{1/2 } + \al^2 || \Om||_2
\les || \Om W||_2 + J^{1/2 } .
\eeq
Plugging $ \al^2 || D_R \chi D_R \Psi ||_{L^2}  = J^{1/2}$, the first inequality in \eqref{eq:F21} and the estimate  \eqref{eq:ref1} in \eqref{eq:ref0}, we establish 
\[
J \les \al^2( J^{1/2} + || \Om W||_2 + J^{1/2 }  )
 \al^{-1}8^{1/\al} C_l S || \Om ||_{L^2} 
\les \al ( J^{1/2} + || \Om W||_2  ) 8^{1/\al} C_l S || \Om W ||_{L^2} .
\]

The above inequality is a quadratic inequality on $B = J^{1/2} / || \Om W||_2$ : $B^2 \les A( B + 1) , A = \al 8^{1/\al} C_l S \leq \al$, which implies $ B \les A^{1/2} $. Thus, 
we prove 
\[
J^{1/2} \les  \al^{1/2}  (8^{1/\al} C_l S)^{1/2} || \Om W||_{L^2} .
\]



Combining the above estimate of $J$ and \eqref{eq:F20}-\eqref{eq:Cllam}, we yield 
\beq\label{eq:ref2}
\bal
|| Z_2 W||_{L^2} & \les (  8^{1/\al} C_l S + \al^{\f{1}{2}}  (8^{1/\al} C_l S )^{ 1/2}
 ) || \Om W||_{L^2}  \\
&\les  \min(\al,  (8^{1/\al} C_l S )^{ 1/2} )|| \Om W||_{L^2}.
\eal
\eeq

Using Lemma \ref{lem:l12}, \eqref{eq:F1}, \eqref{eq:F3} and \eqref{eq:ref2}, we establish
\beq\label{eq:l12_small}
\bal
&|| L_{12}(Z_{\chi_{\lam}}) - \chi_1 L_{12}(Z_{\chi_{\lam}})(0) W||_{L^2} 
\les || Z_{\chi} W||_{L^2}\les || (Z_1 + Z_2 + Z_3) W||_{L^2}\les \al || \Om W ||_{L^2} , \\
\eal
\eeq
where we have used $4^{-1/\al} \al^{-1}\les \al$. Combining \eqref{eq:elli3}, \eqref{eq:elli32} and \eqref{eq:l12_small}, we complete the proof of the first estimate. We remark that we only need the bound $|| Z_2 W ||_{L^2} \les \al || \Om W||_{L^2}$ from \eqref{eq:ref2} in this estimate.

{\bf{Step 3: Estimate of $L_{12}(Z_{\chi_{\lam}})(0)$.}}
Note that the previous estimates hold true for the weights $W = W_k \teq \f{(1+R)^k}{R^k}$ with $k=1$ or $2$. Recall $Z_{\chi_{\lam}} = Z_1 + Z_2 + Z_3$ \eqref{eq:elli2}. Using Lemma \ref{lem:l12}, \eqref{eq:F1}, \eqref{eq:F3} and \eqref{eq:ref2}, we prove
\[
 | L_{12}(Z_{\chi_{\lam}})(0) | \les  || Z_{\chi_{\lam}} W_1||_{L^2} \les || (Z_1 + Z_2 + Z_3) W||_{L^2}
\les ( 4^{- \f{1}{\al}} \al^{-1} + \min(\al,  (8^{1/\al} C_l S )^{ 1/2} ) )  || \Om W_1||_{L^2},
\]
which is \eqref{eq:elli_euler1_l12}. Using \eqref{eq:l12Z0}, we yield that $L_{12}(Z_{\chi_{\nu}})$ is independent of $\nu$ for $\nu \geq S(\tau)^{\al}$.
\end{proof}

\begin{prop}\label{prop:elli_euler}
Suppose that $\Psi$ is the solution of \eqref{eq:elli_euler} and $\Om \in \cH^3$. If $\al <\al_2$ \eqref{eq:al2}, $\lam = 2^{-13} C_l^{-\al}, \ C_l S <  \al \cdot (2^{13})^{-1/\al -1}$, then we have 
\[
\bal
&\al^2 || R^2 \pa_{RR} \Psi_{\chi_{\lam}} ||_{\cH^3} + 
\al || R \pa_{R \b} \Psi_{\chi_{\lam}}  ||_{\cH^3}  +|| \pa_{\b\b} ( \Psi_{\chi_{\lam}} - \f{\sin(2\b)  }{\al \pi} 
(L_{12}(\Om) + \chi_1 L_{12}(Z_{\chi_{\lam}})(0) )
)   ||_{\cH^3} \les  || \Om ||_{\cH^3} , \\
 &| L_{12}(Z_{\chi_{\lam}})(0) |  \les  3^{- \f{1}{\al}}  || \Om \f{1+R}{R}||_{L^2}.
\eal
\]
Moreover, $L_{12}(Z_{\chi_{\nu}})(0)$ does not depend on $\nu$ for $\nu \geq ( S(\tau) )^{\al}$.
\end{prop}

The small factor $3^{-1/\al}$ will be used later to absorb $\al^{-k}$ for several $k \in Z_+$, i.e. $3^{-1/\al} \al^{-k} \les_k 1$. The estimate of $ L_{12}(Z_{\chi_{\lam}})(0) $ follows from \eqref{eq:elli_euler1_l12}. The proof of the first inequality follows from the idea discussed at the beginning of Section \ref{sec:local_H3} and estimates similar to the Step 1 in the proof of Proposition \ref{prop:elli_euler1}. Suppose that $ \Om$ has size $1$. Using the smallness of $C_l, C_l S$ (see Remark \ref{rem:order1}), Lemma \ref{lem:far} and Proposition \ref{prop:elli_euler1}, formally, we get that $C_l \rho \Om$ has size $0$, $C_l \rho  \cT \Psi_{\chi}$ has size $ \approx 0$ for $\cT =\pa_{\b}, D_R$ or $Id$, $ \al^2 D_R\chi D_R \Psi$ has size $\al^2 \al^{-1} = \al$ and $D_R^k \chi \Psi$ has size $\approx 0$. Hence, $Z_{\chi}, \f{1}{\al}  ( L_{12}(Z_{\chi}) - \chi_1 L_{12}(Z_{\chi})(0)) $ have size less than $\al, 1$, respectively, which enables us to treat them as perturbation. 
Moreover, the terms $Z_1, Z_2$ in $Z_{\chi}$ \eqref{eq:elli2} have  derivatives whose orders are lower than $D_R^2 \Psi_{\chi}, \pa_{\b}^2 \Psi_{\chi}$. The term $Z_3 = -C_l \rho \sin(\b) \Om$ in \eqref{eq:Zchi} does not involve $\Psi$ and its estimate is trivial. These allow us to use induction to establish higher order estimates.

Denote $ \Psi_{\chi_{\lam}, *} = \Psi_{\chi_{\lam}} - \f{\sin(2\b)  }{\al \pi} 
(L_{12}(\Om) + \chi_1 L_{12}(Z_{\chi_{\lam}})(0) )$. To simplify our discussions, we introduce some notations for different elliptic estimates. Recall $\cH^m$ defined in \eqref{norm:H2} and $\cH^0 = L^2(\vp_1)$.
For some weight $\wt W$, differential operator $\cT =D^j_{\b} D_R^i $ and constant $ \mu$, we denote by $\cP( \wt W, \cT , \mu,  \al, C_l, \Psi, \Om )$ the following elliptic estimate for the solution $\Psi$ of \eqref{eq:elli_euler2}
\beq\label{eq:elli_p}
\al^2 || \cT D_R^2 \Psi_{\chi_{\lam}} \wt W^{ \f{1}{2}}||_{2}
+ \al || \cT D_R \pa_{\b} \Psi_{\chi_{\lam}} \wt W^{\f{1}{2}}||_{2} 
+ || \cT \pa_{\b}^2 \Psi_{\chi_{\lam}, *} \wt W^{ \f{1}{2}} ||_{2}  \les  || D_R^i \Om ||_{\cH^j}  + || \Om||_{\cH^j},
\eeq
where $\lam = \mu C_l^{-\al}$ is the parameter for the cutoff function. We put $D_R^i \Om$ in the upper bound since $D_R$ commutes with the elliptic operator $\cL_{\al}$ \eqref{eq:elli}, which was observed in \cite{elgindi2019finite}. The upper bound controls the $D_{\b}^j$ derivatives of $D_R^i \Om$. We simplify $\cP( \wt W, \cT , \mu,\al, C_l,  \Psi, \Om  )$  as $\cP( \wt W, \cT , \mu)$.

Recall the weights $\vp_i$ \eqref{def:wg} and the $\cH^3$ norm \eqref{norm:H2}. Denote $\cP_0 = \cP( \f{(1+R)^4}{R^4}, Id, 2^{-3} )$, 
\beq\label{eq:elli_pj}
\bal
&\cP_{1+j} = \cP( \vp_1, (D_R)^j, 2^{-4-j]} ), \ 0 \leq j \leq 3, \quad \cP_{5 + j} =  \cP( \vp_2, (D_R)^jD_{\b}, 2^{-8-j } ),   0\leq j \leq 2,  \\
& \cP_{8 +j} = \cP( \vp_2, (D_R)^j D^2_{\b}, 2^{-11-j } ), \ j = 0, 1,  \quad \cP_{10 } = \cP( \vp_2, D^3_{\b}, 2^{-13 } ), 
\eal
\eeq

We establish $\cP_i$ in an increasing order by induction. In other words, we first establish \eqref{eq:elli_p} for $\cT$ being $D_R$ derivatives, then for $\cT$ including one $D_{\b}$ derivatives, and so on. Estimate $\cP_0$ is established in Proposition \ref{prop:elli_euler1}, and it serves as the base case. This order of estimates has been used in \cite{elgindi2019finite}. The support of the cutoff function in $\cP_l$ satisfies 
\beq\label{eq:elli_chi}
\chi^{(l)} \equiv 1 \mathrm{ \quad in \quad  } \supp( \chi^{(l+1)}) \cup \supp(\Om), \quad \chi^{(l)} \teq \chi_1(R / ( 2^{-3 -l} C_l^{-\al}  ) ) .
\eeq

Hence, to prove $\cP_n$, we can apply $\cP_l, l\leq n-1$. The $\cH^3$ elliptic estimate follows from all $\cP_i$.

\begin{proof}
We demonstrate the ideas in the induction by mainly proving the $L^2(\vp_1)$ elliptic estimate $\cP_1$.
To establish $\cP_n, n \geq 1$, in step I, we use the $\cP_n$ version of the elliptic estimates in Proposition \ref{prop:key} with source term $\Om + Z_{\chi}$, which can be proved by the argument in \cite{elgindi2019finite}. We simplify $\chi^{(n)}$ as $\chi$. In the case of $n=1$, using the $\cP_1$ elliptic estimates (or the $L^2(\vp_1)$ estimates), we obtain 
\beq\label{eq:step1}
\bal
&\al^2 || R^2 \pa_{RR} \Psi_{\chi_{\lam}} \vp_1^{1/2}||_{L^2} + 
\al || R \pa_{R \b} \Psi_{\chi_{\lam}} \vp_1^{1/2} ||_{L^2}  \\
& +|| \pa_{\b\b} (\Psi_{\chi_{\lam}} - (\pi \al)^{-1} \sin(2\b) 
(L_{12}(\Om_{\chi} + Z_{\chi}) ) \cdot \vp_1^{1/2} ||_{L^2}  
\les  || (\Om_{\chi}+Z_{\chi}) \vp_1^{1/2}||_{L^2}.
\eal
\eeq

In step II, we apply Lemma \ref{lem:l12} to the $L_{12}(\cdot)$ terms and the elliptic estimate we have obtained, i.e. $\cP_i, i \leq n-1$, to control the $Z_{\chi}$ terms. In the case of $n=1$, $\cP_i, i \leq n-1$ is  $\cP_0$, which has been established in Proposition \ref{prop:elli_euler1}. Our goal is to establish 
\begin{align}
|| Z_{\chi} \vp_1^{1/2} ||_{L^2} & \les || \Om \vp_1^{1/2} ||_{L^2 }, \label{eq:step21} \\
|| \pa_{\b \b} \B( \f{\sin(2\b)  }{\al \pi} 
(L_{12}( Z_{\chi}) - \chi_1 L_{12}(Z_{\chi})(0) ) \B) \vp_1^{1/2} ||_{L^2}  & \les || \Om \vp_1^{1/2} ||_{L^2}, \label{eq:step22}
\end{align}
in the case of $n=1$, and similar estimates in the case of $n>1$.

Recall $Z_{\chi} = Z_1 + Z_2 +Z_3$ and \eqref{eq:Zchi}. By triangle inequality, it suffices to establish the above estimates for $Z_i$ separately. Note that $Z_3$ in \eqref{eq:Zchi} does not involve $\Psi$ and contains the small factor $C_l \rho$ (see \eqref{eq:small_Cl} below). 
The above estimates (and similar estimates in the case of $n>1$) for $Z_3$ are straightforward by applying Lemma \ref{lem:l12} to estimate the $L_{12}(\cdot)$ term. The above estimates (and similar estimates appeared in the proof of $\cP_n, n > 1$) for $Z_1, Z_2$ are established by the following substeps.

Firstly, $Z_1, Z_2$ defined in \eqref{eq:Zchi} only contain the first order derivative $D_R ,\pa_{\b}$ of $\Psi_{\chi}$, which are lower order than the leading terms $D_R^2 \Psi_{\chi}, \pa_{\b}^2 \Psi_{\chi}$ in \eqref{eq:elli2}. Hence, we can apply the previous elliptic estimates, e.g. $\cP_0$ or Proposition \ref{prop:elli_euler1} for $n=1$, to estimate the norm of higher order derivatives of $Z_1, Z_2$ or the norm of $Z_1, Z_2$ with more singular weight. To estimate the $\Psi$ terms in $Z_1, Z_2$ that do not involve $D_R$ derivative, e.g. $\pa_{\b} \Psi_{\chi}, \Psi_{\chi}$, we decompose $\Psi_{\chi}$ into 
\beq\label{eq:elli_euler_step1}
 \Psi_{\chi,*} \teq \Psi_{\chi} -  \f{\sin(2\b)}{\pi\al} (L_{12}(\Om) + \chi_1 L_{12}(Z_{\chi})(0)),  \
\Psi_{\chi , 2} \teq \f{\sin(2\b)}{\pi\al} (L_{12}(\Om) + \chi_1 L_{12}(Z_{\chi})(0)) . 
\eeq
We apply the elliptic estimates to estimate $\Psi_{\chi,*}$, Lemma \ref{lem:l12} and Proposition \ref{prop:elli_euler1} for $L_{12}(Z_{\chi})(0)$ to estimate $\Psi_{\chi,2}$. Formally, compared to $\Om$, $\Psi_{\chi, *}$ and $\Psi_{\chi, 2}$ have size $1, \al^{-1}$, respectively.

Secondly, $Z_1$ and $Z_2$ contain small factors. For $Z_1$ \eqref{eq:Zchi}, since $\lam = C C_l^{-\al}$ for $C \in [2^{-13},2^{-4}]$, in $\supp(\chi)$, we get a small factor 
\beq\label{eq:small_Cl}
C_l \rho \chi \leq  \one_{R \leq 2\lam} C_l  \rho \leq C_l ( 2\lam)^{1/\al} \leq 8^{-1/\al} 
\les_k \al^{k}
\eeq
for any $k \in Z_+$.
For $Z_2$ defined in \eqref{eq:Zchi}, the first term in $Z_2$ also contains the small factor $C_l \rho$, the second and the fourth terms contains a small factor $\al^2$ and the third term contains $\al$. 
These small factors cancel the factor $\al^{-1}$ in $\Psi_{\chi, 2}$ in \eqref{eq:elli_euler_step1}. 
In the case of $n=1$, we estimate typical terms, $C_l \rho r^{-1}\pa_{\b} \Psi_{\chi}, \al D_R \chi \Psi$,  in $Z_{\chi}$ \eqref{eq:Zchi}. Denote 
\beq\label{eq:W}
W = \f{(1+R)^2}{R^2}.
\eeq

 Recall $\vp_1 = ( f(\b)  W)^2,  f(\b)= \sin(2\b)^{-\s/2} , \s = \f{99}{100}$ \eqref{def:wg}. 
 Using 
 \[||g(R,\b) f(\b)||_{L^2(\b)} \les || g(R, \b )||_{L^{\infty}(\b)} \les || \pa_{\b} g(R, \cdot ) ||_{L^2(\b)}, \quad   g = \pa_{\b}\Psi_{\chi}, \ \Psi,
 \]
 $\rho = R^{1/\al}$, $D_R \chi  = 0$ for $|R| \leq 1$, Proposition \ref{prop:elli_euler1} and Lemma \ref{lem:l12}, we get 
\[
\bal
&|| C_l \rho r^{-1} \pa_{\b} \Psi_{\chi} \vp_1^{1/2}||_2
\les \al^4 || \pa_{\b} \Psi_{\chi} f   ||_2
\les \al^4 || \pa_{\b\b} \Psi_{\chi}  ||_2 \les \al^3 || \Om W||_2 \les \al^3 || \Om \vp_1^{1/2}||_{L^2},
 \\
&|| \al D_R\chi \Psi \vp_1^{1/2}||_2
\les  \al || D_R \chi \Psi  f||_2
\les \al || D_R \chi   \pa_{\b} \Psi ||_2  \les \al \cdot \al^{-1} || \Om W||_2
\les || \Om \vp_1^{1/2}||_{L^2},
\eal
\]
where we have used \eqref{eq:elli_chi} with $l=0$ so that we can apply Proposition \ref{prop:elli_euler1} to estimate $D_R \chi   \pa_{\b} \Psi$.
Other terms in $Z_{\chi}$ can be estimated similarly. We prove \eqref{eq:step21}. Estimates similar to \eqref{eq:step21} in the case of $n>1$ are proved similarly.





Thirdly, we consider \eqref{eq:step22} and similar estimates appeared in the proof of $\cP_n$ with $n>1$, which are more difficult to prove since they contain $\al^{-1}$. 
Recall  $\vp_1, \vp_2$ in \eqref{wg:L2_R0} and $W$ in \eqref{eq:W}. Using Lemma \ref{lem:l12} and \eqref{eq:L12Hk} in its proof, for any $p, l\geq 0$ and $q = 1,2$, we obtain
\beq\label{eq:step3}
|| D_{\b}^{p} D_R^l  \pa_{\b}^2 \f{\sin(2\b)}{\pi \al} (L_{12}(Z_1 + Z_2) - \chi_1 L_{12}(Z_1 + Z_2)(0)) \vp_q^{1/2}||_2 
\les \al^{-1} \sum_{ i \leq \max (l-1, 0)} || D_R^i (Z_1 + Z_2) W||_2 .
\eeq
We need to further estimate the right hand side. The most difficult term in $ Z_1, Z_2$ \eqref{eq:elli2},\eqref{eq:Zchi} is $\al D_R \chi \Psi$, since other terms contain smaller factors $\al^2, C_l \rho$ and their weighted Sobolev norm can be bounded by $ \al ( || D_R^l \Om||_{\cH^p} + ||\Om||_{\cH^p})$ using the same argument as that in Step 2. Formally, $\al \Psi$ has size $1$ compared to $\Om$. Exploiting the factor $D_R \chi$, we show that $\al D_R \chi \Psi$ has size $\al$.

To estimate $\al D_R^i (D_R\chi \Psi)$, we have two types of terms $I_1 \teq \al D_R^{i+1} \chi \Psi$ and $I_{j, m} \teq \al D_R^j \chi D_R^m \Psi$ with $j, m \geq 1$ and $j+ m = i+1$. Note that $|\log( C_l \lam^{1/\al})| \les \al^{-1}, \ S \lam^{-1/ \al} \leq 2^{13 /\al} C_l S \les \al.$ Applying Lemma \ref{lem:far} with $M =\lam$ to 
$D_R^{i+1}\chi \Psi$, we get 
\beq\label{eq:step31}
||  \al  D_R^{i+1} \chi \Psi W ||_2 \les ||  \al  D_R^{i+1} \chi \Psi  ||_2 \les \al || \Om ||_2.
\eeq
When $i=0$, we do not have $I_{j,m}$. Recall $i\leq \max(l-1,0)$ in the summation in \eqref{eq:step3}. Thus, in the case of $n=1$, combining \eqref{eq:step3} and \eqref{eq:step31} implies \eqref{eq:step22}. The same argument applies to the case of $ l \leq 1$.

It remains to estimate $I_{j,m}$ with $j,m\geq 1$, in the case of $l \geq 2$. Recall $L_{12}(\cdot)$ from \eqref{eq:biot3}, $\Psi_{\chi, 2}$ from \eqref{eq:elli_euler_step1} and $\chi = \chi^{(n)}$. Since $\supp(D_R^j \chi ) = \{ \lam_n \leq R \leq 2\lam_n \}$ is away from $\supp(\Om)
\cup \supp(\chi_1) \subset \{ R \leq S^{\al} \}$ (see \eqref{eq:elli_chi}),  we get 
\[
D_R^j \chi  D_R^m \Psi_{\chi^{(m)}, 2} =  D_R^j \chi \cdot  \f{\sin(2\b)}{\pi \al}
\B(
 - \int_0^{\pi/2} D_R^{m-1} \Om(R, \b)  \sin(2\b) d \b 
+ D_R^m \chi_1 L_{12}( Z_{\chi^{(m)} })(0)\B)= 0. 
\] 
Thus, we can subtract the singular term from $D_R^j \chi D_R^m \Psi$ 
\beq\label{eq:step32}
D_R^j \chi D_R^m \Psi = D_R^j \chi D_R^m ( \Psi - \Psi_{\chi^{(m)}, 2}).
\eeq
Formally, compared to $\Om$, $D_R^j \chi D_R^m \Psi$ has size $1$. In the summation in \eqref{eq:step3}, we have $i \leq l-1$. Since $j + m = i+1$ and $j \geq 1$, we get $m = i+1 - j \leq l-1$. By definition \eqref{eq:elli_pj}, the weighted $D_{\b}^p D_R^l$ elliptic estimates appear in $\cP_n$, if $ n \geq l$  and $l\leq 3$. Thus, using the elliptic estimate $\cP_{m}$ ($m \leq l-1 \leq n-1)$ in the induction hypothesis, i.e. $\cT = D_R^m, \td W = \vp_1$ in \eqref{eq:elli_p}, we yield
\beq\label{eq:step33}
|| \al D_R^j \chi D_R^m  ( \Psi - \Psi_{\chi^{(m)}, 2}) W ||_2 
\les \al  ||  D_R^m  ( \Psi_{\chi^{ (m) }} - \Psi_{\chi^{(m)}, 2}) W ||_2 
\les \al ( || D_R^m \Om ||_{\cH^0} + || \Om ||_{\cH^0}),
\eeq
where $\cH^0 = L^2(\vp_1)$. Combining \eqref{eq:step3}-\eqref{eq:step33}, we establish the $\cP_n$ version of \eqref{eq:step22} for $Z_1+Z_2$.




Therefore, combining \eqref{eq:step1}-\eqref{eq:step22}, we obtain the $L^2(\vp_1)$ elliptic estimate, i.e. $\cP_1$.  Repeating this argument, we can obtain the $\cP_l, 2\leq l \leq 10$ and $\cH^3$ elliptic estimates.

Note that the assumption on $\lam, C_l, S$, i.e. $C_l S < \al \cdot (2^{13})^{-1/\al-1}$, implies $\lam \geq S^{\al}$ and 
\[ 
4^{-1/\al} \al^{-1} \les 3^{-1/\al}, \quad  (8^{1/\al} C_l S )^{ 1/2}
\leq ( (2^{10})^{-1/\al} )^{1/2} \leq 3^{-1/\al}.
\]
Since the estimate \eqref{eq:elli_euler1_l12} in Proposition \ref{prop:elli_euler1} does not depend on $\lam $ as long as $ \lambda \geq ( S(\tau) )^{\al}$, using \eqref{eq:elli_euler1_l12} and the above calculation, we establish the desired estimate on $L_{12}(Z_{\chi_{\lam}})(0)$.
\end{proof}

\begin{remark}
The term $D_R^j \chi D_R^m \Psi$ can also be estimated using an argument similar to that in the Step 2 of the proof of Proposition \ref{prop:elli_euler1}. We find the above approach simpler.
\end{remark}


Recall $\bar \Om$ in \eqref{eq:profile}. We have a result similar to Proposition \ref{prop:psi}. 
\begin{prop}\label{prop:psi2}
Let $\bar{\Psi}_0(t)$ be the solution of \eqref{eq:elli_euler} with source term $\bar{\Om}_0 = \bar{\Om} \chi(R / \nu)$. If $\al < \al_2$ \eqref{eq:al2}, $\lam = 2^{-13} C_l^{-\al} ,  \ 
C_l S < \al (2^{13})^{-1/\al - 1}, \ 2 \nu < \lam$, then we have 
\[
\bal
& \al || \f{1+R}{R} D_R^2 \bar{\Psi}_{0,\chi_{\lam}} ||_{\cW^{5,\infty}} + 
\al || \f{1+R}{R} R \pa_{R \b} \bar{\Psi}_{0, \chi_{\lam}}  ||_{\cW^{5,\infty}} \\
&  +||\f{1+R}{R} \pa_{\b\b} ( \bar{\Psi}_{0,\chi_{\lam}} - \f{\sin(2\b)  }{\al \pi} 
(L_{12}( \bar \Om_0) + \chi_1 L_{12}( \bar Z_{\chi_{\lam}})(0) )
)   ||_{\cW^{5,\infty}} \les  \al , \\
 &| L_{12}( \bar Z_{\chi_{\lam}})(0) |  \les  3^{- \f{1}{\al}}  ,
\eal
\]
where $\bar Z_{\chi_{\lam}}$ associated to $\bar \Psi_0$ is
defined in \eqref{eq:elli2},\eqref{eq:Zchi}.
Moreover, $L_{12}( \bar Z_{\chi_{\mu}})(0)$ does not depend on $\mu$ for $\mu \geq ( S(\tau) )^{\al}$ and enjoys the above estimate for $L_{12} ( \bar Z_{\chi_{\lam}})$. 
\end{prop}
\begin{remark}
Although $\bar{\Om}_0 = \bar{\Om} \chi_{\nu}$ is time-independent, the equation \eqref{eq:euler2} is not and $\bar{\Psi}_0(t)$ depends on how we rescale the space. The factor $2\nu$ is the support size of $\bar{\Om}_0$. We impose $\lam > 2 \nu$ so that $\chi_{\lam} = 1$ in the support of $\bar \Om_0$.
\end{remark}
The proof follows from the argument in the proof of Propositions \ref{prop:psi}, \ref{prop:elli_euler1} and \ref{prop:elli_euler}.

\subsection{Nonlinear stability}\label{sec:3Dnon}
We apply the nonlinear stability analysis of the 2D Boussinesq equations to prove Theorem \ref{thm:euler}. In Section \ref{subsec:non_boot}, we impose the bootstrap assumption on the support size. In Section \ref{subsec:non_appr}, we construct the approximate steady state and impose the normalization conditions, which are small perturbations to those in the 2D Boussinesq. In Section \ref{sec:lot}, we estimate the terms in the 3D Euler \eqref{eq:euler4} that are different from 
the 2D Boussinesq \eqref{eq:bousdy2}. These terms contain factors that are much smaller than $\al$ and we treat them as perturbations. In Section \ref{subsec:non_stab}, we generalize the nonlinear stability estimates in the 2D Boussinesq to the 3D Euler. In Section \ref{sec:grow_supp}, we use the ideas described in Section \ref{sec:idea_supp} to control the growth of the support. In Section \ref{subsec:non_blowup}, we prove finite time blowup.

\subsubsection{Bootstrap assumption on the support size}\label{subsec:non_boot}

Recall $\al_2$ defined in \eqref{eq:al2} in Lemma \ref{lem:in}. We first require $\al < \al_2$. We impose the first bootstrap assumption: for $t \geq 0$, we have 
\beq\label{eq:boot5}
\bal
C_l( t) \max( S(t ), S(0)) < \al \cdot (2^{13})^{- \f{1}{\al} - 1} \teq K(\al).  
\eal
\eeq
Under the above Bootstrap assumption, the support of $\om, \th$ in $D_1$ does not touch the symmetry axis and $z = \pm 1$, and the assumption in Proposition \ref{prop:elli_euler} is satisfied. We will choose $C_l(0)$ at the final step, which guarantees the smallness in \eqref{eq:boot5}.

\subsubsection{Approximate steady state and the normalization condition}\label{subsec:non_appr}
Since the rescaled domain $\td D_1$ \eqref{eq:rescale_D} is bounded, we construct approximate steady state with bounded support. We localize $\bar{\Om}, \bar{\th}$ defined in \eqref{eq:profile}  to construct the approximate steady state for \eqref{eq:euler4}
\beq\label{eq:profile2}
\bar{\Om}_0 \teq \chi_{\nu} \bar{\Om}, \quad \bar{\th}_0 \teq \chi_{\nu} \bar{\th}
= \chi_{\nu}  (1 + x J(\bar{\eta}) ),
\eeq
where $\chi_{\nu} = \chi_1( R / \nu)$ and we have applied the integral operator $J(f)$ in Lemma \ref{lem:small}. We can choose $\chi_1 = \td \chi_1^2$ for another smooth cutoff function $\td \chi_1$ such that $\chi_1^{1/2} = \td \chi_1$ is smooth. We use $\bar \om_0$ to denote $\bar \Om_0$ in the $(x,y)$ coordinates. Clearly, the support size of $\bar{\Om}_0, \bar{\th_0}$ is $2\nu$.
Using the computation in \eqref{eq:lemsm2}, we have 
\beq\label{eq:profile22}
\bal
&\bar{\eta}_{0} = \pa_x (\chi_{\nu} \bar{\th}) 
 =   \al \cos^2(\b) D_R \chi_{\nu} \cdot J(\bar{\eta}) +
\al \cos(\b) r^{-1}  D_R \chi_{\nu}  + 
  \chi_{\nu}  \bar{\eta},   \\
&\bar{\xi}_{0}(R, \b) = \pa_y (\chi_{\nu} \bar{\th}) 
 = \al \sin(\b) \cos(\b) D_R \chi_{\nu} \cdot J(\bar{\eta})  
+ \al \sin(\b) r^{-1} D_R \chi_{\nu}   + \chi_{\nu}  \bar{\xi}, 
 \eal
\eeq
Let $\bar{\Psi}_0(t)$ be the solution of \eqref{eq:elli2} with source term $\bar{\Om}_0$. 
Applying Lemma \ref{lem:small} and the analysis in its proof, we know that $\bar{\Om}_0, \bar{\eta}_0, \bar{\xi}_0$ enjoys the same estimates as that of $\bar{\Om}, \bar{\eta}, \bar{\xi}$ in Lemmas \ref{lem:bar} and \ref{lem:xi}.


We need to adjust the time-dependent normalization condition for $c_{\om}(t), c_l(t)$. 
Firstly, we choose the time-dependent cutoff radial $\lam(t) = 2^{-13} (C_l(t))^{-\al}$ according to Proposition \ref{prop:elli_euler}. 

Define $\bar{Z}_{\chi_{\lam(0)}}(t)$ according to \eqref{eq:Zchi}, or equivalently \eqref{eq:l12Z0},
with $\Psi = \bar{\Psi}_0(t), \Om = \bar{\Om}_0$ and $\chi = \chi_{\lam(0)}$. It does not depend on the cutoff radial as long as $\lam(0) \geq (2\nu)^{\al}$, where $2\nu$ is the size of support of $\bar{\Om}_0$.
We use the following conditions
\beq\label{eq:normal20}
\bal
\bar{c}_{\om}(t) &=  -1 - \f{2}{\pi \al} L_{12}( \bar{\Om}_0- \bar{\Om} +  \bar{Z}_{\chi_{\lam(0)}})(0) 
 \quad \bar{c}_l(t) = 
\f{1}{\al} + 3  -\f{1-\al}{\al}\f{2}{\pi \al} L_{12} (\bar{\Om}_0 - \bar{\Om} + \bar{Z}_{\chi_{\lam(0)}})(0) .\\
\eal
\eeq

We remark that $\bar{c}_{\om}(t), \bar{c}_{l}(t)$ is time-dependent. Without the $Z$ term, the above conditions for $\bar{c}_{\om}, \bar{c}_l$ are the same as 
that in \eqref{eq:profile} with a correction due to the difference between the profiles $(\bar{\Om}, \bar{\eta})$ in \eqref{eq:profile} and $\bar{\Om}_0, \bar{\eta}_0$ in \eqref{eq:profile2}-\eqref{eq:profile22}.
For this difference, we use \eqref{eq:normal} to correct $\bar{c}_{\om}, \bar{c}_l$. 

For any perturbation $\Om(t)$, we use the following conditions for $c_{\om}(t), c_l(t)$
\beq\label{eq:normal21}
 c_{\om }(t) =  -\f{2}{\pi \al} L_{12}(\Om(t) + Z_{\chi_{\lam(t)}}(t))(0)
 ,  \quad c_l(t) = \f{1-\al}{\al} c_{\om}(t).
\eeq
Without the $Z$ term, the above conditions for $c_{\om}(t), c_l(t)$ are the same as that in \eqref{eq:normal}. 

We add the $Z$ terms in \eqref{eq:normal20}, \eqref{eq:normal21} since the behavior of 
$\Psi$, which is the solution of \eqref{eq:elli_euler}, is characterized by $L_{12}(\Om + Z_{\chi})(0)$ for $R$ close to $0$ according to the elliptic estimate in Proposition \ref{prop:elli_euler}. 
For the 2D Boussinesq equation, we use $L_{12}(\Om)(0)$ to determine $c_{\om}, c_l$ since it also characterizes the behavior of $\Psi$ near $R=0$ according to Proposition \ref{prop:key}.


We choose the above conditions so that the error of the approximate steady state vanishes quadratically in $R$ near $R=0$ and that the update of $\Om(t), \eta(t) (\om, \th_x)$ in equation \eqref{eq:euler4} also vanishes quadratically in $R$ near $R=0$ if the initial perturbation $\Om(\cdot ,0), \eta(\cdot ,0)$ $(\th_x(0))$ vanishes quadratically. We also determine $\bar{c}_{\om}, \bar{c}_l$ in \eqref{eq:profile} and $c_{\om}, c_l$ in \eqref{eq:normal} based on this principle. 

\begin{remark}\label{rem:order2}
We will choose $\nu$ to be very large, relative to $\al^{-1}$. 
Therefore, we treat $\bar \Om_0 \approx \bar \Om, \bar \th_0 \approx \bar \th$. Due to the small factor $3^{-1/\al}$ in Propositions \ref{prop:elli_euler}, \ref{prop:psi2}, we treat $L_{12}(Z_{\chi})(0)$, $L_{12}( \bar Z_{\chi} )(0)\approx 0$. From Remark \ref{rem:order1} and the bootstrap assumption \eqref{eq:boot5}, we also have $C_l \approx 0, C_l S \approx 0, r \approx 1$. 
We treat the error terms in these approximations as perturbation.
\end{remark}

\subsubsection{Estimate of the lower order terms}\label{sec:lot}

The equations \eqref{eq:euler4} are slightly different from \eqref{eq:bousdy2} for the Boussinesq systems. We show how to estimate their differences. Suppose that $\om(t), \th(t)$ are the perturbations and the support size of $\bar{\om}_0 + \om(t), \bar{\th}_0 + \th(t)$ is $S(t)$.

Assume that the bootstrap assumption \eqref{eq:boot5} holds true. For the term $\f{1-r^4}{r^4} \th_x$, within the support of $\om, \th$, we have $ \rho \leq S(t), r = 1 - C_l \rho \sin(\b) \in [3/4, 1]$. We get
\beq\label{eq:small_s}
|\f{1-r^4}{r^4} | \les 1 - r  \leq C_l \rho  \leq C_l  S(t) < \al (2^{13})^{-1/\al-1},
\eeq
which is extremely small compared to $\al$. Since $\rho = R^{1/\al}$, the factor $C_l \rho , 1 - r^4$ vanish high order in $R$ near $R=0$. 
Hence, $\f{1-r^4}{r^4}\th_x$ is a smooth (near $R=0$) small error term.


For the term $\f{1}{r} C_l \psi$ in $u = -\psi_y + \f{1}{r} C_l \psi$ defined in \eqref{eq:euler4}. Under the $(R,\b)$ coordinates, it becomes $\f{C_l \rho}{r} (\rho \Psi(R, \b))$. Compared to $-\psi_y = -( \rho^2 \Psi)_y$ in \eqref{eq:simp2}, $\f{C_l \rho}{r} (\rho \Psi(R, \b))$ vanishes on $\b = 0, \pi/2$ and contains a small smooth factor $C_l \rho = C_l R^{1/ \al}$ within the support of $ \om, \th$. 

The last difference is the elliptic estimate between Propositions \ref{prop:key} and \ref{prop:elli_euler}. Notice that in \eqref{eq:euler4}, we only use $\Psi(R, \b)$ for $(R,\b)$ within the support of $\om, \th$. 
We have $\Psi_{\chi_{\lam(t)}}(R,\b) = \Psi(R,\b)$ for $\lam(t) = 2^{-13} C_l^{-\al}, R \leq S(t)$. Finally, $\chi_1 L_{12}(Z_{\chi_{\lam(t)}})(0)$ in Proposition \ref{prop:elli_euler} 
only affects the equation near $R=0$.
 Since $\f{1+R}{R}\bar{\Om} \in L^2$, using the estimate in Proposition \ref{prop:elli_euler}, we get
\beq\label{eq:boot70}
| L_{12}(\bar{Z}_{\chi_{\lam(0)}})(0) | = | L_{12}(\bar{Z}_{\chi_{\lam(t)}})(0)| \les \al 3^{-1/\al} , \quad
| L_{12}(Z_{\chi_{\lam(t)}}) (0)| \les  3^{-1/\al}  || \f{1+R}{R} \Om||_{L^2}, 
\eeq
where we have used $\lam(t) \geq (S(t))^{\al}$ due to \eqref{eq:boot5} to obtain the first identity, and used \eqref{eq:profile},\eqref{eq:profile2} and $|| \f{1+R}{R}\bar{\Om}_0||_{L^2}\les \al$ to obtain the first inequality. The terms in \eqref{eq:boot70} are treated as small terms with amplitude close to $0$.
 
Using the argument in Section \ref{sec:non}, we can estimate these lower order terms in $\cH^3, \cH^3(\psi)$ or $\cC^1$ norm accordingly and obtain a small constant in the estimate bounded by $C(1  + \al^{-\kp}) (3^{-1/\al} + C_l S)$, where $\kp , C>0$ are some absolute constant.

\subsubsection{Nonlinear stability}\label{subsec:non_stab}
Notice that the domain $\td D_1$ \eqref{eq:rescale_D} of the dynamic rescaling equation is bounded and is different from $\R_2^+$.
We cannot apply directly the estimates in Sections \ref{sec:lin}-\ref{sec:non}
because in these estimates, we linearize the equations around $\bar{\Om}, \bar{\eta}, \bar{\xi}$ which are defined globally.

We consider the system of $\th_x, \th_y, \om$ obtained from \eqref{eq:euler4} and then linearize it around the approximate steady state $\bar{\Om}_0, \bar{\eta}_0, \bar{\xi}_0, \bar{c}_{\om}, \bar{c}_l$ constructed in Section \ref{subsec:non_appr} to obtain a system similar to \eqref{eq:lin21}-\eqref{eq:lin23} for the perturbation $(\Om, \eta, \xi)$ with $\bar{\Om}, \bar{\eta}, \bar{\xi}, \f{3}{1+R} (= \f{2}{\pi \al} L_{12}(\bar{\Om}))$ replaced by $\bar{\Om}_0, \bar{\eta}_0 , \bar{\xi}_0, \f{2}{\pi\al} L_{12}(\bar{\Om}_0)$. We also put the lower order terms discussed in Section \ref{sec:lot} into the remaining terms $\cR_{\Om}, \cR_{\eta}, \cR_{\xi}$. 


According to Lemma \ref{lem:small}, we know that $\bar{\Om}_0, \bar{\eta}_0, \bar{\xi}_0$ converges to $\bar{\Om}, \bar{\eta}, \bar{\xi}$ in the $\cH^3, \cH^3(\psi)$ norm as $\nu \to \infty$ ($\nu$ is the cutoff radial in \eqref{eq:profile2}). 
Moreover, we can easily generalize the $\cH^3, \cH^3(\psi)$ convergence to the higher order convergence. We choose the same weights and the same energy norm as that in Section \ref{sec:lin}-\ref{sec:non}. Then for sufficient large $\nu$, due to these convergence results, under the bootstrap assumption \eqref{eq:boot5},  we can obtain the following $\cH^3, \cH^3(\psi)$ estimates similar to that in Corollary \ref{cor:H3}
\[
\f{1}{2}\f{d}{dt} E_3^2(\Om, \eta,\xi) 
\leq  (- \f{1}{13} + C \al ) E_3^2  + \cR_3,
\]
where $E_3, \cR_3$ are defined in \eqref{eg:H3}. We have a slightly weaker estimate ($\f{1}{13} < \f{1}{12}$) due to the small difference between $(\bar{\Om}_0, \bar{\eta}_0, \bar{\xi}_0)$ and $(\bar{\Om}, \bar{\eta}, \bar{\xi})$. 
\begin{remark}
The choice of $\nu$ is independent of $C_l(0)$. We will choose initial data with the size of the support $S(0) \geq 2\nu$.  Though $S(0) > \nu$ is large, we choose $C_l(0)$ small enough at the final step and verify \eqref{eq:boot5}.
\end{remark}

Recall the equation \eqref{eq:xi0_2} for the 2D Boussinesq equation in the $\cC^1$ estimate of $\xi$.
The damping part in \eqref{eq:xi0_2} is $(-2 - \f{3}{1+R}) \xi$. For the 3D Euler equation, it is replaced by $(-2 - \f{2}{\pi \al} L_{12}(\bar{\Om}_0) )\xi$. For sufficient large $\nu$, using the convergence results, we can obtain estimates similar to \eqref{eq:xi_inf1}, \eqref{eq:xi_inf2R},
\eqref{eq:xi_inf2b} with slightly larger constants, e.g. $-2 , 3$ are replaced by  $-2 + \f{1}{100}, 3 + \f{1}{100}$.

There exists a large absolute constant $\nu_0$, such that for $\nu  > \nu_0$, $\nu$ satisfies the above requirements, and we have
\beq\label{eq:boot60}
 |\f{2}{\pi \al} L_{12}(\bar{\Om} -\bar{\Om}_0 )(0) | \leq \f{1}{100}.
\eeq
To estimate the $\cH^3$ norm of $\cR_{\Om}, \cR_{\eta}$ and $\cH^3(\psi), \cC^1$ norms of $\cR_{\xi}$, we apply the estimates in Section \ref{sec:non} and the argument in Section \ref{sec:lot}. Therefore, for $\nu > \nu_0$, under the bootstrap assumption, we obtain the following nonlinear estimate for compactly supported perturbations $\Om(t), \eta(t), \xi(t)$ around $(\bar{\Om}_0, \bar{\eta}_0,\bar{\xi}_0)$, which is similar to \eqref{eq:boot4},
\beq\label{eq:boot6}
\bal
\f{1}{2} \f{d}{dt} E^2(\Om, \eta,\xi)
\leq  -\f{1}{13} E^2 + C( \al^{1/2} E^2 +  \al^{-3/2} E^3 +  \al^2 E )
+ C(\al, C_l(t), S(t)) (E^2 + E + E^3 ) , 
\eal
\eeq 
where the energy $E$ is defined in \eqref{eg}. The last term is from the estimates of the lower order terms discussed in Section \ref{sec:lot}, e.g. $\f{1-r^4}{r^4}\th_x, \f{1}{r}C_l \psi$, and $C(\al, C_l(t), S(t)) = C (1 + \al^{-\kp})(3^{-1/\al} + C_l(t) S(t) )$), for some universal constant $C$. Under the bootstrap assumption \ref{eq:boot5}, we further obtain
\beq\label{eq:boot62}
C(\al, C_l(t), S(t)) \les  (1 + \al^{-\kp}) 3^{-1/\al}  \les \al^3.
\eeq

Combining \eqref{eq:boot6}, \eqref{eq:boot62}, we obtain that there exist $\al_3$ with $0 < \al_3 <\al_2$ ($\al_2$ is the constant in \eqref{eq:al2} in Lemma \ref{lem:in}) and an absolute constant $\td{K}>0$, such that if $E(\Om(0), \eta(0), \xi(0)) < \td{K} \al^2$, under the bootstrap assumption \ref{eq:boot5}, we have 
\beq\label{eq:boot71}
E(\Om(t), \eta(t), \xi(t)) < \td{K} \al^2. 
\eeq

Recall $c_{\om}, c_l, \bar{c}_{\om}, \bar{c}_l$ defined in \eqref{eq:normal20}, \eqref{eq:normal21}. Using \eqref{eq:boot70}, \eqref{eq:boot60}, $|L_{12}(\Om)(0)| \les || \Om||_{\cH^3} \les E \les \al^2$, we obtain 
\[
| c_{\om} + \bar{c}_{\om} + 1 |<   \f{1}{100} +C 3^{-1/\al} + C \al , \quad 
c_l + \bar{c}_l > \f{1}{\al} + 3 - \f{1}{100 \al} - C 3^{-1/\al} \al^{-1} - C.
\]


We further choose $\al_4$ with $0 < \al_4 < \al_3$, such that for $\al < \al_4$, 
\beq\label{eq:boot72}
-\f{3}{2} <c_{\om} + \bar{c}_{\om} <  -\f{1}{2}, \quad c_l + \bar{c}_l > \f{3}{4\al}.
\eeq


\subsubsection{Growth of the support}\label{sec:grow_supp}
Recall Definition \ref{def:supp} of $S(\tau)$. Finally, we use the idea in Section \ref{sec:idea_supp} to estimate the growth of the support $S(\tau)$ of the solutions $\om + \bar{\om}_0, \th + \bar{\th}_0$. Denote 
\[
\wh{\uu}(t) = \uu(t) + \bar{\uu}(t), \  \wh{\Psi}(t) = \Psi(t) + \bar{\Psi}_0(t) , \  \wh{c}_{l}(t) = c_l(t) + \bar{c}_l .
\]

Applying \eqref{eq:simp1}-\eqref{eq:simp2} and \eqref{eq:trans} to $\wh{\Psi}$, 
we can rewrite the transport term $\wh{\uu} \cdot \na$ in \eqref{eq:euler4} as  
\[
\wh{\uu} \cdot \na = (-\pa_y \wh{\psi} + C_l r^{-1} \wh{\psi}) \pa_x  + \pa_x \wh{\psi} \pa_y 
= (\f{ \al C_l \rho \cos(\b)}{r} R \wh{\Psi}   - \al R \pa_{\b} \wh{\Psi} ) \pa_R + ( 2 \wh{\Psi} +  \al R \pa_R \wh{\Psi}
 - \f{C_l \rho \sin(\b)}{r} \wh{\Psi}) \pa_{\b},
\]
where $\rho = R^{1/\al}, r = 1 - C_l \rho \sin(\b)$. The above formula is different from \eqref{eq:trans} due to the extra term $C_l r^{-1} \wh{\psi} \pa_x$. Notice that $\wh{c}_{l} \xx \cdot \na$ becomes $\al \wh{c}_{l} R \pa_R$ under the $(R,\b)$ coordinates. For a point which is inside the support of $\om + \bar{\om}, \th + \bar{\th}$ and has coordinates $(R(t)), (\b(t))$, its trajectory under the flow $(\wh{c}_{l} \xx + \wh{\uu}) \cdot \na $ is governed by 
\beq\label{eq:lang}
\bal
 \f{d}{dt} R(t)  = (\al \wh{c}_{l} R(t) + \f{ \al C_l(t) \rho(t) \cos(\b(t)) }{r(t)} R(t) \wh{\Psi}(R(t), \b(t))   - \al R(t) \pa_{\b} \wh{\Psi}( R(t), \b(t)),
\eal
\eeq
where the relation between $\wh{c}_{l}(t), C_l(t)$ is given in \eqref{eq:rescal42}.

\begin{lem}\label{lem:decay}
Under the assumption of Propositions \ref{prop:elli_euler}, \ref{prop:psi2} and that $\Om \in \cH^3$, for $R \leq S(t)$, we have 
\[
 |(1+R^{1/3}) \wh{\Psi}(R, \b)   |  +  | (1+R^{1/3}) \pa_{\b} \wh{\Psi}(R, \b) | \les \al^{-1}  || \Om||_{\cH^2} + 1
 \les  \al^{-1} E( \Om(t), \eta(t), \xi(t)) + 1.
\]
\end{lem}

Recall the weights $\vp_i$ in Definition \ref{def:wg} for the $\cH^3$ norm \eqref{norm:H2}. Denote by $\td \cH^3$ the modified $\cH^3$ space with radial weight $\f{(1+R)^2}{R^2}$ in the $\cH^3$ space replaced by $\f{1+R}{R}$. The $\td \cH^3$ version of the elliptic estimates in Proposition \ref{prop:elli_euler} can be obtained by the same argument. 
Since $\bar{\Om}_0 + \Om $ is in $\td \cH^3$ space ($\bar{\Om}_0$ vanishes linearly near $R=0$), 
applying the $\td \cH^3$ elliptic estimate to $\wh{\Psi} - \wh{\Psi}_2$ and $L_{12}(Z_{\chi})(0)$, where $\wh \Psi_2 = \f{\sin(2\b)}{\pi \al} ( L_{12}(\Om) + \chi_1 L_{12}(Z_{\chi})(0))$, and Lemma \ref{lem:l12} to $\wh \Psi_2$, we obtain 
\[
 || \pa_{\b \b} \wh \Psi ||_{\td \cH^3} \les \al^{-1}  || \Om + \bar \Om_0 ||_{\td \cH^3}
 \les \al^{-1} || \Om ||_{\cH^3} + 1
 \les \al^{-1} E( \Om(t), \eta(t), \xi(t)) + 1.
\]
Applying the argument in the proof of Lemma \ref{lem:xi_decay}, we establish the decay estimate.


Now we assume that the initial data satisfies $E(\Om(0), \eta(0), \xi(0) ) < \td{K} \al^2$. Under the bootstrap assumption \eqref{eq:boot5}, we have {\it a priori} estimates \eqref{eq:boot71}, \eqref{eq:boot72}.

Plugging the bootstrap assumption \ref{eq:boot5}, \eqref{eq:boot71} and Lemma \ref{lem:decay} in \eqref{eq:lang}, 
 we derive
\[
\f{d}{dt} R(t)\leq  \al \wh{c}_l R(t) + C \al ( \al^{-1} E + 1) R(t)^{2/3}
\leq \al \wh{c}_l R(t) + C \al R(t)^{2/3},
\]
where we have used $C_l(t) \rho(t) \leq  C_l(t) S(t) \les 1 , r^{-1} \les 1$. From the formula of $C_l(t)$, we know $\f{d}{dt} C_l(t) = -\wh{c}_l(t) C_l(t)$. Multiplying $C^{\al}_l(t)$ on both sides,
 we get 
\[
\f{d}{dt} C_l^{\al} R(t) \leq C\al  C_l^{\al} R^{2/3}(t) =  C\al (C_l^{\al} R )^{2/3}  C_l(t)^{\al/3}.
\]
From the {\it a priori} estimate \eqref{eq:boot72} and the formula of $C_l$ in \eqref{eq:rescal42}, we know $C^{\al}_l(t) \leq C^{\al}_l(0) \exp(-\f{t}{2})$. Then solving this ODE, we yield 
\[
\bal
(C_l^{\al} R(t))^{1/3} \leq (C_l(0)^{\al} S(0)^{\al})^{1/3} + C \al \int_0^{\infty} C_l^{\al/3}(0) \exp(-\f{b}{6}) d b
\leq C_l(0)^{\al/3}( S(0)^{\al/3} + C\al ).
\eal
\]
Taking the supremum over $(R(t), \b(t))$ within the support of $\Om, \th$, we prove 
\beq\label{eq:supp}
C_l(t) S(t) \leq  C( \al, S(0)) C_l(0).
\eeq



\subsubsection{Finite time blowup}\label{subsec:non_blowup}
For fixed $\al < \al_4, \nu > \nu_0$, we choose zero initial perturbation $\Om(0) = 0, \eta(0)= 0,\xi(0) = 0$. Then the initial data is $(\bar{\Om}_0 ,\bar{\th}_0)$ defined in \eqref{eq:profile2} which has compact support with support size $S(0) = 2\nu$. 
We choose initial rescaling $C_l(0)$ such that  $ C( \al, S(0))   C_l(0) < K(\al) / 2$, where $K(\al)$ is defined in \eqref{eq:boot5}. Using the {\it a priori} estimates \eqref{eq:boot71}, \eqref{eq:boot72} and \eqref{eq:supp}, we know that the bootstrap assumption in \eqref{eq:boot5} can be continued. Thus these estimates hold true for all time.

Since $-\f{3}{2} < c_{\om} + \bar{c}_{\om} < -\f{1}{2}$ (\eqref{eq:boot72}) and the solutions $\om, \th$ are close to $\bar{\om}_0, \bar{\th}_0$ for all time in the dynamic rescaling equation, using the argument in Subsection \ref{sec:blowup} and the BKM blowup criterion in \cite{beale1984remarks}, 
we prove that the solutions remain in the same regularity class as that of the initial data before $T^* <+\infty$ and develop a finite time singularity at $T^*$,
where $T^* = t(\infty) = \int_0^{\infty}
C_{\om}(\tau) d\tau < +\infty$.

Since $\bar{\th}_0 + \th(t) \geq 0$ and the support of $\om, \th$ is away from the axis, we can recover $u^{\th} =  \td \th^{1 /2}/r ,\om^{\th}$ from $\th, \om$ via \eqref{eq:omth}, \eqref{eq:rescal41}. From \eqref{eq:profile2} and the discussion below \eqref{eq:profile2}, $\chi_{\nu}^{1/2}(R) = \td \chi_1(R / \nu)$ is smooth. Since $\bar \th(x, y) \geq 1$ \eqref{eq:th}, it is even in $x$, and $ \bar \th \in C^{1,\al}$, we get $\bar \th^{1/2} \in C^{1,\al}$. It follows that $\bar \th_{0}^{1/2}=\chi_{\nu}^{1/2} \bar \th^{1/2} 
= \td \chi_1(R / \nu) \bar \th^{1/2}  \in C^{1,\al}$. Using $u_0^{\th} =  \td \th_0^{1 /2}/r $, the relation \eqref{eq:rescal41}, and $r \geq \f{1}{2}$ in $\supp(u_0^{\th})$, we get  $u_0^{\th} \in C^{1,\al}$ and that $u_0^{\th}$ is even in $z$. Due to the regularity on  $u_0^{\th}, \om_0^{\th}$ and the fact that in $D_1$, they are supported near $(r, z) = (1,0)$, the solutions have finite energy in $D_1$. 


\section{Concluding Remarks}\label{sec:remark}

We have proved finite time blowup of the 2D Boussinesq and the 3D axisymmetric Euler equations with solid boundary and large swirl using $C^{\al}$ initial data with small $\al$ for $(\om, \na \th)$ in the case of the 2D Boussinesq equations and for $(\om^{\th}, \na (u^{\th})^2 )$ in the case of the 3D Euler equations, respectively. For the 2D Boussinesq, the singularity is asymptotically self-similar. In particular, we showed that the velocity field is in $C^{1,\alpha}$ and has finite energy.

The results presented in this paper can be generalized to prove finite time singularity of some related problems. First of all, the proof of Theorem \ref{thm:euler} with minor modifications on Lemma \ref{lem:biot1} implies similar results for the 3D Euler equations in a bounded domain $\td D$ with smooth boundary and the following properties:  $\td D$ is symmetric with respect to the plane $z=0$, and satisfies $ D \cap \{ (r,z) : |z| \leq \e\} =  \{ (r,z) : r\in [0, 1] , |z|\leq \e\} $ for some $\e>0$. Formally, $\td D$ is a cylinder near $z=0$. Secondly, almost the same analysis can be applied to prove finite time blowup of the 3D axisymmetric Euler equations in a domain outside the cylinder $\{ (r,z) : r \geq 1, z\in \R\}$. The proof is easier since the domain is away from the symmetry axis and the term $-\f{1}{r} \td \psi + \f{1}{r^2} \td \psi$ in $\cL \td \psi$ is of lower order ( $r^{-1}\leq 1$ ).

Thirdly, our method of analysis can be applied to prove the finite time blowup of the following modified 2D Boussinesq equation on the {\it whole} space for $C^{\al}$ initial data $\om, \f{\th}{x}$ with small $\al$:
\[
\om_t +  \uu \cdot \na \om  = \th / x,  \quad \th_t + \uu \cdot  \na \th  =  0 ,\quad u= \na^{\perp} (-\D)^{-1} \om.
\]
The above modified Boussinesq equations with a simplified Biot-Savart law have been studied in \cite{HOANG20187328}, \cite{KISELEV201834}. Note that the above equations are a closed system for $\om, \th/x$. We can derive the corresponding dynamic rescaling formulation for the above system and reformulate problem using the $(R,\b)$ variables. We consider the equations for the variable $\Om(R,\b) = \om(x, y), \eta(R, \b) = (\th / x)(x, y)$. The approximate steady state for $\bar{\Om}, \bar{\eta}$ is similar to \eqref{eq:profile} with $\cos(\b)^{\al}$ replaced by $(\sin(2\b))^{\al/2}$, which is $C^{\al/2}$ globally on $\R^2$. Moreover, the scaling parameters are $\bar{c}_l = \f{1}{\al}, \bar{c}_{\om} = -1$. The leading order part of the linearized operator of this system is exactly the same as that in \eqref{eq:lin21}-\eqref{eq:lin22}. The same analysis in Section \ref{sec:lin}-\ref{sec:non} applies to the above system and the proof is much easier since the $\th_y$ variable appeared in \eqref{eq:lin21}-\eqref{eq:lin23} is not present in this  system. 

We would like to point out that the results presented in this paper do not provide a full justification of the finite time singularity of the 3D axisymmetric Euler equations with solid boundary considered in \cite{luo2013potentially-1,luo2013potentially-2}. The method of analysis presented here relies heavily on the assumption that the initial velocity field is in $C^{1,\alpha}$ with a small $\alpha$. Under this assumption, several important nonlocal terms in the perturbation analysis can be made arbitrarily small by choosing a sufficiently small $\alpha$. For smooth initial data considered in \cite{luo2013potentially-1,luo2013potentially-2},  it is almost impossible to obtain an analytic expression of an approximate steady state with a small residual error for the dynamic rescaling equations. Even if we use a numerically constructed approximate steady state, there are several essential difficulties in proving nonlinear stability of this approximate steady state.
In particular, the most difficult part is to control the nonlocal terms in the linear stability analysis. The standard energy estimates are simply too crude to control the nonlocal terms. 

Recently, in collaborator with De Huang, we have been able to prove the finite time self-similar singularity of the HL model with $C_c^{\infty}$ initial data by using the method of analysis presented in \cite{chen2019finite} and a computer assisted analysis. We are now working to extend this computer assisted analysis to prove the finite time self-similar singularity of the 2D Boussinesq and 3D axisymmetric Euler equations in the presence of boundary with smooth initial data in the same setting as that considered in  \cite{luo2013potentially-1,luo2013potentially-2}. We will report these results in a forthcoming paper.

\vspace{0.2in}
{\bf Acknowledgments.} The research was in part supported by NSF Grants DMS-1613861, DMS-1907977 and DMS-1912654. We would like to thank De Huang for his stimulating discussion on and contribution to Lemma \ref{lem:biot1}. We are grateful to Dongyi Wei for telling us the estimate of the mixed derivative terms related to Proposition \ref{prop:prod1}. We would also like to thank Tarek Elgindi, Dongyi Wei and Zhifei Zhang for their valuable comments and suggestions on our earlier version of the manuscript. We are also grateful to the two referees for their constructive comments on the original manuscript, which improve the quality of our paper.


\appendix
\section{}

In Appendix \ref{app:gam}, we estimate $\G(\b)$ and the constant $c$ appeared in the approximate profile \eqref{eq:profile}. In Appendix \ref{app:comp}, we perform the derivations and establish several inequalities in the linear stability analysis in Section \ref{sec:L2_less}.
In Appendix \ref{app:singular}, we derive the singular term \eqref{eq:G} in the elliptic estimates.
In Appendix \ref{sec:l12}, we will establish several estimates of $L_{12}(\Om)$
that are used frequently in the nonlinear stability analysis. 
Notice that we only have the formula of $\bar{\eta} = \bar{\th}_x$ in \eqref{eq:profile}. We need to recover $\bar{\th}, \bar{\xi} =\bar{\th}_y$ from $\bar{\eta}$ via integration. Yet, we do not have a simple formula  to perform integration. Alternatively, we derive useful estimates for $\bar{\xi}$ in Appendix \ref{sec:xi}. Some estimates of $\bar{\Om}, \bar{\eta}$ are also obtained there. 
In Appendix \ref{app:lemma}, we show that the truncation of the approximate steady state would contribute only to a small perturbation under the norm we use, and we prove Lemma \ref{lem:biot1}.
In Appendix \ref{app:stream}, we prove Lemma \ref{lem:biot1}. In Appendix \ref{app:toy}, we study the toy model introduced in \cite{elgindi2019finite}.

\subsection{Estimates of $\G(\b)$ and the constant $c$}\label{app:gam}
\begin{lem}\label{lem:one}
For $x \in [0,1]$, the following estimate holds uniformly for $\lam \geq 1 / 10$,
\beq\label{eq:one}
(1-x^{\kappa}) x^{\lam}  \leq \f{\kappa}{\lam}.
\eeq
Consequently, for $\b \in [0, \pi /2]$, $ 2 \geq \lam \geq 1/10 $, we have 
\[
 | ( \G(\b) - 1) (\sin(2\b))^{\lam} | 
 \les  | (\cos^{\al}(\b) - 1) (\cos(\b))^{\lam} |  \les \al,
\]
and
\[
  \B|c -  \f{2}{\pi} \B| = \B|\f{2}{\pi} \int_0^{\pi/2} (\G(\b) - 1) \sin(2\b) d\b \B| 
\leq 2 \al.
\]
\end{lem}

\begin{proof}
Using change of a variable $t = x^{\kappa}$, it suffices to show that for $t \in [0, 1],
(1 - t ) t^{ \lam/ \kappa} \leq \f{\kappa}{\lam}.$ Notice that $\lam \geq 1/10$ and $t \leq 1$. Using Young's inequality, we derive 
\[
\bal
&(1 - t ) t^{ \lam/ \kp} = \f{\kp}{\lam} \cdot (  \f{ \lam}{\kp}(1 -t)  ) t^{\lam / \kp}
\leq \f{\kp}{\lam}  \lt(  \f{ \f{ \lam}{\kp}(1 -t) +  \f{\lam}{\kp} t  }{ 1 + \f{\lam}{\kp}} \rt)^{1 + \lam / \kp}
=  \f{\kp}{\lam} \lt( \f{ \lam}{\lam +\kp}  \rt)^{1 +\lam /\kp} \leq  \f{\kp}{\lam} ,
\eal
\]
which implies \eqref{eq:one}. The remaining inequalities in the Lemma follows directly from \eqref{eq:one}.
\end{proof}

\subsection{Computations in the linear stability analysis}\label{app:comp}
We perform the derivations and establish several inequalities in the linear stability analysis in Section \ref{sec:L2_less}.

The calculations and estimates presented below can also be verified using {\it Mathematica} 
\footnote{The Mathematica code for these calculations can be found via the link 
\url{https://www.dropbox.com/s/y6vfhxi3pa8okvr/Calpha_calculations.nb?dl=0}.}
since we have simple and explicit formulas.

\subsubsection{ Derivations of \eqref{eq:L2_2}}\label{comp:L2_dp}

Recall the formulas of $\psi_0, \vp_0$ in \eqref{wg:L2_R0}. A direct calculation yields 
\[
\bal
  & \f{1}{2}(R\vp_0)_R  - \vp_0
= \B(\f{1}{2}  \B(  R \cdot \f{ (1+R)^3}{R^3} \B)_R -  \f{(1+R)^3}{R^3} \B) \sin(2\b)\\
=&  \B(  \f{1}{2} \B( -2R^{-3} -  3 R^{-2} + 1  \B) - \f{(1+R)^3}{R^3} \B)\sin(2\b) 
   = - \B(   2R^{-3} + \f{9}{2}  R^{-2}  + 3 R^{-1} + \f{1}{2} \B) \sin(2\b).
\eal
\]

Denote $\psi_0 = A(R) \G(\b)^{-1}$. For the coefficient in the $\eta$ integral in \eqref{eq:L2_2}, we have
\[
\f{1}{2} (R\psi_0 )_R + (- 2 + \f{3}{1+R}) \psi_0
= \B(\f{1}{2}  (R A(R))_R
+ (- 2 + \f{3}{1+R}) A(R)  \B)\G(\b)^{-1} \teq ( I + II) \G(\b)^{-1}.
\]

Note that $A(R) =\f{3}{16}  \lt( \f{(1+R)^3}{R^4}  + \f{3}{2} \f{(1+R)^4}{R^3} \rt) $ \eqref{wg:L2_R0}. A direct calculation implies 
\[
\bal
 I & =\f{3}{32}\lt(  \f{(1+R)^3}{R^3} + \f{3}{2} \f{(1+R)^4}{R^2}  \rt)_R 
= \f{3}{32}\lt(3 \f{ (1+R)^2}{R^3} - 3\f{(1 + R)^3}{R^4} + 6 \f{(1+R)^3}{R^2} - 3	 \f{(1+R)^4}{R^3} \rt) \\
& = \f{3}{32}\lt( \f{(1+R)^2}{R^4} ( 3 R - 3(1+R) + 6(1+R)R^2 - 3(1+R)^2 R )  \rt) =  \f{3(1+R)^2}{32R^4} ( -3 - 3 R + 3R^3 )  , \\
II & =   \lt(-2 + \f{3}{1+R} \rt)  \f{3}{32} \lt( 2 \f{(1+R)^3}{R^4}  + 3 \f{(1+R)^4}{R^3}  \rt)  = \f{ 3 (1+R)^2}{ 32 R^4} (  -2 -2 R + 3 ) ( 2 + 3R(1+R)) , \\
I + II & =  \f{3 (1+R)^2}{32R^4} ( -3 - 3 R + 3R^3 + (1-2R)( 2 + 3R +3R^2) )
= \f{3(1+R)^2}{32R^4} ( -1 - 4 R -3R^2 -3 R^3 ).
\eal
\]
The above calculations imply \eqref{eq:L2_2}.

\subsubsection{ Derivations of \eqref{eq:L2_42} }\label{comp:L2_cw1}
From \eqref{eq:profile}, we know 
\[
\f{\bar{\eta} - R\pa_R \bar{\eta} }{\bar{\eta}} =  \f{(1+R)^3}{ 6R} \B( \f{6R}{ (1+R)^3} 
 - R\cdot \f{6}{(1+R)^3} + R \cdot \f{18 R}{ (1+R)^4}  \B) = \f{3R}{1+R}.
\]
Using the above identity, \eqref{wg:L2_R0} and $c_{\om} = -\f{2}{\pi\al} L_{12}(\Om)(0)$ \eqref{eq:normal}, we can compute 
\[
\bal
& (\bar{\eta} - R\pa_R \bar{\eta} )  \psi_0 c_{\om}
=  \f{\bar{\eta} - R\pa_R \bar{\eta} }{\bar{\eta}} \f{9}{8}\f{\al}{c} 
\B(  R^{-3} + \f{3}{2} \f{1+R}{R^2}    \B) c_{\om}
= \f{27\al}{8c}  \f{R}{1+R}  \B(  R^{-3} + \f{3}{2} \f{1+R}{R^2}    \B) c_{\om} \\
 =&\B( \f{27\al}{8c}  \f{1}{(1+R)R^2}  + \f{81 \al}{ 16 c} \f{1}{R} \B) \cdot \f{-2}{\pi\al} L_{12}(\Om)(0)
=  \B(-\f{27}{4\pi c} \f{1}{(1+R)R^2}  - \f{81}{8\pi c} \f{1}{R}\B) L_{12}(\Om)(0),
\eal
\]
which implies \eqref{eq:L2_42}.

\subsubsection{ Derivation of the ODE \eqref{eq:L2_44} for $L_{12}(\Om)(0)$}\label{comp:L2_cw2}

Multiplying $ \sin(2\b) / R$ on both sides of \eqref{eq:lin21} and then integrating \eqref{eq:lin21}, we derive 
\[
\bal
\f{d}{dt} L_{12} (\Om)(0)
=&  - \B\la R\pa_R \Om , \f{\sin(2\b)}{R} \B\ra  - L_{12}(\Om)(0) 
+ c_{\om} \B\la \bar{\Om} - R \pa_R \bar{\Om},  \f{\sin(2\b)}{R} \B\ra \\
 &+ \B\la \eta,  \f{\sin(2\b)}{R} \B\ra  - \B\la \f{3}{1+R} D_{\b} \Om,\f{\sin(2\b)}{R} \B\ra 
 + \B\la \cR_{\Om}, \f{\sin(2\b)}{R} \B\ra.
 \eal
\]
The first term vanishes by an integration by parts argument. Using \eqref{eq:profile} and \eqref{eq:normal},
we can compute the third term
\[
\bal
c_{\om}\B\la \bar{\Om} - R \pa_R \bar{\Om},  \f{\sin(2\b)}{R} \B\ra 
&= \f{\al}{c} c_{\om} \int_0^{\infty}  \int_0^{\pi/2}\G(\b) \f{6R^2}{(1+R)^3}\cdot \f{\sin(2\b)}{R} d \b d R
=\f{\pi \al}{2} c_{\om} \int_0^{\infty} \f{6R}{(1+R)^3} d R \\
 &= 3 \pi \al  c_{\om}\lt( - (1+R)^{-1} + \f{1}{2} (1+R)^{-2} \rt)\B|_0^{\infty} = \f{3\pi \al}{2} c_{\om}
 = -  3 L_{12}(\Om)(0) .
\eal
\]
It follows that
\[
\f{d}{dt} L_{12} (\Om)(0) = - 4 L_{12} (\Om)(0) +  \B\la \eta,  \f{\sin(2\b)}{R} \B\ra  - \B\la \f{3 \sin(2\b)}{(1+R)R} , D_{\b} \Om \B\ra  + \B\la \cR_{\Om}, \f{\sin(2\b)}{R} \B\ra.
\]
Multiplying $ \f{81}{ 4 \pi c}  L_{12}(\Om)(0)$ to the both sides, we derive \eqref{eq:L2_44}.

\subsubsection{ Computations of the integrals in \eqref{eq:L2_520} }\label{comp:L2_cw3}

A simple calculation implies that for any $k > 2$
\beq\label{eq:ing}
\int_0^{\infty}  (1+R)^{-k} d R = \f{1}{k-1}, \quad  \int_0^{\infty} \f{R} {(1+R)^{k}} d R 
= \int_0^{\infty} \f{1}{(1+R)^{k-1}} - \f{1}{(1+R)^k} d R= \f{1}{(k-1)(k-2)}.
\eeq

For the integral in $\b$, we get 
\[
\int_0^{\pi /2} (1 - 2 \sin(2\b))^2  d \b = \f{\pi}{2} - 4 \int_0^{\pi/2} \sin( 2\b ) d \b
+ 4 \int_0^{\pi/2}  ( \sin(2\b) )^2 d \b = \f{\pi}{2} - 4 +  4 \cdot \f{\pi}{4}  
= \f{3\pi}{2} - 4. \\
\]

Using  \eqref{eq:ing} with $k=4$ and the above calculation, we can compute
\[
\B| \B| \f{R^{3/2}}{(1+R)^2} \f{1}{R}(1 - 2  \sin(2\b))\B|\B|_2^2
= \int_0^{\infty} \f{R}{(1+R)^4} dR \cdot \int_0^{\pi /2} (1 - 2 \sin(2\b))^2  d \b 
= \f{1}{6}(\f{3\pi}{2} - 4  ).
\]

For $A_1$ in \eqref{eq:L2_42}, we apply the Cauchy-Schwarz inequality directly to yield 
\[
A_1 = -\f{27}{4\pi c}  L_{12}(\Om) (0)  \B\la  \eta , \f{1 }{(1+R)R^2} \B\ra  
\leq \f{27}{ 4\pi c} | L_{12}(\Om) (0)| \B\la \eta^2, \f{(1+R)^3}{R^{4}}  \B\ra^{1/2}
\B\la \f{R^{4}}{(1+R)^3} , \f{1}{(1+R)^2 R^4}  \B\ra^{1/2} . 
\]

Using \eqref{eq:ing}, we can calculate 
\[
\B| \B| \f{R^2}{(1+R)^{3/2}} \cdot \f{1}{(1+R)R^2}\B|\B|_2^2 
= \int_0^{\pi/2} 1 d\b \cdot \int_0^{\infty} (1+R)^{-5} dR = \f{\pi}{8}.
\]

\subsubsection{ Estimates of $D(\Om), D(\eta)$ and the proof of \eqref{eq:L2_72p}}\label{comp:L2_sum}
We introduce
\[
\bal
D_1(\eta) & \teq - \f{3(1+R)^2}{ 32 R^4} ( 1 + 4 R + 3 R^2 + 3 R^3) , \\
D_2(\eta) & \teq \lt( \f{3}{16} R^{-3}  + \f{3}{8}\f{ (1 + R)^2}{R^2} +  \f{3 R}{4(1+R)} \rt) + \f{3}{16} \lt( \f{1}{6}\f{(1+R)^4}{R^3} + \f{3}{8} \f{(1+R)^3}{R^4} \rt).
\eal
\]

Recall $D(\Om), D(\eta)$ in \eqref{eq:L2_712} and the weights $\vp_0,\psi_0$ defined in \eqref{wg:L2_R0}. By definition, $D(\eta) = D_1(\eta) \G(\b)^{-1}+ D_2(\eta)$. Thus, \eqref{eq:L2_72p} is equivalent to 
\beq\label{eq:L2_72}
\sin(2\b) D(\Om) \leq -\f{1}{6} \vp_0, \quad  D_1(\eta) \G(\b)^{-1}+ D_2(\eta)   \leq -\f{1}{8} \psi_0.
\eeq

To prove the first inequality, it suffices to prove
\[
D(\Om)  = - 2 R^{-3} - \f{9}{2} R^{-2} - 3 R^{-1} - \f{1}{2}
+   \f{4}{3} R^{-3}  + 6 R^{-2} + \f{1+R}{ 3 R}   \leq  - \f{(1+R)^3}{ 6 R^3},
\]
which is equivalent to proving
\[
 (-2 + \f{4}{3} + \f{1}{6}) R^{-3} + (-\f{9}{2} + 6+ \f{1}{2}) R^{-2} + (-3 + \f{1}{3} + \f{1}{2}) R^{-1}
+ (-\f{1}{2} + \f{1}{3} + \f{1}{6})  \leq 0.
\]
It is further equivalent to 
\[
-\f{1}{2} R^{-3} + 2 R^{-2} - \f{13}{6} R^{-1} \leq 0,
\]
which is valid since $2 \sqrt{ \f{1}{2} \times \f{13}{6}} > 2$. Hence, we prove the first inequality in \eqref{eq:L2_72}.

For the second inequality in \eqref{eq:L2_72}, firstly, we use $\G(\b) D_2(\eta) \leq D_2(\eta)$ ($\G(\b) = \cos^{\al}(\b)$ \eqref{eq:profile}) to obtain
\beq\label{eq:L2_721}
\bal
&D_3(\eta) \teq D_1 (\eta) + D_2(\eta) \G(\b)  \leq D_1(\eta) + D_2(\eta)  \\
=&  \f{3}{16} \lt\{ 
- \f{(1+R)^2}{ 2 R^4} ( 1 + 4 R + 3 R^2 + 3 R^3)
+  R^{-3}  + 2 \f{ (1 + R)^2}{R^2} +  \f{ 4 R}{1+R} 
+\f{1}{6}\f{(1+R)^4}{R^3} + \f{3}{8} \f{(1+R)^3}{R^4}
 \rt\}.
 \eal
\eeq

Recall the definition of $\psi_0$ in \eqref{wg:L2_R0}. 
Multiplying both sides of the second inequality in \eqref{eq:L2_72} by $\G(\b)$, we obtain that the inequality is equivalent to
\beq\label{eq:L2_722}
D_3(\eta) \leq \f{3}{16} \lt( - \f{1}{8} \f{(1+R)^3}{R^4} - \f{3}{16} \f{(1+R)^4}{R^3}  \rt).
\eeq
We split the negative term in the upper bound of $D_3(\eta)$ in \eqref{eq:L2_721} as follows 
\[
\bal
& -\f{(1+R)^2}{ 2 R^4} ( 1 + 4 R + 3 R^2 + 3 R^3)
= - \f{(1+R)^2}{2 R^4} \lt\{ (1 +R) + (3 R^2) + R(1+R)^2 + R( 2 -2 R + 2R^2) \rt\}\\
=& - \f{ (1+R)^3}{ 2 R^4} - \f{3}{2} \f{(1+R)^2}{R^2} - \f{(1+R)^4}{ 2 R^3} - \f{ (1+R)^2(1-R + R^2)}{R^3} .
\eal
\]
It follows that
\[
\bal
&D_3(\eta) \leq 
\f{3}{16}\lt\{
\f{ (1+R)^3}{  R^4} \lt(  -\f{1}{2} + \f{3}{8}\rt)
+ \f{(1+R)^4}{R^3} \lt( -\f{1}{2} + \f{1}{6}\rt)  + \f{1}{2} \f{(1+R)^2}{R^2}
- \f{ (1+R)^2(1-R + R^2)}{R^3} \rt. \\
 &\lt.+ \f{1}{R^3} + \f{4 R}{1+R}  \rt\} =\f{3}{16} \lt\{ -\f{1}{8} \f{(1+R)^3}{R^4} - \f{1}{3}  \f{(1+R)^4}{R^3}
 + \f{1}{2}\f{(1+R)^2}{R^2}- \f{(1+R)(1+R^3)}{R^3}  + \f{1}{R^3} + \f{4R}{1+R}  \rt\}.
\eal
\]
Observe that
\[
\bal
&-\f{1}{3}\f{(1+R)^4}{R^3} + \f{1}{2} \f{(1+R)^2}{R^2} 
= -\f{3}{16}\f{(1+R)^4}{R^3} +  \lt(- \f{7}{48} \f{(1+R)^4}{R^3} + \f{1}{2} \f{(1+R)^2}{R^2} \rt)
\leq -\f{3}{16}\f{(1+R)^4}{R^3} , \\
&- \f{(1+R)(1+R^3)}{ R^3}  + \f{1}{R^3} + \f{ 4 R}{1+R}  =  -\f{1}{R^2} -(1+R) + \f{4 R}{1+R}
= -\f{1}{R^2}-\f{ (R-1)^2}{ (1+R)}  \leq 0  ,
\eal
\]
where we have used $ \f{7}{48} \f{(1+R)^2}{R} \geq  \f{7}{48} \times 4  \geq 1/ 2$ to derive the first inequality. Therefore, we prove \eqref{eq:L2_722}, which further implies the second inequality in \eqref{eq:L2_72}.

\subsection{Derivation of the singular term \eqref{eq:G} in the elliptic estimates}\label{app:singular}

Suppose that $\Psi$ is the solution of \eqref{eq:elli}. Consider $\td{\Psi} = \Psi + G \sin(2\b)$. Notice that if $\al = 0$, $\sin(2\b)$ is the kernel of the operator $\cL_{\al}$ in \eqref{eq:elli}
(it is self-adjoint if $\al =0$). We have 
\[
\cL_{\al}(\td{\Psi}) = \Om + \cL_{\al}(G \sin(2\b)) 
= \Om - (\al^2 R^2 \pa_{RR} G  + \al (\al + 4) R \pa_R G ) \sin(2\b).
\]
We look for $G(R)$ that satisfies $G(R) \to 0$ as $R \to +\infty$ and $\cL_{\al}(\td{\Psi})$ is orthogonal to $\sin(2\b)$: 
\[
0 = \int_0^{\pi/2} \sin(2\b) ( \Om - (\al^2 R^2 \pa_{RR} G  + \al (\al + 4) R \pa_R G ) \sin(2\b)) d \b 
\]
for every $R$, which implies 
\beq\label{eq:elli_G}
 \al^2 R^2 \pa_{RR} G + \al(\al+4) R \pa_R G = \f{4}{\pi} \Om_*,
\eeq
where $\Om_*(R) = \int_0^{\pi/2} \Om(R, \b) \sin(2\b) d \b$ and we have used $\int_0^{\pi/2} \sin^2(2\b) d \b = \f{\pi}{4}$. The above ODE is first order with respect to $\pa_R G$ and can be solved explicitly. Multiplying the integrating factor $\f{1}{\al^2} R^{-2 + \f{4+\al}{\al}}$ to both sides and then integrating from $0$ to $R$ yield
\[
R^{\f{4+\al}{\al}} \pa_R G =  \f{4}{\al^2 \pi } \int_0^{R} \Om_*(t) t^{ \f{4}{\al} -1} d t.
\]
Imposing the vanishing condition $G(R) \to 0$ as $R \to +\infty$, we yield 
\[
G = -\f{4}{\al^2 \pi } \int_R^{\infty} s^{ - \f{4+\al}{\al}} \int_0^s \Om_*(t) t^{ \f{4}{\al} -1} d t ds.
\]
Using integration by parts, we further derive
\[
G = \f{1} {\al\pi} \int_R^{\infty} \pa_s ( s^{ -\f{4}{\al}}) \int_0^s \Om_*(t) t^{ \f{4}{\al} -1} d t ds
= -\f{1}{ \al \pi} \int_R^{\infty} \f{\Om_*(s)}{s}  ds  - \f{1}{\al \pi} R^{ -\f{4}{\al}} \int_0^{R} \Om_*(s) s^{\f{4}{\al} - 1} ds.
\]
Using the above formula and the notation $L_{12}(\Om)$ \eqref{eq:biot3}, we derive \eqref{eq:G}.

\subsection{Estimates of $L_{12}(\Om)$}\label{sec:l12}


Recall $\td{L}_{12}(\Om) = L_{12}(\Om) - L_{12}(\Om)(0)$.  We have the following important cancellation between $\td L_{12}(\Om)$ and $\Om$.

\begin{lem}\label{lem:cancel}
For $ k \in [3/2,4]$ and any $\lam > 0$, we have 
\beq\label{eq:cancel}
\bal
\la \sin(2\b) \Om \td{L}_{12}(\Om), R^{-k} \ra& =  - \f{k-1}{2}  \B| \B| \td{L}_{12}(\Om) R^{-k/2} \B| \B|^2_{L^2{(R)}} ,\\
\la (\sin(2\b) \Om + \lam \td{L}_{12}(\Om))^2 , R^{-k} \ra &=  \la R^{-k} (\sin(2\b) )^2,  \Om^2 \ra - ( (k-1) \lam -  \f{\pi}{2} \lam^2 ) \B| \B| \td{L}_{12}(\Om) R^{-k/2} \B| \B|^2_{L^2{(R)}} .
\eal
\eeq
\end{lem}
\begin{proof}
From the definition of $\td{L}_{12}(\om)(R)$ in \eqref{eq:nota_ux}, we know that it does not depend on $\b$ and 
\[
\int_0^{\pi / 2} \Om(s, \b) \sin(2\b) d\b = - (\pa_{R} \td{L}_{12}(R) ) R.
\]
Using integration by parts, we obtain 
\[
\la \sin(2\b) \Om \td{L}_{12}(\Om), R^{-k} \ra = \int_0^{\infty}  (- (\pa_{R} \td{L}_{12}(R) ) R ) \td{L}_{12}(\Om) R^{-k} d R 
= - \f{k-1}{2} \int_0^{\infty} \td{L}_{12}(\Om)^2 R^{-k} d R,
\]
which is exactly the first identity in \eqref{eq:cancel}. The second identity in \eqref{eq:cancel} is a direct consequence of 
$\la \td{L}^2_{12}(\Om), R^{-k}\ra = \f{\pi}{2} || \td{L}_{12}(\Om) R^{-k/2} ||^2_{L^2{(R)}} $
and the first identity. 
\end{proof}

To estimate $\td{L}_{12}(\Om) g$ in $\cL_i$, we use the following simple Lemma.
\begin{lem}\label{lem:ux}
Let $g$ be some function depending on $\bar{\Om}, \bar{\eta}, \bar{\xi}$ and $\vp$ be some weights.
We have 
\beq\label{eq:ux}
\bal
\la \td{L}^2_{12} (\Om) g^2 , \vp \ra & \les ||R^{-1} \td{L}_{12} (\Om) ||^2_{L^2(R)} 
  \B| \B| \int_0^{\pi/2} R^2 g^2(R, \b)  \vp(R, \b) d \b  \B|\B|_{L^{\infty}(R)}  , \\
\la ( D^k_R \td{L}_{12} (\Om) )^2 g^2 , \vp \ra & \les ||  R^{-1} D^{k-1}_R\Om   ||^2_{L^2} 
  \B| \B| \int_0^{\pi/2} R^2 g^2(R, \b)  \vp(R, \b) d \b  \B|\B|_{L^{\infty}(R)}  , \\
\eal
\eeq
for $k\geq 1$, 
provided that the upper bound is well-defined, where $D_R  = R \pa_R$.
\end{lem}


\begin{proof}
The first inequality follows directly from that $\td{L}_{12}(\Om)$ does not dependent on $\b$.
Recall the definition of $\td{L}_{12}(\Om)$ in \eqref{eq:nota_ux} and $D_R = R \pa_R$. Notice that for $k \geq 1$, we have
\[
\bal
D_R^k \td{L}_{12}(\Om) &= - \int_0^{\pi / 2} D^{k-1}_R\Om(R, \b) \sin(2\b) d\b.
\eal
\]
Using the Cauchy-Schwarz inequality, we prove
\[
\bal
&\la ( D^k_R \td{L}_{12} (\Om) )^2 g^2 , \vp \ra 
= \int_0^{\infty} \lt(   ( \int_0^{\pi / 2} D_R^{k-1}\Om(R, \b) \sin(2\b) d\b )^2  \int_0^{\pi /2} g^2 \vp d \b \rt) d R \\
\les &   
 \int_0^{\infty} ( \int_0^{\pi/2} (D_R^{k-1}\Om)^2 d \b ) ( \int_0^{\pi /2} g^2 \vp d \b ) d R
 \leq  ||   R^{-1} D_R^{k-1} \Om ||^2_{L^2} 
  \B| \B| \int_0^{\pi/2} R^2 g^2 \vp (R, \b) d \b  \B|\B|_{L^{\infty}(R)}.
 \eal
\]
\end{proof}


\begin{lem}\label{lem:l12}
Let $\chi(\cdot) : [0, \infty) \to [0, 1]$ be a smooth cutoff function, such that $\chi(R) = 1$ for $R \leq 1$ and $\chi(R) = 0$ for $R \geq 2$. For $k=1,2$, we have
\beq\label{eq:l12}
\bal
& 
 || L_{12}(\Om) ||_{L^{\infty}}\les || \f{1+R}{R} \Om||_{L^2}, \quad || \td{L}_{12}(\Om) (R^{-2} + R^{-3})^{1/2} ||^2_{L^2(R)}  \les || \Om \f{(1+R)^2}{R^2} ||^2_{L^2} ,
 \\
& || L_{12}(\Om)||_2 \les || \Om||_2, \quad 
|| \f{(1+R)^k}{R^k} ( L_{12}(\Om) - L_{12}(\Om)(0) \chi ) ||_{L^2(R)} \les || \f{(1+R)^k}{R^k} \Om||_{L^2}.
\eal
\eeq
provided that the right hand side is bounded. 
Moreover, if $\Om \in \cH^3$, then for $ 0 \leq k \leq 3, 0\leq l \leq 2$, we have 
\beq\label{eq:l12X}
\bal
& ||   L_{12}(\Om) - L_{12}(\Om)(0) \chi ||_{\cH^3} + ||  D_R( L_{12}(\Om) - L_{12}(\Om)(0) \chi) ||_{\cH^3} \les || \Om ||_{\cH^3},   \\
& 
|| D^k_R L_{12}(\Om) ||_{\infty} 
+ || D^k_R ( L_{12}(\Om) -\chi L_{12}(\Om)(0)) ||_{\infty} \les || \Om||_{\cH^3},  \\
&
|| (1 +R) \pa_R D^l_R L_{12}(\Om) ||_{\infty} 
+ || (1+R) \pa_R D^l_R ( L_{12}(\Om) -\chi L_{12}(\Om)(0)) ||_{\infty} \les || \Om||_{\cH^3},  \\
& || L_{12}(\Om) ||_{X} + || D_R L_{12}(\Om) ||_X \les  || \Om||_{\cH^3},
\eal
\eeq
where $X \teq \cH^3 \oplus \cW^{5,\infty}$ is defined in \eqref{norm:X}.
\end{lem}
\begin{remark}
We subtract $\chi L_{12}(\Om)(0)$ near $R=0$ since $L_{12}(\Om)$ does not vanishes at $R=0$.
\end{remark}

\begin{proof}
Recall $L_{12}(\Om)$ in \eqref{eq:biot3} and $\td{L}_{12}(\Om)$ in \eqref{eq:nota_ux}. Using the Cauchy-Schwarz and the Hardy inequality, we get
\beq\label{eq:l12_pf1}
\bal
|| L_{12}(\Om)||_{L^{\infty}}& \les \la |\Om|, \f{1}{R} \ra \les || \f{1+R}{R} \Om||_{L^2} || \f{1}{1+R}||_{L^2(R) }  \les || \f{1+R}{R} \Om||_{L^2}  ,  \\
 || \f{1}{R^l} \td{L}_{12}(\Om) ||_{L^2(R)}  &\les \int_0^{\infty} \f{1}{R^{2l}} \td{L}^2_{12}(\Om) d R
\les \int_0^{\infty} \f{1}{R^{2l-2}} (\pa_R \td{L}_{12}(\Om) )^2 d R \les \la \Om^2, R^{-2l} \ra,
\eal
\eeq
for $l = 1,\f{3}{2},2$, which implies the first two inequalities in \eqref{eq:l12}. For $k=1,2$, observe that 
\[
\bal
&|| \f{(1+R)^k}{R^k} ( L_{12}(\Om) - L_{12}(\Om)(0) \chi ) ||_{L^2(R)} 
\les || \f{(1+R)^k}{R^k} \td{L}_{12}(\Om)  \chi||_{L^2(R)}
+ || \f{(1+R)^k}{R^k} L_{12}(\Om) (1-  \chi )||_{L^2(R)}   \\
\les &  || \f{1}{R^k} \td{L}_{12}(\Om)  ||_{L^2(R)}  + ||  L_{12}(\Om) ||_{L^2(R)}
\les  || \Om \f{(1+R)^k}{R^k} ||_{L^2} +   ||  L_{12}(\Om) ||_{L^2(R)} ,
\eal
\]
where we have used  \eqref{eq:l12_pf1} in the last inequality. Denote $\Om_* = \int_0^{\pi/2} \Om d \b$. From \eqref{eq:biot3}, we know 
\[
L_{12}(\Om)(R) = \int_R^{\infty} \f{\Om_*(S)}{S}  d S = \int_0^{\infty} K(R, S) \Om_*(S) d S, 
\quad K(R, S) = \f{1}{S} \one_{R \leq S}.
\]
The $L^2$ boundedness of $L_{12}$ is standard. Notice that $K$ is homogeneous of degree $-1$, i.e.
$K(\lam R, \lam S) = \lam^{-1} K(R, S)$ for $\lam >0$. Using change of a variable $S = R z$
, we get
\[
L_{12}(\Om)(R) = \int_0^{\infty} \f{1}{R} K(1, z) \Om_*( R z) R dz =  \int_0^{\infty}  K(1, z) \Om_*( R z)  dz.
\]
Then, the Minkowski inequality implies 
\[
|| L_{12}(\Om)||_{L^2} \leq \int_0^{\infty} K(1, z) || \Om_*( R z)||_{L^2(R)}  dz
\les \int_0^{\infty} K(1, z) z^{-1/2} ||\Om||_{L^2}  dz = || \Om||_{L^2}\int_{z \geq 1} z^{-3/2} dz \les || \Om||_{L^2}.
\]
We complete the proof of \eqref{eq:l12}. Notice that $D_R L_{12}(\Om) = - \Om_* , || D_R^k \chi||_{L^2} \les 1$ for $1 \leq k\leq 4$ and $D_{\b} L_{12}(\Om) = 0, D_{\b} \chi = 0$. 
Using that $\sin(2\b)^{-\s}$ in the weight $\vp_1 = \sin(2\b)^{-\s} \f{(1+R)^4}{R^4}$ is integrable in the $\b$ direction and \eqref{eq:l12}, we yield
\beq\label{eq:L12Hk}
\bal
 &|| (L_{12}(\Om) - L_{12}(\Om)(0) \chi ) \vp_1^{1/2} ||_{L^2} 
 +  || D_R^k (L_{12}(\Om) - L_{12}(\Om)(0) \chi ) \vp_1^{1/2} ||_{L^2} \\
 \les &|| (L_{12}(\Om) - L_{12}(\Om)(0) \chi ) \f{(1+R)^2}{R^2} ||_{L^2} 
+|| D_R^k (L_{12}(\Om) - L_{12}(\Om)(0) \chi )\f{(1+R)^2}{R^2} ||_{L^2}  \\
\les & || \Om \f{(1+R)^2}{R^2}||_{L^2} 
+ || D^{k-1}_R \Om_* \f{(1+R)^2}{R^2} ||_{L^2} 
+ |L_{12}(\Om)(0)| || D_R^k \chi \f{(1+R)^2}{R^2}||_{L^2} \\
\les & || \Om \f{(1+R)^2}{R^2}||_{L^2} + || D_R^{k-1} \Om \f{(1+R)^2}{R^2}||_{L^2} \les || \Om||_{\cH^3},
\eal
\eeq
which implies the first estimate in \eqref{eq:l12X}. From the definition of $L_{12}(\Om)$ in \eqref{eq:biot3}, we have $D_R L_{12}(\Om)= L_{12}(D_R \Om)$. Notice that $| D_R^k \chi(R) | \les 1$. Using \eqref{eq:l12}, we prove for $ k \leq 3$
\[
|| D_R^k L_{12}(\Om)||_{L^{\infty}} +  | L_{12}(\Om)(0)| \cdot || D_R^k \chi||_{L^{\infty}} \les || \Om||_{\cH^3},
\]
which implies the second estimate in \eqref{eq:l12X}. Similarly, since $\pa_R D_R^l L_{12}(\Om) = \pa_R L_{12}(D_R^l \Om) =  - R^{-1} D_R^l \Om_*(R) $, where $\Om_*(R) = \int_0^{\pi/2} \Om(R, \b) d \b$, and that $l \leq 2$, we have
\[
|| \pa_R D_R^l L_{12}(\Om) ||_{L^{\infty}} =
||  R^{-1} D_R^l \Om_*||_{L^{\infty}(R)} \les || R^{-1}  D_R^l \Om_* ||^{1/2}_{L^2(R) }
|| \pa_R (R^{-1}  D_R^l \Om_* ) ||^{1/2}_{L^2(R) } 
\les || \Om||_{\cH^3},
\]
which along with the second estimate in \eqref{eq:l12X} and $| \pa_R D_R^l \chi L_{12}(\Om)(0) | \les | L_{12}(\Om)(0) | \les || \Om||_{\cH^3} $ completes the proof of the third estimate in \eqref{eq:l12X}.

Since $\chi L_{12}(\Om)(0)$ does not depend on $\b$, we apply the first two estimates in \eqref{eq:l12X} to yield 
\[
\bal
|| D_R^i L_{12}(\Om) ||_X
&\leq || D_R^i ( L_{12}(\Om) - \chi L_{12}(\Om)(0)) ||_{\cH^3} 
+  || D_R^i \chi L_{12}(\Om)(0) ||_{\cW^{5,\infty}} \\
& \les || \Om||_{\cH^3} + | L_{12}(\Om)(0) | \les || \Om||_{\cH^3}
\eal
\]
for $i= 0, 1$. We complete the proof of \eqref{eq:l12X}.
\end{proof}

\subsection{Estimate of the approximate self-similar solution}\label{sec:xi} In appendix \ref{sec:omth}, we estimate some norm of $\bar{\Om}, \bar{\eta}$ using the explicit formulas. 
For $\bar{\xi}$, it is given by an integration of $\bar{\eta}$ that does not have an explicit formula. We estimates $\bar{\xi}$, its derivatives and some norm in subsection \ref{sec:sbxi}.

\subsubsection{Estimate of $\bar{\Om}, \bar{\eta}$}\label{sec:omth}
Recall the formula of $\bar{\Om}, \bar{\eta}$ in \eqref{eq:profile}. A simple calculation yields 
\beq\label{eq:bar0}
\bar{\Om} = \f{\al}{c}  \f{3 R \G(\b) }{(1+R)^2},  \
\bar{\eta} = \f{\al}{c} \f{6 R \G(\b)  }{(1+R)^3},  \ 
\bar{\Om} - D_R \bar{\Om} = \f{\al}{c}  \f{6R^2 \G(\b) }{(1+R)^3}, \
\bar{\eta} - D_R \bar{\eta} = \f{\al}{c}  \f{18R^2 \G(\b) }{(1+R)^4} .
\eeq
Without specification, in later sections, we assume that $R \geq 0, \b \in [0,\pi/2]$. 
\begin{lem}\label{lem:bar}
The following results apply to any $ k \leq 3, 0 \leq i + j \leq 3, j \neq 1$.
(a) For $f = \bar{\Om}, \bar{\eta}, \bar{\Om} - D_R \bar{\Om}, \bar{\eta} - D_R \bar{\eta}$, we have
\beq\label{eq:bar}
| D_R^k f | \les f , \quad  |D_R^i D^j_{\b} f | \les \al \sin(\b) f.
\eeq
(b) Let $\vp_i$ be the weights defined in \eqref{wg}. For $g = \bar{\Om}, \bar{\eta}$, we have
\beq\label{eq:bar_ux}
\int_0^{\pi /2} R^2 (D_R^k g )^2 \vp_1 d \b \les \al^2,  \quad
\int_0^{\pi /2} R^2 (D_R^i D^j_{\b} g )^2 \vp_2  d \b \les \al^3, 
\eeq
uniformly in $R$ and 
\beq\label{eq:bar_ing}
\la  (D^k_R  ( g - D_R g   ) )^2 , \vp_1 \ra  \les \al^2, \quad
\la  (D^i_R D^j_{\b} ( g - D_R g   ) )^2 , \vp_2  \ra  \les \al^3.
\eeq
\end{lem}

\begin{proof}
Recall $D_{\b}= \sin(2\b) \pa_{\b}, D_R = R \pa_R$. Using $\G(\b)=\cos(\b)^{\al}$, \eqref{eq:Dg} and a direct calculation gives 
\beq\label{eq:gamR}
| D^j_{\b} \G(\b) | \les \al \sin(\b) \G(\b), \quad  | D_R^i \f{R}{(1+R)^m} | \les \f{R}{ (1+R)^m},
\quad | D_R^i \f{R^2}{(1+R)^m} | \les \f{R^2}{ (1+R)^m}.
\eeq
for $ 1 \leq j \leq 5$, $ 0\leq i \leq 5$ and $m = 2,3, 4 $. Combining these estimates and the formulas in \eqref{eq:bar0} implies \eqref{eq:bar}. As a result, we have the following pointwise estimates for $g = \bar{\Om}$ or $\bar{\eta}$
\[
\bal
&|D_R^k  g| \les g \les  \al \G(\b) \f{R}{(1+R)^2}, \quad 
|D_R^i D^j_{\b}  g | \les \al\sin(\b) g \les  \al^2  \sin(\b)\G(\b) \f{R}{(1+R)^2},  \\
&|D_R^k (g - D_R g) | \les g - D_R g \les  \al \f{R^2 \G(\b) }{(1+R)^3} ,
|D_R^i D^j_{\b}  (g - D_R g) | \les \al \sin(\b) (g - D_R g) \les  \al^2 \sin(\b)\f{R^2 \G(\b) }{(1+R)^3}, 
\eal
\]
for $k \leq 3$, $i+j \leq 3, j \neq 0$, where we have used $ c \approx \f{2}{\pi}$ in Lemma \ref{lem:one}.
Recall $\vp_i$ in Definition \ref{def:wg}.
\[
\vp_1 \teq (1+R)^4 R^{-4} \sin(2\b)^{ - \s}, \quad 
\vp_2 \teq (1+R)^4R^{-4}  \sin(2\b)^{ - \g} .
\]
Notice that for $\s = \f{99}{100}, \g = 1 + \f{\al}{10}$, we have 
\[
\int_0^{\pi/2} \G(\b)^2 \sin(2\b)^{-\s} d \b \les 1, \quad
\int_0^{\pi/2} \al^2 \sin(\b)^2\G(\b)^2 \sin(2\b)^{-\g} d \b 
\les \al^2 \int_0^{\pi/2} \cos(\b)^{2\al -1 - \al/ 10} d \b  \les \al.
\]
Combining the pointwise estimates, the estimates of the angular integral and a simple calculation then gives \eqref{eq:bar_ux}, \eqref{eq:bar_ing}.
\end{proof}

Recall the $\cW^{l,\infty}$ norm in \eqref{norm:W}. We have 
\begin{prop}\label{prop:gam} It holds true that
$\G(\b) ,\bar{\Om}, \bar{\eta} \in \cW^{7, \infty}$ with
\[
\bal
&|| \G(\b)||_{\cW^{7,\infty}} \les 1, \quad
||\f{(1+R)^2}{R} \bar{\Om}||_{\cW^{7,\infty}}  + || \f{(1+R)^2}{R}\bar{\eta} ||_{\cW^{7,\infty}} \les \al ,  \\
&|| D_{\b} \bar{\Om}||_{\cW^{7,\infty}}  + || D_{\b}\bar{\eta} ||_{\cW^{7,\infty}} \les \al^2 .
\eal
\]
\end{prop}
\begin{proof}
The proof follows directly from the calculation \ref{eq:gamR} and $\sin(\b) \G(\b) \sin(2\b)^{-\al/5} \les 1$.
\end{proof}

\subsubsection{Estimates of $\bar{\xi}$ }\label{sec:sbxi}
Recall that the approximate self-similar profile $\bar{\eta}$ \eqref{eq:profile} is given by 
\beq\label{eq:eta}
\bal
(\bar{\th}_x)(x, y) & = \bar{\eta}(R, \th) = \f{\al}{c} \f{6R}{(1+R)^3} \cos^{\al}(\b) = 
\f{ 6  \al}{c} \f{ x^{\al}}{(1+(x^2 +y^2)^{\al/2})^3} .
\eal
\eeq
We also use $\bar{\eta}(x, y)$ to denote the above expression.
Throughout this section, we use the following notation 
\beq\label{eq:nota_RS}
R = (x^2 + y^2)^{\al/2}, \quad \b = \arctan(y/x),  \quad S = (z^2  + y^2)^{\al/2} , \quad \tau = \arctan( y /z),
\eeq
where $z$ will be used in the integral.  $\bar{\th}(x, y)$, $\bar{\xi}(R, \th) = \bar{\th}_y(x, y)$ can be obtained from $\bar{\eta}(x, y)$ (or $\bar{\th}_x $) 
as follows
\beq\label{eq:th}
\bar{\th} = 1 + \int_0^x \bar{\eta}(z, y) dz ,  \quad \bar{\xi} = \bar{\th}_y  =  \int_0^x  \bar{\eta}_y(z, y) dz.
\eeq
We can choose $\bar \th(0, y) \equiv c$ for other constant $c>0$, and it does not change $\na \bar \th$. 
Observe that 
\beq\label{eq:barxi2}
\bal
 \bar{\eta}_y(z, y) & = - \f{ 6  \al}{c}  \cdot  \f{3\al y}{y^2 + z^2} 
 \f{  (z^2 +y^2)^{\al /2}  z^{\al }} {(1+(z^2 +y^2)^{\al/2})^4}  \\
 &= - \f{1}{z}   \f{  3 \al yz }{y^2 + z^2}   \f{  (z^2 +y^2)^{\al /2}}{  1+(z^2 +y^2)^{\al/2}  }\bar{\eta}(z, y)  
 = -\f{1}{z} \f{ 3\al \sin(2 \tau ) S}{ 2 (1+S)} \bar{\eta} ,
\eal
\eeq
where we have used the notation $S, \tau$ defined in \eqref{eq:nota_RS}.
Hence, we get 
\beq\label{eq:barxi}
\bar{\xi}= \int_0^x - \f{ 6  \al}{c}  \cdot  \f{3\al y}{y^2 + z^2} 
 \f{  (z^2 +y^2)^{\al /2}  z^{\al }} {(1+(z^2 +y^2)^{\al/2})^4} dz
=  \int_0^{x} \f{1}{z} \lt( - \f{ 3\al \sin(2 \tau ) S}{2 (1+S)} \bar{\eta} \rt) dz .
 \eeq
These integrals cannot be calculated explicitly for general $\al$. We have the following estimates for $\bar{\xi}$.


\begin{lem}\label{lem:xi}
Assume that $0\leq \al \leq \f{1}{1000}$. For $R \geq0, \b \in [0, \pi/2]$ and $0 \leq i + j \leq 5$,
we have
\begin{align}
    | D^i_R D^j_{\b} \bar{\xi} | \les  - \bar{\xi}, & \quad  | D^i_R D^j_{\b} (3\bar{\xi} - R\pa_R \bar{\xi}) | \les -\bar{\xi}, \label{eq:xi0} \\
 |\bar{\xi} |   \les    \f{ \al^2(x^2 + y^2)^{\al/2}}{ (1 + (x^2 + y^2)^{\al/2}) } \f{y^{\al}}{( 1 + y^{\al})^3}  \min \lt(  1 ,  \f{x^{1+\al}}{y^{1+\al}}  \rt)  
&\les  \f{\al^2R^2}{1+R} \lt( \one_{\b < \pi /4} \f{ \sin^{\al}(\b)}{ (1+R \sin^{\al}(\b) )^3} 
+ \one_{\b \geq \pi/4} \f{\cos^{\al+1}(\b)}{(1+R)^3}  \rt)
\label{eq:xi} , \\
  -\bar{\xi } \les \al^2 \cos(\b) , \quad &  || \bar{\xi} ||_{\cC^1} \les  || \f{1+R}{R} 
( 1 + ( R \sin( 2 \b)^{\al} )^{-\f{1}{40} } ) \bar{\xi} ||_{L^{\infty}} \les \al^2, \notag
\end{align}
where $||\cdot ||_{\cC^1}$ is defined in \eqref{norm:c1}. Let $\psi_1, \psi_2$ be the weights defined in \eqref{wg}. We have 
\beq\label{eq:xi_ux}
\int_0^{\pi /2} R^2 (D^i_R D^j_{\b} \bar{\xi} )^2 \psi_k d \b  
\les \al^4  
\eeq
uniformly in $R$, and 
\beq\label{eq:xi_cw}
\la  (D^i_R D^j_{\b} ( 3\bar{\xi} - R \pa_R \bar{\xi}  ) )^2 , \psi_k \ra  
\les \al^4, \quad  \la  (D^i_R D^j_{\b}  \bar{\xi})^2 , \psi_k \ra  
\les \la \bar{\xi}^2 , \psi_k \ra \les \al^4,
\eeq
where $ (D^i_R D^j_{\b},  \psi_k )$ represents 
$ ( D^i_R, \psi_1)$ for $0\leq i \leq 5$,  and $ (D^i_R D^j_{\b},  \psi_2 )$ for $i+j \leq 5, j \geq 1$.
\end{lem}

\begin{remark}
Using \eqref{eq:barxi}, we have $ -\bar{\xi} \geq 0$ for $R \geq0, \b \in [0, \pi/2]$.
\end{remark}

We have several commutator estimates which enable us to exchange the derivative and integration in \eqref{eq:barxi} so that we can estimate $D_R^i D^j_{\b}\bar{\xi}$ easily.

Recall the relation between $\pa_x, \pa_y$ and $\pa_R, \pa_{\b}$ in \eqref{eq:simp1}. 
We have the following relation 
\beq\label{eq:deri2}
D_R = R \pa_R = \f{1}{\al} (x \pa_x  + y \pa_y), \quad D_{\b} = \sin(2\b) \pa_{\b} 
= 2 y \pa_y - 2 \al \sin^2 (\b) D_R .
\eeq
The first relation holds because $R = r^{\al}, R\pa_R = \f{1}{\al} r\pa_r$, and the second relation is obtained by multiplying $\pa_y = \f{\sin(\b)}{r} \al D_R + \f{\cos (\b)}{ r} \pa_{\b} $ by $y$ and then using $y/r = \sin(\b) , x / r = \cos(\b)$.

\begin{lem}\label{lem:com}
Suppose that $f(0, y) =0$ for any $y$. Denote 
\beq\label{eq:lemxi0}
I(f)(x ,y) =  \int_0^x \f{1}{z} f(z, y) dz.
\eeq
We have  
\begin{align}
D_R I(f)(x, y) & = I( D_S f)(x,y) , \label{eq:com1} \\
D_{\b} I(f)(x, y) - I( D_{\tau} f)(x,y) & = -2 \al \sin^2 (\b) \cdot I( D_S f) 
+  2 \al I( \sin^2(\tau)  D_S f) \label{eq:com2} ,
\end{align}
where $R, \b, S , \tau$ are defined in \eqref{eq:nota_RS}, provided that $f$ is sufficiently smooth.
\end{lem}

\begin{proof}
Notice that $y \pa_y$ commutes with the $z$ integral. From \eqref{eq:deri2}, it suffices to prove 
\[
x \pa_x I(f)(x, y) = I( z\pa_z f ).
\]
A directly calculation yields 
\[
\bal
x \pa_x I(f)(x,y) &= x \pa_x (  \int_0^x  \f{1}{z} f(z, y) dz) =   f (x, y) , \quad 
I(z \pa_z f)(x,y) =  \int_0^x \f{1}{z} \cdot z \pa_z f(z, y) dz   = f(x, y).
\eal
\]
It follows \eqref{eq:com1}. Using the fact that both $y \pa_y$ and $R\pa_R$ commute with the $z$ integral and the formula of $D_{\b}$ \eqref{eq:deri2} twice, we derive
\[
\bal
&D_{\b} I(f)(x,y) = (2 y \pa_y  - 2\al \sin^2(\b) D_R)  I(f)  
= I( 2y\pa_y f)  - 2 \al \sin^2(\b) I(D_S f)  \\
= & I( D_{\tau} f  + 2\al \sin^2(\tau) D_S f) - 2 \al \sin^2(\b) I(D_S f) 
= I(D_{\tau} f) + 2\al I( \sin^2(\tau)  D_S f)- 2 \al \sin^2(\b) I(D_S f) .
\eal
\]
\eqref{eq:com2} follows by rearranging the above identity.
\end{proof}

Next, we prove Lemma \ref{lem:xi}. 

\begin{proof}[Proof of Lemma \ref{lem:xi}]

\textit{Step 1.} 
Recall $D_R = R\pa_R, D_{\b} = \sin(2\b) \pa_{\b}$. First, we show that 
\beq\label{eq:lemxi3}
| D_R^i D^j_{\b} \bar{\xi} | \les  \al  \int_0^{x} \f{1}{z} \sin(2\tau) \f{S}{1+S}\bar{\eta}(z, y) dz  \asymp -\bar{\xi} 
\eeq
for $0\leq i+ j \leq 5$. Using $\G(\b) =\cos(\b)^{\al}$,\eqref{eq:Dg} and a direct calculation yields
\beq\label{eq:lemxi31}
\B| D^i_R \f{R^2}{(1+R)^4} \B|  \les \f{R^2}{(1+R)^4} , \quad |D^i_{\b}\G(\b)| \les \al \sin(\b) \G(\b), \quad  |D^i_{\b}\sin(2\b)| \les \sin(2\b)
\eeq
for $ i \leq 5$.
Denote 
\beq\label{eq:lemxi33}
f(S, \tau) =  \f{3\al}{2} \sin(2\tau) \f{S}{1+S} \bar{\eta} =  \f{ 9 \al^2}{c} \sin(2\tau) \G(\tau) \f{S^2}{(1+S)^4}.
\eeq
We remark that $f = - z \bar{\eta}_y(z, y)$ according to \eqref{eq:barxi2}. Obviously, $f(S, \tau ) \geq 0 $. Using the above estimates, we get 
 \beq\label{eq:lemxi32}
 | D^i_S D^j_{\tau} f| \les f 
\eeq
for $i+j \leq 5$. Notice that \eqref{eq:barxi} implies $\bar{\xi} = - I(f)$ and that $I(\cdot)$ \eqref{eq:lemxi0} is a positive linear operator for $x \geq 0$. 
We further derive 
\beq\label{eq:lemxi1}
  | I( D^i_S D^j_{\tau}f|) | \leq I( |D^i_S D^j_{\tau}f |) \les  I(f )
\eeq
for $i + j \leq 5$. Using \eqref{eq:com1} and the above estimates, we yield 
\[
|D_R^i \bar{\xi} | = | D_R^i I(f) | = | I(D_S^i f) | \les I(f).
\]
For other derivatives $D_R^i D^j_{\b}$ with $j \geq 1, i+ j \leq 5$, we estimate $D^2_{\b} \bar{\xi}$, which is representative. Using \eqref{eq:com2}, we have
\[
\bal
D^2_{\b} \bar{\xi} =& D^2_{\b} I(f) 
 = D_{\b} \lt(  I(D_{\tau} f) -2\al \sin^2(\b) \cdot I(D_S f) + 2\al I( D_S f \sin^2(\tau) ) \rt) \\
  =& I(D^2_{\tau} f) 
 -2\al \sin^2(\b) \cdot I(D_S D_{\tau} (f)) + 2\al I( \sin^2(\tau) D_S D_{\tau} f  )  \\
 &+ D_{\b} \lt(  -2\al \sin^2(\b) \cdot I(D_S f) \rt)+ D_{\b}\lt( 2\al I( D_S f \sin^2(\tau) ) \rt)
= J_1 + J_2 + J_3 + J_4 + J_5.
\eal
\]
For $J_1, J_2, J_3$, we simply use $\sin^2(\b), \sin^2(\tau) \leq 1$ and \eqref{eq:lemxi1} to obtain
\beq\label{eq:lemxi2}
I_1, J_2, J_3 \les I( | D^i_R D^j_{\tau} f |) \les I(f)
\eeq
for $(i,j) = (0,2), (1,1), (1,1)$ respectively. For $J_4$, if $D_{\b}$ acts on $\sin^2(\b)$, we obtain 
$ \al D_{\b} (\sin^2(\b)) \cdot  I(D_S f)$, which can be bounded as before using \eqref{eq:lemxi1}. For the remaining parts in $J_4$ and $J_5$, $D_{\b}$ acts on 
$I(\cdot)$ and we can use \eqref{eq:com2} again to obtain several terms. Each term can be bounded using 
\eqref{eq:lemxi1} and an argument similar to \eqref{eq:lemxi2}. The estimates of other derivatives $D_R^i D^j_{\b}$ 
can be done similarly. We omit these estimates. Since the right hand side of \eqref{eq:lemxi3} is $\f{2}{3} I(f) = -\f{2}{3}\bar{\xi} \asymp -\bar{\xi}$, the above estimates imply \eqref{eq:lemxi3}. 

\vspace{0.1in}
\textit{Step 2.}
The estimate \eqref{eq:lemxi3} can be generalized to $i+j \leq 6$ easily. 
Hence, we get 
\[
| D_R^i D^j_{\b} (3 \bar{\xi} - R\pa_R \bar{\xi} ) | \les |D_R^i D^j_{\b} \bar{\xi}|
+ |D_R^{i+1} D^j_{\b} \bar{\xi}| \les -\bar{\xi},
\]
for any $i+j\leq 5$,  which proves \eqref{eq:xi0}.


\vspace{0.1in}
\textit{Step 3: Pointwise estimate.} In this step, we prove \eqref{eq:xi}. 
From \eqref{eq:barxi}, we know that the first inequality in \eqref{eq:xi} is equivalent to 
\[
 \int_0^{x}  \f{y}{y^2 + z^2} \f{  z^{\al} (y^2 + z^2)^{\al/2} }{(1 + (y^2 +z^2)^{\al/2})^4 } d z \les   \f{  (x^2 + y^2)^{\al/2}}{ (1 + (x^2 + y^2)^{\al/2}) } \f{y^{\al}}{( 1 + y^{\al})^3}  \min \lt(  1 ,  \f{x^{1+\al}}{y^{1+\al}}  \rt)  .
\]
For $z \in [0, x]$, we have $z^2 + y^2 \leq x^2 + y^2$. Since $\f{t}{1+t}$ is increasing with respect to $t \geq 0$, we yield  
\[
\f{   (y^2 + z^2)^{\al/2} }{ 1 + (y^2 +z^2)^{\al/2} }
\les \f{ (y^2 + x^2)^{\al/2} }{ 1 + (y^2 +x^2)^{\al/2} }.
\]
Therefore, it suffices to prove 
\beq\label{eq:lemxi5}
J(x, y) \teq \int_0^{x}  \f{y}{y^2 + z^2} \f{  z^{\al} }{(1 + (y^2 +z^2)^{\al/2})^3 } d z \les   \f{y^{\al}}{( 1 + y^{\al})^3}  \min \lt(  1 ,  \f{x^{1+\al}}{y^{1+\al}}  \rt)  .
\eeq

\paragraph{Case 1 : $x \leq 1 + y$} Observe that
\[
J \leq \f{1}{ (1+y^{\al})^3 } \int_0^x \f{y z^{\al}}{y^2 + z^2} dz
  =\f{y^{\al}}{ (1+y^{\al})^3 }  \int_0^{ \f{x}{y}} \f{ t^{\al} } {1+t^2} dt ,
 \]
 where we have used change of a variable $z = y t$ to derive the identity. 
 Since $\al \leq 1/10$, we get
 \[
 \int_0^{ \f{x}{y}} \f{ t^{\al} } {1+t^2} dt \leq \int_0^{\infty} \f{ t^{\al} } {1+t^2} dt \les 1, \quad
  \int_0^{ \f{x}{y}} \f{ t^{\al} } {1+t^2} dt \leq \int_0^{\f{x}{y}} t^{\al} dt \les \f{x^{1+\al}}{y^{1+\al}}.
 \]
Combining the above estimates, we prove \eqref{eq:lemxi5} for $x \leq 1 + y$.

\paragraph{Case 2 : $ x > 1 + y$} Firstly, we have
\[
J(x, y) = \int_0^{1 + y}  \f{y}{y^2 + z^2} \f{  z^{\al} }{(1 + (y^2 +z^2)^{\al/2})^3 } d z
+ \int_{1+y}^x \f{y}{y^2 + z^2} \f{  z^{\al} }{(1 + (y^2 +z^2)^{\al/2})^3 } d z \teq J_1 + J_2.
\]
We apply the result in Case 1 to estimate $J_1$
\[
J(1 + y, y) \les \f{y^{\al}}{ (1+y^{\al})^3} \min\lt( 1, \f{(1+y)^{1+\al}}{y^{1+\al}} \rt)
\les \f{y^{\al}}{ (1+y^{\al})^3} .
\]
For $J_2$, we have 
\[
\bal
J_2 &  \leq \int_{1+y}^x \f{y}{y^2 + z^2} \f{z^{\al}}{z^{3\al}} dz 
 = y^{- 2\al} \int_{ \f{1+y}{y}}^{\f{x}{y}}  \f{t^{-2\al}}{1+t^2} dt
 \les y^{-2\al} \int_{ \f{1+y}{y} }^{\infty} t^{-2\al -2 } dt  \\
 &\les y^{-2\al}  \lt( \f{ 1+y}{y} \rt)^{ - 1 - 2\al} 
 = \f{y}{(1+y)^{1+2\al}} 
=  \f{y^{\al}}{ (1+y)^{3\al}}  \f{ y^{1-\al}}{ (1+y)^{1-\al}}
\les \f{y^{\al}} { (1 +y^{\al})^3} ,
 \eal
\]
 where we have used change of a variable $z = y t$ to derive the first identity. Noting that $x \geq y$ in this case. We conclude 
 \[
J(x, y)  = J_1 + J_2  \les \f{y^{\al}}{ (1 + y^{\al})^3} 
\leq \f{y^{\al}}{ (1+y^{\al})^3} \min\lt( 1, \f{{x}^{1+\al}}{y^{1+\al}} \rt).
 \]
Combining the above two cases, we prove \eqref{eq:lemxi5}, which implies the first inequality in \eqref{eq:xi}. 

Finally, we prove the second inequality in \eqref{eq:xi}. Using the notation \eqref{eq:nota_RS}, we have 
\[
 R = (x^2 + y^2)^{\al/2}, \quad \f{(x^2 + y^2)^{\al/2}}{1 +(x^2 + y^2)^{\al/2} } = \f{R}{1+R},
 \quad  y^{\al} =  R\sin^{\al}(\b), \quad  \f{y^{\al}}{ (1+y^{\al})^3} 
 = \f{  R\sin^{\al}(\b)} { (1 + R\sin^{\al}(\b))^3}.
\]
For $x \leq y$, we have $\b \geq \pi / 4, \ 1 \les \sin(\b), \  x^2 + y^2 \les y^2$. Hence, 
\[
\f{y^{\al}}{ (1+y^{\al})^3} \f{x^{1+\al}}{y^{1+\al}} 
\les \f{y^{\al}}{  (1 + (x^2 + y^2)^{\al/2} )^3 }  \f{x^{1+\al}}{y^{1+\al}} 
= \f{R\sin^{\al}(\b) }{ (1 +R)^3}  \cdot \f{ \cos^{1+\al}(\b)}{ \sin^{1+\al}(\b) }
\les \f{R\cos^{1+\al}(\b) }{ (1 +R)^3} .
\]
Combining the above identity and the estimate, we prove the second inequality in \eqref{eq:xi}.
The last inequality in \eqref{eq:xi} follows directly from \eqref{eq:xi0} and the first two inequalities in \eqref{eq:xi}.


\vspace{0.1in}
\textit{Step 4: Estimates of the integral} Now, we are in a position to prove \eqref{eq:xi_ux} and \eqref{eq:xi_cw}. We are going to prove 
\beq\label{eq:lemxi6}
\int_0^{\pi/2}  \bar{\xi}^2(R, \b) \psi_k d \b \les \f{\al^4}{(1+R)^2}.
\eeq
Clearly, \eqref{eq:xi_ux} and \eqref{eq:xi_cw} follow from the above estimate and \eqref{eq:xi0}.

Notice that $\psi_i$ defined in \eqref{wg} satisfies 
\beq\label{eq:lemxi51}
\psi_1, \psi_2 \les \f{(1+R)^4}{R^4} \sin(\b)^{-\s} \cos(\b)^{-\g},
\eeq
where $\g = 1 + \f{\al}{10}, \s = \f{99}{100}$. Using \eqref{eq:xi}, $1 + R\sin^{\al}(\b) \geq  (1+R) \sin^{\al}(\b) $, we yield
\[
\bal
&(1+R)^2 \int_0^{\pi/2}  |  \bar{\xi } |^2 \psi_k  d \b
\les   (1+R)^2 \f{ \al^4 R^4}{(1+R)^2}\cdot \lt\{  \int_0^{\pi/4}  \f{ \sin^{ 2\al}(\b)}{ ( ( 1+ R) \sin^{\al}(\b) )^6}  \psi_k d\b  +  \int_{\pi/4}^{\pi/2} \f{ \cos^{2\al + 2}}{(1+R)^6} \psi_k d \b \rt\} \\
 \les & \f{\al^4 R^4}{(1+R)^6} \f{(1+R)^4}{R^4}
  \lt\{ \int_0^{\pi/4}  \sin(\b)^{-4\al} \sin(\b)^{-\s} \cos(\b)^{-\g}
 +\int_{\pi/4}^{\pi/2}  \cos(\b)^{2 + 2\al }  \sin(\b)^{-\s} \cos(\b)^{-\g} d \b 
\rt\}  \\
  \les&  \al^4 
  \lt(\int_0^{\pi/4}   \sin(\b)^{-\s - 4\al} d\b  
  +\int_{\pi/4}^{\pi/2} \cos(\b)^{2 + 2\al - \g} d \b \rt) \les \al^4,
 \eal 
\]
where we have used $ \al \leq \f{1}{1000}$, $4 \al + \s < \f{199}{200}$, $2 + 2 \al - \g \geq 1$, to derive the last inequality which does not depend on $\al$ for $\al  \leq \f{1}{1000}$. It follows \eqref{eq:lemxi6}.
\end{proof}

\subsection{Other Lemmas}\label{app:lemma}
We use the following Lemma to construct small perturbation. 
\begin{lem}\label{lem:small}
Let $\chi(\cdot) : [0, \infty) \to [0, 1]$ be a smooth cutoff function, such that $\chi(R) = 1$ for $R \leq 1$ and $\chi(R) = 0$ for $R \geq 2$. Denote 
\beq\label{eq:chilam}
\chi_{\lam}(R) = \chi(R / \lam) , \quad \bar{\Om}_{\lam} = \chi_{\lam} \bar{\Om}, \quad \bar{\eta}_{\lam} = \pa_x ( \chi_{\lam} \bar{\th}) ,
\quad \bar{\xi}_{\lam} =\pa_{y} (\chi_{\lam} \bar{\th} ), 
\eeq
where $\bar{\th}$ is obtained in \eqref{eq:th}. We have
\beq\label{eq:lim1}
\lim_{\lam \to +\infty}
|| \bar{\Om}_{\lam} - \bar{\Om}||_{\cH^3} +
 || (1+R) ( \bar{\eta}_{\lam} -\bar{\eta} )||_{\cH^3}  + || \bar{\xi}_{\lam} -\bar{\xi}||_{\cH^3(\psi)}
 = 0, \quad   
\overline{\lim}_{\lam \to +\infty}    || \bar{\xi}_{\lam} -\bar{\xi}  ||_{\cC^1}  \leq K_{10} \al^{ 2},
\eeq
where $K_{10} >0$ is some absolute constant. In particular, we also have 
\beq\label{eq:lim2}
\lim_{\lam \to +\infty} L^2_{12}(\bar{\Om}_{\lam} - \bar{\Om})(0) + \la  (\bar{\Om}_{\lam} - \bar{\Om})^2, \vp_0 \ra
+ \la  (\bar{\eta}_{\lam} - \bar{\eta} )^2, \psi_0 \ra = 0.
\eeq
\end{lem}
We need a Lemma similar to Lemma \ref{lem:com}.
\begin{lem}\label{lem:com2}
Suppose that $f(0, y) =0$ for any $y$. Denote  $J(f)(x ,y) = \f{1}{z} \int_0^x f(z, y) dz.$
We have  
\[
\bal
D_R J(f)(x, y) & = J( D_S f)(x,y) ,  \\ 
D_{\b} J(f)(x, y) - J( D_{\tau} f)(x,y)  & = -2 \al \sin^2 (\b) \cdot J( D_S f) 
+  2 \al J( \sin^2(\tau)  D_S f) ,
\eal
\]
where $R, \b, S , \tau$ are defined in \eqref{eq:nota_RS}, provided that $f$ is sufficiently smooth.
\end{lem}
The first identity follows from a direct calculation and the proof of the second is similar to that in Lemma \ref{lem:com}. We omit the proof.

\begin{proof}[Proof of Lemma \ref{lem:small}]

\textit{Step 1: Estimate of $J(\bar  \eta)$.} 
Using \eqref{eq:th} and the operator $J$ in Lemma \ref{lem:com2}, we get $\bar \th  = 1 + x J(\bar{\eta})$. We have the following estimate for $J(\bar \eta)$
\beq\label{eq:lemsm1}
 | D^i_R  D^j_{\b}  J(\bar{\eta}) |  \les  J(\bar{\eta}) = \f{1}{x} \int_0^x \bar{\eta}(z, y) dz \les \bar{\eta},
\eeq
for $0 \leq i+ j \leq 5$. The proof of the first inequality follows from Lemma \ref{lem:com2} and the argument in the proof of \eqref{eq:lemxi3}. The proof of the second inequality is similar to that of \eqref{eq:lemxi5} by considering $x \leq 1+ y$ and $ x > 1+ y$. We omit the proof. 

\textit{Step 2: Estimate of $\bar{\eta}_{\lam} -\bar{\eta}, \bar{\xi}_{\lam}-\bar{\xi}$.} Recall $\bar{\eta}_{\lam} = \pa_ x(  \chi_{\lam}  \bar{\th})$, $\bar{\xi}_{\lam} = \pa_ y(  \chi_{\lam}  \bar{\th})$ and the formula of $\pa_x , \pa_y$ \eqref{eq:simp1}. A direct calculation yields 
\beq\label{eq:lemsm2}
\bal
\bar{\eta}_{\lam}(R, \b) -\bar{\eta} & 
 = \al \f{\cos(\b)}{r} D_R \chi_{\lam} \cdot \bar{\th}  + (\chi_{\lam} - 1) \bar{\eta}
 =   \al \cos^2(\b) D_R \chi_{\lam} \cdot J(\bar{\eta}) 
+ \al \f{\cos(\b)}{r} D_{R} \chi_{\lam} + (\chi_{\lam} - 1) \bar{\eta},  \\
\bar{\xi}_{\lam}(R, \b) -\bar{\xi}& 
 = \al\f{\sin(\b) }{r} D_R \chi_{\lam} \cdot \bar{\th}  + (\chi_{\lam} - 1) \bar{\xi}
  =  \al \sin(\b) \cos(\b) D_R \chi_{\lam} \cdot J(\bar{\eta}) 
+ \al \f{\sin(\b)}{r} D_{R} \chi_{\lam} + (\chi_{\lam} - 1) \bar{\xi}, 
 \eal
 \eeq
where we have used $\pa_x \bar{\th} = \bar{\eta}, \pa_y \bar{\th} = \bar{\xi} , 
\ \bar \th = 1 + x J(\bar \eta) 
= 1 + r \cos(\b) J(\bar \eta)$.
From \eqref{eq:chilam}, we have
\[
 D_{\b} \chi_{\lam} = 0,  \quad | D_R \chi_{\lam}  | =  (R / \lam)  | \chi^{\prime}( R /  \lam)  | \les 1.
\]
Similarly, we have 
\beq\label{eq:lemsm3}
| D^k_R \chi_{\lam}|  \les 1 ,
\eeq
for $ k =1,2,3,4$. 
For $G_{\lam}(R, \b) = \al \f{g(\b)}{r} D_R \chi_{\lam} $ with $g(\b) = \sin(\b)$ or $\cos(\b)$ and $i+j\leq 3$, since $R \geq \lam$ in $\supp( D_R^k \chi_{\lam})$ and $r = R^{-1/\al}$, we get 
\[
|D_{\b}^i D_R^j G_{\lam}| \les_{\al}    R^{-1/\al} \one_{R \geq \lam}, 
\quad |D_{\b}^{i+1} D_R^j G_{\lam}|  \les_{\al} \sin(2\b) R^{-1/\al} \one_{R \geq \lam}.
\]
Recall the $\cC^1$ norm \eqref{norm:c1} and $\cH^3$  \eqref{norm:H2}. Using the fast decay of $G_{\lam}$ in $R$ and the smoothness in $\b$, we get
\beq\label{eq:lemsm_add}
\lim_{\lam \to \infty} || (1 + R) G_{\lam} ||_{\cH^3} = 0, \quad \lim_{\lam \to \infty} || G_{\lam} ||_{\cC^1}  = 0.
\eeq
Notice that $\pa_R \chi_{\lam}, \ (\chi_{\lam} - 1) = 0$ for $R \leq \lam$. From the formula of $\bar{\eta}$ and \eqref{eq:xi_cw} in Lemma \ref{lem:xi}, we know $(\chi_1 -1)  (1+R)\bar{\eta}\in \cH^3$ ($\bar{\eta}$ decays $R^{-2}$ for large $R$) and $\bar{\xi} \in \cH^3(\psi)$. Using the estimates of $J(\bar{\eta})$ in \eqref{eq:lemsm1}, we also have $(\chi_1 - 1) J(\bar{\eta}) \in \cH^3 \subset \cH^3(\psi)$. Therefore, applying \eqref{eq:lemsm2}, \eqref{eq:lemsm3} to $\chi_{\lam}$, \eqref{eq:lemsm_add}, the fact that the $\cH^3$ norm is stronger than the $\cH^3(\psi)$ norm \eqref{norm:H22}, and the Dominated Convergence Theorem yields 
\[
 \lim_{\lam \to \infty }|| (1+R) (\bar{\eta}_{\lam} - \bar{\eta}) ||_{\cH^3} = 0 , \quad  \lim_{\lam \to \infty }|| \bar{\xi}_{\lam} - \bar{\xi} ||_{\cH^3(\psi)} = 0 .
\]

Similarly, we have 
\[
\lim_{\lam \to \infty} ||  \bar{\Om}_{\lam} - \bar{\Om} ||_{\cH^3} = 0.
\]

Using \eqref{eq:lemsm1}, \eqref{eq:lemsm3} and the fact that $\bar{\eta}$ decays for large $R$ (see \eqref{eq:profile}), we have 
\[
\overline{\lim }_{\lam \to \infty}  ||  \sin(\b) \cos(\b)  D_R \chi_{\lam} \cdot J(\bar{\eta}) ||_{\cC^1} = 0.
\]
Using \eqref{eq:xi0}-\eqref{eq:xi} in Lemma \ref{lem:xi}, and \eqref{eq:lemsm2}-\eqref{eq:lemsm_add}, we conclude 
\[
|| (\chi_{\lam} - 1)  \bar{\xi}||_{\cC^2} \les \al^2, \quad 
\overline{\lim}_{\lam \to +\infty}    || \bar{\xi}_{\lam} -\bar{\xi}  ||_{\cC^1} \les \al^2.
\]
We complete the proof of \eqref{eq:lim1}. 

Recall that the $\cH^3$ norm is stronger than $L^2(\vp_1)$. Using Lemma \ref{lem:l12} for $L_{12}(\Om)(0)$, the fact that $\vp_0 \les \vp_1,  \psi_0 \les (1+R) \vp_1$ (see Definition \ref{def:wg}, \ref{def:L2}) and the limit obtained in \eqref{eq:lim1}, we prove \eqref{eq:lim2}.
\end{proof}

Let $C^{ \f{\al}{40}}$ be the standard H\"older space. Recall the $\cC^1$ norm defined in \eqref{norm:c1}. We have the following embedding.
\begin{lem}\label{lem:h_embed}
Suppose that $f \in \cC^1(R, \b)$ and $f(R, \pi/2) = 0$ for $R \geq 0$. We have
\[
|| f||_{C^{  \f{\al}{ 40} }} \leq C_{\al} || f ||_{\cC^1}
\]
for some constant $C_{\al}$ depending on $\al$ only.
\end{lem}

\begin{proof}
Recall the relation between the Cartesian coordinates $(x, y)$ and the polar coordinates $(r, \b), (R, \b)$. Since $f$ vanishes on the axis $\b = \f{\pi}{2}$. It suffices to prove that $f$ is H\"older in $\R^2_{++}$. Let $  (R_1, \b_1),  (R_2, \b_2 )$ be arbitrary two different points in $\R^2_{++}$, i.e. $R_1, R_2 \geq 0, \b_1, \b_2 \in [0, \pi /2]$, and $ r_1 = R_1^{ 1 /\al}, r_2 = R_2^{1/\al}$. Without loss of generality, we assume $R_1 \leq R_2$, $\b_1 \leq \b_2$ and $||f ||_{\cC^1} = 1$. 
From \eqref{norm:c1}, we have $ | f | \leq 1, |\pa_R f| \leq \f{1}{1+R} , \ 
| \pa_{\b} f  | \leq  R^{1/40} \sin(2\b)^{\al/40 -1}$. 
 Using 
\[
 \sin(2\b)^{\al / 40 -1} 
\les  ( \sin(\b)^{\al/40-1} + \cos(\b)^{\al/ 40 -1} )
\les  (  \b^{\al /40-1} + (\pi/ 2 -\b)^{\al/40 -1} )
\]
and the estimates of the derivatives, we obtain 
\[
\bal
| f(R_1, \b_1)  - f(R_1, \b_2) |
&\leq \int_{\b_1}^{\b_2}
| \pa_{\b} f(R_1, \b) | d \b 
\leq  C R_1^{ \f{1}{40}} \int_{\b_1}^{\b_2} \lt( \b^{ \f{\al}{40}-1} + ( \f{\pi}{2} -\b)^{ \f{\al}{40} -1} \rt) d \b \\
&\leq C_{\al} R_1^{ \f{1}{40}} 
( \b_2^{ \f{\al}{40}} - \b_1^{ \f{\al}{40} }
+ ( \f{\pi}{2} - \b_1)^{ \f{\al}{40}} - 
( \f{\pi}{2} - \b_2)^{ \f{\al}{40}} )
\leq C_{\al}  R_1^{ \f{1}{40}}  |\b_2 - \b_1|^{\f{\al}{40}} , \\
| f(R_1, \b_2)  - f(R_2, \b_2) |
&\leq \int_{R_1}^{R_2} |\pa_R f(R, \b_2) | d R
\leq \int_{R_1}^{R_2} \f{1}{1+R} d R = \log \f{1+R_2}{1+R_1}
\les (R_2 - R_1)^{1/40},
\eal
\]
where we have used $\log \f{1+R_2}{1+R_1}\leq \log(1 + R_2 - R_1)$ and $\log(1 + x) \les x^{1/40}$ for $x \geq 0$ in the last inequality. The distance $d$ between two points is
\[
\bal
d^2 & = (r_1 \cos(\b_1) - r_2 \cos(\b_2) )^2
+ (r_1 \sin(\b_1) - r_2 \sin(\b_2) )^2
= (r_1 - r_2)^2 + 2r_1 r_2 (1 - \cos(\b_1 - \b_2)) \\
& =  |R_1^{1/\al} - R_2^{1/\al}|^2 + 4 R_1^{1/\al} R^{1/\al}_2 \sin( \f{1}{2} (\b_1 - \b_2))^2
\geq C_{\al} (  | R_1  - R_2|^{2/\al}  +  R_1^{2/\al} |\b_1 -\b_2|^2 ),
\eal
\]
where we have used $R_1 \leq R_2$ in the last inequality. Using the triangle inequality and the above estimates, we conclude $| f(R_1, \b_1) - f(R_2, \b_2)|
\les C_{\al} d^{ \f{\al} {40}}$.
\end{proof}

\subsection{ Proof of Lemmas on the stream function}\label{app:stream}

\begin{proof}[Proof of Lemma \ref{lem:biot1}]
We simplify $\om^{\th}$ as $\om$ and denote by $\vth = \arctan(x_2  /x_1) $ the angular variable. Recall the cylinder $D_1 = \{ (r, z) : r \in [0,1], |z| \leq 1\}$. We extend $\om \one_{(r,z)\in D_1}$ to $\R^3$ as follows :
\beq\label{eq:w_ext}
\om_e( r, z) = \om(r, z)  \mathrm{ \ for\ } (r, z) \in D_1, \quad \om_e(r, z) = 0
 \mathrm{ \ for \ }  (r,z) \notin D_1.
\eeq
Note that $\om_e$ is only supported in $D_1$, which is different from $\om$. Denote 
\beq\label{eq:w_ext2}
\bal
&\om_{\pm} = \max( \pm \om_e, 0), \quad
 \cL =  -\pa_{rr} - \f{1}{r} \pa_r - \pa_{zz} + \f{1}{r^2},  \quad \D = \pa_{rr} + \f{1}{r} \pa_r +  \pa_{zz} + \f{1}{r^2} \pa_{\vth \vth} , \\
&\psi_{\pm}(r,z) = \f{1}{4\pi} 	\int_0^{\infty} \int_{-\infty}^{\infty} 
\int_0^{2\pi} \f{ \sin(\vth) \om_{\pm}(r_1, z_1)  } { \lt( (z-z_1)^2 + r^2 + r_1^2 - 2 \sin(\vth) r r_1 \rt)^{1/2} } r_1 d r_1 dz_1 d \vth,
\eal
\eeq
where $\D$ is the Laplace operator in $\R^3$ in cylindrical coordinates. Clearly, $\psi_{\pm}$ solve the Poisson equation in $\R^3$: $-\D (\sin(\vth) \psi_{\pm}(r, z) ) = \om_{\pm}(r,z) \sin(\vth)$, which can be verified easily using the Green function of $-\D$. Since $\om_{\pm} \geq 0$, using the above formula and 
$\f{ \sin(\vth)   } { \lt( (z-z_1)^2 + r^2 + r_1^2 - 2 \sin(\vth) r r_1 \rt)^{1/2} } 
- \f{ \sin(\vth)   } { \lt( (z-z_1)^2 + r^2 + r_1^2 + 2 \sin(\vth) r r_1 \rt)^{1/2} } \geq 0$ for $\vth \in [0, \pi ]$, we get $\psi_{\pm} \geq 0$.

Let $\td{\psi}$ be a solution of \eqref{eq:euler2}-\eqref{eq:euler21}. By definition of $\cL$, we have 
\[
-\D( \td{\psi} \sin(\vth)) = \sin(\vth) \cL \td{\psi} = \om \sin(\vth).
\]
Consider the domain $D_1^+ = \{ (r, z, \vth) : r \in [0, 1], |z|\leq 1 , \vth \in [0, \pi] \}$, which is a half of the cylinder $D_1$. Next, we compare $\td \psi \sin(\vth)$ and $\psi_+ \sin(\vth)$ in $D_1^+$ using the maximal principle for the Laplace operator $\D$.

Recall from \eqref{eq:w_ext} that $\om_e = \om$ in $D_1^+ \subset D_1$. For $(r, z , \vth) \in D_1^+$, we have $\sin(\vth) \geq 0$ and
\beq\label{eq:max}
-\D ( (\td{\psi} - \psi_+ ) \sin(\vth ) ) = (\om - \om_+) \sin(\vth )  \leq 0.
\eeq

On the boundary of $\pa D_1^+$, we have $\vth \in \{ 0, \pi\}$, $r=1$ or $z \in \{-1, 1\}$. The boundary related to $\vth \in \{0, \pi \}$ is $\{ (r,z, \vth) : r\in [0,1], |z|\leq 1, \vth=0, \pi\}$, or equivalently $\{ (x,y,z) : |x|\leq 1, y=0, |z|\leq 1\}$ in the Cartesian coordinates. It contains the symmetry axis $r=0$.
Recall that $\td \psi$ is odd  and 2-periodic in $z$. We obtain \eqref{eq:euler22} $\td \psi( r, \pm 1) =0$. Recall the boundary condition \eqref{eq:euler21} $\td{\psi}(1, z)=0$ and the fact that $\psi^+$ is nonnegative. We have
\[
 (\td{\psi} - \psi_+ ) \sin(\vth )  = 0   \mathrm{ \quad for \quad } \vth \in \{ 0, \pi\}, \qquad 
  (\td{\psi} - \psi_+ ) \sin(\vth )  \leq 0   \mathrm{ \quad for \quad }   r = 1 \textrm{ \ or \ } z \in \{ -1, 1\} ,
\]
where we have used $\sin(\vth) \geq 0$ in $D_1^+$. 
Applying the maximal principle to \eqref{eq:max} in the bounded domain $D_1^+$, we yield $(\td{\psi}(r,z) - \psi_+(r,z) ) \sin(\vth ) \leq 0 $ in $D_1^+$, which further implies $\td{\psi}(r,z) \leq \psi_+(r,z)$ for $r\leq 1, |z|\leq 1 $. Similarly, we have $\td{\psi} + \psi_- \geq 0$. Hence $ | \td{\psi}| \leq \psi_+ + \psi_-$.

Recall from \eqref{eq:w_ext},\eqref{eq:w_ext2} that $\supp( \om_{\pm} ) \subset \supp(\om) \cap D_1$ and the assumption $\supp(\om) \cap D_1 \subset \{ (r ,z) : (r-1)^2 + z^2 < 1/4  \}$ in Lemma \ref{lem:biot1}. Thus, for $r > \f{1}{4}$,  $(r_1, z_1)$ in the support of $\om_{\pm}$ and $|\vth| \leq \pi$, we have $r_1 > \f{1}{2}$ and 
\[
(z-z_1)^2 + r^2 + r_1^2 - 2 \cos(\vth) r r_1 
= (z-z_1)^2 + (r- r_1)^2 + 4 \sin^2(\vth / 2) r r_1 \asymp( ( (z-z_1)^2 + (r- r_1)^2)^{1/2} +  |\vth|)^2.
\]
We have similar estimate with $\cos(\vth)$ replaced by $\sin(\vth)$.
Using this estimate and integrating the $\vth$ variable in the integral about $\psi_{\pm}$ in \eqref{eq:w_ext2}, we complete the proof.
\end{proof}

\begin{remark}
The above proof can also be established in the Cartesian coordinates, which is essentially the same up to change of variables. 
\end{remark}

\begin{proof}[Proof of Lemma \ref{lem:far}]
Recall $r = 1 -C_l y = 1 -C_l \rho \sin(\b)$ in \eqref{eq:label_rs} and $\td{\om }(r, z) = \om^{\th}(r, z) / r$ in \eqref{eq:omth}. Since the support size satisfies $C_l(\tau) S(\tau) < 1/4$, within the support of $\om^{\th}(r, z)$, we have $r \geq 1/2$. 
Hence, $| \td{\om} | \asymp  | \om^{\th} | $ and $|r \om^{\th}(r,z)| \les | \td \om(r,z)| $. From \eqref{eq:euler_polar} and \eqref{eq:label_rs}, $R \leq (2C_l)^{-\al}$ implies $\rho \leq (2C_l)^{-1}$ and $r \geq 1-\f{1}{2} =\f{1}{2}$. Therefore, for $(x, y)$ with $R = (x^2 + y^2)^{\al/2} \leq (2C_l)^{-\al}$, we can apply Lemma \ref{lem:biot1} and \eqref{eq:rescal41} to get
\[
\bal
|\psi(x, y) | & \les C_{\om} C_l^{-2} \int   |\td{\om}(r_1, z_1)| \B( 1+ |\log( (r_1- (1 - C_l y) )^2 + (z_1 - C_l x)^2) | \B)  d r_1 d z_1  \\
&=  C_{\om}   \int  |\td{\om}(1 - C_l y_1, C_l x_1) | \B( 1 + | \log ( C_l^2 ( ( y_1 -y)^2 + (x_1- x)^2 ) | \B) d y_1 d x_1, 
\eal
\]
where we have used Lemma \ref{lem:biot1} and $r_1 \leq 1 $ in the first inequality, and used change of variables $r_1 = 1 - C_l y_1, z_1 = C_l x_1$ in the second identity. From \eqref{eq:rescal41}, $C_{\om} \td{\om}(1 - C_l y_1, C_l x_1) $ in the integrand becomes $\om(x_1, y_1)$. For $(x_1, y_1)$ within the support of $\om$, we have $x_1^2 + y_1^2 \leq ( S(\tau))^2$. Hence, for any $ x^2 + y^2> 4 ( S(\tau) )^2 $, or equivalently, $R> (2S(\tau))^{\al}$, we get
\[
(x_1 - x)^2 + (y_1 - y)^2 \asymp  x^2 + y^2 = \rho^2 = R^{2/\al}, \quad  
1 + |\log ( C_l^2 ( ( y_1 -y)^2 + (x_1- x)^2 ) | \les
|\log (  C_l^2  R^{2/\al} ) | + 1.
\]
Using $ \Psi(R, \b) = \f{1}{\rho^2} \psi(x, y) = R^{-2/\al} \psi(x, y)$ and the above estimates, we get
\[
|\Psi(R, \b)| \les R^{-2/\al} (1 + |\log(  C_l^2 R^{2/\al} ) | ) \int | \om(x_1, y_1)  | d x_1 d y_1.
\]
Passing to the $(R, \b)$ coordinates, we have $dx_1 dy_1 = \td{\rho} d \td{ \rho} d \td{\b} =\al^{-1} \td{R}^{2/\al - 1} d \td{R} d \td{\b}$. Using the Cauchy-Schwarz inequality, we get
\[
\bal
 \int | \om(x_1, y_1)  | d x_1 d y_1 &= 
 \al^{-1} \int_0^{\pi}\int_{\td{R} \leq ( S(\tau) )^{\al} } | \Om( \td{R}, \td{\b})  | \td{R}^{\f{2}{\al}- 1} d \td{R} d \td{\b}
 \les \al^{-1} || \Om||_{L^2}  \B( \int_0^{\pi}\int_0^{ ( S(\tau) )^{\al}} \td{R}^{ \f{4}{\al} - 2} d \td{R} d \b \B)^{1/2} \\
 & \les \al^{-1/2} || \Om||_{L^2}  ( S(\tau) )^{2 - \al/2}.
 \eal
\]
It follows that
\[
| \Psi(R, \b)| \les R^{-2/ \al} (1 + |\log(  C_l^2 R^{2/\al} ) |) \al^{-1/2} || \Om||_{L^2}  ( S(\tau) )^{2 - \al/2}.
\]

For $ M \leq R \leq (2C_l)^{-\al} $, we have $ C_l M^{1/\al}< C_l R^{1/\al}  <1$ and $|\log(  C_l^2 R^{2/\al} ) | \leq |\log(  C_l^2 M^{2/\al} ) |
\les |\log(  C_l M^{1/\al} ) |$. Integrating $|\Psi(R,\b)|^2$ over $(R, \b) \in [M, (2C_l)^{-\al}  ] \times [0, \pi / 2]$ yields 
\[
\bal
&\int_{0}^{\pi/2}\int_M^{(2C_l)^{-\al} }
 | \Psi(R, \b)|^2 d R d \b \les
 (1+|\log(  C_l M^{1/\al} ) |)^2 \al^{-1} S^{4 - \al}  || \Om||_{L^2}^2 \int_M^{\infty} R^{-4/\al} dR  \\
  \les &  (1+|\log(  C_l M^{1/\al} ) |)^2  S^{4 - \al} M^{-4/\al + 1} || \Om||_{L^2}^2
  \les (1+|\log(  C_l M^{1/\al} ) |)^2 (S / M^{1/\al})^2|| \Om||_{L^2}^2,
 \eal
\]
where we have used the assumption $M^{1/\al} > S$ and $S^{4 - \al} M^{-4/\al + 1} = ( S / M^{1/\al})^{4-\al} \leq ( S / M^{1/\al})^{2}$ to obtain the last inequality. We conclude the proof. 
\end{proof}

\subsection{ A toy model for 2D Boussinesq}\label{app:toy}

We consider the toy model introduced in \cite{elgindi2019finite}
\[
\bal
&\om_t - (x_1 \lam(t), -x_2 \lam(t)) \cdot \na \om = \pa_1 \th ,   \\
&  \th_t - (x_1 \lam(t), -x_2 \lam(t)) \cdot \na \th =  0 , \quad \lam(t) = \int_{\R^2} \f{y_1 y_2}{ |y|^4} \om(y, t) dy ,\\
\eal
\]
where $\pa_1 \th = \pa_{x_1 }\th$. This model can be derived from the 2D Boussinesq equations by approximating the velocity $(u,v)$ by $u_{x_1}(0,0,t) \cdot (x_1, -x_2)$ and rescale the solution by a constant. We assume that $\om$ is odd in $x_1$ and $x_2$, and $\th$ is even in $x_1$ and odd in $x_2$. We show that for initial data $\om_0 ,  \na \th_0\in C_c^{\al}(\R_2) $, the solution exists globally. We follow the argument in \cite{elgindi2019finite}. Without loss of generality, we assume $\supp( \pa_1 \th_0) \subset [-1,1]^2$. Using the derivation in \cite{elgindi2019finite}, we get 
\beq\label{eq:toy_solution}
\bal
&\om(x_1, x_2, t) = (\pa_1 \th_0) ( \mu(t) x_1, \f{x_2}{\mu(t)}) \int_0^t \mu(s) ds,  \quad \mu(t) \teq \exp( \int_0^t \lam(s) ds) , \\
&\f{ \dot \mu}{\mu} = 4 \int_0^t \mu(s) ds  J(t), 
\quad J(t) \teq  \int_0^{\infty} \int_0^{\infty} \f{y_1 y_2}{ |y|^4} ( \pa_1 \th_{0} ) ( \mu(t) y_1, \f{y_2}{\mu(t)} ) d y_1 dy_2.
\eal
\eeq



Next, we estimate $J(t)$. Denote $\td \th(x_1,x_2) = \th_0(x_1, x_2) - \th_0(0, x_2)$. Clearly, we have $ \pa_1 \td \th = \pa_1 \th$.
We simplify $\mu(t)$ as $\mu$. Since $ (\pa_1 \td \th_{0} )( \mu y_1, \f{y_2}{\mu }) =\mu^{-1} \pa_1( \td \th_0( \mu y_1, \f{y_2}{\mu}) )$, $\supp( \pa_1 \td \th_{0}) = \supp( \pa_1 \th_{0}) \subset [-1,1]^2$, using integration by parts and $\pa_1 \f{y_1y_2}{|y|^4} = \f{y_2(y_2^2 - 3y_1^2)}{|y|^6}$, we yield 
\[
\bal
J  &= \mu^{-1} \int_0^{ \mu^{-1}} \int_0^{\infty} \f{y_1 y_2}{ |y|^4}  \pa_1 \B( \td \th_{0}( \mu(t) y_1, \f{y_2}{\mu(t)} ) \B) d y_1 dy_2 \\
&=  \mu^{-1} \int_0^{\infty} \f{ \mu^{-1} y_2}{ (\mu^{-2} + y_2^2)^2}  \td \th_0( 1, \f{y_2}{\mu}) dy_2
- \mu^{-1} \int_0^{\mu^{-1}} \int_0^{\infty} \f{y_2(y_2^2 - 3y_1^2)}{|y|^6} \td \th_0( \mu y_1, \f{y_2}{\mu}) dy_1 dy_2  \teq J_1 + J_2.
\eal
\]

Since $\td \th_0 \in C^{1,\al}$, $\td \th_0(0, x_2) = 0$ and $\td \th_0(x_1, 0) = 0$, we have $| \td \th_0( x_1, x_2)| \les |x_1|^{\al} |x_2|$. It follows 
\[
\bal
|J_1| &\les \mu^{-1} \int_0^{\infty} \f{\mu^{-1} y_2}{ (\mu^{-2} + y_2^2)^2} \f{y_2}{\mu} dy_2
 = \mu^{-2}  \int_0^{\infty} \f{z^2}{ (1+z^2)^2} dz \les \mu^{-2} , \\
 |J_2| & \les \mu^{-2}\int_0^{\mu^{-1}}  (\mu y_1)^{\al} \int_0^{\infty} 
  \B|\f{y_2^2 ( y_2^2 - 3y_1^2 ) }{|y|^6}\B| dy_2 dy_1
  \les \mu^{-2}\int_0^{\mu^{-1}}  (\mu y_1)^{\al} y_1^{-1} dy_1 
  \les_{\al} \mu^{-2}. 
 \eal
\]

Plugging the above estimates in \eqref{eq:toy_solution}, we obtain 
\[
\B|\f{ \dot \mu}{\mu} \B| \les_{\al}  \mu^{-2} \int_0^t \mu(s) ds .
\]
Thus, $\mu$ remains bounded for all time. Formula \eqref{eq:toy_solution} implies that the solution exists globally.

\bibliographystyle{plain}
\bibliography{selfsimilar}

\end{document}